\documentclass{article}

\newif\ifshownavigationpage
\newif\ifshowreminders
\newif\ifshownotationindex
\newif\ifshowtheoremlinks
\newif\ifshowtheoremtree
\newif\ifshowtheoremlist
\newif\ifshowequationlist

\newif\ifshowcomments
\newif\ifhighlight 
\newif\ifelaborate

\newif\ifshowaddressedcomments
\newif\ifshowrvin
\newif\ifshowrvout

\shownotationindextrue
\showrvintrue
\showrvouttrue

\showcommentstrue

\usepackage{amsthm, amsmath, amssymb}
\usepackage{thmtools}
\usepackage{soul}
\usepackage{authblk}
\usepackage{framed}
\usepackage{mathrsfs}
\usepackage{hyperref}
\hypersetup{
    colorlinks,
    linkcolor={red!50!black},
    citecolor={blue!50!black},
    urlcolor={blue!80!black}
}

\usepackage[dvipsnames]{xcolor}
\usepackage{graphicx, multicol}
\usepackage{marvosym, wasysym}
\usepackage{fancyhdr}
\usepackage{stackengine}
\usepackage{algorithm}
\usepackage{bm}
\usepackage[noend]{algpseudocode}
\usepackage{bbm}
\usepackage{verbatim}
\usepackage{mdframed}
\usepackage{needspace}
\makeatletter
\renewcommand{\ALG@beginalgorithmic}{\scriptsize}
\makeatother
\usepackage{xcolor}
\allowdisplaybreaks

\DeclareFontFamily{U}{mathx}{}
\DeclareFontShape{U}{mathx}{m}{n}{ <-> mathx10 }{}
\DeclareSymbolFont{mathx}{U}{mathx}{m}{n}
\DeclareFontSubstitution{U}{mathx}{m}{n}

\renewcommand{\dj}[1]{\bm{d}_{J_1}^{_{#1}}}

\ifhighlight
\else
    
\fi

\newcommand{\rvoutopacity}{20}
\ifshowrvin

\else

\fi

\ifshowrvout
    \newcommand{\rvout}[1]{{\color{red!\rvoutopacity}{#1}} }
    \newcommand{\chrout}[1]{{\color{blue!\rvoutopacity}{#1}} }    
    \newcommand{\rvoutm}[1]{{\color{black!\rvoutopacity}{\ifmmode\text{\sout{\ensuremath{\displaystyle#1}}}\else\sout{#1}\fi}} } 
\else
    \newcommand{\rvout}[1]{}
    \newcommand{\chrout}[1]{}
\fi

\ifshowcomments
    \newcommand{\summ}[1]{{\color{blue}[summary: #1]} } 
    \newcommand{\chr}[1]{{\color{PineGreen}[CR: #1]} } 
    \newcommand{\xw}[1]{{\color{RoyalBlue}[XW: #1]} } 
    \ifshowaddressedcomments
        \newcommand{\chra}[1]{{\color{PineGreen}\sout{[CR: #1]}} } 
        \newcommand{\xwa}[1]{{\color{RoyalBlue}\sout{[XW: #1]}} } 
    \else
        \newcommand{\chra}[1]{} 
        \newcommand{\xwa}[1]{} 
    \fi
\else
    \newcommand{\summ}[1]{} 
    \newcommand{\chr}[1]{} 
    \newcommand{\chra}[1]{} 
    \newcommand{\xw}[1]{} 
    \newcommand{\xwa}[1]{} 
\fi

\usepackage[shortlabels]{enumitem}

\newlist{thmdependence}{itemize}{10}
\setlist[thmdependence]{nosep,label=-}

\newcommand{\thmtreenode}[5]{\item[#1] \linkdest{location, thm tree #3} {#2}~\ref{#3} \linktopf{#3} \thmsum{#4}{#5}}
\newcommand{\thmtreenodewopf}[5]{\item[#1] \linkdest{location, thm tree #3} {#2}~\ref{#3} \thmsum{#4}{#5}}
\newcommand{\thmtreeref}[2]{\item[\elsewhere] {{\hyperlink{location, thm tree #2}{\color{gray}#1}}}~\ref{#2}\thmsum{0.5}{}}

\ifshowtheoremlinks
    \newcommand{\linksinthm}[1]{\emph{\linkdest{location, #1}\linktopf{#1} \linktothmtree{location, thm tree #1} }}
    \newcommand{\linksinthmwopf}[1]{\emph{\linkdest{location, #1} \linktothmtree{location, thm tree #1} }}
    \newcommand{\linksinpf}[1]{\linkdest{location, proof of #1}\linktothm{#1} \linktothmtree{location, thm tree #1} }
\else
    \newcommand{\linksinthm}[1]{}
    \newcommand{\linksinthmwopf}[1]{}
    \newcommand{\linksinpf}[1]{}
\fi

\ifshownotationindex
    \newcommand{\notationdef}[2]{\linkdest{location, notation definition of #1}\hyperlink{location, notation index of #1}{#2}}
    \newcommand{\notationidx}[2]{\linkdest{location, notation index of #1}\hyperlink{location, notation definition of #1}{#2}}
\else
    \newcommand{\notationdef}[2]{#2}
\fi
    
\newcommand{\linktopf}[1]{\hyperlink{location, proof of #1}{\pflinksymbol}}
\newcommand{\linktothm}[1]{\hyperlink{location, #1}{\thmlinksymbol}}
\newcommand{\linktothmtree}[1]{\hyperlink{#1}{\thmtreelinksymbol}}
\newcommand{\thmlinksymbol}{{\tiny [Theorem]}}
\newcommand{\pflinksymbol}{{\tiny [Proof]}}
\newcommand{\thmtreelinksymbol}{{\tiny [ThmTree]}}

\newcommand{\complete}{{\color{black}\checkmark}}

\newcommand{\issue}{{\color{red}\checkmark}}
\newcommand{\elsewhere}{}

\newcommand{\thmsum}[2]{\quad{\color{gray}\begin{minipage}[t]{#1\linewidth}{#2}\vspace{0.5\baselineskip}\end{minipage}}}

\makeatletter
\newcommand{\linkdest}[1]{%
    \Hy@raisedlink{\raisebox{5mm}[0pt][0pt]{\hypertarget{#1}{}}}%
}
\makeatother

\newcommand{\elaborateopacity}{50}
\newcommand{\elaboratecolor}{RawSienna}
\ifelaborate
    \newcommand{\elaborate}[1]{{\color{\elaboratecolor!\elaborateopacity}{
    \begin{framed}
    \noindent {\footnotesize[Elaboration]}
    #1 
    \end{framed}
    }}\noindent}
\else
    \newcommand{\elaborate}[1]{}
\fi

\usepackage[margin=1.2in]{geometry}

\newtheorem{theorem}{Theorem}
\newtheorem{lemma}[theorem]{Lemma}
\newtheorem{corollary}[theorem]{Corollary}
\newtheorem{proposition}[theorem]{Proposition}

\newtheorem{definition}[theorem]{Definition}

\newtheorem{assumption}{Assumption}
\newtheorem{remark}{Remark}
\newtheorem{condition}{Condition}

\newtheorem*{theorem-nonumber}{Theorem}
\newtheorem*{condition-nonumber}{Condition}
\newtheorem*{proposition-nonumber}{Proposition}

\usepackage{mathtools}
\DeclarePairedDelimiter{\ceil}{\lceil}{\rceil}
\DeclarePairedDelimiter\floor{\lfloor}{\rfloor}

\newcommand{\D}{\mathbb D}

\newcommand{\cmt}[1]{#1} 
\renewcommand{\cmt}[1]{} 
\renewcommand{\P}{\mathbf{P}}

\newcommand{\E}{\mathbf{E}}
\newcommand{\RV}{\mathcal{RV}}
\newcommand{\R}{\mathbb{R}}
\newcommand{\Z}{\mathbb{Z}}
\renewcommand{\S}{\mathbb{S}}
\newcommand{\C}{\mathbb{C}}
\newcommand{\M}{\mathbb{M}}
\newcommand{\I}{\mathbbm{I}}
\renewcommand{\complement}{c}

\newcommand{\lo}{\mathit{o}}

\usepackage[normalem]{ulem}

\usepackage{multirow}
\usepackage{longtable}

\usepackage{array}
\def\delequal{\mathrel{\ensurestackMath{\stackon[1pt]{=}{\scriptscriptstyle\Delta}}}}
\def\distequal{\mathrel{\ensurestackMath{\stackon[1pt]{=}{\scriptstyle d}}}}
\newcommand{\norm}[1]{\left\lVert#1\right\rVert}
\usepackage{mathtools}

\algrenewcommand\algorithmicrequire{\textbf{Require:}}
\algrenewcommand\algorithmicensure{\textbf{Postcondition:}}

\usepackage{subfigure}
\usepackage{float}
\usepackage{subfiles}
\title{Large Deviations and Metastability Analysis\\ for Heavy-Tailed Dynamical Systems}

\DeclareMathAccent{\widecheck}{0}{mathx}{"71}

\author[1]{Xingyu Wang} 
\author[2]{Chang-Han Rhee}
\affil[1]{Quantitative Economics, University of Amsterdam\\
    Amsterdam, 1018 WB, NL}
\affil[2]{Industrial Engineering and Management Sciences, Northwestern University\\
    Evanston, IL, 60613, USA}
\begin{document}
\maketitle

\begin{abstract}
\noindent
This paper introduces novel frameworks for large deviations and metastability analysis in heavy-tailed stochastic dynamical systems. We develop and apply these frameworks within the context of stochastic difference equation $X^\eta_{j+1}(x) = X^\eta_{j}(x) + \eta a\big( X^\eta_{j}(x)\big) + \eta \sigma\big( X^\eta_{j}(x)\big)Z_{j+1}$ and its variation with truncated dynamics $X^{\eta|b}_{j+1}(x) =  X^{\eta|b}_{j}( x) + \varphi_b\big(\eta  a\big(  X^{\eta|b}_{j}( x)\big) + \eta \sigma\big(  X^{\eta|b}_{j}( x)\big) Z_{j+1}\big)$, where $\varphi_b(x) = (x/\norm{x})\max\{\norm{x}, b\}$. The truncation operator $\varphi_b(\cdot)$ is often introduced as a modulation mechanism in heavy-tailed systems, such as stochastic gradient descent algorithms in deep learning.  We establish locally uniform sample-path large deviations for both processes and translate these asymptotics into precise characterizations of the joint distributions of the first exit times and exit locations. Our large deviations asymptotics are sharp enough to rigorously characterize \emph{the catastrophe principle} by establishing the distributional limit of the sample paths conditional on the rare events of interest, thereby revealing the most likely paths through which rare events arise in heavy-tailed dynamical systems. Moreover the resulting limit theorem unveils a discrete hierarchy of phase transitions (i.e., exit times) as the truncation threshold $b$ varies. Together, these developments lead to comprehensive heavy-tailed counterpart of the classical Freidlin-Wentzell theory. We present our results in the context of discrete time processes $X^\eta_{j+1}(x)$ and $X^{\eta|b}_{j+1}(x)$, as they more directly model the stochastic algorithms in deep learning that inspired this work. Nontheless, the same approach applies straightforwardly to continuous-time processes, and we include the corresponding results for the L\'evy-driven SDEs in the appendix. 
\end{abstract}

\counterwithin{equation}{section}
\counterwithin{lemma}{section}
\counterwithin{corollary}{section}
\counterwithin{theorem}{section}
\counterwithin{definition}{section}
\counterwithin{proposition}{section}
\counterwithin{figure}{section}
\counterwithin{table}{section}

\tableofcontents

\section{Introduction}
Large deviations and metastability analysis in stochastic dynamical systems are deeply interconnected and have a rich history in probability theory and related fields. 
Since the foundational works of Kramers and Eyring \cite{eyring1935chemical, kramers1940brownian, glasstone1941theory}, which analyzed phase transitions in stochastic dynamical systems in the context of chemical reaction-rate theory, extensive theoretical advancements have been made. 
One of the most notable breakthroughs is the now-classical Freidlin-Wentzell theory \cite{freidlin1970onsmall, freidlin1973some, freidlin1984random}, which introduced large deviations machinery to the analysis of exit times and global behaviors of small random perturbations of dynamical systems. 
Further extensions of this approach in the context of statistical physics were pioneered in \cite{cassandro1984metastable} and described in detail in \cite{Olivieri_Vares_2005}. 
One of the key advantages of this approach---often called \emph{the pathwise approach}---is its ability to describe in detail the scenarios that lead to phase transitions. 
In particular, the large deviations formalism at the sample-path level enables precise identification of the most likely paths out of the metastability sets. 
This ensures that, asymptotically, whenever the dynamical system escapes from the metastability set, the escape routes almost always closely resemble these most likely paths.
However, the sample-path-level large deviations are typically available only in the form of logarithmic asymptotics, and hence, the asymptotic scale of the exit time can be determined only up to its exponential rate, requiring different approaches to identify the prefactor.
Another breakthrough is \emph{the potential-theoretic approach} initiated in \cite{bovier2001metastability, bovier2004metastability, bovier2005metastability} and later summarized in \cite{bovier2016metastability}. 
Instead of relying on large deviations machinery, this approach leverages potential-theoretic tools:
the scale of exit times for Markov processes can be expressed in terms of capacity, which, in turn, can be bounded using variational principles. 
The key advantage of this approach, compared to the pathwise approach, is that it is often possible to find test functions that tightly bound the capacity of the Markov chains, thereby yielding \emph{precise} asymptotics---rather than merely logarithmic asymptotics as in the pathwise approach---of the scales of exit times. 
Although the potential theoretic approach does not provide as much information---such as the most likely paths---as the pathwise approach beyond the asymptotics of the exit times, its sharpness has inspired extensive research activity. 
The early works in the potential theoretic approach were focused on reversible Markov processes. 
However, recent developments have extended the scope of the approach to enable the analysis of non-reversible Markov processes; see, for example, \cite{slowik2012note, landim2014metastability, gaudilliere2014dirichlet, lee2022non}.

While these developments provide powerful means to understand rare events and metastability of light-tailed systems, 
heavy-tailed systems exhibit a fundamentally different large deviations and metastability behaviors and call for a different set of technical tools for successful analysis.
For example, early foundational works in heavy-tailed context \cite{imkeller2006first, imkeller2008levy, pavlyukevich2008metastable, imkeller2010first} 
proved that the exit times of the stochastic processes driven by heavy-tailed noises scale polynomially with respect to the scaling parameter.  
These papers also reveal that the exit events are almost always driven by a single disproportionately large jump, while the rest of the system's behavior remains close to its nominal behavior. 
Here, nominal behavior refers to the functional law of large numbers limit of the scaled processes. 
This is in stark contrast to the light-tailed counterparts, where the exit times scale geometrically, and the exit events are driven by smooth tilting of the entire system from its nominal behavior.
One can view this as a manifestation of \emph{the principle of a single big jump}, a well-known folklore in extreme value theory. 
For stochastic processes with independent increments over a finite time horizon, \cite{hult2005functional} systematically characterized the principle of a single big jump with an early formulation of heavy-tailed sample-path large deviations.

However, many heavy-tailed rare events in machine learning, finance, operations research, and other disciplines cannot be driven by a single big jump;
see e.g.\ \cite{Albrecher_Chen_Vatamidou_Zwart_2020,tankov2003financial,foss2006heavy,doi:10.1287/moor.1120.0539,wang2022eliminating}.
A notable example arises in the context of deep learning. 
Stochastic gradient descent (SGD) and its variants are the methods of choice in training deep neural networks (DNNs). 
Heavy-tailed SGDs have attracted significant attention in the recent past because of their ability to escape local minima with a single big jump, enabling them to explore non-convex energy landscapes within realistic training time horizons. 
Such ability is widely believed to have fundamental connection to DNNs' remarkable generalization performance on test data.
However, the pure form of SGD is rarely employed in practice.
In particular, when the gradient noise appears to exhibit heavy-tailed behavior causing SGD to occasionally attempt to travel a long distance in a single step, the step size is truncated at a threshold. 
This is a common practice known as gradient clipping; see, e.g., \cite{Engstrom2020Implementation, merity2018regularizing, graves2013generating, pascanu2013difficulty,zhang2020why}.
With gradient clipping, the exit event from a large attraction field cannot be solely driven by a single big jump. 
In general---as we rigorously confirm in this paper---when a single big jump is insufficient to cause the rare event of interest, it is driven by the minimal number of big jumps required to trigger it, while the rest of the system remains close to its nominal dynamics. 
This portrayal provides a more complete picture than the principle of a single big jump and is referred to as \emph{the catastrophe principle}. 
More recently, a rigorous mathematical characterization of the catastrophe principle for L\'evy processes and random walks was established in the form of heavy-tailed sample-path large deviations \cite{rhee2019sample}, leveraging the $\M$-convergence theory originally introduced in \cite{lindskog2014regularly}. 
The results in \cite{rhee2019sample} can be viewed as the heavy-tailed counterpart of the Mogulskii's theorem \cite{lynch1987large, mogulskii1993large}.
Notably, the new large deviations formulation in \cite{rhee2019sample} provides precise asymptotics for heavy-tailed processes, in contrast to the logarithmic asymptotics of the classical large deviation principle; see~\eqref{intro: LD for RW}. 
This raises the hope that, for heavy-tailed dynamical systems, it may be possible to simultaneously obtain \emph{both} detailed descriptions of the scenarios leading to phase transitions (as in the pathwise approach \cite{freidlin1984random, Olivieri_Vares_2005}) and sharp asymptotics for the exit time (as in the potential-theoretic approach \cite{bovier2016metastability}).
Successfully implementing this strategy for practical systems requires establishing strong enough sample-path large deviations and developing tools to translate these results into exit-time analyses tailored for heavy-tailed dynamical systems with transition dynamics potentially modulated by truncation.

In this paper, we propose a new formulation of large deviations, along with systematic tools to establish them and translate them into exit-time asymptotics for heavy-tailed dynamical systems.
Using this framework, we derive precise sample-path large deviations and sharp scaling limits of the joint distributions of the exit times and locations for heavy-tailed stochastic difference equations. 
In particular, we characterize the asymptotics of processes whose step sizes are modulated by truncation; see \eqref{def, intro, clipped SGD} and \eqref{def: objects for defining Y eta b, clipped SDE, 1} for the precise definitions.
It turns out that such modulation introduces phase transition within phase transition: the polynomial rate of exit time's asymptotic scale changes discontinuously w.r.t.\ the truncation parameter, changing the way the exit events occur qualitatively;
see Theorem~\ref{theorem: first exit time, unclipped} and the form of $\mathcal J_b^I$ in \eqref{def: first exit time, J *}.
This behavior sharply contrasts with the light-tailed counterpart, where truncation does not affect the large deviations behavior.
This is another manifestation of the dichotomy between the catastrophe principle and conspiracy principle. 
In view of these, our results provide comprehensive heavy-tailed counterparts to the Freidlin–Wentzell theory for stochastic dynamical systems.
More precisely, the main contributions of this article can be summarized as follows:
\begin{itemize}
    \item 
        {\bf Heavy-tailed Large Deviations:}
        We establish sample-path large deviations for heavy-tailed dynamical systems. 
        We propose a new heavy-tailed large deviations formulation that is locally uniform w.r.t.\ the initial values.
        We accomplish this by formulating a uniform version of $\M(\S \setminus \C)$-convergence \cite{lindskog2014regularly,rhee2019sample}.
        Our large deviations characterize the catastrophe principle,
        which reveals a discrete hierarchy governing the causes and probabilities of a wide variety of rare events associated with heavy-tailed stochastic difference and differential equations;
        see Theorems~\ref{corollary: LDP 2},~\ref{theorem: LDP 1, unclipped},~\ref{theorem: LDP, SDE, uniform M convergence}, and~\ref{theorem: LDP, SDE, uniform M convergence, clipped}.
        We also obtain the precise distributional limit of the scaled sample paths conditional on the rare events in Corollary~\ref{corollary: conditional limit, SGD} and \ref{corollary: conditional limit, SDE}. 
        In the second half of this paper, we focus on their implications for the exit-time (and exit-location) analysis. 
        However, we emphasize that these results provide general, systematic tools for heavy-tailed rare-event analysis far beyond exit times.

    \item 
        {\bf Metastability Analysis:} 
        We establish a scaling limit of the exit-time and exit-location for stochastic difference equations. 
        We accomplish this by developing a machinery for local stability analysis of general (heavy-tailed) Markov processes.
        Central to the development is the concept of asymptotic atoms, where the process recurrently enters and asymptotically regenerates.
        Leveraging the locally uniform version of sample-path large deviations over these asymptotic atoms, we derive sharp asymptotics for the joint distribution of (scaled) exit-times and exit-locations for heavy-tailed processes, as detailed in Theorem~\ref{theorem: first exit time, unclipped} and Corollary~\ref{corollary: first exit time, untruncated case}.
        Notably, the scaling rate parameter reflects an intricate interplay between the truncation threshold and the geometry of the drift, which is a feature absent in both the principle of a single big jump regime (heavy-tailed systems without truncation) and the conspiracy principle regime (light-tailed systems).

\end{itemize}
A culmination of metastability analysis is the sharp characterization of the global dynamics of heavy-tailed processes, often established in the form of process-level convergence of scaled processes to simpler ones, such as Markov jump processes on a discrete state space; see, for example, \cite{Olivieri_Vares_2005, bovier2016metastability, imkeller2008levy, pavlyukevich2008metastable, lee2022non-b, rezakhanlou2023scaling}. 
For unregulated processes such as \eqref{defSDE, initial condition x}, which are governed by the principle of a single big jump, it is well known 
that the scaling limit is a Markov jump process with a state space consisting of the local minima of the potential function; see e.g., \cite{doi:10.1142/S0219493715500197, imkeller2008levy, pavlyukevich2008metastable}. 
In a companion paper \cite{wangSGDpaper2}, 
we demonstrate that the framework developed in this paper is strong enough to extend the above-mentioned results to the systems \emph{not} governed by the principle of a single big jump---such as \eqref{def, intro, clipped SGD} and \eqref{def: objects for defining Y eta b, clipped SDE, 1}---within a multi-well potential, by identifying scaling limits and characterizing their global behavior at the process level. 
In particular, the scaling limit for the truncated heavy-tailed dynamics is a Markov jump process that \emph{only visits the widest minima}. 
This is in sharp contrast to the untruncated cases \cite{doi:10.1142/S0219493715500197, imkeller2008levy, pavlyukevich2008metastable} where the limiting Markov jump process visits all the local minima with certain fractions.
As a result, the fraction of time such processes spend in narrow attraction fields converges to zero as the scaling parameter (often called learning rate in the machine learning literature) tends to zero.
Precise characterization of such phenomena is of fundamental importance in understanding and further leveraging the curious effectiveness of the stochastic gradient descent (SGD) algorithms in training deep neural networks.

\subsection{Overview of the Paper}\label{subsec:overview-of-the-paper}

In this paper, we focus on the class of heavy-tailed phenomena captured by the notion of regular variation.
To be specific, let $(\bm Z_i)_{i \geq 1}$ be a sequence of iid random vectors in $\R^d$ such that $\E \bm Z_1 = \bm 0$ and $\P(\norm{\bm Z_i} > x)$ is regularly varying with index $-\alpha$ as $x \to \infty$ for some $\alpha>1$.
That is, there exists some slowly varying function $\phi$ such that $\P(\norm{\bm Z_1} > x) = \phi(x)x^{-\alpha}$.
For any $\eta > 0$ and $\bm x \in \R^m$, let $\big(\bm X^\eta_j(\bm x)\big)_{j \geq 0}$ be the solution of the following stochastic difference equation
\begin{align}
    \bm X^\eta_0(\bm x) = \bm x;\qquad 
    \bm X^\eta_{j+1}(\bm x) & = \bm X^\eta_{j}(\bm x) + \eta \bm a\big( \bm X^\eta_{j}(\bm x)\big) + \eta \bm \sigma\big( \bm X^\eta_{j}(\bm x)\big)\bm Z_{j+1}\quad \forall j \geq 0.
    \label{intro, def for X eta j}
\end{align}
Throughout this paper, we adopt the convention that the subscript denotes the time, and the superscript $\eta$ denotes the scaling parameter that tends to zero. 
Furthermore, we consider a truncated variation of $\bm X^\eta_{j+1}(\bm x)$. 
Let $\varphi_b(\cdot)$ be the projection operator from $\R^m$ onto the closed ball centered at the origin with radius $b$.
Define 
\begin{align}
     \bm X^{\eta|b}_0(\bm x) = \bm x;\qquad 
     \bm X^{\eta|b}_{j+1}(\bm x) & = \bm X^{\eta|b}_{j}(\bm x) + \varphi_b\Big(\eta \bm a\big( \bm X^{\eta|b}_{j}(\bm x)\big) + \eta \bm\sigma\big( \bm X^{\eta|b}_{j}(\bm x)\big)\bm Z_{j+1}\Big)\quad \forall j \geq 0.
     \label{def, intro, clipped SGD}
\end{align}
In other words,
$\bm X^{\eta|b}_j(\bm x)$ is a modulated version of $\bm X^{\eta}_j(\bm x)$ where the distance traveled at each step is truncated at $b$,
and the dynamics of $\bm X^{\eta}_j(\bm x)$ is recovered by setting the truncation threshold $b$ as $\infty$.
As mentioned above, such dynamics arise in the training of DNNs, and their global behaviors are closely connected to the performance of the trained models. 
In particular, if $\bm a$ is the negative gradient of the training loss $f$, then the argument of $\varphi_b(\cdot)$ in \eqref{def, intro, clipped SGD}, $\eta \bm a\big(\bm X_j^\eta(\bm x)\big) + \eta \bm \sigma \big(\bm X_j^\eta(\bm x)\big) \bm Z_{j+1} = -\eta \big(\nabla f\big(\bm X_j^\eta(\bm x)\big) - \bm \sigma \big(\bm X_j^\eta(\bm x)\big) \bm Z_{j+1} \big)$, represents the state-dependent stochastic gradient of $f$ at $\bm X_j^\eta(\bm x)$, scaled by the negative learning rate $-\eta$, which corresponds to the one-step displacement of SGD.
Therefore, \eqref{intro, def for X eta j} and \eqref{def, intro, clipped SGD} serve as models for the dynamics of heavy-tailed SGD and its variation with gradient clipping, respectively.
See, for example, \cite{wang2022eliminating, pascanu2013difficulty, zhang2020why, pmlr-v202-koloskova23a} and the references therein for more details.
Note that \eqref{intro, def for X eta j} and \eqref{def, intro, clipped SGD} can be viewed as discretizations of small-noise SDEs driven by L\'evy processes. 
In this paper, we primarily focus on these discrete-time processes, as they provide more accurate models of the stochastic algorithms in deep learning compared to the continuous counterparts. 
Furthermore, \eqref{intro, def for X eta j} and \eqref{def, intro, clipped SGD} do not require the $Z_i$'s to be $\alpha$-stable to model SGDs and impose no restrictions on the choice of regular variation. 
In contrast, approximating SGDs with SDEs becomes obscure when $Z_i$'s are not $\alpha$-stable.
Nevertheless, we emphasize that all the results we establish for \eqref{intro, def for X eta j} and \eqref{def, intro, clipped SGD} in this paper can also be established for the stochastic differential equations driven by regularly-varying L\'evy processes, with a straightforward adaptation of the machinery we develop here. 
We present the results for L\'evy-driven SDEs in Appendix~\ref{sec: appendix, SDE reults}.
Finally, note also that although \eqref{intro, def for X eta j} and \eqref{def, intro, clipped SGD} are probably the most natural scaling regime in many contexts, more general scaling can be considered as well. 
In Appendix~\ref{sec: appendix, SGD, general scaling, results},
we present the corresponding results for
\begin{equation}
\begin{aligned}
    &\bm X^\eta_0(\bm x) = \bm x,\qquad
    \bm X^\eta_j(\bm x) = \bm X^\eta_{j - 1}(\bm x) +  \eta \bm a\big(\bm X^\eta_{j - 1}(\bm x)\big) + \eta^\gamma\bm \sigma\big(\bm X^\eta_{j - 1}(\bm x)\big)\bm Z_j\quad \forall j \geq 1;
    \\
    &\bm X^{\eta|b}_0(\bm x) = \bm x,\qquad {\bm X^{\eta|b}_j(\bm x)} = \bm X^{\eta|b}_{j - 1}(\bm x) +  \varphi_b\Big(\eta \bm a\big(\bm X^{\eta|b}_{j - 1}(\bm x)\big) + \eta^\gamma \bm \sigma\big(\bm X^{\eta|b}_{j - 1}(\bm x)\big)\bm Z_j\Big)\quad \forall j \geq 1
\end{aligned}
\label{eq:gneral-scaling}
\end{equation}
with some $\gamma > 0$.

At the crux of this study is a fundamental difference between light-tailed and heavy-tailed stochastic dynamical systems. 
This difference lies in the mechanism through which system-wide rare events arise.
In light-tailed systems, the system-wide rare events are characterized by the \emph{conspiracy principle}: the system deviates from its nominal behavior because the entire system behaves subtly differently from the norm, as if it has conspired.
In contrast, \emph{the catastrophe principle} governs the rare events in heavy-tailed systems: catastrophic failures (i.e., extremely large deviations from the average behavior) in a small number of components drive the system-wide rare events, 
and the behavior of the rest of the system is indistinguishable from the nominal behavior.

The classical large deviations principle (LDP) framework \cite{MR2571413, MR997938, MR2260560, MR758258} has been wildly successful in providing systematic tools for studying rare events. 
In particular, the sample-path large deviation principle rigorously characterize the conspiracy principle. 
Notable developments include the Mogulskii's theorem \cite{lynch1987large, mogulskii1993large}, the Freidline and Wentzell theory \cite{freidlin1970onsmall, freidlin1973some, freidlin1984random}, and various extensions 
for discrete-time processes \cite{
7b45313d-69b0-37a6-ac55-aaa69381f337, 
10.1214/aop/1176990641}
for finite dimensional processes under relaxed assumptions
\cite{ 
doi:10.1142/9789812770639_0007, 
donati2004large, 
donati2008large, 
BALDI20111218, 
dupuis2011weak},
and for infinite dimensional processes 
\cite{
budhiraja2000variational, 
10.1214/07-AOP362, 
dadbc9bb-cda8-3dea-89b9-08e2a516e833, 
10.1214/aop/1079021473, 
mohan2021wentzell}.

On the other hand, due to the fundamental differences in the way rare events arise, sample-path large deviations for heavy-tailed processes has been developed much later. 
Instead, the principle of a single big jump, a special case of the catastrophe principle, has been discussed in the heavy-tail and extreme value theory literature for a long time. 
That is, in many heavy-tailed systems, the system-wide rare events arise due to exactly one catastrophe. 
This line of investigation was initiated in the classical works \cite{nagaev1969limit,nagaev1978property}, 
and \cite{hult2005functional} confirmed the principle of a single big jump systematically at the sample-path level for random walks.
The summary of the subsequent developments in the context of processes with independent increments can be found in, for example, 
\cite{borovkov_borovkov_2008,denisov2008large,embrechts2013modelling,foss2011introduction}.
More recently, \cite{rhee2019sample} established a general catastrophe principle for regularly varying L\'evy processes and random walks, which goes beyond the principle of a single big jump and characterizes the rare events driven by any number of catastrophes.
For example, let $\D\big([0,1],\R\big)$ be the space of real-valued càdlàg functions over $[0,1]$, 
$S_j \triangleq Z_1+\cdots+Z_j$ be a mean-zero random walk, and $\bm S^n \triangleq \{ S_{\lfloor nt\rfloor}^n/n: t \in [0,1]\}$ be the scaled sample path.
Under regularly varying $Z_i$'s,
the sample path large deviations established in \cite{rhee2019sample}
takes the following form for ``general'' $B\in \D\big([0,1],\R\big)$,
\begin{equation}\label{intro: LD for RW}
    \begin{split}
    0 
    < 
    \mathbf{C}_k( B^\circ)
    & 
    \leq  \liminf_{n \to \infty}
        \frac{ \P({\bm{S}}^{n} \in B) }{  
    (n\P(|Z_1| > n) )^{k}
    }
   \leq  
   \limsup_{n \to \infty}
   \frac{\P({\bm{S}}^{n} \in B) }{ 
   (n\P(|Z_1| > n))^{k}
   } 
    \leq 
    \mathbf{C}_{k}( B^-)
    <
    \infty,
    \end{split}
\end{equation}
where $k$ is the minimal number of jumps that a step function must possess in order to belong to $B$,
$\mathbf C_k(\cdot)$ is a measure supported on the set of step functions with $k$ jumps, 
and $B^\circ$ and $B^-$ are the interior and closure of $B$, respectively.
Here, the index $k$, as a function of $B$, plays the role of the infimum of rate function over $B$ in the classical light-tailed large deviation principle (LDP) formulation. 
See also \cite{bernhard2020heavy} for analogous results for random walks under more general scaling.

Note that in contrast to the standard log-asymptotics in the classical LDP framework, \eqref{intro: LD for RW} provides exact asymptotics. 
This formulation provides a powerful framework in heavy-tailed contexts; for instance, it has enabled the design and analysis of strongly efficient rare-event simulation algorithms for a wide variety of rare events associated with $\bm S^n$, as demonstrated in \cite{chen2019efficient}. 
Moreover, \cite[Section 4.4]{rhee2019sample} proves that it is impossible to establish the classical LDP w.r.t.\ $J_1$ topology at the sample-path level for regularly varying L\'evy processes. 
On a related note, by relaxing the upper bound of the standard LDP, an alternative formulation known as ``extended LDP'' was proposed in \cite{borovkov2010large}, and such a formulation is also feasible for heavy-tailed processes; see, for example, \cite{borovkov2011large, bazhba2020weibull, bazhba2022large}.
However, the extended LDP only provides log-asymptotics. 
For regularly varying processes, it is often desirable and possible to obtain exact asymptotics; for example, the extended LDP wouldn't suffice for analyzing the strong efficiency of the aforementioned rare-event simulation algorithm in \cite{chen2019efficient}. 
We will also see that exact asymptotics are crucial in Section~\ref{sec: first exit time simple version} and Section~\ref{sec: first exit time simple version, proof} of this paper for sharp exit time and exit location analysis.
In fact, it demands an even stronger version of \eqref{intro: LD for RW}, which we will introduce in \eqref{intro: LD of SGD} shortly.
Below, we describe the main contributions of this paper.
\\

\noindent{\bf Large Deviations for Heavy-Tailed Dynamical Systems.}
Our first contribution is to characterize the catastrophe principle for a general class of heavy-tailed stochastic dynamical systems in the form of a \emph{locally uniform} heavy-tailed large deviations at the sample-path level.
This turns out to be the right large deviations formulation for the purpose of the subsequent metastability analysis.
To be specific, let $\bm{X}^{\eta|b}_{[0,1]}(\bm x) \delequal \{ \bm X^{\eta|b}_{ \floor{t/\eta} }(\bm x):\ t\in[0,1] \}$ be the time-scaled version of the sample path of $\bm X^{\eta|b}_j(\bm x)$ defined in \eqref{def, intro, clipped SGD}, embedded in the continuous time.
Note that $\bm{X}_{[0,1]}^{\eta|b}(\bm x)$ is a random element in $\D = \D\big([0,1],\R^m\big)$.
As $\eta$ decreases,
$\bm X_{[0,1]}^{\eta|b}(\bm x)$ converges to a deterministic limit $\{\bm y_t(\bm x): t \in[0,1]\}$, where $d\bm y_t(\bm x)/dt = \bm a(\bm y_t(\bm x))$ with initial value $\bm y_0(\bm x) = \bm x$.
Let $B \subseteq \D$ be a Borel set w.r.t.\ the $J_1$ topology and $A \subset \R^m$ be a compact set.
We establish the following asymptotic bound for each $k \geq 0$:
\begin{equation}\label{intro: LD of SGD}
    \begin{split}
        \inf_{\bm x \in A}
    \mathbf{C}^{(k)|b}( B^\circ;\bm  x )
& \leq  \liminf_{\eta \downarrow 0}\frac{ \inf_{\bm x \in A}\P\big({\bm{X}}_{[0,1]}^{\eta|b}(\bm x) \in B \big) }{  
    \big(\eta^{-1}\P(\norm{\bm Z_1} > \eta^{-1}) \big)^{ k }
} 
\\
   & 
   \leq  \limsup_{\eta \downarrow 0}\frac{ \sup_{\bm x \in A}\P\big({\bm{X}}_{[0,1]}^{\eta|b}(\bm x) \in B\big) }{ 
   \big(\eta^{-1}\P(\norm{\bm Z_1} > \eta^{-1})\big)^{ k }
   } 
    \leq 
    \sup_{\bm x \in A}
    \mathbf{C}^{(k)|b}( B^-; \bm x).
    \end{split}
\end{equation}
%
The precise statement and the definition of $\mathbf C^{(k)|b}$ can be found in Section~\ref{subsubsec: untrucated, LD, SGD}. Here, we note that $C^{(k)|b}$ is precisely identified, its intuitive meaning is clear, and its computation is straightforward using Monte Carlo simulation.

Additionally, we point out that 
the index $k$ leading to non-degenerate upper and lower bounds in \eqref{intro: LD of SGD} 
represents the minimal number of jumps (with sizes truncated under $b$) that must be added to the path of $\bm{y}_t(\bm x)$ for it to enter the set $B$, given $\bm x\in A$. 
Such an index $k$ dictates the precise polynomial decay rate of the rare-event probability and corresponds to the infimum of rate function of the classical large deviations framework.
Note also that as the set $A$ shrinks to a singleton, the upper and lower bounds in \eqref{intro: LD of SGD} become tighter, and hence, \eqref{intro: LD of SGD} is a \emph{locally uniform} version of the large deviations formulation in \eqref{intro: LD for RW}.

\begin{figure}[p]
\vskip 0.2in
\begin{center}
\begin{tabular}{c}
$(a)$ \textbf{Experiment Setting, and Estimations of $\P(\bm X^{\eta|b}_{[0,1]}(x_\text{init})\in B)$} 
\\
\includegraphics[width=0.7\textwidth]{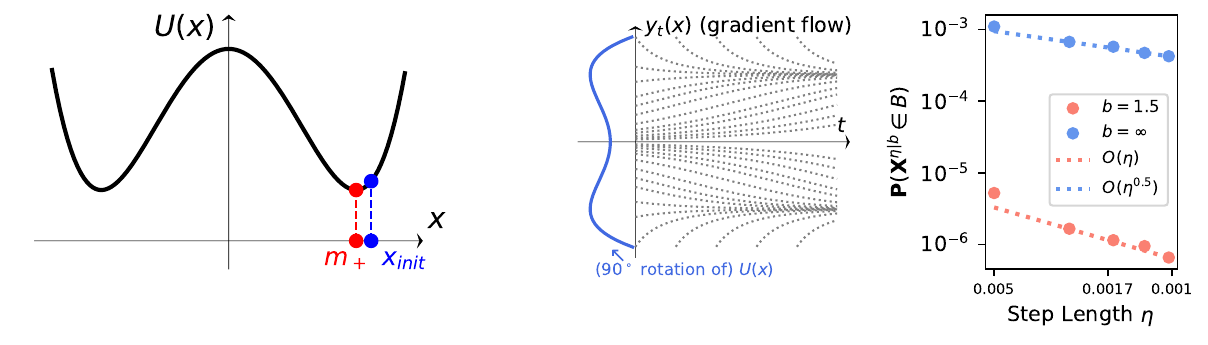} 

\\
$(b)$ \textbf{Heavy-Tailed $Z_i$, No Truncation $(b = \infty)$}
\\
\includegraphics[width=0.85\textwidth]{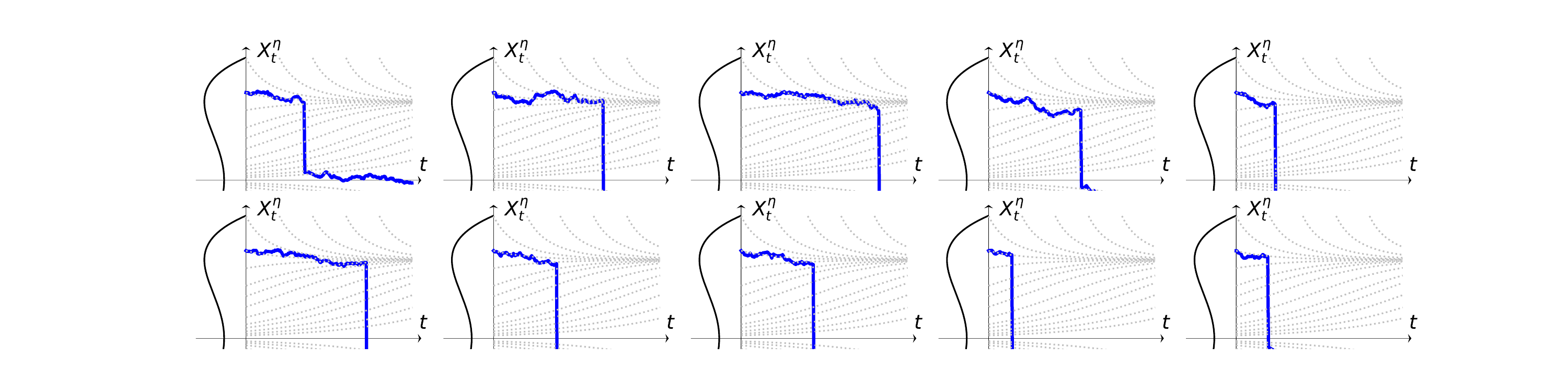} 
\\

 $(c)$ \textbf{Heavy-Tailed $Z_i$, with Truncation $(b = 1.5)$} 
\\
\includegraphics[width=0.85\textwidth]{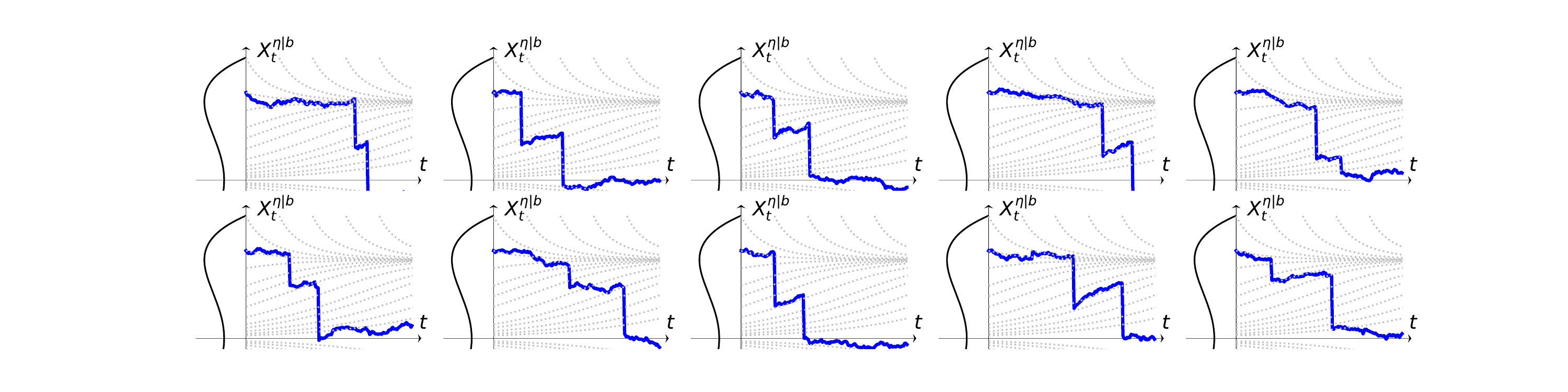} 
\\
$(d)$ \textbf{Light-Tailed $Z_i$, No Truncation $(b = \infty)$}
\\
\includegraphics[width=0.85\textwidth]{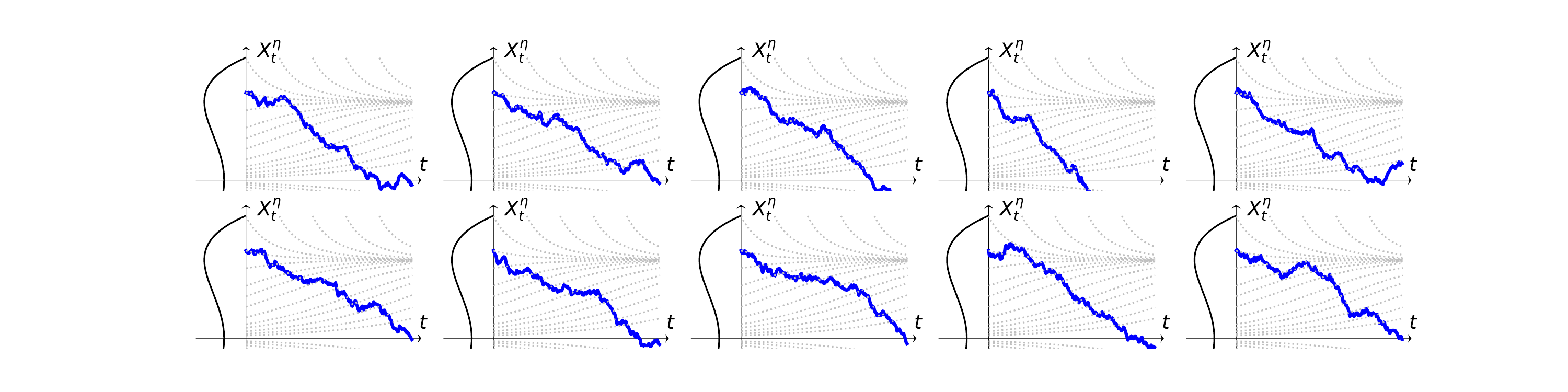} 
\\

$(e)$ \textbf{Light-Tailed $Z_i$, with Truncation $(b = 1.5)$}
\\
\includegraphics[width=0.85\textwidth]{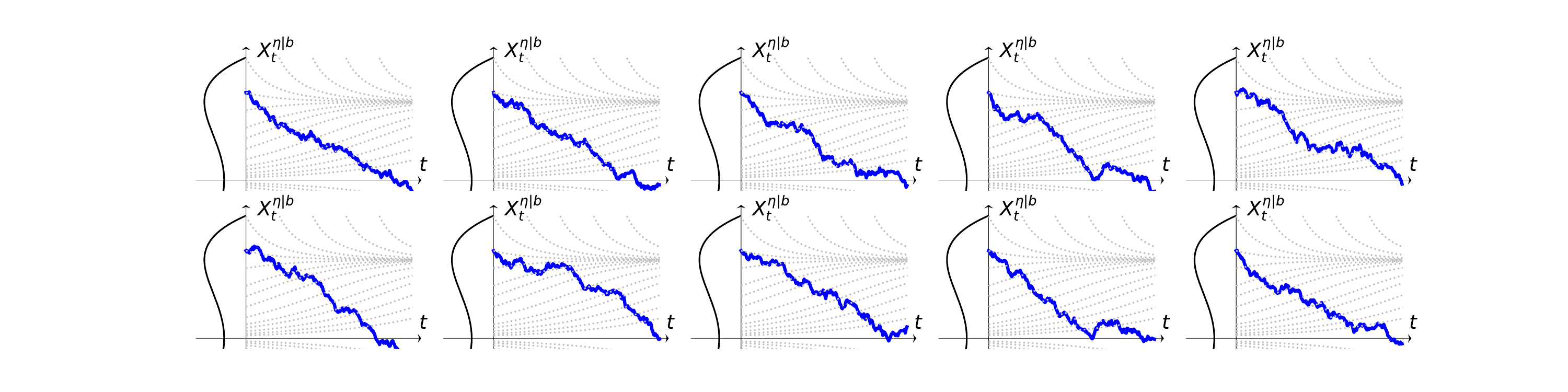}

\end{tabular}
\caption{ 
Numerical examples for large deviations and the catastrophe principle.
\textbf{(a, Left)} the potential function $U(\cdot)$ defined in \eqref{numerical example: def potential U, LD}.
\textbf{(a, Middle)} gradient flows under $-U^\prime(\cdot)$.
\textbf{(a, Right)} Estimation of 
$
\P(\bm X^{\eta|b}_{[0,1]}(\bm x_\text{init}) \in B)
$
through Monte-Carlo simulation; dashed lines are predictions according to our large deviations asymptotics.
$\textbf{(b)--(e)}$: Samples from
$
\P(\bm X^{\eta|b}_{[0,1]}(\bm x_\text{init}) \in \ \cdot\ | \bm X^{\eta|b}_{[0,1]}(\bm x_\text{init}) \in B),
$
$\eta = \frac{1}{200}$.
}
\label{fig: exit path summary}
\end{center}
\vskip -0.2in
\end{figure}

An important implication of \eqref{intro: LD of SGD} is the sharp characterization of the catastrophe principle. 
Specifically, Section~\ref{subsubsec: conditional limit theorem} proves that the conditional distribution of $\bm X^{\eta|b}_{[0,1]}(\bm x)$ given the rare event of interest converges to the distribution of a piecewise deterministic random function $\bm X^{*|b}_{|B}(\bm x)$ with precisely $k$ random jumps whose sizes are bounded from below:
\begin{align}
    \mathscr L\big(\bm X^{\eta|b}_{[0,1]}(\bm x)\big |\bm X^{\eta|b}_{[0,1]}(\bm x)\in B\big)  \to \mathscr L \big(\bm X^{*|b}_{|B}(\bm x)\big).
    \label{intro: catastrophe principle, claim}
\end{align}
We give the formal statement of the catastrophe principle in Corollary~\ref{corollary: conditional limit, SGD}.
Here, we note that the perturbation associated with $\bm Z_i$ is $\eta\bm \sigma(\bm X_{i-1}^{\eta|b}(\bm x)\bm ) \bm Z_i$. 
Hence, the jump size associated with $\bm Z_i$ being bounded from below implies that $\bm Z_i$ is of order $1/\eta$.
This means that the rare event $\{\bm X_{[0,1]}^{\eta|b}(\bm x) \in B\}$ is driven by $k$ jumps of size $O(1/\eta)$, whereas the rest of the system behaves close to the law-of-large-numbers limit of the system.
Figure~\ref{fig: exit path summary} illustrates the catastrophe principle in a univariate setting where the drift is given by the negative gradient of a potential: $\bm b(\cdot) = -U'(\cdot)$.
In (a, Left) of Figure~\ref{fig: exit path summary}, we show the potential function $U: \R \to \R$, while (a, Middle) shows its gradient flows starting from different initial points. 
By gradient flows, we refer to the solution of the ODE $d \bm y_t(\bm x)/dt = -U^\prime\big(\bm y_t(\bm x)\big)$ with initial condition $\bm y_0(\bm x) = \bm x$. 
Given an initial value $\bm x_\text{init}$, 
suppose that we are interested in the
conditional law
\begin{align}
    \P\big(
        \bm X^{\eta|b}_{[0,1]}(\bm x_\text{init}) \in \ \cdot\ \big|
        \bm X^{\eta|b}_{[0,1]}(\bm x_\text{init}) \in B
    \big),
    \label{def: event B for numerical example, LD, intro}
\end{align}
where
$
    B = \big\{ \xi \in \D\big([0,1],\R\big):\ \xi(t) \leq 0\ \text{for some }t \leq 1 \big\}.
$
That is, the behavior of \eqref{def, intro, clipped SGD} when they escape from the attraction field $(0,\infty)$ associated with the local minimum $m_+=\sqrt 5$ within $\floor{1/\eta}$ steps.
As shown in (b) of 
Figure~\ref{fig: exit path summary},
when driven-by heavy-tailed $Z_i$'s, 
the untruncated dynamics
$
\bm X^\eta_j
$
closely resembles the gradient flows and stays close to $m_+$ until a single large $Z_j$ sends $\bm X^\eta_j$ outside of $(0,\infty)$ in one shot. 
In comparison, under small enough choices of the truncation threshold $b$ in \eqref{def, intro, clipped SGD}, the process $\bm X^{\eta|b}_j$ can no longer exit $(0,\infty)$ from $m_+$ in one step.
Indeed, (c) of Figure~\ref{fig: exit path summary} depicts a case where the sample paths of $\bm X^{\eta|b}_j$ resembles the gradient flow with two large perturbations truncated at $b$.
This clearly confirms the catastrophe principle \eqref{intro: catastrophe principle, claim}: 
the rare event $\{\bm X^{\eta|b}_{[0,1]}(\bm x) \in B\}$ arises almost always because of $k=2$ catastrophically large---i.e., $O(1/\eta)$---perturbations, whereas the rest of the system is indistinguishable from its nominal behavior; here,
the index $k$ is the minimum number of jumps required by the nominal path (i.e., gradient flow) to enter the set $B$.
Compare this to (b) of Figure~\ref{fig: exit path summary}, which is governed by the principle of a single big jump, i.e., the catastrophe principle with $k=1$.
Note also that both of these sharply contrast with the light-tailed cases predicted by the classical Freidlin-Wentzell theory, where the rare events arise as the SGD fights against the negative gradient in each step, climbing up the potential hill to transition into the adjacent potential well; see part (d) and (e) of Figure~\ref{fig: exit path summary}.
It is also worth noting that, unlike the light-tailed exit scenarios, which closely follow a single deterministic path defined by the solution of a variational problem associated with the rate function,
heavy-tailed scenarios exhibit significant stochasticity in the location and size of the big jumps with only the number of jumps being deterministically $k$.
This reflects the fact that the distributional limit of the scaled process described in \eqref{intro: catastrophe principle, claim} is non-degenerate.
See Section~\ref{subsec: numerical examples} for more details of this numerical example.

The notion of $\M(\S \setminus \C)$-convergence, introduced in \cite{lindskog2014regularly} and further developed in \cite{rhee2019sample}, 
was a key technical tool behind \eqref{intro: LD for RW} in \cite{rhee2019sample}.
In this paper,
we introduce a uniform version of the $\M(\S \setminus \C)$-convergence and prove an associated Portmanteau theorem in Section~\ref{sec: M convergence, asymptotic equivalence}.
These developments form the backbone that supports our proofs in Section~\ref{subsec: LD, SGD} for the uniform sample-path large deviations of the form \eqref{intro: LD of SGD}.
\\

\noindent {\bf Metastability Analysis.}
The second contribution of this paper is the first exit-time analysis for heavy-tailed systems. 
As described at the beginning of this section, two modern approaches to the analysis of exit times for light-tailed stochastic dynamical systems are
the Freidlin-Wentzell theory (or pathwise approach) detailed in the monographs \cite{freidlin1984random, Olivieri_Vares_2005} and the potential theoretic approach summarized in the monograph \cite{bovier2016metastability}.
Despite their success in the light-tailed contexts, neither the pathwise approach nor the potential theoretic approach readily extends to heavy-tailed contexts. 
In particular, for truncated heavy-tailed dynamics such as $\bm X_{[0,1]}^{\eta|b}(\bm x)$, the explicit formula for the stationary distribution is rarely available, and its generator lacks the simplicity of the Brownian case, making the adaptation of potential theoretic approach to our context challenging.
Meanwhile, the pathwise approach hinges on the large deviation principles at the sample-path level.
Historically, however, the heavy-tailed large deviations at the sample-path level have been unavailable and considered to be out of reach until recently.

Successful results in heavy-tailed contexts are relatively recent. 
For one-dimensional L\'evy driven SDEs, \cite{imkeller2006first, pavlyukevich2008metastable} proved that the exit times from metastability sets scale at a polynomial rate and the prefactor of the of the scale depend on the width of the potential wells rather than the height of the potential barrier.
Similar results have been established in more general settings,
such as the multi-dimensional analog in \cite{imkeller2010first},
exit times for a global attractor instead of a stable point \cite{hogele2014exit}
(see also \cite{doi:10.1142/S0219493715500197} for its application in characterizing the limiting Markov chain of hyperbolic dynamical systems driven by heavy-tailed perturbations),
exit times under multiplicative noises in $\R^d$ \cite{pavlyukevich2011first},
extensions to infinite-dimensional spaces \cite{debussche2013dynamics},
and the (discretized) stochastic difference equations driven by $\alpha$-stable noises \cite{NEURIPS2019_a97da629}, to name a few.
Such metastability analyses were applied in \cite{simsekli2019tail} to study the generaliztion performance of DNNs trained by SGDs with heavy-tailed dynamics and, more recently, in \cite{JMLR:v25:21-1343} to analyze the sample efficiency of policy gradient algorithms in reinforcement learning.
It should be noted that these results focus on the events associated with the principle of a single big jump.

In contrast, this paper develops a systematic tool for analyzing the exit times and locations, even in cases where the principle of a single big jump fails to account for the exit events, and more complex patterns arise during the exit process.
The process $\bm X_j^{\eta|b}(\bm x)$ exemplifies such a scenario, as the truncation operator $\varphi_c(\cdot)$ prevents exit events driven by a single big jump. 
We reveal phase transitions in the first exit times of $\bm X^{\eta|b}_j(\bm x)$, which depend on a notion of the ``discretized widths'' of the attraction fields.
Specifically,
we consider \eqref{intro, def for X eta j}
with drift coefficients $\bm a(\cdot) = - \nabla U(\cdot)$ 
for some potential function $U \in \mathcal{C}^1(\R^m)$.
Without loss of generality, let $I \subseteq \R^m$ be some open and bounded set containing the origin.
Suppose that the entire domain $I$ falls within the attraction field of the origin,
and the gradient field $-\nabla U(\cdot)$ is locally contractive around the origin.
In other words,
when initialized within $I$, the deterministic gradient flow
$d\bm y_t(\bm x)/dt = -\nabla U(\bm y_t(\bm x))$ (under the initial condition $\bm y_0(\bm x) =\bm  x$)
will be attracted to and remain trapped near the origin.
However, due to the presence of random perturbations, 
$\bm X^{\eta|b}_j(\bm x)$ 
will eventually escape from $I$ after a sufficiently long time.
Of particular interest are the asymptotic scale of the first exit times as $\eta \downarrow 0$.
Theorem \ref{theorem: first exit time, unclipped} proves that
the joint law of the first exit time $\tau^{\eta|b}(\bm x) = \min\{j \geq 0:\ \bm X^{\eta|b}_j(\bm x) \notin I \}$
and the exit location $\bm X^{\eta|b}_{\tau}(\bm x) \triangleq \bm X^{\eta|b}_{\tau^{\eta|b}(\bm x)}(\bm x)$
admits the limit (uniformly for all $\bm x$ bounded away from $I^\complement$):
\begin{align}
    \Big(\lambda^I_b(\eta) \cdot  \tau^{\eta|b}(\bm x),\ \bm X^{\eta|b}_{\tau}(\bm x)\Big)
    \Rightarrow 
    (E,V_b)\qquad\text{ as }\eta \downarrow 0
    \label{result, intro, first exit time}
\end{align}
with some (deterministic) time-scaling function $\lambda^I_b(\eta)$.
Here, $E$ is an exponential random variable with the rate parameter 1, 
$V_b$ is some random element independent of $E$ and supported on $I^\complement$,
and the scaling function 
$\lambda_b^I(\eta)$ is regularly varying with index $-[1 + \mathcal J^I_b(\alpha-1)]$ as $\eta \downarrow 0$,
where $\mathcal J^I_b$ is the aforementioned discretized width of domain $I$ relative to the truncation threshold $b$.
The precise definition of $\mathcal J_b^I$ is provided in \eqref{def: first exit time, J *} in Section~\ref{subsec: first exit time, results, SGD}. 
However, we note that in the special case $b= \infty$, one can immediately verify that $\mathcal J_b^I = 1$, regardless of the geometry of $U$. 
Consequently, \eqref{result, intro, first exit time} reduces to the principle of a single big jump, as expected. 
When the drift is contractive so that $\nabla U(\bm x) \cdot \bm x \geq 0$ for all $\bm x \in I$, it is also straightforward to see that $\mathcal J_b^I = \lceil r / b \rceil$ where $r$ is the distance between $\bm 0$ and $I^c$, and hence, $\mathcal J_b^I$ is indeed precisely the discretized width of the attraction field $I$ relative to $b$. 
In particular, note that the drift is contractive within any attraction field in the one-dimensional cases.  
However, in general multi-dimensional spaces, $\mathcal J_b^I$ reflects a much more intricate interplay between the geometry of the drift $\bm a(\cdot)$ (or the potential $U(\cdot)$) and the truncation threshold $b$.  
\begin{figure}[t]
\vskip 0.2in
\begin{center}
\begin{tabular}{cccc}
\footnotesize{(i)} & \footnotesize{(ii)} & \footnotesize{(iii)} & \footnotesize{(iv)}
\\
\includegraphics[width=0.27\textwidth]{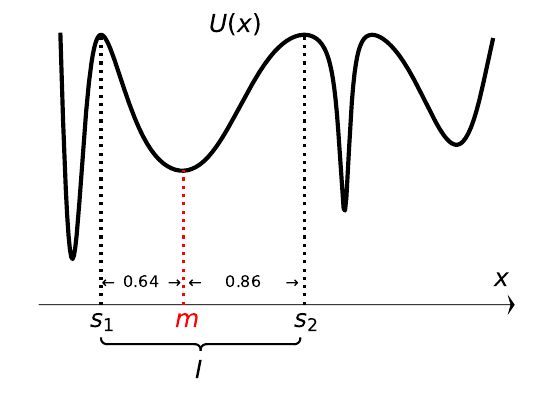}  &
\includegraphics[width=0.27\textwidth]{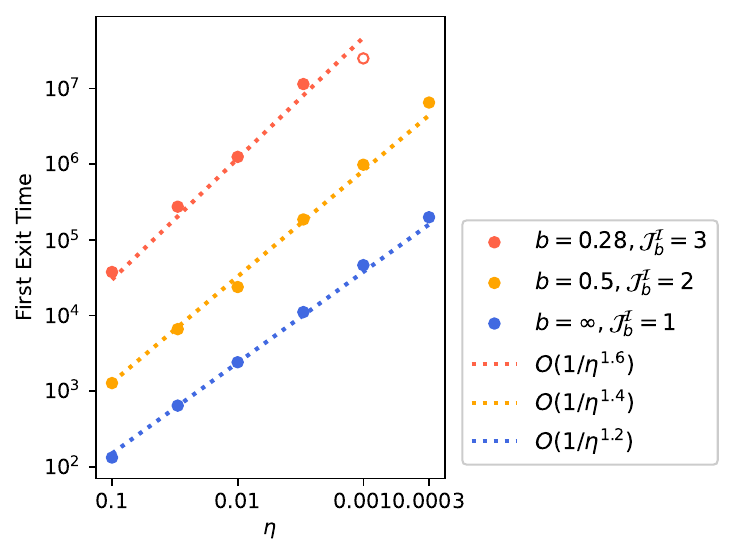}  & 
\includegraphics[width=0.18\textwidth]{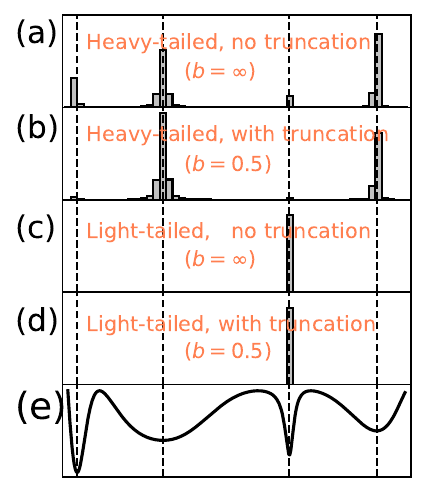} &
\includegraphics[width=0.18\textwidth]{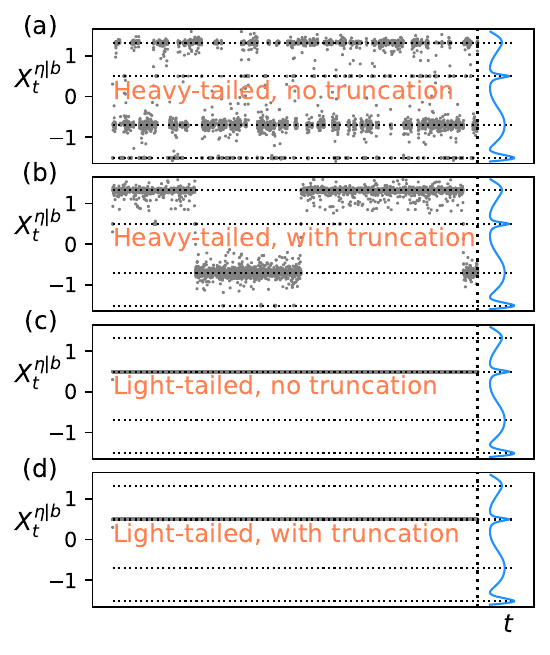}
\end{tabular}
\caption{ 
Numerical examples of the metastability analysis.
\textbf{(i)} The univariate potential $U(\cdot)$ defined in \eqref{aeq: potential U, first exit time}.
\textbf{(ii)} First exit times $\tau^{\eta|b}(m)$ from $I$ under different truncation thresholds $b$ and scale parameters $\eta$.
Dashed lines are predictions from our results in Section~\ref{sec: first exit time simple version}, whereas
the dots are the exit times estimated using 20 samples.
It can be observed that the predictions and estimates align well.
\textbf{(iii)} Histograms of locations within the potential potential $U(\cdot)$ visited by $X^{\eta|b}_t(x)$.
Note that in (b), the sharp minima are almost completely eliminated from the trajectory of the SGD.
\textbf{(iv)} Sample paths of $X^{\eta|b}_t(x)$.
Dashed lines in (iii) and (iv) are added as references for the locations of local minima.
Driven by truncated heavy-tailed noise, $X^{\eta|b}_t(x)$ almost completely avoids the sharp minima of $U(\cdot)$ in (b).
}
\label{fig: exit time}
\end{center}
\vskip -0.2in
\end{figure}
Figure~\ref{fig: exit time} illustrates the key role of the relative width $\mathcal J_b^\mathcal{I}$ in one dimension.
Specifically, we consider a one-dimensional case with a potential function $U:\R \to \R$ depicted in 
Figure~\ref{fig: exit time} (i), where $I = (s_1,s_2)$ is the attraction field of the local minimum $m$.
Since $m$ is closer to the left boundary $s_1$, the minimal number of steps required to exit $I$ when starting from $m$ is $\mathcal J^\mathcal{I}_b = \ceil{ |s - m_1|/b }$ where $b \in (0,\infty)$.  
In the untruncated case \eqref{intro, def for X eta j} (i.e., with $b = \infty$), 
we simply have $\mathcal J^\mathcal{I}_\infty = 1$.
Figure~\ref{fig: exit time} (ii) illustrates the discrete structure of phase transitions in \eqref{result, intro, first exit time}, 
where
the first exit time $\tau^{\eta|b}(\bm x)$ is (roughly) of order
$
1/\eta^{ 1 + \mathcal J^I_b \cdot (\alpha - 1)  }
$
for small $\eta$, with $\alpha = 1.2$ being the index of $\bm Z_i$'s regular variation.
This means that the order of the first exit time $\tau^{\eta|b}(\bm x)$ does not vary continuously with respect to the truncation threshold $b$.
Instead, it exhibits a discrete dependence on $b$ through the integer-valued quantity $\mathcal J^I_b$.
Consequently, the wider the domain $I$, the \emph{asymptotically longer} the exit time $\tau^{\eta|b}(\bm x)$ will be.
In the companion paper \cite{wangSGDpaper2}, we build on these phase transitions in exit times to reveal an intriguing global behavior of $\bm X_j^{\eta|b}$ over a multi-well potential: the distribution of $\bm X_j^{\eta|b}$'s sample path closely resembles a Markov chain that \textbf{completely avoids narrow local minima}; see Figure~\ref{fig: exit time} (iii) and (iv).
More importantly, we demonstrate in \cite{wangSGDpaper2} that
such global dynamics under truncated heavy tails are intimately related to the generalization performance of deep neural networks.
See Section~\ref{subsec: numerical examples} for more details of the numerical experiments presented in Figure~\ref{fig: exit time}.


Our approach to the metastability analysis hinges on the concept of asymptotic atoms,
a general machinery we develop in Section \ref{subsec: framework, first exit time analysis}.
Asymptotic atoms are nested regions of recurrence at which the process asymptotically regenerates upon each visit.
Our locally uniform sample-path large deviations then prove to be the right tool in this framework,
empowering us to characterize the behavior of the stochastic processes uniformly for all initial values over the asymptotic atoms.
It should be noted that
\cite{imkeller2009exponential} also investigated the exit events driven by multiple truncated jumps. 
However, in their context, the mechanism through which multiple jumps arise is due to a different tail behavior of the increment distribution that is lighter than any polynomial rate---more precisely, a Weibull tail---and it is fundamentally different from that of the regularly varying case.
(See also \cite{imkeller2010hierarchy} for the summary of the hierarchy in the asymptotics of the first exit times for heavy-tailed dynamics.)
Our results complement the picture and provide a missing piece of the puzzle by unveiling the phase transitions in the exit times under truncated regularly varying perturbations.
\\

Some of the the metastability analysis in Section~\ref{sec: first exit time simple version} of this paper have been presented in a preliminary form at a conference \cite{wang2022eliminating}.
The main focus of \cite{wang2022eliminating} was the connection between the metastability analysis of stochastic gradient descent (SGD) and its generalization performance in the context of training deep neural networks. 
Compared to the brute force approach in \cite{wang2022eliminating},
the current paper provides a systematic framework to characterize the global dynamics for significantly more general class of heavy-tailed dynamical systems.\\

The rest of the paper is organized as follows. 
Section~\ref{sec: main results} presents the main results of this paper,
with numerical examples collected in Section~\ref{subsec: numerical examples}.
Section~\ref{sec: LD of SGD, proof} and Section~\ref{sec: first exit time simple version, proof}
provide the proofs of 
Sections~\ref{sec: M convergence, asymptotic equivalence}, \ref{subsec: LD, SGD}, and \ref{sec: first exit time simple version}.
Results for stochastic difference equations under more general scaling regimes are presented in Appendix~\ref{sec: appendix, SGD, general scaling, results}.
Results for SDEs driven by L\'evy processes with regularly varying increments are collected in Appendix~\ref{sec: appendix, SDE reults}.

\section{Main Results}
\label{sec: main results}
This section presents the main results of this paper and discusses their implications. 
Section \ref{sec: M convergence, asymptotic equivalence} introduces the uniform version of $\M(\S\setminus \C)$-convergence and presents an associated portmanteau theorem.
Section~\ref{subsec: LD, SGD} develops the sample-path large deviations,
and
Section~\ref{sec: first exit time simple version} carries out the metastability analysis.
Section~\ref{subsec: numerical examples} presents numerical examples of our theoretical results.
All the proofs are deferred to the later sections.

Before presenting the main results, we set frequently used notations.
Let $\notationdef{set-for-integers-below-n}{[n]} \delequal \{1,2,\cdots,n\}$ for any positive integer $n$.
Let $\notationdef{notation-non-negative-numbers}{\mathbb{N}} = \{0,1,2,\cdots\}$ be the set of non-negative integers.
Let $(\mathbb{S},\bm{d})$ be a metric space with 
$\notationdef{notation-borel-sigma-algebra}{\mathscr{S}_\mathbb{S}}$
being the corresponding Borel $\sigma$-algebra.
For any $E\subseteq \mathbb{S}$,
let
$\notationdef{notation-interior-of-set-E}{E^\circ}$ and $\notationdef{notation-closure-of-set-E}{E^-}$ be the interior and closure of $E$, respectively.
For any $r > 0$, 
let
$\notationdef{notation-epsilon-enlargement-of-set-E}{E^r} \delequal 
\{ y \in \mathbb{S}:\ \bm{d}(E,y)\leq r\}$ be the $r$-enlargement of a set $E$.
Here for any set $A \subseteq \mathbb{S}$ and any $x \in \mathbb{S}$,
we define $\bm{d}(A,x) \delequal \inf\{\bm{d}(y,x):\ y \in A\}$.
Also, let
$
\notationdef{notation-epsilon-shrinkage-of-set-E}{E_{r}} \delequal
((E^c)^r)^\complement
$
be the $r$-shrinkage
of $E$.
Note that for any $E$, the enlargement $E^r$ of $E$ is closed, and the shrinkage $E_r$ of $E$ is open.
We say that set $A \subseteq \mathbb{S}$ is bounded away from another set $B \subseteq \mathbb{S}$
if $\inf_{ x\in A,y\in B }\bm{d}(x,y) > 0$.
For any Borel measure $\mu$ on $(\mathbb{S},\mathscr{S}_{\mathbb{S}})$,
let the support of $\mu$
(denoted as $\notationdef{notation-support-of-mu}{\text{supp}(\mu)}$)
be the smallest closed set $C$ such that $\mu(\S \setminus C)= 0$.
For any function $g: \mathbb{S} \to \R$, let
$
\notationdef{notation-support-of-function-g}{\text{supp}(g)} \delequal \big(\{ x \in \mathbb{S}:\ g(x) \neq 0 \}\big)^-.
$
Given two sequences of positive real numbers $(x_n)_{n \geq 1}$ and $(y_n)_{n \geq 1}$, 
we say that $x_n = \bm{O}(y_n)$ (as $n \to \infty$) if there exists some $C \in [0,\infty)$ such that $x_n \leq C y_n\ \forall n\geq 1$.
Besides, we say that $x_n = \bm{o}(y_n)$ if $\lim_{n \rightarrow \infty} x_n/y_n = 0$.

\subsection{Uniform $\mathbb M(\mathbb S\setminus \mathbb C)$-Convergence}
\label{sec: M convergence, asymptotic equivalence}
This section extends the notion of $\mathbb M(\mathbb S \setminus \mathbb C)$-convergence \cite{lindskog2014regularly, rhee2019sample} to a uniform version and prove an associated portmanteau theorem. 
Such developments pave the way to the locally uniform heavy-tailed sample-path large deviations.

Specifically, in this section we consider some metric space $(\mathbb{S},\bm{d})$ that is complete and separable.
Given any Borel measurable subset $\mathbb{C} \subseteq \mathbb{S}$,
let $\mathbb{S}\setminus \mathbb{C}$
be the metric subspace of $\mathbb{S}$ in the relative topology with
$\sigma$-algebra
$
\mathscr{S}_{ \mathbb{S}\setminus \mathbb{C}}
\delequal 
\{ A \in \mathscr{S}_{\mathbb{S}}:\ A \subseteq \mathbb{S}\setminus \mathbb{C}\}.
$
Let
\begin{align*}
    \notationdef{notation-M-S-exclude-C}{\mathbb{M}(\S\setminus \mathbb{C})}
    \delequal 
    \{
    \nu(\cdot)\text{ is a Borel measure on }\mathbb{S}\setminus \mathbb{C} :\ \nu(\mathbb{S}\setminus \mathbb{C}^r) < \infty\ \forall r > 0
    \}.
\end{align*}
$\mathbb{M}(\mathbb{S}\setminus \mathbb{C})$\linkdest{notation-M-convergence}
can be topologized by the sub-basis constructed using sets of form
$
\{
\nu \in \mathbb{M}(\mathbb{S}\setminus \mathbb{C}):\ \nu(f) \in G
\},
$
where $G \subseteq [0,\infty)$ is open, $f \in \mathcal{C}({ \mathbb{S}\setminus \mathbb{C} })$,
and
$
\notationdef{notation-mathcal-C-S-exclude-C}{\mathcal{C}({ \mathbb{S}\setminus \mathbb{C} })}
$
is the set of all real-valued, non-negative, bounded and continuous functions with support bounded away from $\mathbb{C}$ (i.e., $f(x) = 0\ \forall x \in \mathbb{C}^r$ for some $r > 0$).
Given a sequence $\mu_n \in \M(\S\setminus\C)$ and some $\mu \in \M(\S\setminus\C)$,
we say that $\mu_n$ converges to $\mu$ in $\M(\S\setminus\C)$ as $n \to \infty$
if
$
\lim_{n \to \infty}|\mu_n(f) - \mu(f)| = 0
$
for all $f \in \mathcal C(\S\setminus\C)$.
See \cite{lindskog2014regularly} for equivalent definitions in the form of a Portmanteau Theorem.
When the choice of $\mathbb S$ and $\mathbb C$ is clear from the context, we simply refer to it as $\mathbb M$-convergence.
As demonstrated in \cite{rhee2019sample}, the sample path large deviations for heavy-tailed stochastic processes can be formulated in terms of $\mathbb M$-convergence of the scaled process in the Skorokhod space. 
In this paper, we introduce a stronger version of $\mathbb M$-convergence, which facilitates the metastability analysis in the later sections. 
\begin{definition}
[Uniform $\mathbb{M}(\mathbb S \setminus \mathbb C)$-convergence]
\label{def: uniform M convergence}
Let $\Theta$ be a set of indices.
Let $\mu^\eta_\theta,\ \mu_\theta \in \mathbb{M}(\mathbb{S}\setminus \mathbb{C})$ for each $\eta>0$ and $\theta\in \Theta$.
We say that $\mu_\theta^\eta$ converges to $\mu_\theta$ in $\mathbb M(\mathbb S\setminus\mathbb C)$ uniformly in $\theta$ on $\Theta$ as $\eta\downarrow 0$ if
\begin{align*}
    \lim_{\eta \downarrow 0}\sup_{\theta \in \Theta}| \mu^\eta_\theta(f) - \mu_\theta(f) | = 0
    \qquad \forall f \in \mathcal{C}({ \mathbb{S}\setminus \mathbb{C} }).
\end{align*}
\end{definition}
If $\{\mu_\theta:\ \theta \in \Theta\}$ is sequentially compact, a Portmanteau-type theorem holds.
The proof is provided in Section~\ref{subsec: proof of portmanteau theorem for uniform M convergence}.

\begin{theorem}[Portmanteau theorem for uniform $\M(\S \setminus \C)$-convergence]
\label{theorem: portmanteau, uniform M convergence}
\linksinthm{theorem: portmanteau, uniform M convergence}
Let $\Theta$ be a set of indices.
Let $\mu^\eta_\theta,\ \mu_\theta \in \mathbb{M}(\mathbb{S}\setminus \mathbb{C})$ for each $\eta>0$ and $\theta\in \Theta$.
Suppose that for any sequence of measures $(\mu_{\theta_n})_{n \geq 1}$,
there exist a sub-sequence $(\mu_{\theta_{n_k}})_{k \geq 1}$
and some $\theta^* \in \Theta$ such that
\begin{align}
    \lim_{k \to \infty}\mu_{\theta_{n_k}}(f) = \mu_{\theta^*}(f)
    \qquad 
    \forall f \in \mathcal C(\S \setminus \C).
    \label{assumption in portmanteau, uniform M convergence}
\end{align}
Then the next three statements are equivalent:
\begin{enumerate}[(i)]
    \item  
        $\mu_\theta^\eta$ converges to $\mu_\theta$ in $\mathbb M(\mathbb S\setminus\mathbb C)$ uniformly in $\theta$ on $\Theta$ as $\eta \downarrow 0$;
    \item 
        $\lim_{\eta \downarrow 0}\sup_{\theta \in \Theta}| \mu^\eta_\theta(f) - \mu_\theta(f) | = 0$ for each $f \in \mathcal{C}(\mathbb S\setminus \mathbb C)$
        that is also uniformly continuous on $\mathbb S$;
         
    \item 
        $\limsup_{\eta \downarrow 0}\sup_{\theta \in \Theta}\mu^\eta_\theta(F) - \mu_\theta(F^\epsilon) \leq 0$
        and
        $\liminf_{\eta \downarrow 0}\inf_{\theta \in \Theta}\mu^\eta_\theta(G) - \mu_\theta(G_\epsilon) \geq 0$
        for all $\epsilon > 0$, all closed $F \subseteq \mathbb{S}$ that is bounded away from $\mathbb{C}$,
        and all open
        $G \subseteq \mathbb{S}$ that is bounded away from $\mathbb{C}$.

\end{enumerate}
Furthermore, any of the claims $(i)$--$(iii)$ implies the following.
\begin{enumerate}[(i)]\addtocounter{enumi}{3}
    \item  
         $\limsup_{\eta \downarrow 0}\sup_{\theta \in \Theta}\mu^\eta_\theta(F) \leq \sup_{\theta \in \Theta} \mu_\theta(F)$
         and
         $\liminf_{\eta \downarrow 0}\inf_{\theta \in \Theta}\mu^\eta_\theta(G) \geq \inf_{\theta \in \Theta} \mu_\theta(G)$
         for all closed $F \subseteq \mathbb{S}$ that is bounded away from $\mathbb{C}$
        and all open
        $G \subseteq \mathbb{S}$ that is bounded away from $\mathbb{C}$.
\end{enumerate}
\end{theorem}

\begin{remark}
   We provide two additional remarks regarding Theorem~\ref{theorem: portmanteau, uniform M convergence}.
First, it is generally not possible to strengthen statement $(iii)$ and assert that
\begin{align}
    \limsup_{\eta \downarrow 0}\sup_{\theta \in \Theta}\mu^\eta_\theta(F) - \mu_\theta(F) \leq 0,
    \qquad 
    \liminf_{\eta \downarrow 0}\inf_{\theta \in \Theta}\mu^\eta_\theta(G) - \mu_\theta(G) \geq 0
    \label{false claim, portmanteau theorem for uniform M convergence}
\end{align}
        for all closed $F \subseteq \mathbb{S}$ bounded away from $\mathbb{C}$
        and all open
        $G \subseteq \mathbb{S}$ bounded away from $\mathbb{C}$.
In other words, in statement $(iii)$ the $\epsilon$-fattening in $F^\epsilon$ and $\epsilon$-shrinking in $G_\epsilon$ are indispensable.
Indeed, we demonstrate through a counterexample that, due to the infinite cardinality of the collections of measures $\{\mu^\eta_\theta: \theta \in \Theta\}$ and $\{\mu_\theta:\ \theta \in \Theta\}$, the claims in \eqref{false claim, portmanteau theorem for uniform M convergence} fall apart while statements $(i)$--$(iii)$ hold true.
Specifically, by setting $\mathbb C = \emptyset$ and $\mathbb S = \R$, the $\mathbb M(\mathbb S\setminus\mathbb C)$-convergence degenerates to the weak convergence of
Borel measures on $\R$.
Set $\Theta = [-1,1]$ and
\begin{align*}
    \mu^\eta_\theta \delequal \bm \delta_{\theta - \eta},
    \qquad 
    \mu_\theta\delequal \bm \delta_{\theta},
\end{align*}
where $\bm \delta_x$ is the Dirac measure at $x$. For closed set $F = [-1,0]$ and any $\eta \in (0,2)$,
\begin{align*}
    \sup_{\theta \in \Theta}\mu^\eta_\theta(F) - \mu_\theta(F)
    & \geq 
   \bm \delta_{ -\eta/2  }\big([-1,0]\big)
    -
    \bm{\delta}_{\eta/2}\big([-1,0]\big)
    \qquad\qquad 
    \text{by picking } \theta = \eta/2
    \\ 
    & = \mathbbm{I}\bigg\{ \frac{-\eta}{2} \in [-1,0] \bigg\}
    -
    \mathbbm{I}\bigg\{ \frac{\eta}{2} \in [-1,0] \bigg\}
    =1,
\end{align*}
thus implying 
$\limsup_{\eta \downarrow 0}\sup_{\theta \in \Theta}\mu^\eta_\theta(F) - \mu_\theta(F) \geq 1$.

Secondly, while statement $(iv)$ holds as the key component when establishing the sample-path large deviation results,
it is indeed strictly weaker than the other claims for one obvious reason:
unlike statements $(i)$--$(iii)$, the content of statement $(iv)$ does not require $\mu^\eta_\theta$ to converge to $\mu_\theta$ for any given $\theta \in \Theta$.
To illustrate that $(iv)$ does not imply $(i)$--$(iii)$, it suffices to examine the case where $\mathbb C = \emptyset$, $\mathbb S = \R$, $\Theta = [-1,1]$, $\mu^\eta_\theta = \bm{\delta}_{-\theta}$, and $\mu_\theta = \bm{\delta}_\theta$.   
\end{remark}


To conclude this section, we note that 
the proof of $\mathbb M$-convergence (and hence the application of Theorem~\ref{theorem: portmanteau, uniform M convergence})
is often facilitated by
the notion of asymptotic equivalence between two families of random objects.
Specifically,
we consider the following version of asymptotic equivalence that generalizes Definition 2.9 in \cite{10.1214/24-EJP1115},
which is particularly useful in the context of Lemma~\ref{lemma: asymptotic equivalence when bounded away, equivalence of M convergence}.
The proof of  Lemma~\ref{lemma: asymptotic equivalence when bounded away, equivalence of M convergence}
will be provided in Section~\ref{sec: M convergence, asymptotic equivalence}.

\begin{definition}[Asymptotic Equivalence]
\label{def: uniform asymptotic equivalence}
Let $X_n$ and $Y^\delta_n$ be random elements taking values in a complete separable metric space $(\mathbb{S},\bm{d})$ and supported on the same probability space.
Let $\epsilon_n$ be a sequence of positive real numbers.
Let $\mathbb{C} \subseteq \mathbb{S}$ be Borel measurable.
$X_n$ is said to be \textbf{\notationdef{asymptotic-equivalence}{asymptotically equivalent} to $Y^\delta_n$ in $\mathbb M(\mathbb S \setminus \mathbb{C})$ with respect to $\epsilon_n$ as $\delta \downarrow 0$} 
if the following holds:
given $\Delta > 0$ and $B \in \mathscr{S}_\mathbb{S}$ that is bounded away from $\mathbb{C}$,
\begin{align*}
    \lim_{\delta \downarrow 0}\lim_{ n \rightarrow \infty } \epsilon^{-1}_n {\P\Big( \bm{d}\big(X_n, Y^\delta_n\big)\mathbbm{I}\big( X_n \in B\text{ or }Y^\delta_n \in B \big) > \Delta \Big)}= 0.
\end{align*}
\end{definition}


\begin{lemma}
\label{lemma: asymptotic equivalence when bounded away, equivalence of M convergence}
\linksinthm{lemma: asymptotic equivalence when bounded away, equivalence of M convergence}
Let $X_n$ and $Y^\delta_n$ be random elements taking values in a complete separable metric space $(\mathbb{S},\bm{d})$ and supported on the same probability space.
Let $\mu \in \mathbb M(\mathbb S \setminus \mathbb C)$.
Suppose that 
\begin{enumerate}[(i)]
    \item 
        $X_n$ is asymptotically equivalent to $Y^\delta_n$ in $\mathbb M(\mathbb S\setminus\mathbb{C})$ with respect to $\epsilon_n$ as $\delta \downarrow 0$,

    \item 
        Given $B \in \mathscr{S}_\mathbb{S}$ that is bounded away from $\mathbb{C}$,
        it holds for all $\delta > 0$ small enough that
    \begin{align*}
        \limsup_{\delta \downarrow 0}\limsup_{n \to \infty}\epsilon^{-1}_n\P(Y^\delta_n \in B) \leq \mu(B^{-}),
        \qquad
        \liminf_{\delta \downarrow 0}
        \liminf_{n \to \infty}\epsilon^{-1}_n\P(Y^\delta_n \in B) \geq \mu(B^\circ).
    \end{align*}
\end{enumerate}
Then
$\epsilon^{-1}_n \P(X_n \in\cdot) \rightarrow \mu(\cdot)$ in $\mathbb{M}(\mathbb{S}\setminus \mathbb{C})$.
\end{lemma}

\subsection{Heavy-Tailed Large Deviations}
\label{subsec: LD, SGD}
In 
Section~\ref{subsubsec: untrucated, LD, SGD}, we study the sample-path large deviations for stochastic difference equations driven by heavy-tailed dynamics.
Section~\ref{subsubsec: conditional limit theorem} then characterizes the catastrophe principle of heavy-tailed systems by presenting
the conditional limit theorems.
The results reveal a discrete hierarchy of the most likely scenarios and probabilities of rare events in heavy-tailed stochastic difference equations.
We note that analogous results under more general scaling regimes and for stochastic differential equations are collected in Sections~\ref{sec: appendix, SGD, general scaling, results} and \ref{sec: appendix, SDE reults} of the Appendix.

\subsubsection{Sample-Path Large Deviations}
\label{subsubsec: untrucated, LD, SGD}

Let \notationdef{notation-Z-iid-noise-LDP}{$\bm Z_1,\bm Z_2,\ldots$} be iid copies of some random vector $\bm Z$ taking values in $\R^d$, and let
$\notationdef{notation-sigma-algebra-F}{\mathcal{F}}$ be the $\sigma$-algebra generated by $(\bm Z_j)_{j \geq 1}$. 
Henceforth in this paper, all vectors in Euclidean spaces are understood as column vectors unless stated otherwise.
Let $\mathcal{F}_j$ be the $\sigma$-algebra generated by $\bm Z_1,\bm Z_2,\cdots,\bm Z_j$ and $\mathcal{F}_0 \delequal \{\emptyset,\Omega\}$.
Let $(\Omega,\mathcal{F},\mathbb{F},\P)$ be {a} filtered probability space
with filtration $\notationdef{notation-F}{\mathbb{F}} = (\mathcal{F}_j)_{j \geq 0}$.
Given $b \in (0,\infty)$,
the drift coefficient $\notationdef{a}{\bm a}: \mathbb{R}^m \to \mathbb{R}^m$,
and the diffusion coefficient $\notationdef{sigma}{\bm \sigma}:\mathbb{R}^m\to \mathbb{R}^{m\times d}$,
our goal is to study the sample-path large deviations for the discrete-time process $\big\{ \notationdef{notation-X-eta-j-truncation-b-LDP}{\bm X^{\eta|b}_t(\bm x)}: t \in \mathbb N\big\}$ in $\R^m$ driven by the recursion
\begin{align}
    \bm X^{\eta|b}_0(\bm x) = \bm x,\qquad {\bm X^{\eta|b}_t(\bm x)} = \bm X^{\eta|b}_{t - 1}(\bm x) +  \varphi_b\Big(\eta \bm a\big(\bm X^{\eta|b}_{t - 1}(\bm x)\big) + \eta \bm \sigma\big(\bm X^{\eta|b}_{t - 1}(\bm x)\big)\bm Z_t\Big)\ \ \forall t \geq 1,
    \label{def: X eta b j x, clipped SGD}
\end{align}
where
the truncation operator $\varphi_\cdot(\cdot)$ is defined by
\begin{align}
    \notationdef{notation-truncation-operator-level-b}{\varphi_b}(\bm w) 
    \delequal{} 
    \Big(\frac{b}{\norm{\bm w}} \wedge 1\Big)\cdot \bm w
    \ \ \ \forall \bm w \neq \bm 0,
    \qquad
    {\varphi_b}(\bm 0) \delequal \bm 0. \label{defTruncationClippingOperator}
\end{align}
Here, $u\wedge v = \min\{u,v\}$ and $u \vee v = \max\{u,v\}$.
For any $\bm w \neq \bm 0$,
we have 
$
\varphi_b(\bm w) = (b \wedge \norm{\bm w}) \cdot \frac{\bm w}{\norm{\bm w}}.
$
In other words, the truncation operator $\varphi_b(\bm w)$ in \eqref{def: X eta b j x, clipped SGD} maintains the direction of the vector $\bm w$ but rescales it to ensure that the norm would not exceed the threshold value $b$.
In particular, we are interested in the case where $\bm Z_i$'s are heavy-tailed.
In this paper, we capture the heavy-tailed phenomena
with the notion of regular variation.
For any measurable function $\phi:(0,\infty) \to (0,\infty)$, we say that $\phi$ is regularly varying as $x \rightarrow\infty$ with index $\beta$ (denoted as $\phi(x) \in \RV_\beta(x)$ as $x \to \infty$) if $\lim_{x \rightarrow \infty}\phi(tx)/\phi(x) = t^\beta$ for all $t>0$. 
For details of the definition and properties of regularly varying functions, see, for example, 
\cite{bingham1989regular, resnick2007heavy, foss2011introduction, buraczewski2016stochastic}. 
Throughout this paper, we say that a measurable function $\phi(\eta)$
is regularly varying as $\eta \downarrow 0$ with index $\beta$ 
if $\lim_{\eta \downarrow 0} \phi(t\eta)/\phi(\eta) = t^\beta$ for any $t > 0$.
We denote this as $\phi(\eta) \in \notationdef{notation-RV-LDP}{\RV_{\beta}}(\eta)$ as $\eta \downarrow 0$.
Besides, we adopt the $L_2$ norm 
$
\norm{(x_1,\cdots,x_k)} = \sqrt{\sum_{j = 1}^k x^2_k}
$
on Euclidean spaces.
Let
\begin{align}
    \notationdef{notation-H}{H(x)} \delequal \P(\norm{\bm Z} > x). \label{def: H, law of Z_j}
\end{align}
For any $\alpha > 0$, let $\notationdef{notation-measure-nu-alpha}{\nu_\alpha}$ be the (Borel) measure on $(0,\infty)$ with
\begin{align}
     \nu_\alpha[x,\infty) = x^{-\alpha}. \label{def: measure nu alpha}
\end{align}
Let $\notationdef{notation-R-d-unit-sphere}{\mathfrak N_d} \delequal \{\bm x \in \R^d:\ \norm{\bm x} = 1\}$ be the unit sphere of $\R^d$.
Let $\Phi: \R^d \to [0,\infty) \times \mathfrak N_d$ be
\begin{align}
    \notationdef{notation-Phi-polar-transform}{\Phi(\bm x)} \delequal 
    \begin{cases}
         \Big(\norm{\bm x},\frac{\bm x}{\norm{\bm x}}\Big) &\text{ if }\bm x \neq 0,
         \\
         \big( 0, (1,0,0,\cdots,0)\big) & \text{ otherwise.}
    \end{cases}
    \label{def: Phi, polar transform in Rm}
\end{align}
Note that the origin is included in the domain of $\Phi$ simply to lighten the notations in the proofs. 
However, $\Phi(\bm x)$ will not be applied at $\bm x = \bm 0$ in our proofs. 
Thus, $\Phi$ can be interpreted as the polar transform with domain extended to $\bm 0$.
We impose the following multivariate regular variation assumption regarding the law of $\bm Z$.
\begin{assumption}[Regularly Varying Noises]\label{assumption gradient noise heavy-tailed}
$\E \bm Z = \bm 0$. 
Besides, there exist some $\notationdef{alpha-noise-tail-index-LDP}{\alpha} > 1$ and 
a probability measure $\mathbf S(\cdot)$ on the unit sphere $\mathfrak N_d$ such that
\begin{itemize}
    \item 
        $H(x) \in \RV_{-\alpha}(x)$ as $x \to \infty$,
    \item 
        for the polar coordinates $(R,\bm \Theta) \delequal \Phi(\bm Z)$, we have (as $x \to \infty$)
        \begin{align}
            \frac{
                \P\Big( (x^{-1}R, \bm \Theta) \in\ \cdot\ \Big)
            }{
                H(x)
            }
            \rightarrow
            \nu_\alpha \times \mathbf S
            \qquad 
            \text{in $\mathbb M\Big( 
            \big([0,\infty) \times \mathfrak N_d \big)
            \setminus
            \big( \{0\} \times \mathfrak N_d \big)
            \Big)$}.
            \label{claim, Rd heavy tailed assumption}
        \end{align}
\end{itemize}
\end{assumption}
\elaborate{
This is the $R^1$ version.
\begin{assumption}[Regularly Varying Noises]\label{assumption gradient noise heavy-tailed}
$\E Z = 0$. Besides, there exist $\notationdef{alpha-noise-tail-index-LDP}{\alpha} > 1$ and $\notationdef{notation-p-plus-and-minus}{p^{(+)},p^{(-)}} \in (0,1)$
with $p^{(+)} + p^{(-)} = 1$ such that
\begin{align*}
    H(x) \in \RV_{-\alpha}(x)\ \text{ as }x \rightarrow \infty;\ \ \ 
    \lim_{x \rightarrow \infty}\frac{H^{(+)}(x)}{H(x)} = p^{(+)};\ \ \ 
    \lim_{x \rightarrow \infty}\frac{H^{(-)}(x)}{H(x)} = p^{(-)} = 1 - p^{(+)}.
\end{align*}
\end{assumption}
}
\begin{remark}
The multivariate regular variation condition \eqref{claim, Rd heavy tailed assumption} is typically stated in terms of vague convergence; see, e.g., \cite{Resnick_2004, hult2006regular}.
    While vague convergence is generally weaker than $\mathbb M$-convergence (see Lemma~2.1 of \cite{lindskog2014regularly}),
    due to $\alpha > 1$ we have 
    $
    (\nu_\alpha \times \mathbf S)(A) < \infty
    $
    for any Borel set $A \subseteq (0,\infty) \times \mathfrak N_d$ that is bounded away from $\{0\} \times \mathfrak N_d$.
    Therefore, it is easy to verify that the $\mathbb M$-convergence stated in \eqref{claim, Rd heavy tailed assumption} is equivalent to vague convergence.
    Furthermore, by the alternative definitions for multivariate regular variation (see \cite{Resnick_2004, hult2006regular}),
    Assumption~\ref{assumption gradient noise heavy-tailed} is equivalent to the vague convergence of
    $
    H^{-1}(x)\P(x^{-1}\bm Z\in \ \cdot\ ) 
    $
    to some Borel measure $\mu(\cdot)$ in $\mathbb M(\R^d \setminus \{\bm 0\})$,
    where $\mu(\cdot)$ exhibits self-similarity in terms of $\mu(\lambda A) = \lambda^{-\alpha}\mu(A)$ for any Borel set $A \subseteq \R^d$ that is bounded away from the origin.
\end{remark}

Next,
we introduce the assumptions on
the drift coefficient $
\bm a(\cdot) = \big(a_1(\cdot),\cdots,a_m(\cdot)\big)^T
$ and
the
diffusion coefficient $
\bm\sigma(\cdot) = \big(\sigma_{i,j}(\cdot)\big)_{i \in [m], j \in [d]}.
$
Henceforth, we adopt the $L_2$ vector norm induced matrix norm
$
\norm{\textbf A} = \sup_{ \bm x \in \R^q:\ \norm{\bm x} = 1 }\norm{\textbf A\bm x}
$
for any $\textbf A \in \R^{p \times q}$.
{
Obviously, the lower bound for $D$ in Assumption \ref{assumption: lipschitz continuity of drift and diffusion coefficients} is not necessary,
and it is imposed w.l.o.g.\ for the notational simplicity in the proof.
}
\begin{assumption}[Lipschitz Continuity]
\label{assumption: lipschitz continuity of drift and diffusion coefficients}
There exists some $\notationdef{notation-Lipschitz-constant-L-LDP}{D} \in [1,\infty)$ such that
$$\norm{\bm \sigma(\bm x) - \bm \sigma(\bm y)} \vee \norm{\bm a(\bm x)-\bm a(\bm y)} \leq D\norm{\bm x - \bm y}\ \ \ \forall \bm x,\ \bm y \in \mathbb{R}^m.$$
\end{assumption}

To present the main results, we set a few notations.
Let $(\notationdef{notation-D-0T-cadlag-space}{\mathbb{D}{[0,T]}},\notationdef{notation-D-J1}{\dj{[0,T]}})$
be the metric space where $\mathbb{D}[0,{T}] = \D\big([0,T],\R^m\big)$ is the space of all càdlàg functions with domain $[0,{T}]$ and codomain $\R^m$,
and $\dj{[0,T]}$ is the Skorodkhod $J_1$ metric
\begin{align}
    \dj{[0,T]}(x,y) \delequal 
\inf_{\lambda \in \Lambda_T} \sup_{t \in [0,T]}|\lambda(t) - t|
\vee \norm{ x(\lambda(t)) - y(t) }.
\label{def: J 1 metric on D 0 T}
\end{align}
Here,
$
\Lambda_T
$
is the set of all homeomorphism on $[0,T]$.
Throughout this paper, we fix some $m$ and $d$ and consider $\bm X^{\eta|b}_t(\bm x)$ taking values in $\R^m$ driven by $\bm Z_t$'s in $\R^d$.
Given $A \subseteq \R$, 
let
$
\notationdef{order-k-time-on-[0,t]}{A^{k \uparrow}} \delequal 
\{
(t_1,\cdots,t_k) \in A^k:\ t_1 < t_2 < \cdots < t_k
\}
$
be the set of sequences of increasing real numbers on $A$ with length $k$.
For any $b$, $T \in (0,\infty)$ and $k \in \mathbb{N}$,
define the mapping 
$\notationdef{notation-h-k-t-bar-mapping-truncation-level-b-LDP}{\bar h^{(k)|b}_{[0,T]}}: \mathbb{R}^m\times \mathbb{R}^{d \times k} \times \R^{m \times k} \times (0,{T}]^{k\uparrow} \to \mathbb{D}{[0,T]}$ as follows.
Given
$\bm x \in \mathbb{R}^m$,
$\textbf{W} = (\bm w_1,\cdots,\bm w_k) \in \mathbb{R}^{d \times k}$, 
$\textbf V = (\bm v_1,\cdots,\bm v_k) \in \R^{m \times k}$,
and $\bm{t} = (t_1,\cdots,t_k)\in (0,T]^{k\uparrow}$, let $\xi = \bar h^{(k)|b}_{[0,T]}(\bm x,\textbf{W},\textbf V,\bm{t})$ be the solution to
\begin{align}
    \xi_0 & = \bm x; \label{def: perturb ode mapping h k b, 1}
    \\
    \frac{d\xi_s}{d s} & = \bm a(\xi_s)\ \ \ \forall s \in [0,{T}],\ s \neq t_1,t_2,\cdots,t_k; \label{def: perturb ode mapping h k b, 2}
    \\
    \xi_s & = \xi_{s-} + \bm v_j + \varphi_b\big( \bm \sigma(\xi_{s-} + \bm v_j)\bm w_j\big)\ \ \ \text{ if }s = t_j\text{ for some }j\in[k] \label{def: perturb ode mapping h k b, 3}
\end{align}
Similarly, define the mapping
$
\notationdef{notation-h-k-b-t-mapping-LDP}{h^{(k)|b}_{[0,T]}}:\R^m \times \R^{d \times k} \times (0,T]^{k \uparrow} \to \D[0,T]
$
by
\begin{align}
    h^{(k)|b}_{[0,T]}(\bm x, \bm W, \bm t)
    \delequal 
    \bar h^{(k)|b}_{[0,T]}\big(\bm x,\bm W, (\bm 0,\cdots,\bm 0), \bm t\big).
    \label{def: perturb ode mapping h k b, 4}
\end{align}
In essence,  $h^{(k)|b}_{[0,T]}(\bm x,\textbf{W},\bm{t})$ produces an ODE path perturbed by jumps $\bm w_1,\cdots,\bm w_k$ (with sizes modulated by $\bm \sigma(\cdot)$ and then truncated under threshold $b$) at times $t_1,\cdots,t_k$,
and the mapping $\bar h^{(k)|b}_{[0,T]}$ further includes perturbations $\bm v_j$'s right before each jump.
For $k = 0$,
we adopt the convention that $\xi = \bar h^{(0)|b}_{[0,T]}(\bm x)$ is the solution to the ODE
${d\xi_s}/{d s} = \bm a( \xi_s)\ \forall s \in [0,T]$
with the initial condition $\xi_0 = \bm x$.
For each $r > 0$ and $\bm x \in \R^m$, let $\notationdef{notation-ball-r-x}{\bar B_r(\bm x)} \delequal \{ \bm{y}\in\R^m:\ \norm{\bm y - \bm x} \leq r \}$
be the closed ball with radius $r$ centered at $\bm x$.
Given $b,T \in (0,\infty)$, $\epsilon \geq 0$, $A \subseteq \R^m$ and $k \in \mathbb N$,
let
\begin{align}
    \notationdef{notation-D-A-k-t-truncation-b-LDP}{\mathbb{D}_{A}^{(k)|b}{[0,T]}(\epsilon) } \delequal \bar h^{(k)|b}_{[0,T]}\Big( A \times \mathbb{R}^{m \times k} \times \big(\bar B_\epsilon(\bm 0)\big)^k \times (0,T]^{k\uparrow} \Big)
    \label{def: l * tilde jump number for function g, clipped SGD}
\end{align}
be the set that contains all the ODE path with $k$ jumps by time $T$, 
i.e., the image of the mapping $\bar h^{(k)|b}_{[0,T]}$ defined in \eqref{def: perturb ode mapping h k b, 1}--\eqref{def: perturb ode mapping h k b, 3}, 
under small perturbations $\norm{\bm v_j} \leq \epsilon$ for all $j \in [k]$.
By our definition of $\bar h^{(0)|b}_{[0,T]}$ above, $\D^{(0)|b}_A[0,T](\epsilon)$ simply contains all ODE paths under vector field $\bm a(\cdot)$ with initial values over $A$.
For $k = -1$, we adopt the convention that
$\mathbb{D}_{A}^{(-1)|b}[0,T](\epsilon) \delequal \emptyset$.
Also, note that
$
\mathbb{D}^{(k)|b}_A{[0,T]}(\epsilon) \subseteq \mathbb{D}^{(k)|b}_A{[0,T]}(\epsilon^\prime)
$
for any $0 \leq \epsilon < \epsilon^\prime$
and $k \geq -1$.
We state useful properties of $h^{(k)|b}_{[0,T]}$ and $\D^{(k)|b}_A[0,T](\epsilon)$ in Section~\ref{sec: appendix, mapping h} of the appendix.

For any $t > 0$, let 
$\notationdef{notation-lebesgue-measure-restricted}{\mathcal{L}_t}$ be the Lebesgue measure restricted on $(0,t)$ and $\notationdef{notation-lebesgue-measure-on-ordered-[0,t]}{\mathcal{L}^{k\uparrow}_t}$ be the Lebesgue measure restricted on $(0,t)^{k \uparrow}$.
Given $\bm x \in \mathbb{R}^m$, $k \in \mathbb N$ and $b,T \in (0,\infty)$,
define the Borel measure
\begin{align}
 \notationdef{notation-measure-C-k-t-truncation-b-LDP}{{\mathbf{C}}^{(k)|b}_{{[0,T]}}(\ \cdot \ ;\bm x)} \delequal &
   \int \mathbbm{I}\Big\{ h^{(k)|b}_{{[0,T]}}\big( \bm x,\textbf W,\bm t  \big) \in\ \cdot\  \Big\} 
   \big((\nu_\alpha \times \mathbf S)\circ \Phi\big)^k(d \textbf W) \times\mathcal{L}^{k\uparrow}_{ {T} }(d\bm t),
   \label{def: measure mu k b t}
\end{align}
where
$\mathbf S$ is the probability measure on the unit sphere $\mathfrak N_d$ characterized in Assumption~\ref{assumption gradient noise heavy-tailed},
$\nu_\alpha$ is specified in \eqref{def: measure nu alpha},
$(\nu_\alpha \times \mathbf S) \circ \Phi$
is the composition of the product measure $\nu_\alpha \times \mathbf S$ with the polar transform $\Phi$, i.e.,
\begin{align}
    \big((\nu_\alpha \times \mathbf S)\circ \Phi\big)(B) \delequal 
    (\nu_\alpha \times \mathbf S)\big( \Phi(B) \big)\qquad \forall \text{Borel set }B \subseteq \R^d\setminus\{\bm 0\},
    \label{def, nu alpha times S composition polar transform}
\end{align}
and $\big((\nu_\alpha \times \mathbf S)\circ \Phi\big)^k$ is the $k$-fold of $(\nu_\alpha \times \mathbf S)\circ \Phi$.
In other words, for $\textbf W = (\bm w_1,\cdots,\bm w_k) \in \R^{d \times k}$ with $\bm w_j \neq 0\ \forall j \in [k]$, we have
$
\big((\nu_\alpha \times \mathbf S)\circ \Phi\big)^k(d\textbf W) 
    = 
\bigtimes_{j \in [k]}\big((\nu_\alpha \times \mathbf S)\circ \Phi\big)(d \bm w_j).
$
Note that for any $\bm x \in A$, the measure ${{\mathbf{C}}^{(k)|b}_{{[0,T]}}(\ \cdot \ ;\bm x)}$ is supported on $\D^{(k)|b}_A[0,T](0)$.
Next, define the rate functionc
\begin{align*}
    \notationdef{notation-lambda-scale-function}{\lambda(\eta)} \delequal \eta^{-1}H(\eta^{-1})
\end{align*}
with $H(x) = \P(\norm{\bm Z} > x)$ defined in \eqref{def: H, law of Z_j}.
By Assumption~\ref{assumption gradient noise heavy-tailed}, $\lambda(\eta) \in \RV_{\alpha - 1}(\eta)$ as $\eta \downarrow 0$.
We write $\lambda^k(\eta) = \big(\lambda(\eta)\big)^k$.
For any $T,\ \eta,\ b \in (0,\infty)$, and $\bm x \in \R^m$,
let
$$
 \notationdef{notation-scaled-X-eta-mu-truncation-b-LDP}{{\bm{X}}^{\eta|b}_{[0,T]}(\bm x)} \delequal \{ \bm X^{\eta|b}_{ \floor{ t/\eta } }(\bm x):\ t \in [0,T] \}
$$
be the time-scaled version of $\bm X_j^{\eta|b}(\bm x)$ embedded in $\D[0,T]$,
with $\notationdef{floor-operator}{\floor{t}}\delequal \max\{n \in \mathbb{Z}:\ n \leq t\}$
and
$\notationdef{ceil-operator}{\ceil{t}} \delequal \min\{n \in \Z:\ n \geq t\}$.
In case that $T = 1$, we suppress the time horizon $[0,1]$ and write 
$
\notationdef{notation-D-0T-cadlag-space-T=1}{\mathbb{D}} \delequal \mathbb D[0,1],
$
$
\notationdef{notation-D-J1-T=1}{
\bm{d}_{J_1}
} \delequal \dj{[0,1]},
$
$\notationdef{notation-scaled-X-eta-mu-truncation-b-LDP-T=1}{{\bm{X}}^{\eta|b}(\bm x)} \delequal {\bm{X}}^{\eta|b}_{[0,1]}(\bm x)$,
$\notationdef{notation-h-k-t-sigma-mapping-truncation-level-b-LDP-T=1}{h^{(k)|b}}  \delequal h^{(k)|b}_{[0,1]}$,
$\notationdef{notation-D-A-k-t-truncation-b-LDP-T=1}{\mathbb{D}_{A}^{(k)|b}(\epsilon)}  \delequal \mathbb{D}_{A}^{(k)|b} [0,1](\epsilon)$,
and
$\notationdef{notation-measure-C-k-t-truncation-b-LDP-T=1}{\mathbf{C}^{(k)|b}}    \delequal \mathbf{C}^{(k)|b}_{[0,1]}$.
Now, we are ready to state Theorem~\ref{corollary: LDP 2},
which establishes the uniform $\mathbb M$-convergence for the law of $\bm X^{\eta|b}_{[0,T]}(\bm x)$ to $\mathbf C^{(k)|b}_{[0,T]}(\ \cdot\ ;\bm x)$ and a uniform version of the sample path large deviations for $\bm X^{\eta|b}_{[0,T]}(\bm x)$.

\begin{theorem} 
\label{corollary: LDP 2}
\linksinthm{corollary: LDP 2}
Under Assumptions \ref{assumption gradient noise heavy-tailed} and \ref{assumption: lipschitz continuity of drift and diffusion coefficients},
it holds for any $k \in \mathbb N$, any $b,T,\epsilon \in (0,\infty)$, and any compact $A\subset \R^m$ that
$$
 \lambda^{-k}(\eta) \P\big( \bm{X}^{\eta|b}_{[0,T]}(\bm x) \in\ \cdot\ \big) 
    \rightarrow 
    \mathbf{C}^{(k)|b}_{[0,T]}   (\ \cdot\ ; \bm x)
\quad
\text{in
$
\mathbb{M}\Big( \mathbb{D}[0,T]\setminus \mathbb{D}^{(k-1)|b}_{A}[0,T](\epsilon) \Big)
$
uniformly in {$\bm x$ on $A$}
}
$$
as $\eta \downarrow 0$.
Furthermore,
for any
$B \in \mathscr{S}_{\mathbb{D}[0,T]}$
that is 
bounded away from ${ \mathbb{D}}_{A}^{(k - 1)|b}[0,T](\epsilon)$ for some (and hence all) $\epsilon > 0$ small enough,
\begin{equation}\label{claim, uniform sample path LD, corollary: LDP 2}
    \begin{split}
        \inf_{\bm x \in A}
    \mathbf{C}^{(k)|b}_{[0,T]}\big( B^\circ; \bm x \big)
& \leq  \liminf_{\eta \downarrow 0}\frac{ \inf_{\bm x \in A}\P\big({\bm{X}}^{\eta|b}_{[0,T]}(\bm x) \in B \big) }{  \lambda^k(\eta)  } 
\\
   & \leq  \limsup_{\eta \downarrow 0}\frac{ \sup_{\bm x \in A}\P\big({\bm{X}}^{\eta|b}_{[0,T]}(\bm x) \in B \big) }{  \lambda^k(\eta)  } 
    \leq 
    \sup_{\bm x \in A}
    \mathbf{C}^{(k)|b}_{[0,T]}\big( B^-; \bm x \big)
    <
    \infty.
    \end{split}
\end{equation}
\end{theorem}

We provide the proof of Theorem~\ref{corollary: LDP 2} in Section~\ref{subsec: LDP clipped, proof of main results}.
Furthermore, by sending $b \to \infty$ in Theorem~\ref{corollary: LDP 2}, we are able to establish uniform sample path large deviations for the process
$\big\{ \notationdef{notation-X-j-eta-x}{\bm X^{\eta}_t(\bm x)}: t \in \mathbb N\big\}$ driven by the recursion
\begin{align}
    \bm X^\eta_0(\bm x) = \bm x;\qquad
    \bm X^\eta_t(\bm x) = \bm X^\eta_{t - 1}(\bm x) +  \notationdef{notation-eta}{\eta} \bm a\big(\bm X^\eta_{t - 1}(\bm x)\big) + \eta\bm \sigma\big(\bm X^\eta_{t - 1}(\bm x)\big)\bm Z_t\ \ \forall t \geq 1.
     \label{def: X eta b j x, unclipped SGD}
\end{align}
Note that the scalar version of the stochastic difference equation in \eqref{def: X eta b j x, unclipped SGD} is
\begin{align}
    X^\eta_{t,i}(x)
    =
    X^\eta_{t - 1,i}(x)
    +
    \eta a_i\big(\bm X^\eta_{t - 1}(x)\big)
    +
    \eta \sum_{j \in [d]} \sigma_{i,j}\big(\bm X^\eta_{t - 1}(x)\big)Z_{t,j}
    \qquad 
    \forall t \geq 1,\ i \in [m],    
    \nonumber
\end{align}
where
$
\bm a(\cdot) = \big(a_1(\cdot),\cdots,a_m(\cdot)\big)^T,
$
$
\bm\sigma(\cdot) = \big(\sigma_{i,j}(\cdot)\big)_{i \in [m], j \in [d]},
$
$
\bm X^\eta_t(x) = \big( X^\eta_{t,1}(x),\cdots,X^\eta_{t,m}(x) \big)^T,
$
and
$
\bm Z_t = (Z_{t,1},\cdots,Z_{t,d})^T.
$
By interpreting $\varphi_{\infty}(\bm w) = \bm w$ as the identity mapping in \eqref{defTruncationClippingOperator},
the definition of $\bm X^\eta_t(\bm x)$ in \eqref{def: X eta b j x, unclipped SGD} coincides with that of
$\bm X^{\eta|\infty}_t(\bm x)$
in \eqref{def: X eta b j x, clipped SGD} under the choice of $b = \infty$.
Analogously, we adopt the notations
$\notationdef{notation-h-k-t-bar-mapping-LDP}{\bar h^{(k)}_{[0,T]}} \delequal {\bar h^{(k)|\infty}_{[0,T]}}$,
$\notationdef{notation-h-k-t-mapping-LDP}{h^{(k)}_{[0,T]}}\delequal {h^{(k)|\infty}_{[0,T]}}$,
$
\notationdef{notation-measure-C-k-t-mu-LDP}{\mathbf{C}^{(k)}_{ {[0,T]} }(\ \cdot\ ;\bm x)} \delequal 
{\mathbf{C}^{(k)|\infty}_{ {[0,T]} }(\ \cdot\ ;\bm x)},
$
and
$
\notationdef{notation-D-A-k-t-LDP}{\mathbb{D}^{(k)}_A{[0,T]}(\epsilon) } \delequal {\mathbb{D}^{(k)|\infty}_A{[0,T]}(\epsilon) }.
$
For $k = -1$,
we again adopt the convention that $\mathbb{D}^{(-1)}_{A}[0,T](\epsilon) \delequal \emptyset$.
Define the time-scaled version of the sample path as
\begin{align}
    \notationdef{notation-scaled-X-0T-eta-LDP}{\bm{X}^\eta_{{[0,T]}}(\bm x)} \delequal \big\{ \bm X^\eta_{ \floor{ t/\eta }}(\bm x):\ t \in [0,{T}] \big\}\quad \forall T > 0. \label{def: scaled SGD, LDP}
\end{align}
In case that $T = 1$, we suppress the time horizon $[0,1]$ and write
$
\notationdef{notation-h-k-t-sigma-mapping-T=1}{h^{(k)}},
$
$
\notationdef{notation-measure-C-k-t-mu-LDP-T-1}{\mathbf{C}^{(k)}},
$
$
\notationdef{notation-D-A-k-t-LDP-T=1}{\mathbb{D}^{(k)}_A(\epsilon)},
$
and
$
\notationdef{notation-scaled-X-eta-LDP}{\bm{X}^{\eta}(x)}
$
to denote
$
h^{(k)}_{[0,1]},
$
$
\mathbf{C}^{(k)}_{[0,1]},
$
$
\mathbb{D}^{(k)}_A[0,1](\epsilon),
$
and
$
\bm{X}^{\eta}_{[0,1]}(x),
$
respectively.
Under Assumption~\ref{assumption: boundedness of drift and diffusion coefficients}, we establish in Theorem~\ref{theorem: LDP 1, unclipped} the uniform $\M$-convergence and sample path large deviations for $\bm X^\eta_{[0,T]}(\bm x)$.
Again, the lower bound for $C$ in Assumption~\ref{assumption: boundedness of drift and diffusion coefficients} is imposed w.l.o.g.\ simply for the convenience of the proof.

\begin{assumption}[Boundedness]  
\label{assumption: boundedness of drift and diffusion coefficients}
There exists some $\notationdef{notation-constant-C-boundedness-assumption}{C} \in [1,\infty)$ such that
\begin{align*}
    \norm{\bm a(\bm x)} \vee \norm{\bm \sigma(\bm x)} \leq  C\qquad \forall \bm x \in \mathbb{R}^m.
\end{align*}
\end{assumption}

\begin{theorem} 
\label{theorem: LDP 1, unclipped}
\linksinthm{theorem: LDP 1, unclipped}
Under Assumptions \ref{assumption gradient noise heavy-tailed}, \ref{assumption: lipschitz continuity of drift and diffusion coefficients}, and \ref{assumption: boundedness of drift and diffusion coefficients},
it holds for any $k \in \mathbb N$, $T>0$, $\epsilon > 0$, and any compact $A \subseteq \R^m$ that
$$
\lambda^{-k}(\eta) \P\big( \bm{X}^{\eta}_{[0,T]}(\bm x) \in\ \cdot\ \big) 
    \rightarrow 
    \mathbf{C}^{(k)}_{[0,T]}(\ \cdot\ ; \bm x)
    \quad 
    \text{in $\mathbb{M}\Big(\mathbb{D}[0,T]\setminus \mathbb{D}^{(k-1)}_A[0,T](\epsilon)\Big)$
uniformly in $\bm x$ on $A$}
$$
 as $\eta \downarrow 0$.
Furthermore, 
for any
$B \in \mathscr{S}_{\mathbb{D}[0,T]}$ that is
 bounded away from ${ \mathbb{D}}_{A}^{(k - 1)}[0,T](\epsilon)$ for some (and hence all) $\epsilon > 0$ small enough,
\begin{equation} \label{claim, uniform sample path LD, theorem: LDP 1, unclipped}
    \begin{split}
        \inf_{\bm x \in A}
    \mathbf{C}^{(k)}_{[0,T]}( B^\circ; \bm x)
& \leq  \liminf_{\eta \downarrow 0}\frac{ \inf_{\bm x \in A}\P\big({\bm{X}}^{\eta}_{[0,T]}(\bm x) \in B \big) }{  \lambda^k(\eta)  } 
\\
   & \leq  \limsup_{\eta \downarrow 0}\frac{ \sup_{\bm x \in A}\P\big({\bm{X}}^{\eta}_{[0,T]}(\bm x) \in B \big) }{  \lambda^k(\eta)  } 
    \leq 
    \sup_{\bm x \in A}
    \mathbf{C}^{(k)}_{[0,T]}( B^-; \bm x)
    <
    \infty.
    \end{split}
\end{equation}
\end{theorem}

\begin{remark}
  We add a remark on the connection between \eqref{claim, uniform sample path LD, corollary: LDP 2}
 \eqref{claim, uniform sample path LD, theorem: LDP 1, unclipped} and the classical LDP framework.
Given a measurable set $B \subseteq \D[0,T]$, there is a particular $k$ that plays the role of the rate function. 
Specifically, let $\D^{(k)}_A[0,T] = \D^{(k)}_A[0,T](0)$ and
$\mathcal J_{A}(B) \delequal \min\{k \in \mathbb N:\ B \cap \D^{(k)}_A[0,T] \neq \emptyset\}$.
In great generality, this coincides with the smallest possible value of $k \in \mathbb N$ for which the lower bound 
$
 \inf_{\bm x \in A}
    \mathbf{C}^{(k)}_{[0,T]}( B^\circ; \bm x)
$
in \eqref{claim, uniform sample path LD, theorem: LDP 1, unclipped}
is strictly positive,
and $\lambda^{\mathcal J_{A}(B)}(\eta)$ characterizes the exact rate of decay for both $\inf_{\bm x \in A}\P({\bm{X}}^{\eta}_{[0,T]}(\bm x) \in B)$
and
$\sup_{\bm x \in A}\P({\bm{X}}^{\eta}_{[0,T]}(\bm x) \in B)$
as $\eta \downarrow 0$.
It should be noted these results are exact asymptotics as opposed to the log asymptotics in classical LDP framework. 
In case that the set $A$ is a singleton (e.g., $A = \{\bm 0\}$), $T = 1$, $\bm a \equiv 0$, and $\bm \sigma \equiv \textbf I_m$ (i.e., the identity matrix in $\R^m$), the process ${\bm{X}}^{\eta}_{[0,T]}(\bm x)$ will degenerate to a L\'evy process, and
 $\mathcal J_{A}(\cdot)$ will reduce to $\mathcal J(\cdot)$ defined in equation (3.3) of \cite{rhee2019sample}.
 Furthermore, the condition of $B$ being bounded away from $\D^{(k-1)}_A(\epsilon)$ (for small $\epsilon > 0$) will reduce to that 
 $B$ is bounded away from the set of step functions (i.e., piece-wise constant functions) in $\D$, vanishing at the origin, with at most $k-1$ jumps.
 This confirms that Theorems~\ref{corollary: LDP 2} and \ref{theorem: LDP 1, unclipped} are proper generalizations of the heavy-tailed large deviations for L\'evy processes and random walks in \cite{rhee2019sample}.
\end{remark}

We provide the proof of Theorem~\ref{theorem: LDP 1, unclipped} in Section~\ref{subsec: LDP clipped, proof of main results}.
Here, we give a high-level description of the proof strategy for Theorems~\ref{corollary: LDP 2} and \ref{theorem: LDP 1, unclipped}.
\begin{itemize}
    \item 
        We first establish the asymptotic equivalence between $\bm X^{\eta|b}_{[0,T]}(\bm x)$ and an ODE perturbed by the top-$k$ ``largest'' noises in $(\bm Z_j)_{j \leq T/\eta}$ in terms of $\M$-convergence.
        The key technical tools are the concentration inequalities in Lemma \ref{lemma LDP, small jump perturbation} that tightly control the fluctuations in $\bm X^{\eta|b}_j(\bm x)$ between any two ``large'' $\bm Z_j$'s.

    \item
        Then, to complete the proof of Theorem~\ref{corollary: LDP 2}, it suffices to study the $\M$-convergence of this perturbed ODE.
        The foundation of this analysis is 
        the asymptotic law of the top-$k$ largest noises in $(\bm Z_j)_{j \leq T/\eta}$ studied in Lemma~\ref{lemma: weak convergence of cond law of large jump, LDP}.

    \item
        For $b$ sufficiently large,  $\bm X^{\eta}_j(x)$ would coincide with $\bm X^{\eta|b}_j(x)$ for the entire period of $j \leq T/\eta$, unless we have a large $\bm Z_j$ during this period.
        By sending $b \to \infty$ and analyzing the limits involved, we obtain the sample path large deviations for $\bm X^\eta_j(\bm x)$ and prove Theorem~\ref{theorem: LDP 1, unclipped}.
\end{itemize}
See Section \ref{subsec: LDP clipped, proof of main results} for the detailed proof and the rigorous definitions of the concepts involved.

\subsubsection{Catastrophe Principle}
\label{subsubsec: conditional limit theorem}
Perhaps the most important implication of the large deviations bounds is the identification of conditional distributions of the stochastic processes given the rare events of interest. 
This section precisely identifies the distributional limits of the conditional laws of $\bm X^\eta_{[0,T]}(\bm x)$ and $\bm X^{\eta|b}_{[0,T]}(\bm x)$.
In fact, the conditional limit theorem below follows immediately from the sample-path large deviations established above, i.e., \eqref{claim, uniform sample path LD, theorem: LDP 1, unclipped} and \eqref{claim, uniform sample path LD, corollary: LDP 2},
and Portmanteau Theorem.
While all the results in Section \ref{subsubsec: conditional limit theorem} can be easily extended to $\mathbb D[0,T]$ with arbitrary $T \in (0,\infty)$, we focus on $\mathbb D = \mathbb D\big([0,1],\R^m\big)$ for the sake of clarity of the presentation.

\begin{corollary}\label{corollary: conditional limit, SGD}
Let Assumptions \ref{assumption gradient noise heavy-tailed} and \ref{assumption: lipschitz continuity of drift and diffusion coefficients} hold.
\begin{enumerate}[(i)]
 \item Given $b > 0$, $k \in \mathbb N$, $\bm x \in \R^m$, and measurable $B \subseteq \mathbb D$, suppose that $B$ is bounded away from $\mathbb D^{(k-1)|b}_{\{\bm x\}}(\epsilon)$ for some (and hence all) $\epsilon > 0$ small enough, and
    $
    \mathbf{C}^{(k)|b}(B^\circ;\bm x) = \mathbf{C}^{(k)|b}(B^-;\bm x) > 0.
    $
    Then
    $$
    \P\big(\bm X^{\eta|b}_{[0,1]}(\bm x)\in \cdot\,|\,\bm X^{\eta|b}_{[0,1]}(\bm x)\in B\big) \Rightarrow
    \frac{\mathbf C^{(k)|b}(\,\cdot\cap B;\bm x)}{\mathbf C^{(k)|b}(B;\bm x)}
    \qquad \text{as }\eta \downarrow 0.
    $$ 

    \item Furthermore, suppose that Assumption \ref{assumption: boundedness of drift and diffusion coefficients} holds.
    Given $k\in \mathbb N$, $\bm x \in \R^m$, and measurable $B \subseteq \mathbb D$, suppose that $B$ is bounded away from $\mathbb D^{(k-1)}_{\{\bm x\}}(\epsilon)$ for some (and hence all) $\epsilon > 0$ small enough,
    and 
    $
    \mathbf{C}^{(k)}(B^\circ;\bm  x) = \mathbf{C}^{(k)}(B^-;\bm x) > 0.
    $
    Then
    $$\P\big(\bm X^{\eta}_{[0,1]}(\bm x)\in \cdot \,\big|\,\bm X^{\eta}_{[0,1]}(\bm x)\in B \big)\Rightarrow 
    \frac{\mathbf C^{(k)}(\cdot\cap B;\bm x)}{\mathbf C^{(k)}(B;\bm x)}
    \qquad \text{as }\eta \downarrow 0.
    $$ 
\end{enumerate}
\end{corollary}

\begin{remark}
Note that Corollary~\ref{corollary: conditional limit, SGD} is a sharp characterization of \emph{catastrophe principle} for $\bm X^{\eta|b}_{[0,1]}(\bm x)$ and $\bm X^{\eta}_{[0,1]}(\bm x)$. 
By definition of $\mathbf C^{(k)|b}$ in \eqref{def: measure mu k b t}, 
its support belongs to the set of paths of the form
\[
    h^{(k)|b}\big(\bm x, (\bm w_1,\cdots,\bm w_k ),(t_1,\cdots,t_k)\big),
\]
where the mapping $h^{(k)|b}$ is defined in \eqref{def: perturb ode mapping h k b, 1}--\eqref{def: perturb ode mapping h k b, 3}, 
and the norms $\norm{\bm w_j}$'s are bounded from below; see, for instance, Lemma~\ref{lemma: LDP, bar epsilon and delta} and \ref{lemma: LDP, bar epsilon and delta, clipped version}.
This is a clear manifestation of the catastrophe principle: 
whenever the rare event arises, the conditional distribution resembles the nominal path (i.e., the solution of the associated ODE) perturbed by precisely $k$ jumps.
In fact, the definition of $\mathbf C^{(k)|b}$ also implies that the the jump sizes are Pareto (modulated by $\bm \sigma(\cdot)$) and the jump times are uniform, conditional on the perturbed path belonging to $B$.
Similar interpretation applies to $\bm X^\eta_{[0,1]}(\bm x)$ in part (ii) of Corollary~\ref{corollary: conditional limit, SGD}.
\end{remark}

\subsection{Metastability Analysis}
\label{sec: first exit time simple version}
This section analyzes the metastability of $\bm X_j^\eta(\bm x)$ and $\bm X_j^{\eta|b}(\bm x)$.
Section~\ref{subsec: first exit time, results, SGD} establishes the scaling limits of their exit times.
Section~\ref{subsec: framework, first exit time analysis} introduces a framework that facilitates such analysis for general Markov chains.
Again, the results for stochastic differential equations and/or under more general scaling regimes are collected in the Appendix.

\subsubsection{First Exit Times and Locations}
\label{subsec: first exit time, results, SGD}
In this section,
we analyze the first exit times and locations of $\bm X^\eta_j(\bm x)$ and $\bm X^{\eta|b}_j(\bm x)$
from an attraction field of some potential with a unique local minimum at the origin.
Specifically, throughout Section \ref{subsec: first exit time, results, SGD}, we fix an open set $\notationdef{notation-exit-domain-I}{I} \subset \R^m$ 
that is bounded and contains the origin,
i.e., $\sup_{\bm x \in I}\norm{\bm x} < \infty$
and
$\bm 0 \in I$.
Let $\notationdef{notation-continuous-gradient-descent}{\bm{y}_t(\bm x)}$ be the solution of ODE
\begin{align}
\bm y_0(\bm x) = \bm x,\qquad
   \frac{d\bm{y}_t(\bm x)}{dt} = \bm a\big(\bm{y}_t(\bm x)\big) \ \ \forall t \geq 0.
   \label{def ODE path y t}
\end{align}
We impose the following assumption on the gradient field $\bm a:\ \R^m \to \R^m$.

\begin{assumption}
\label{assumption: shape of f, first exit analysis}
$\bm a(\bm 0) = \bm 0$.
The open set $I \subset \R^m$ contains the origin and is bounded, i.e.,
$\sup_{\bm x \in I}\norm{\bm x} < \infty$
and
$\bm 0 \in I$.
For all $\bm x \in I \setminus \{\bm 0\}$,
\begin{align*}
    \bm y_t(\bm x) \in I\ \ \forall t \geq 0,
    \qquad
    \lim_{t \to \infty}\bm y_t(\bm x) = \bm 0.
\end{align*}
Besides, it holds for all $\epsilon > 0$ small enough that
$
\bm a(\bm x)\bm x < 0\ \forall \bm x \in \bar B_\epsilon(\bm 0) \setminus \{\bm 0\}.
$

\end{assumption}

An immediate consequence of the condition $\lim_{t \to \infty}\bm y_t(\bm x) = \bm 0\ \forall \bm x \in I\setminus\{\bm 0\}$ is that $\bm a(\bm x) \neq \bm 0$ for all $\bm x \in I\setminus \{\bm 0\}$.
Of particular interest is the case where $\bm a(\cdot) = -\nabla U(\cdot)$
for some potential $U\in \mathcal{C}^1(\R^m)$
that has a unique local minimum at $\bm x = 0$ over the domain ${I}$.
In particular, Assumption~\ref{assumption: shape of f, first exit analysis} holds if $U$ is also locally $\mathcal C^2$ around the origin, and the Hessian of $U(\cdot)$ at the origin $\bm x = \bm 0$ is positive definite.
We note that Assumption~\ref{assumption: shape of f, first exit analysis} is a standard one in existing literature;
see e.g.\ \cite{doi:10.1142/S0219493711003413,imkeller2010first}.

Define 
\begin{align}
    \notationdef{notation-tau-eta-x-first-exit-time}{\tau^\eta(\bm x)} \delequal \min\big\{j \geq 0:\ \bm X^\eta_j(\bm x) \notin I\big\},
    \qquad
    \notationdef{notation-tau-eta-b-x-first-exit-time}{\tau^{\eta|b}(\bm x)} \delequal \min\big\{j \geq 0:\ \bm X^{\eta|b}_j(\bm x) \notin I \big\},
    \label{def: first exit time for heavy tailed SGD}
\end{align}
as the first exit time of $\bm X^\eta_j(\bm x)$ and $\bm X^{\eta|b}_j(\bm x)$ from $I$, respectively.
To facilitate the presentation of the main results, we introduce a few concepts.
Define the mapping 
$
\bar g^{(k)|b}: \R^m \times \R^{d \times k} \times \R^{m \times k} \times (0,\infty)^{k\uparrow} \to \R^m
$
as the location of the (perturbed) ODE with $k$ jumps at the last jump time:
\begin{align}
    \notationdef{notation-mapping-bar-g-k-b}{\bar g^{(k)|b}\big( \bm x, \textbf W, \textbf V, (t_1,\cdots,t_k)\big)}
    \delequal 
    \bar h^{(k)|b}_{ [0,t_k + 1] }
    \Big(
        \bm x,
        \textbf W,
        \textbf V,
        (t_1,\cdots,t_k)
    \Big)(t_k),
    \label{def: bar g k b mapping, metastability}
\end{align}
where $\bar h^{(k)|b}_{[0,T]}$ is the perturbed ODE mapping defined in \eqref{def: perturb ode mapping h k b, 1}--\eqref{def: perturb ode mapping h k b, 3}.
Note that the definition remains the same if, in \eqref{def: bar g k b mapping, metastability}, we use mapping $\bar h^{(k)|b}_{[0,T]}$
with any $T \in [t_k,\infty)$ instead of $\bar h^{(k)|b}_{[0,t_k + 1]}$.
We include a $+1$ offset only to extend the time range of the mapping and simplify some arguments in our proofs.
Besides,
define 
$
\widecheck{g}^{(k)|b}: \R^m \times \R^{d\times k} \times (0,\infty)^{k\uparrow} \to \R^m
$
by
\begin{align}
    \notationdef{notation-check-g-k-b}{\widecheck{g}^{(k)|b}(\bm x,\textbf W,\bm t)}
    \delequal 
    \bar g^{(k)|b}\big(\bm x, \textbf W, (\bm 0,\cdots,\bm 0), \bm t\big)
    =
    h^{(k)|b}_{[0,t_k+1]}(\bm x,\textbf W,\bm t)(t_{k}),
    \label{def: mapping check g k b, endpoint of path after the last jump, first exit analysis}
\end{align}
where $\bm t = (t_1,\ldots,t_k) \in (0,\infty)^{k\uparrow}$, and the mapping 
$h^{(k)|b}_{[0,T]}$ is defined in \eqref{def: perturb ode mapping h k b, 4}.
\elaborate{
That is,
for any $T > t_k$ and $\xi =  h^{(k)|b}_{[0,T]}(x,w_1,\cdots,w_k,t_1,\cdots,t_k)$,
we set $\widecheck{g}^{(k)|b}(x,w_1,\cdots,w_k,t_1,\cdots,t_k) = \xi(t_k)$, i.e., the value of path $\xi$ right after the last jump.
}
For $k = 0$,
we adopt the convention that $\bar{g}^{(0)|b}(\bm x) = \bm x$.
With mappings $\bar g^{(k)|b}$ defined, we are able to introduce 
(for any $k \geq 1$, $b> 0$, and $\epsilon \geq 0$)
\begin{align}
    \notationdef{notation-set-G-k-b-epsilon}{\mathcal G^{(k)|b}(\epsilon)} 
    & \delequal 
    \bigg\{
    \bar g^{(k - 1)|b}
    \Big( \bm v_1 + \varphi_b\big(\bm \sigma(\bm v_1)\bm w_1\big),
    (\bm w_2,\cdots, \bm w_k), (\bm v_2,\cdots,\bm v_k), \bm t 
    \Big):
    \nonumber
    \\ 
    & \qquad\qquad
    \textbf W = (\bm w_1,\cdots, \bm w_k) \in \R^{d\times k},
    \textbf V = (\bm v_1,\cdots, \bm v_k) \in \Big(\bar B_\epsilon(\bm 0)\Big)^k,
    \bm t \in (0,\infty)^{k - 1 \uparrow}
    \bigg\}
    \label{def: set G k b epsilon}
\end{align}
as the set covered by the $k^\text{th}$ jump of along ODE path initialized at the origin, with each jump modulated by $\bm \sigma(\cdot)$ and truncated under $b$ (and an $\epsilon$ perturbation right before each jump).
Here, the truncation operator $\varphi_b$ is defined in \eqref{defTruncationClippingOperator}, and $\bar B_{r}(\bm 0)$ is the closed ball with radius $r$ centered at the origin.
Under $\epsilon = 0$, we write
\begin{align*}
    \notationdef{notation-set-G-k-b}{\mathcal G^{(k)|b}}
    \delequal 
    \mathcal G^{(k)|b}(0)
    = 
    \bigg\{
    \widecheck{g}^{(k - 1)|b}
    \Big( \varphi_b\big(\bm \sigma(\bm 0)\bm w_1\big),
    (\bm w_2,\cdots, \bm w_k), \bm t 
    \Big):\ 
    \textbf W = (\bm w_1,\cdots,\bm w_k) \in \R^{d \times k},
    \bm t \in (0,\infty)^{k - 1 \uparrow}
    \bigg\}.
\end{align*}
Furthermore, as a convention for the case with $k = 0$, we set
\begin{align}
    \mathcal{G}^{(0)|b}(\epsilon) \delequal \bar B_{\epsilon}(\bm 0).
    \label{def: 0 jump coverage set, first exit times}
\end{align}
We note that $\mathcal G^{(k)|b}(\epsilon)$ is monotone in $\epsilon$, $k$, and $b$, in the sense that
$
\mathcal{G}^{(k)|b}(\epsilon) \subseteq \mathcal{G}^{(k)|b}(\epsilon^\prime) 
$
for all $0 \leq \epsilon \leq \epsilon^\prime$,
$
\mathcal{G}^{(k)|b}(\epsilon) \subseteq \mathcal{G}^{(k+1)|b}(\epsilon), 
$
and
$
\mathcal{G}^{(k)|b}(\epsilon) \subseteq \mathcal{G}^{(k)|b^\prime}(\epsilon) 
$
for all $0 < b \leq b^\prime$.

The intuition behind our metastability analysis (in particular, Theorem~\ref{theorem: first exit time, unclipped})
is as follows.
The characterization of the $k$-jump-coverage sets of form $\mathcal G^{(k)|b}$ reveals that, due to the truncation of  $\varphi_b(\cdot)$, the space reachable by ODE paths would expand as more jumps are added to the ODE path.
This leads to an intriguing phase transition for the law of the first exit times $\tau^{\eta|b}(\bm x)$ (as $\eta \downarrow 0$) in terms of the minimum number of jumps required for exit.
More precisely, let
\begin{align}
    \notationdef{notation-J-*-first-exit-analysis}{\mathcal J^I_b} \delequal
    \min\big\{ k \geq 1:\ \mathcal G^{(k)|b} \cap I^\complement \neq \emptyset \big\}
    \label{def: first exit time, J *}
\end{align}
be the smallest $k$ such that, under truncation at level $b$, the $k$-jump-coverage sets can reach outside the attraction field $I$.
Theorem~\ref{theorem: first exit time, unclipped} reveals a discrete hierarchy that 
the order of the first exit time $\tau^{\eta|b}(\bm x)$ and the limiting law of the exit location $\bm X^{\eta|b}_{ \tau^{\eta|b}(\bm x) }(\bm x)$ are dictated by this ``discretized width'' metric $\mathcal J^I_b$ of the domain $I$, relative to the truncation threshold $b$.
Here, the limiting law is characterized by measures
\begin{align}
     \notationdef{notation-check-C-k-b}{\widecheck{ \mathbf C }^{(k)|b}(\ \cdot\ )}
    & \delequal 
    \int \mathbbm{I}\bigg\{ \widecheck{g}^{(k-1)|b}\Big( \varphi_b\big(\bm\sigma(\bm x)\bm w_1\big),(\bm w_2,\cdots,\bm w_k),\bm t \Big) \in \ \cdot \  \bigg\}
     \big((\nu_\alpha \times \mathbf S)\circ \Phi\big)^k(d \textbf W) \times \mathcal{L}^{k-1\uparrow}_\infty(d\bm t),
    \label{def: measure check C k b}
\end{align}
where 
$\alpha > 1$ is the heavy-tail index in Assumption~\ref{assumption gradient noise heavy-tailed},
$\textbf W = (\bm w_1, \bm w_2, \cdots, \bm w_k) \in \R^{d \times k}$,
$\big((\nu_\alpha \times \mathbf S)\circ \Phi\big)^k$ is the $k$-fold of $(\nu_\alpha \times \mathbf S)\circ \Phi$
defined in \eqref{def, nu alpha times S composition polar transform},
and
\notationdef{notation-measure-L-k-up-infty}{$\mathcal{L}^{k\uparrow}_\infty$}
is the Lebesgue measure restricted on $\{ (t_1,\cdots,t_k) \in (0,\infty)^k:\ 0 < t_1 < t_2 < \cdots < t_k \}$.
Section~\ref{subsec: lemma for measure check C} collects useful properties of the mapping $\widecheck g^{(k)|b}$ and the measure ${\widecheck{ \mathbf C }^{(k)|b}}$.

Recall that $H(\cdot) = \P(\norm{\bm Z_1} > \cdot)$, $\lambda(\eta) = \eta^{-1}H(\eta^{-1})$,
and for any $k \geq 1$ we write $\lambda^k(\eta) = \big(\lambda(\eta)\big)^k$.
Recall that 
$
\notationdef{notation-I-epsilon-shrinkage}{I_\epsilon} = \{ \bm y:\ \norm{\bm x - \bm y} < \epsilon\ \Longrightarrow\ \bm x \in I \}
$
is the $\epsilon$-shrinkage of $I$.
As the main result of this section,
Theorem~\ref{theorem: first exit time, unclipped} provides sharp asymptotics for the joint law of first exit times and exit locations of  $\bm X^{\eta|b}_j(\bm x)$ and $\bm X^\eta_j(\bm x)$.
The proof of Theorem~\ref{theorem: first exit time, unclipped} is based on a general framework developed in Section~\ref{subsec: framework, first exit time analysis},
and 
we detail the proof in Section~\ref{sec: proof of proposition: first exit time}.

\begin{theorem}
\label{theorem: first exit time, unclipped}
\linksinthm{theorem: first exit time, unclipped}
\textbf{(First Exit Times and Locations: Truncated Case)}
    Let Assumptions \ref{assumption gradient noise heavy-tailed}, \ref{assumption: lipschitz continuity of drift and diffusion coefficients}, and \ref{assumption: shape of f, first exit analysis} hold.
    Let $b > 0$.
        Suppose that 
        $\mathcal J^I_b < \infty$,
        $I^c$ is bounded away from $\mathcal G^{(\mathcal J^I_b - 1)|b}(\epsilon)$ for some (and hence all) $\epsilon > 0$ small enough,
        and
        $
        \widecheck{\mathbf C}^{( \mathcal J^I_b )|b}(\partial I) = 0.
        $
        Then 
        $\notationdef{notation-C-b-*}{C^I_b} \delequal  \widecheck{ \mathbf{C} }^{ (\mathcal{J}^I_b)|b }(I^\complement) < \infty$.
        Furthermore, if $C^I_b \in (0,\infty)$,
        then
        for any $\epsilon > 0$, $t \geq 0$, and measurable set $B \subseteq I^c$,
    \begin{align*}
    \limsup_{\eta\downarrow 0}\sup_{\bm x \in I_\epsilon}
    \P\bigg(
        C^I_b \eta\cdot \lambda^{ \mathcal J^I_b }(\eta)\tau^{\eta|b}(\bm x) > t;\ 
        \bm X^{\eta|b}_{ \tau^{\eta|b}(\bm x)}(\bm x) \in B
     \bigg)
     & \leq \frac{ \widecheck{\mathbf{C}}^{ (\mathcal J^I_b)|b }(B^-) }{ C^I_b }\cdot\exp(-t),
     \\
     \liminf_{\eta\downarrow 0}\inf_{\bm x \in I_\epsilon}
    \P\bigg(
        C^I_b \eta\cdot \lambda^{ \mathcal J^I_b }(\eta)\tau^{\eta|b}(\bm x) > t;\ 
        \bm X^{\eta|b}_{ \tau^{\eta|b}(\bm x)}(\bm x) \in B
     \bigg)
     & \geq \frac{ \widecheck{\mathbf{C}}^{ (\mathcal J^I_b)|b }(B^\circ) }{ C^I_b }\cdot\exp(-t).
    \end{align*}
    Otherwise, we have $C^I_b = 0$, and
    \begin{align*}
        \limsup_{\eta\downarrow 0}\sup_{\bm x \in I_\epsilon}
    \P\bigg(
        \eta\cdot \lambda^{ \mathcal J^I_b }(\eta)\tau^{\eta|b}(\bm x) \leq t
     \bigg) = 0
     \qquad
     \forall \epsilon > 0,\ t \geq 0.
    \end{align*}
\end{theorem}

\begin{remark}
\label{remark: regularity conditions for first exit times}
Regarding the regularity conditions in Theorem~\ref{theorem: first exit time, unclipped},
conditions of form 
$\widecheck{\mathbf C}^{( \mathcal J^I_b )|b}(\partial I) = 0$ are standard even for metastability analyses of untruncated dynamics;
see e.g.\ \cite{doi:10.1142/S0219493715500197,
hogele2014exit}.
Besides, we note that theses conditions hold almost automatically in the non-degenerate one-dimensional settings:
suppose that $m = d = 1$ (so $\bm Z_j$'s and $\bm X^{\eta|b}_j(\bm x)$'s are random variables in $\R^1$) and for the diffusion coefficient $\sigma: \R \to \R$ we have $\inf_{x \in I}\sigma(x) > 0$;
then for (Lebesgue) almost every $b \in (0,\infty)$,
$I^c$ is bounded away from $\mathcal G^{(\mathcal J^I_b - 1)|b}(\epsilon)$ (for small $\epsilon > 0$),
$
        \widecheck{\mathbf C}^{( \mathcal J^I_b )|b}(\partial I) = 0,
        $
and $C^I_b \in (0,\infty)$
with $\mathcal J^I_b = \inf_{ x \notin I }\ceil{ |x|/b  }$.
See Lemmas~\ref{lemma: measure check C J * b, continuity, first exit analysis, R1} and \ref{lemma: exit rate strictly positive, first exit analysis, R1} in the Appendix.
\end{remark}
\begin{remark}
As noted in Section~\ref{subsec:overview-of-the-paper} and will be confirmed in Corollary~\ref{corollary: first exit time, untruncated case} below, $\mathcal J_b^I = 1$ when $b= \infty$,  regardless of the geometry of $\bm a(\cdot)$. 
In this case, Theorem~\ref{theorem: first exit time, unclipped} reduces to the manifestation of the principle of a single big jump. 
For $b \neq \infty$ and a contractive drift---i.e., $a(\bm x) \cdot \bm x \leq 0$ for all $\bm x \in I$---note that $\mathcal J_b^I = \lceil r / b \rceil$, where $r \triangleq \inf\{ \|\bm x - \bm 0\|: \bm x \in I^c\}$. 
This is because gradient flow will not bring $\bm X_j^{\eta|b}(\bm x)$ closer to $I^c$, and hence, the most efficient way to escape from $I$ is through $\lceil r / b \rceil$ consecutive jumps in the direction where $I^c$ is closest.
In the general case, however, $\mathcal J_b^I$ is determined as the solution to the discrete optimization problem in \eqref{def: first exit time, J *}, where the geometry of $\bm a(\cdot)$---in particular, gradient flows and their distances from $I^c$---plays a more sophisticated role. 
\end{remark}

We conclude this section by noting that the first exit analysis for untruncated process $\bm X_j^\eta(\bm x)$
(see e.g.\ \cite{imkeller2006first, pavlyukevich2008metastable, imkeller2008levy}
for analogous results for continuous processes) 
follows directly from Theorem~\ref{theorem: first exit time, unclipped}.
Let
\begin{align}
    \notationdef{notation-check-C}{\widecheck{\mathbf C}(\ \cdot\ )} \delequal \int \mathbbm{I}\Big\{ \bm \sigma(\bm 0) \bm w \in\ \cdot\ \Big\}
    \big((\nu_\alpha \times \mathbf S)\circ \Phi\big)(d \bm w).
    \label{def: measure check C}
\end{align}
The asymptotic analysis for exit times and locations of the untruncated dynamics $\bm X^\eta_j(\bm x)$ follows from the result for $\bm X^{\eta|b}_j(\bm x)$ by sending $b$ to $\infty$,
and the limiting laws of the exit location $\bm X^\eta_{\tau^\eta(\bm x)}(\bm x)$ is characterized by $\widecheck{\mathbf C}(\ \cdot\ )$, as presented in Corollary~\ref{corollary: first exit time, untruncated case}.
The proof is straightforward and we collect it in Section~\ref{subsec: lemma for measure check C}
for the sake of completeness.

\begin{corollary}\label{corollary: first exit time, untruncated case}
\linksinthm{corollary: first exit time, untruncated case}
\textbf{(First Exit Times and Locations: Untruncated Case)}
Let Assumptions \ref{assumption gradient noise heavy-tailed}, \ref{assumption: lipschitz continuity of drift and diffusion coefficients}, and \ref{assumption: shape of f, first exit analysis} hold.
        Suppose that $\widecheck{\mathbf C}(\partial I) = 0$ and $\norm{\bm \sigma(\bm 0)} > 0$.
        Then $\notationdef{notation-C-*-first-exit-time}{C^I_\infty} \delequal  \widecheck{ \mathbf{C} }(I^\complement) < \infty$.
        Furthermore,
        if $C^I_\infty > 0$,
        then
        for any $t \geq 0$, $\epsilon > 0$, and measurable set $B \subseteq I^c$,
    \begin{align*}
        \limsup_{\eta\downarrow 0}\sup_{\bm x \in I_\epsilon}
        \P\bigg(
            C^I_\infty H(\eta^{-1})\tau^\eta(\bm x) > t;\ \bm X^\eta_{ \tau^\eta(\bm x)}(\bm x) \in B
        \bigg)
        & \leq
        \frac{ \widecheck{\mathbf{C}}(B^-) }{ C^I_\infty }\cdot\exp(-t),
        \\ 
        \liminf_{\eta\downarrow 0}\inf_{\bm x \in I_\epsilon}
        \P\bigg(
            C^I_\infty H(\eta^{-1})\tau^\eta(\bm x) > t;\ \bm X^\eta_{ \tau^\eta(\bm x)}(\bm x) \in B
        \bigg)
        & \geq
        \frac{ \widecheck{\mathbf{C}}(B^\circ) }{ C^I_\infty }\cdot\exp(-t).
    \end{align*}
    Otherwise, we have $C^I_\infty = 0$, and 
    \begin{align*}
        \limsup_{\eta\downarrow 0}\sup_{\bm x \in I_\epsilon}
        \P\bigg(
            H(\eta^{-1})\tau^\eta(\bm x) \leq t
        \bigg) = 0
        \qquad
        \forall \epsilon > 0,\ t \geq 0.
    \end{align*}
\end{corollary}

\subsubsection{General Framework: Asymptotic Atoms}
\label{subsec: framework, first exit time analysis}
This section proposes a general framework that enables sharp characterization of exit times and exit locations of Markov chains.
The new heavy-tailed large deviations formulation introduced in Section~\ref{subsec: LD, SGD} is conducive to this framework.

Consider a general metric space $(\mathbb S,\bm d)$
and a family of $\mathbb S$-valued Markov chains $\big\{\{V_j^\eta(x): j\geq 0\}:\eta>0\big\}$ parameterized by $\eta$, where $x\in \mathbb S$ denotes the initial state and $j$ denotes the time index. 
We use
$
\bm V_{[0,T]}^\eta(x)\delequal \{V^\eta_{\lfloor t/\eta\rfloor}(x): t  \in [0,T]\}
$ 
to denote the scaled version of $\{V_j^\eta(x): j\geq 0\}$ as a  $\D[0,T]$-valued random element.
For a given set $E$, let $\tau^{\eta}_{E}(x) \delequal \min\{j\geq0: V_j^\eta(x) \in  E\}$ denote $\{V_j^\eta(s): j\geq 0\}$'s first hitting time of $E$.
We consider an asymptotic domain of attraction $I\subseteq \mathbb S$, within which $\bm V_{[0,T]}^\eta(x)$ typically (i.e., as $\eta\downarrow 0$) stays within $I$ throughout any fixed time horizon $[0,T]$ as far as the initial state $x$ is in $I$. 
However, if one considers an infinite time horizon, $V_{\boldsymbol{\cdot}}^\eta(x)$ is typically bound to escape $I$ eventually due to the stochasticity. 
The goal of this section is to establish an asymptotic limit of the joint distribution of the exit time $\tau^{\eta}_{I^\complement}(x)$ and the exit location $V^\eta_{\tau^\eta_{I^\complement}(x)}(x)$.
Throughout this section, we will
denote 
\(
    V^\eta_{\tau^\eta_{I(\epsilon)^\complement}(x)}(x)
\)
and 
\(
    V^\eta_{\tau^\eta_{I^\complement}(x)}(x)
\)
with 
\(
    V^\eta_{\tau_\epsilon}(x)
\)
and
\(
    V^\eta_{\tau}(x)
\),
respectively, for notation simplicity. 




We introduce the notion of asymptotic atoms to facilitate the analyses.
Let $\{I(\epsilon)\subseteq I: \epsilon>0\}$ and $\{A(\epsilon)\subseteq \mathbb S: \epsilon>0\}$ be collections of subsets of $I$ such that $\bigcup_{\epsilon>0}I(\epsilon) = I$ and $\bigcap_{\epsilon>0}A(\epsilon) \neq \emptyset$.
Let $C(\cdot)$ is a Borel measure on $\mathbb S\setminus I$ satisfying $C(\partial I) = 0$ that 
characterizes the (asymptotics limit of the) exit location of $V^\eta_{\cdot}(x)$.
Specifically, we consider two different cases for the location measure $C(\cdot)$:
\begin{enumerate}[$(i)$]
    \item
        $C(I^\complement) \in (0,\infty)$: by incorporating the normalizing constant $C(I^\complement)$ into the scale function $\gamma(\eta)$, we can assume w.l.o.g.\ that $C(\cdot)$ \textbf{is a probability measure}, and $C(B)$ dictates the limiting probability that $\P(V^\eta_\tau(x) \in B)$ as shown in Theorem~\ref{thm: exit time analysis framework};

    \item 
        $C(I^\complement) = 0$: as a result, $C(B) = 0$ for any Borel set $B \subseteq I^\complement$, and it is equivalent to stating that $C(\cdot)$ \textbf{is trivially zero}.
\end{enumerate}

\begin{definition} \label{def: asymptotic atom}
 $\big\{\{V_j^\eta(x): j\geq 0\}:\eta>0\big\}$ possesses an asymptotic atom $\{A(\epsilon)\subseteq \mathbb S: \epsilon>0\}$ associated with the domain $I$, location measure $C(\cdot)$, scale $\gamma:(0,\infty) \to (0,\infty)$, and covering $\{I(\epsilon)\subseteq I: \epsilon>0\}$ if the following holds: 
For each measurable set $B \subseteq \mathbb S$,
there exist $\delta_B:(0,\infty)\times(0,\infty)\to(0,\infty)$, $\epsilon_B>0$, and $T_B:(0,\infty) \to (0,\infty)$ such that 
\begin{align}
C(B^\circ) - \delta_B(\epsilon,T)
\leq\,
&
\liminf_{\eta\downarrow0} \frac{\ \inf_{x \in A(\epsilon)} \P\big(\tau^{\eta}_{I(\epsilon)^\complement}(x) \leq T/\eta;\; V_{\tau_{\epsilon}}^\eta(x)\in B\big)}{\gamma(\eta)T/\eta} 
\label{eq: exit time condition lower bound}
\\
\leq\,
&
\limsup_{\eta\downarrow0} \frac{\sup_{x \in A(\epsilon)} \P\big(\tau^{\eta}_{I(\epsilon)^\complement}(x) \leq T/\eta;\; V_{\tau_{\epsilon}}^\eta(x)\in B\big)}{\gamma(\eta)T/\eta} 
\leq C(B^-) + \delta_B(\epsilon,T)
\label{eq: exit time condition upper bound}
\\
&
\limsup_{\eta\downarrow 0} \frac{\sup_{x\in I(\epsilon)} \P\big( \tau^{\eta}_{(I(\epsilon)\setminus A(\epsilon))^\complement}(x) > T/\eta\big)} {\gamma(\eta) T/\eta} = 0
\label{eq:E3}
\\[7pt]
&
\liminf_{\eta\downarrow 0} \ \inf_{x\in I(\epsilon)} \P\big(\tau^{\eta}_{A(\epsilon)}(x) \leq T/{\eta}\big) = 1
\label{eq:E4}
\end{align}
for any $\epsilon \leq \epsilon_B$ and $T \geq T_B(\epsilon)$,
where $\gamma(\eta)/\eta \to 0$ as $\eta \downarrow 0$ and
$\delta_B$'s
are such that 
$$
\lim_{\epsilon\downarrow 0}\lim_{T\to\infty} {\delta_B(\epsilon,T)}
= 0.
$$
\end{definition}

To see how Definition~\ref{def: asymptotic atom} asymptotically characterize the atoms in $V^\eta_{ \boldsymbol{\cdot}}(x)$ for the first exit analysis from domain $I$,
note that
the condition~\eqref{eq:E4} requires the process to efficiently return to the asymptotic atoms $A(\epsilon)$.
The conditions~\eqref{eq: exit time condition lower bound} and \eqref{eq: exit time condition upper bound} then state that,
upon hitting the asymptotic atoms $A(\epsilon)$,
the process almost regenerates in terms of the law of the exit time $\tau^\eta_{I(\epsilon)^\complement}(x)$ and exit locations $V^\eta_{\tau_\epsilon}(x)$.
Furthermore, the condition~\eqref{eq:E3} prevents the process $V^\eta_{ \boldsymbol{\cdot}}(x)$ from spending a long time without either returning to the asymptotic atoms $A(\epsilon)$ or exiting from $I(\epsilon)$, which covers the domain $I$ as $\epsilon$ tends to $0$.

The existence of an asymptotic atom is a sufficient condition for characterization of exit time and location asymptotics as in Theorem~\ref{theorem: first exit time, unclipped}. 
To minimize repetition, we refer to the existence of an asymptotic atom---with specific domain, location measure, scale, and covering---Condition~\ref{condition E2} throughout the paper.
\begin{condition}\label{condition E2}
A family $\big\{\{V_j^\eta(x): j\geq 0\}:\eta>0\big\}$ of Markov chains possesses an asymptotic atom $\{A(\epsilon)\subseteq \mathbb S: \epsilon>0\}$ associated with the domain $I$, location measure $C(\cdot)$, scale $\gamma:(0,\infty) \to (0,\infty)$, and covering $\{I(\epsilon)\subseteq I: \epsilon>0\}$.
\end{condition}

Recall that, right before Definition~\ref{def: asymptotic atom}, we state that for the location measure $C(\cdot)$ we consider two cases that $(i)$ $C(I^\complement) = 1$ (more generally, $C(\cdot)$ is a finite measure), and $(ii)$ $C(I^\complement) = 0$.
The following theorem is the key result of this section. See Section~\ref{subsec: Exit time analysis framework} for the proof of the theorem.


\begin{theorem}
\label{thm: exit time analysis framework} 
\linksinthm{thm: exit time analysis framework} 
If Condition~\ref{condition E2} holds, then the first exit time $\tau_{I^\complement}^\eta(x)$ scales as $1/\gamma(\eta)$, and the distribution of the location $V_\tau^\eta(x)$ at the first exit time converges to $C(\cdot)$. 
Moreover, the convergence is uniform over $I(\epsilon)$ for any $\epsilon>0$.
That is, 
\begin{enumerate}[$(i)$]
    \item 
        If $C(I^\complement) = 1$, then for each $\epsilon>0$, measurable $B\subseteq I^\complement$, and $t\geq 0$,
        \begin{align*}
            C(B^\circ) \cdot e^{-t}
            & \leq 
            \liminf_{\eta\downarrow 0} \inf_{x\in I(\epsilon)}
            \P\big(
                \gamma(\eta)\tau_{I^\complement}^{\eta}(x)>t,\,V_{\tau}^\eta(x)\in B  
            \big) 
            \\
            & \leq 
            \limsup_{\eta\downarrow 0} \sup_{x\in I(\epsilon)}
            \P\big(
                \gamma(\eta)\tau_{I^\complement}^{\eta}(x)>t,\,V_{\tau}^\eta(x)\in B  
            \big) 
            \leq C(B^-) \cdot e^{-t};
        \end{align*}

    \item 
        If $C(I^\complement) = 0$, then for each $\epsilon,t >0$,
        \begin{align*}
            \lim_{\eta\downarrow 0} \sup_{x\in I(\epsilon)}
            \P\big(
                \gamma(\eta)\tau_{I^\complement}^{\eta}(x) \leq t
            \big) = 0.
        \end{align*}
\end{enumerate}

\end{theorem}

\subsection{Numerical Examples}
\label{subsec: numerical examples}

In this section, we provide the details for numerical samples illustrated in Figures~\ref{fig: exit path summary} and \ref{fig: exit time}.

\medskip
\noindent
\textbf{Large Deviations and the Catastrophe Principle}.
We consider iterates in $\R^1$:
\begin{align}
    \bm X^{\eta|b}_0(\bm x) = \bm x,\qquad
\bm X^{\eta|b}_j(\bm x) = \bm X^{\eta|b}_{j-1}(\bm x) + \varphi_b\big( -\eta U^\prime\big(\bm X^{\eta|b}_{j-1}(\bm x)\big) + \eta \bm Z_j\big)\quad \forall j \geq 1,
\label{def: R1 dynamics, numerical examples}
\end{align}
with the potential function defined as
\begin{align}
    U(\bm x) = \frac{1}{10}\bm x^4 - \bm x^2.
    \label{numerical example: def potential U, LD}
\end{align}
The potential $U(\cdot)$ possess two local minima $m_{\pm} = \pm \sqrt{5} \approx \pm 2.24$.
Here, $(\bm Z_j)_{j \geq 1}$ is an iid sequence of law
\begin{align}
      c_{\text{pareto}} \cdot W_\alpha + c_{\text{normal}}\cdot N(0,1)
      \label{def: law of noises, LD experiment}
\end{align}
where $N(0,1)$ is a standard normal RV,
$
\P(W_\alpha > w) = \P(-W_\alpha > w) = \frac{0.5}{ (1 + w)^\alpha  }
$
for $w > 0$,
.i.e., a Lomax (Pareto Type-II) RV with index $\alpha$ and a random sign, and $W$ and $N(0,1)$ are independent.

Recall that we denote the (time-scaled) sample path by $\bm X^{\eta|b}_{[0,1]}(\bm x) = \{ \bm X^{\eta|b}_{\floor{t/\eta}}(\bm x):\ t \in [0,1] \}$,
Let $B\subseteq \D\big([0,1],\R\big)$ be defined as in \eqref{def: event B for numerical example, LD, intro}.
First, we are interested in the probabilities $\P(\bm X^{\eta|b}_{[0,1]}(\bm x_\text{init}) \in B)$.
Specifically, we fix the initial value at $\bm x_\text{init} = 2.5$,
and consider a heavy-tailed setting where 
$c_\text{pareto} = 0.2$, $c_\text{sigma} = 5$, and $\alpha = 1.5$ in \eqref{def: law of noises, LD experiment}.
By \emph{truncated case}, we mean that the truncation threshold $b$ is set as $1.5$ in \eqref{def: R1 dynamics, numerical examples}.
In this case, we have
$\D^{(2)|b}_{ \{ \bm x_\text{init} \} }[0,1] \cap B \neq \emptyset$ (see \eqref{def: l * tilde jump number for function g, clipped SGD})
and that $B$ is bounded away from 
$
\D^{(1)|b}_{ \{ \bm x_\text{init} \} }[0,1](\epsilon)
$
for small $\epsilon$.
In particular, for any $\epsilon > 0$ small enough and any
$
\xi \in 
\D^{(1)|b}_{ \{ \bm x_\text{init} \} }[0,1](\epsilon),
$
we have $\inf_{t \in [0,1]}\xi_t \geq m_+ - (b + \epsilon) > 0$.
Applying Theorem~\ref{corollary: LDP 2},
we have
$\P(\bm X^{\eta|b}_{[0,1]}(\bm x_\text{init}) \in B)$ is roughly of order $\eta^{ 2 * (1.5 - 1)} = \eta$ as $\eta \downarrow 0$.
By \emph{untruncated case} we mean that $b$ is set as $\infty$, so the projection operator $\varphi_b$ in \eqref{def: R1 dynamics, numerical examples} is superfluous
and the iterates reduces to the stochastic difference equation in \eqref{def: X eta b j x, unclipped SGD}.
In this case,
we have 
$\D^{(1)}_{ \{ \bm x_\text{init} \} }[0,1] \cap B \neq \emptyset$ (see \eqref{def: l * tilde jump number for function g, clipped SGD})
and $B$ is bounded away from
$
\D^{(0)}_{ \{ \bm x_\text{init} \} }[0,1](\epsilon),
$
which (regardless of the value of $\epsilon$) only contains the gradient flow
$
d\bm y_t(\bm x_\text{init})/dt = -U^\prime\big(\bm y_t(\bm x_\text{init})\big).
$
We thus yield that
$\P(\bm X^{\eta}_{[0,1]}(\bm x_\text{init}) \in B)$ is roughly of order $\eta^{1.5 - 1} = \eta^{0.5}$ as $\eta \downarrow 0$.
We confirm these asymptotics through Monte-Carlo simulation.
The results are obtained by collecting 32 positive samples
(i.e., draw independent samples of $\mathbbm{I}\{\bm X^{\eta|b}_{[0,1]}(\bm x_\text{init}) \in B\}$ until the event occurs 32 times), and
are presented in Table~\ref{table: large deviation, probability}
and the log-log scale plot in Figure~\ref{fig: exit path summary} (a, Right).
As shown in the plot, the estimates confirm the asymptotics indicated by the sample path large deviations we developed in Section~\ref{subsec: LD, SGD}.

Next, we inspect the conditional law
$\P\big(\bm X^{\eta|b}_{[0,1]}(\bm x_\text{init})\in \cdot\ \big|\bm X^{\eta|b}_{[0,1]}(\bm x_\text{init}) \in B\big)$,
which reveals the most likely behavior of $\bm X^{\eta|b}_j(\bm x)$ given the rare event.
We are also interested in comparing the heavy-tailed and light-tailed cases.
In the \emph{light-tailed case},
we set 
$c_\text{pareto} = 0$ and $c_\text{sigma} = 10$ in \eqref{def: law of noises, LD experiment}.
The parameters are so chosen that,
under $\eta = 1/200$, the probability of the rare event is of an order comparable to its heavy-tailed counterpart (that is, around $10^{-6}$), which prevents the experiment from running too long.
We present in Figure~\ref{fig: exit path summary} (b)--(e) the samples from the conditional law,
which we obtained by running Monte-Carlo simulation for $\bm X^{\eta|b}_{[0,1]}(\bm x_\text{init})$
and keeping the samples when the event $\{\bm X^{\eta|b}_{[0,1]}(\bm x_\text{init}) \in B\}$ occurs.
Part (c) and (d) confirm the catastrophe principle for heavy-tailed dynamics established in Corollary~\ref{corollary: conditional limit, SGD}: 
the rare event arises due to $k$ catastrophically large components, and the index $k$ is the minimum number of perturbations required for the nominal path to enter the target set.
From part (d) and (e) of the figure, we also observe the sharp contrast between the catastrophe principle of heavy-tailed systems and the conspiracy principle of light-tailed systems.

\begin{table}[t]
  \caption{
  Monte-Carlo estimation for $\P(\bm X^{\eta|b}_{[0,1]} \in B)$ using 32 positive samples.}
  \label{table: large deviation, probability}
  \centering
  \begin{tabular}{lllllll}
    \hline
    $\eta$     &    1/200 & 1/400 & 1/600 & 1/800 & 1/1000 \\
    \hline
    
    (Untruncated) $b = \infty$  &  $1.1 \times 10^{-3}$ & $6.74 \times 10^{-4}$ & $5.78 \times 10^{-4}$ & $4.73 \times 10^{-4}$ & $4.25 \times 10^{-4}$ \\ 
    (Truncated) $b = 1.5$    & $5.25 \times 10^{-6}$ & $1.67 \times 10^{-6}$ & $1.15 \times 10^{-6}$ & $9.53 \times 10^{-7}$ & $6.61\times 10^{-7}$ \\

    \hline
  \end{tabular}
\end{table}


\medskip
\noindent
\textbf{Metastability}.
Consider the one-dimensional iterates in \eqref{def: R1 dynamics, numerical examples} under the potential function
\begin{equation}\label{aeq: potential U, first exit time}
\begin{aligned}
    & U(x)= (x+1.6)(x+1.3)^2(x-0.2)^2(x-0.7)^2(x-1.6)\big(0.05|1.65-x|\big)^{0.6} \\
    & \ \ \cdot \Big( 1 + \frac{1}{ 0.01 + 4(x-0.5)^2  } \Big)\Big( 1 + \frac{1}{0.1 + 4(x+1.5)^2} \Big)\Big( 1 - \frac{1}{4}\exp( -5(x + 0.8)(x + 0.8)  ) \Big).
\end{aligned}
\end{equation}
See Fig~\ref{fig: exit time} (i) for an illustration.
Specifically, we consider the case where the law of $(Z_j)_{j \geq 1}$ is of form \eqref{def: law of noises, LD experiment} with  $c_\text{pareto} = 0.1$, $c_\text{normal} = 0$, and $\alpha = 1.2$,
and
focus on the local minimum $m = -0.66$ and its attraction field $I = (-1.3, 0.2)$.
We initialize the process at $\bm x = m$ and are interested in first exit times $\tau^{\eta|b}(\bm x)$  from $I$; see \eqref{def: first exit time for heavy tailed SGD}. 
In this case,
the index $\mathcal J^I_b$ defined in \eqref{def: first exit time, J *}
reduces to $\mathcal J^I_b = \ceil{ 0.64/b  }$ for any $b \in (0,\infty)$,
and the regularity conditions in Theorem~\ref{theorem: first exit time, unclipped}
hold for (Lebesgue) almost all $b > 0$; see Remark~\ref{remark: regularity conditions for first exit times}.
Therefore, for Lebesgue almost all $b > 0$,
the stopping times $\tau^{\eta|b}(\bm x)$ is roughly of order $1/\eta^{ 1 + \mathcal J^I_b(\alpha - 1) } = 1/\eta^{ 1 + \mathcal J^I_b * 0.2 }$ as $\eta \downarrow 0$.
This characterizes the phase transitions in the order of first exit times depending on the (discretized) relative width $\mathcal J^I_b$.
In case that $b = \infty$, we apply Corollary~\ref{corollary: first exit time, untruncated case} and obtain that 
the exit times $\tau^\eta(\bm x)$ in the untrucated case (see \eqref{def: first exit time for heavy tailed SGD})
is roughly of order 
$1/\eta^{\alpha } = 1/\eta^{ 1.2 }$ for small $\eta$.

We confirm these asymptotics through Monte-Carlo simulation and present the results in 
Fig~\ref{fig: exit time} (ii).
This is a log-log scale plot, where each dot represents an average of 20 samples, and the dashed lines indicate the asymptotics provided by our metastability analysis.
To prevent the experiment from running too long, a stopping criterion of $5\times 10^7$ steps 
is employed.
This stopping criterion was reached only in the case where $b = 0.28$ and $\eta = 0.001$, which is indicated in the plot by the only non-solid dot, highlighting that it is an underestimation.
The plot confirms the asymptotic law of first exit times established in our metastability analysis, as well as the phase transition in first exit times  w.r.t.\ $\mathcal J^I_b$.
Furthermore,
this dependency on the relative width $\mathcal J^I_b$ leads to the  
 intriguing global dynamics shown in Fig~\ref{fig: exit time} (iii) and (iv), where we run $\bm X^{\eta|b}_t(\bm x)$ under the choice of $\eta = 1/1000$ and $x = 0.3$.
In the light-tailed cases, we set 
$c_\text{pareto} = 0$ and $c_\text{normal} = 1$.
As we can see from Fig~\ref{fig: exit time} (iii) and (iv), driven by untruncated heavy-tailed perturbations, $\bm X^{\eta}_j(\bm x)$ frequently traverses all local minima of $U$;
In contrast, under truncated heavy tails, $\bm X^{\eta}_j(\bm x)$ almost completely avoids the sharp minima of $U$.
This phenomenon is formally characterized in a companion paper \cite{wangSGDpaper2},
where we show that, as $\eta \downarrow 0$, the (time-scaled) sample path of $\bm X^{\eta|b}_j(\bm x)$ converges in distribution to a Markov chain that \textbf{only visits the widest minima} (in terms of the relative width $\mathcal J^I_b$) of the potential $U$, and we discuss its connection to the generalization performance of deep neural networks trained with heavy tailed noises.

\section{Uniform $\mathbb M$-Convergence and Sample Path Large Deviations }
\label{sec: LD of SGD, proof}

Here, we collect the proofs for Sections~\ref{sec: M convergence, asymptotic equivalence} and \ref{subsec: LD, SGD}.
Specifically, Section~\ref{subsec: proof of portmanteau theorem for uniform M convergence}
provides the proof of Theorem \ref{theorem: portmanteau, uniform M convergence}
(i.e., the Portmanteau theorem for the uniform $\M(\S \setminus \C)$-convergence)
and Lemma~\ref{lemma: asymptotic equivalence when bounded away, equivalence of M convergence}.
Section~\ref{subsec: proof of lemmas, LD of SGD proof} further develops a set of technical tools,
which will then be applied to establish the sample-path large deviations results (i.e., Theorems~\ref{corollary: LDP 2} and \ref{theorem: LDP 1, unclipped}) in Section~\ref{subsec: LDP clipped, proof of main results}.

\subsection{Proof of Theorem \ref{theorem: portmanteau, uniform M convergence} and Lemma~\ref{lemma: asymptotic equivalence when bounded away, equivalence of M convergence}}
\label{subsec: proof of portmanteau theorem for uniform M convergence}


\begin{proof}[Proof of Theorem \ref{theorem: portmanteau, uniform M convergence}]
\linksinpf{theorem: portmanteau, uniform M convergence}
\textbf{Proof of $(i) \Rightarrow (ii)$}. It follows directly from Definition \ref{def: uniform M convergence}.

\medskip
\noindent
\textbf{Proof of $(ii) \Rightarrow (iii)$}.
We consider a proof by contradiction.
Suppose that the upper bound $\limsup_{\eta \downarrow 0}\sup_{\theta \in \Theta}\mu^\eta_\theta(F) - \mu_\theta(F^\epsilon) \leq 0$ does not hold for some closed $F$ bounded away from $\C$ and some $\epsilon > 0$.
Then there exist a sequence $\eta_n \downarrow 0$, a sequence $\theta_n \in \Theta$, and some $\delta > 0$ such that
$
\mu^{\eta_n}_{\theta_n}(F) - \mu_{\theta_n}(F^\epsilon) > \delta\ \forall n \geq 1.
$
Now, we make two observations.
First, using Urysohn's lemma (see, e.g., lemma 2.3 of \cite{lindskog2014regularly}),
one can identify some $f \in \mathcal C(\S \setminus \C)$, which is also uniformly continuous on $\mathbb S$, such that
$
\mathbbm{I}_F \leq f \leq \mathbbm{I}_{F^\epsilon}.
$
This leads to the bound
$
\mu^{\eta_n}_{\theta_n}(F) - \mu_{\theta_n}(F^\epsilon)
\leq 
\mu^{\eta_n}_{\theta_n}(f) - \mu_{\theta_n}(f)
$
for each $n$.
Secondly,
from statement $(ii)$ we get
$
\lim_{n \to \infty}\big| \mu^{\eta_n}_{\theta_n}(f) - \mu_{\theta_n}(f) \big| = 0.
$
In summary, we yield the contradiction
\begin{align*}
\limsup_{n \to \infty}\mu^{\eta_n}_{\theta_n}(F) - \mu_{\theta_n}(F^\epsilon)
& \leq 
\limsup_{n \to \infty}\mu^{\eta_n}_{\theta_n}(f) - \mu_{\theta_n}(f)
\leq \lim_{n \to \infty}\big|\mu^{\eta_n}_{\theta_n}(f) - \mu_{\theta_n}(f)\big| = 0.
\end{align*}
Analogously,
if the claim
$\liminf_{\eta \downarrow 0}\inf_{\theta \in \Theta}\mu^\eta_\theta(G) - \mu_\theta(G^\epsilon) \geq 0$, supposedly, does not hold for some open $G$ bounded away from $\C$ and some $\epsilon > 0$,
then we can yield a similar contradiction by applying Urysohn's lemma and constructing some uniformly continuous $g \in \mathcal C(\S \setminus \C)$ such that $\mathbbm{I}_{G_\epsilon} \leq g \leq \mathbbm{I}_G$.
This concludes the proof of  $(ii) \Rightarrow (iii)$.

\medskip\noindent\textbf{Proof of $(iii) \Rightarrow (i)$.}
Again, we proceed with a proof by contradiction.
Suppose that the claim
$
\lim_{\eta \downarrow 0}\sup_{\theta \in \Theta}\big| \mu^\eta_\theta(g) - \mu_\theta(g) \big| = 0
$
does not hold for some $g \in \mathcal C(\S \setminus \C)$.
Then, there exist some sequences $\eta_n \downarrow 0$, $\theta_n \in \Theta$ and some $\delta > 0$ such that
\begin{align}
    |\mu^{\eta_n}_{\theta_n}(g) - \mu_{\theta_n}(g) | > \delta\qquad \forall n \geq 1.
    \label{proof, proof by contradiction, ii to i, theorem: portmanteau, uniform M convergence}
\end{align}
To proceed, we arbitrarily pick some
closed $F \subseteq \mathbb{S}$ that is bounded away from $\mathbb{C}$
        and some open
        $G \subseteq \mathbb{S}$ that is bounded away from $\mathbb{C}$.
First, 
using claims in $(iii)$, we get
$
\limsup_{n \to \infty}\mu^{\eta_n}_{\theta_n}(F) - \mu_{\theta_n}(F^\epsilon) \leq 0
$
and
$
\liminf_{n \to \infty}\mu^{\eta_n}_{\theta_n}(G) - \mu_{\theta_n}(G_\epsilon) \geq 0
$
for any $\epsilon > 0$.
Next,
due to condition \eqref{assumption in portmanteau, uniform M convergence},
by picking a sub-sequence of $\theta_n$ if necessary we can find some $\mu_{\theta^*}$ such that
$
\lim_{n \to \infty}\big| \mu_{\theta_n}(f) - \mu_{\theta^*}(f) \big| = 0
$
for all $f \in \mathcal C(\S \setminus \C)$.
By Portmanteau theorem for standard $\M(\S \setminus \C)$-convergence (see theorem 2.1 of \cite{lindskog2014regularly}),
we yield
$
\limsup_{n \to\infty}\mu_{\theta_n}(F^\epsilon) \leq \mu_{\theta^*}(F^\epsilon)
$
and
$
\liminf_{n \to\infty}\mu_{\theta_n}(G_\epsilon) \geq \mu_{\theta^*}(G_\epsilon).
$
In summary, for any $\epsilon > 0$,
\begin{align*}
    \limsup_{n \to \infty}\mu^{\eta_n}_{\theta_n}(F) 
    & \leq 
    \limsup_{n \to \infty} \mu_{\theta_n}(F^\epsilon) + \limsup_{n \to \infty} \mu^{\eta_n}_{\theta_n}(F)  -  \mu_{\theta_n}(F^\epsilon)
    \leq \mu_{\theta^*}(F^\epsilon),
    \\
    \liminf_{n \to \infty}\mu^{\eta_n}_{\theta_n}(G) 
    & \geq
    \liminf_{n \to \infty} \mu_{\theta_n}(G_\epsilon) + \liminf_{n \to \infty} \mu^{\eta_n}_{\theta_n}(G)  -  \mu_{\theta_n}(G_\epsilon)
    \geq \mu_{\theta^*}(G_\epsilon).
\end{align*}
Lastly, note that $\lim_{\epsilon \downarrow 0}\mu_{\theta^*}(F^\epsilon) = \mu_{\theta^*}(F)$
and
$\lim_{\epsilon \downarrow 0}\mu_{\theta^*}(G_\epsilon) = \mu_{\theta^*}(G)$
due to continuity of measures and $\bigcap_{\epsilon > 0} F^\epsilon = F$, $\bigcup_{\epsilon > 0}G_\epsilon = G$.
This allows us to apply  Portmanteau theorem for standard $\M(\S \setminus \C)$-convergence again and obtain
$
\lim_{n \to \infty}\big|\mu^{\eta_n}_{\theta_n}(g) - \mu_{\theta^*}(g)\big| = 0
$
for the $g \in \mathcal C(\S \setminus \C)$ fixed in \eqref{proof, proof by contradiction, ii to i, theorem: portmanteau, uniform M convergence}.
However, recall that we have already obtained 
$
\lim_{n \to \infty}\big| \mu_{\theta_n}(g) - \mu_{\theta^*}(g) \big| = 0
$
using assumption \eqref{assumption in portmanteau, uniform M convergence}.
We now arrive at the contradiction
\begin{align*}
    \lim_{n \to \infty}\big| \mu^{\eta_n}_{\theta_n}(g) - \mu_{\theta_n}(g) \big|
    & \leq 
     \lim_{n \to \infty}\big| \mu^{\eta_n}_{\theta_n}(g) - \mu_{\theta^*}(g) \big|
     +
      \lim_{n \to \infty}\big| \mu_{\theta^*}(g) - \mu_{\theta_n}(g) \big| = 0
\end{align*}
and conclude the proof of $(iv)\Rightarrow (i)$.

\medskip\noindent\textbf{Proof of $(i) \Rightarrow (iv)$.}
Due to the equivalence of $(i)$, $(ii)$, and $(iii)$, it only remains to show that $(i) \Rightarrow (iv)$.
Suppose, for the sake of contradiction, that the claim  $\limsup_{\eta \downarrow 0}\sup_{\theta \in \Theta}\mu^\eta_\theta(F) \leq \sup_{\theta \in \Theta} \mu_\theta(F)$ in $(iv)$ does not hold for some closed $F$ bounded away from $\C$.
Then we can find sequences $\eta_n \downarrow 0$, $\theta_n \in \Theta$ and some $\delta > 0$ such that 
$
\mu^{\eta_n}_{\theta_n}(F) > \sup_{\theta \in \Theta} \mu_\theta(F) + \delta\ \forall n \geq 1.
$
Next,
due to the assumption \eqref{assumption in portmanteau, uniform M convergence},
by picking a sub-sequence of $\theta_n$ if necessary we can find some $\mu_{\theta^*}$ such that
$
\lim_{n \to \infty}\big| \mu_{\theta_n}(f) - \mu_{\theta^*}(f) \big| = 0
$
for all $f \in \mathcal C(\S \setminus \C)$.
Meanwhile, 
$(i)$ implies that 
$
\lim_{n \to \infty}\big| \mu^{\eta_n}_{\theta_n}(f) - \mu_{\theta_n}(f) \big| = 0
$
for all $f \in \mathcal C(\S \setminus \C)$.
Therefore,
\begin{align*}
    \lim_{n \to \infty}\big| \mu^{\eta_n}_{\theta_n}(f) - \mu_{\theta^*}(f)  \big|
    \leq 
    \lim_{n \to \infty}\big| \mu^{\eta_n}_{\theta_n}(f) - \mu_{\theta_n}(f) \big|
    +
    \lim_{n \to \infty}\big| \mu_{\theta_n}(f) - \mu_{\theta^*}(f) \big|
    = 0
\end{align*}
for all $f \in \mathcal C(\S \setminus \C)$.
By Portmanteau theorem for standard $\M(\S \setminus \C)$-convergence,
we yield the contradiction
$
\limsup_{n \to \infty}\mu^{\eta_n}_{\theta_n}(F) \leq \mu_{\theta^*}(F) \leq \sup_{\theta \in \Theta}\mu_\theta(F).
$
In summary, we have established the claim $\limsup_{\eta \downarrow 0}\sup_{\theta \in \Theta}\mu^\eta_\theta(F) \leq \sup_{\theta \in \Theta} \mu_\theta(F)$
for all closed $F$ bounded away from $\C$.
The same approach can also be applied to show
$\liminf_{\eta \downarrow 0}\inf_{\theta \in \Theta}\mu^\eta_\theta(G) \geq \inf_{\theta \in \Theta} \mu_\theta(G)$
for all open $G$ bounded away from $\C$.
This concludes the proof.
\end{proof}

\begin{proof}[Proof of Lemma~\ref{lemma: asymptotic equivalence when bounded away, equivalence of M convergence}]
\linksinpf{lemma: asymptotic equivalence when bounded away, equivalence of M convergence}
We arbitrarily pick some Borel measurable $B \subseteq \mathbb S$ that is bounded away from $\mathbb C$.
Henceforth in this proof, we only consider $\Delta > 0$ small enough
that $\bm d(B,\mathbb C) > \Delta$, and hence $B^\Delta$ is still bounded away from $\mathbb C$.
Observe that
\begin{align*}
    \P(X_n \in B)
    & \leq 
    \P\big(
        X_n \in B;\ \bm d( X_n, Y^\delta_n ) \leq \Delta
    \big)
    +
    \P\big(
        X_n \in B;\ \bm d( X_n, Y^\delta_n ) > \Delta
    \big)
    \\ 
    & \leq 
    \P\big(
        Y^\delta_n \in B^\Delta
    \big)
    +
    \P\big(
        X_n\in B\text{ or }Y^\delta_n \in B;\ \bm d( X_n, Y^\delta_n ) > \Delta
    \big).
\end{align*}
As a result,
\begin{align}
     & \limsup_{n \to \infty}
    \epsilon^{-1}_n
    \P(X_n \in B)
    \nonumber
    \\ 
    & \leq
    \limsup_{\delta \downarrow 0}\limsup_{n \to \infty}
    \epsilon^{-1}_n\P\big(
        Y^\delta_n \in B^\Delta
    \big)
    +
    \limsup_{\delta \downarrow 0}\limsup_{n \to \infty} \epsilon^{-1}_n
    \P\Big( \bm{d}\big(X_n, Y^\delta_n\big)\mathbbm{I}\big( X_n \in B\text{ or }Y^\delta_n \in B \big) > \Delta \Big)
    \nonumber
    \\ 
    & \leq
    \mu(B^\Delta)
    \qquad
    \text{ by conditions (i) and (ii) of Lemma~\ref{lemma: asymptotic equivalence when bounded away, equivalence of M convergence}}.
    \label{proof: upper bound, lemma: asymptotic equivalence when bounded away, equivalence of M convergence}
\end{align}
Analogously, observe the lower bound
\begin{align*}
    \P(X_n \in B)
    & \geq 
    \P\big(X_n \in B;\ \bm d( X_n, Y^\delta_n ) \leq \Delta\big)
    \\ 
    & \geq
    \P\big( Y^\delta_n \in B_\Delta; \ \bm d( X_n, Y^\delta_n ) \leq \Delta\big)
    \\ 
    & \geq 
    \P( Y^\delta_n \in B_\Delta)
    -
    \P\big( Y^\delta_n \in B_\Delta; \ \bm d( X_n, Y^\delta_n ) > \Delta\big)
    \\ 
    & \geq 
    \P( Y^\delta_n \in B_\Delta)
    -
    \P\big( Y^\delta_n \in B\text{ or }X_n \in B; \ \bm d( X_n, Y^\delta_n ) > \Delta\big),
\end{align*}
and hence
\begin{align}
    & \liminf_{n \to \infty}
    \epsilon^{-1}_n
    \P(X_n \in B)
    \nonumber
    \\ 
    & \geq 
    \liminf_{\delta \downarrow 0}\liminf_{n \to \infty}\epsilon^{-1}_n
    \P( Y^\delta_n \in B_\Delta)
    -
    \limsup_{\delta \downarrow 0}\limsup_{n \to \infty} \epsilon^{-1}_n
    \P\Big( \bm{d}\big(X_n, Y^\delta_n\big)\mathbbm{I}\big( X_n \in B\text{ or }Y^\delta_n \in B \big) > \Delta \Big)
    \nonumber
    \\ 
    & \geq 
    \mu(B_\Delta)
     \qquad
    \text{ by conditions (i) and (ii) of Lemma~\ref{lemma: asymptotic equivalence when bounded away, equivalence of M convergence}}.
    \label{proof: lower bound, lemma: asymptotic equivalence when bounded away, equivalence of M convergence}
\end{align}
Since 
$\mu \in \mathbb M(\mathbb S\setminus\mathbb C)$ and
$B^\Delta$ is bounded away from $\mathbb C$,
we have $\mu(B^\Delta) < \infty$.
By sending $\Delta \downarrow 0$ in \eqref{proof: upper bound, lemma: asymptotic equivalence when bounded away, equivalence of M convergence} and \eqref{proof: lower bound, lemma: asymptotic equivalence when bounded away, equivalence of M convergence},
it then follows from the continuity of measure $\mu$ that 
\begin{align*}
    \mu(B^\circ) \leq 
    \liminf_{n \to \infty}\epsilon_n^{-1}\P(X_n \in B)
    \leq 
    \limsup_{n \to \infty}\epsilon_n^{-1}\P(X_n \in B)
    \leq \mu(B^-).
\end{align*}
Due to the arbitrariness of our choice of $B$,
we conclude the proof using Theorem 2.1 of \cite{lindskog2014regularly},
which is the Portmanteau theorem for the standard $\mathbb M(\mathbb S\setminus \mathbb C)$-convergence and, essentially, a special case of Theorem~\ref{theorem: portmanteau, uniform M convergence}.
\end{proof}

\subsection{Technical Lemmas for Theorems \ref{corollary: LDP 2} and \ref{theorem: LDP 1, unclipped}}
\label{subsec: proof of lemmas, LD of SGD proof}

Our analysis hinges on the separation of \textit{large noises} among $(\bm Z_j)_{j \geq 1}$ from the rest, and we pay special attention to $\bm Z_j$'s with norm large enough such that $\eta\norm{\bm Z_j}$ exceed some prefixed threshold level $\delta > 0$.
To be more concrete,
for any $i \geq 1$ and $\eta,\delta > 0$, define the $i$\textsuperscript{th} arrival time
of ``large noises'' and its size as
\begin{align}
    \notationdef{notation-large-jump-time}{\tau^{>\delta}_i(\eta)} & \delequal{} \min\{ n > \tau^{>\delta}_{i-1}(\eta):\ \eta\norm{\bm Z_j} > \delta  \},\quad \tau^{>\delta}_{0}(\eta) = 0 \label{defArrivalTime large jump} \\
    \notationdef{notation-large-jump-size}{\bm W^{>\delta}_i(\eta)} & \delequal{} \bm Z_{\tau^{>\delta}_i(\eta)}. \label{defSize large jump}
\end{align}
For any $\delta > 0$ and $k = 1,2,\cdots$, note that
\begin{align}
    \P\Big( \tau^{>\delta}_k(\eta) \leq \floor{1/\eta} \Big)
    & \leq 
    \P\Big( \tau^{>\delta}_j(\eta) - \tau^{>\delta}_{j-1}(\eta) \leq \floor{1/\eta}\ \ \forall j \in [k] \Big)
    \nonumber
    \\
    & = \Big[ \sum_{i = 1}^{ \floor{1/\eta} }\big(1 - H(\delta/\eta)\big)^{i - 1}H(\delta/\eta)      \Big]^{k}
    \leq \Big[ \sum_{i=1}^{\floor{1/\eta}}H(\delta/\eta)      \Big]^{k}
    \nonumber
    \\
    &
    \leq \Big[ 1/\eta\cdot H(\delta/\eta)      \Big]^{k}.
    \label{property: large jump time probability}
\end{align} 
Recall the definition of filtration ${\mathbb{F}} = (\mathcal{F}_j)_{j \geq 0}$
where 
$\mathcal{F}_j$ is the $\sigma$-algebra generated by $\bm Z_1,\bm Z_2,\cdots,\bm Z_j$ and $\mathcal{F}_0 = \{\emptyset,\Omega\}$.
In the next lemma, we establish a uniform asymptotic concentration bound for the weighted sum of $Z_i$'s where the weights are adapted to the filtration $\mathbb F$. 
For any $M \in (0,\infty)$, let $\bm{\Gamma}_M$ denote the collection of families of random matrices $\textbf V_j = (V_{j;p,q})_{p \in [m], q\ \in [d]}$ taking values in $\R^{m \times d}$, over which we will prove the uniform asymptotics:
\begin{align}
    \notationdef{notation-Gamma-M-adapted-process-bounded-by-M-LDP}{\bm{\Gamma}_M} \delequal 
    \bigg\{ (\textbf V_j)_{j \geq 0}\text{ is adapted to }\mathbb{F}:\ \norm{\textbf V_j} \leq M\ \forall j \geq 0\text{ almost surely} \bigg\}.
    \label{def: Gamma M, set of bounded adapted process}
\end{align}

\linkdest{location, lemma: LDP, small jump perturbation}

\begin{lemma}
\label{lemma LDP, small jump perturbation}
\linksinthm{lemma LDP, small jump perturbation}
Let Assumption \ref{assumption gradient noise heavy-tailed} hold.
\begin{enumerate}[(a)]
    \item 
    Given any $M > 0$, $N > 0$, $t > 0$, and $\epsilon > 0$, there exists $\delta_0 = \delta_0(\epsilon,M,N,t) > 0$ such that 
\begin{align*}
    \lim_{\eta \downarrow 0}\eta^{-N}{ \underset{ (\textbf V_i)_{i \geq 0} \in  \bm{\Gamma}_M}{\sup}
    \P\Bigg(\; \underset{ j \leq \floor{t/\eta} \wedge \big(\tau^{>\delta}_{1}(\eta) - 1\big) }{\max}\ 
    \eta\norm{\sum_{i = 1}^j \textbf V_{i-1}\bm Z_i} > \epsilon \Bigg)   } = 0\qquad \forall \delta\in(0,\delta_0).
\end{align*}

\item 
Furthermore, let Assumption \ref{assumption: boundedness of drift and diffusion coefficients} hold.
For each $i \geq 1$, let 
\begin{align}
    \notationdef{notation-event-A-i-concentration-of-small-jumps}{A_i(\eta,b,\epsilon,\delta,\bm x)}
    &\delequal 
    \left\{\rule{0cm}{0.9cm}
    \underset{j \in I_i(\eta,\delta) }{\max}\  
        \eta\norm{\sum_{n = \tau^{>\delta}_{i-1}(\eta) + 1}^j \bm \sigma\big( \bm X^{\eta|b}_{n-1}(\bm x) \big)\bm Z_n} \leq \epsilon 
    \right\};
    \label{def: event A i concentration of small jumps, 1}
    \\
    \notationdef{notation-A-i-concentration-of-small-jumps-2}{I_i(\eta,\delta)} 
    &\delequal 
    \bigg\{j \in \mathbb{N}:\  \tau^{>\delta}_{i-1}(\eta) + 1 \leq j \leq \big(\tau^{>\delta}_{i}(\eta) - 1 \big) \wedge \floor{1/\eta}\bigg\}.
    \label{def: event A i concentration of small jumps, 2}
\end{align}
Here we adopt the convention that (under $b = \infty$)
$$
{A_i(\eta,\infty,\epsilon,\delta,x)}
    \delequal 
    \left\{\rule{0cm}{0.9cm}
    \underset{j \in I_i(\eta,\delta) }{\max}\  
        \eta\norm{\sum_{n = \tau^{>\delta}_{i-1}(\eta) + 1}^j \bm \sigma\big( \bm X^{\eta}_{n-1}(\bm x) \big)\bm Z_n} \leq \epsilon 
    \right\}.
$$
For any $k \geq 0$, $N > 0$, $\epsilon > 0$ and $b \in (0,\infty]$, there exists $\delta_0 = \delta_0(\epsilon,N) > 0$ such that
\begin{align*}
    \lim_{\eta \downarrow 0}\eta^{-N}{\sup_{\bm x \in \R^m} \P\Bigg( \Big(\bigcap_{i = 1}^k A_i(\eta,b,\epsilon,\delta,\bm x)\Big)^c \Bigg)  } = 0\qquad \forall \delta \in (0,\delta_0).
\end{align*}

\end{enumerate}
\end{lemma}

\begin{proof}
\linksinpf{lemma LDP, small jump perturbation}
$(a)$
Choose some $\beta$ such that $\frac{1}{2 \wedge \alpha} < \beta < 1$. Let
\begin{align*}
    \bm Z^{(1)}_i \delequal \bm Z_i \mathbbm{I}\Big\{ \norm{\bm Z_i} \leq \frac{1}{\eta^\beta}\Big\},\qquad
    \hat{\bm Z}^{(1)}_i \delequal \bm Z^{(1)}_i - \E \bm Z^{(1)}_i,\qquad
    \bm Z^{(2)}_i \delequal \bm Z_i \mathbbm{I}\Big\{ \norm{\bm Z_i} \in (\frac{1}{\eta^\beta},\frac{\delta}{\eta}]\Big\}.
\end{align*}
Note that $\sum_{i=1}^j \bm V_{i-1}\bm Z_i = \sum_{i=1}^j \bm V_{i-1}\bm Z_i^{(1)}  + \sum_{i=1}^j \bm V_{i-1}\bm Z_i^{(2)}$ on $j < \tau^{>\delta}_1(\eta)$, and hence,  
\begin{align*}
&\max_{j \leq  \floor{t/\eta} \wedge \big(\tau^{>\delta}_{1}(\eta) - 1\big) }  \eta\norm{\sum_{i=1}^j \bm V_{i-1}\bm Z_i}
\\
&
\leq
\max_{j \leq  \floor{t/\eta}  }\eta\norm{\sum_{i=1}^j \bm V_{i-1}\bm Z_i^{(1)}}
+
\max_{j \leq  \floor{t/\eta}  }\eta\norm{\sum_{i=1}^j \bm V_{i-1}\bm Z_i^{(2)}}
\\
&
\leq
\max_{j \leq  \floor{t/\eta}}\eta\norm{\sum_{i=1}^j \bm V_{i-1}\E \bm Z_i^{(1)}}
+\max_{j \leq  \floor{t/\eta}}\eta\norm{\sum_{i=1}^j \bm V_{i-1}\widehat{\bm Z}_i^{(1)}}
+\max_{j \leq  \floor{t/\eta}}\eta\norm{\sum_{i=1}^j \bm V_{i-1}\bm Z_i^{(2)}}.
\end{align*}
Therefore, it suffices to show the existence of $\delta_0$ such that for any $\delta \in (0,\delta_0)$,
\begin{align}\label{proof: LDP, small jump perturbation, asymptotics part 1}
    \limsup_{\eta \downarrow 0} 
        \underset{ (\bm V_i)_{i \geq 0} \in  \bm{\Gamma}_M}{\sup}
        \underset{ j \leq \floor{t/\eta} }{\max}\ \eta\norm{\sum_{i=1}^j \bm V_{i-1}\E \bm Z_i^{(1)}}   &< \frac{\epsilon}{3},
    \\
    \label{proof: LDP, small jump perturbation, asymptotics part 2}
    \lim_{\eta \downarrow 0} 
        \eta^{-N} \underset{ (\bm V_i)_{i \geq 0} \in  \bm{\Gamma}_M}{\sup}
            \P\Bigg( \underset{ j \leq \floor{t/\eta} }{\max}\ 
                \eta\norm{\sum_{i = 1}^j\bm V_{i-1}\widehat{\bm Z}^{(1)}_i} > \frac{\epsilon}{3} \Bigg)   &= 0,
    \\
    \label{proof: LDP, small jump perturbation, asymptotics part 3}
    \lim_{\eta \downarrow 0} 
        \eta^{-N} \underset{ (\bm V_i)_{i \geq 0} \in  \bm{\Gamma}_M}{\sup}
            \P\Bigg( \underset{ j \leq \floor{t/\eta} }{\max}\ 
                \eta\norm{\sum_{i = 1}^j\bm V_{i-1}{\bm Z}^{(2)}_i} > \frac{\epsilon}{3} \Bigg)   &= 0.
\end{align}
For \eqref{proof: LDP, small jump perturbation, asymptotics part 1}, 
first observe that
\begin{align*}
\norm{\E \bm Z^{(1)}_i} 
&= 
\norm{ \E \bm Z_i\mathbbm{I}\{ \norm{\bm Z_i}> 1/\eta^\beta \} }
\qquad\text{ due to $\E \bm Z_i = \bm 0$}
\\
& \leq 
\E \Big[\norm{\bm Z_i}\mathbbm{I}\{ \norm{\bm Z_i}> 1/\eta^\beta \} \Big]
\\&
=
\E\Big[(\norm{\bm Z_i}-1/\eta^\beta)\mathbbm{I}\{ \norm{\bm Z_i} - 1/\eta^\beta>0 \}\Big]  + 1/\eta^\beta\cdot\P( \norm{\bm Z_i}> 1/\eta^\beta ).
\end{align*}
Since $(\norm{\bm Z_i}-1/\eta^\beta)\mathbbm{I}\{ \norm{\bm Z_i}- 1/\eta^\beta>0 \}$ is non-negative,
\begin{align*}
\E (\norm{\bm Z_i}-1/\eta^\beta)\mathbbm{I}\{ \norm{\bm Z_i}- 1/\eta^\beta>0 \} 
&
=
\int_0^\infty \P\big((\norm{\bm Z_i} - 1/\eta^\beta)\I\{\norm{\bm Z_i} - 1/\eta^\beta\}> x\big)dx
\\
&
=
\int_0^\infty \P(\norm{\bm Z_i} - 1/\eta^\beta> x)dx
=
\int^{\infty}_{1/\eta^\beta}\P(\norm{\bm Z}> x)dx.
\end{align*}
Recall that $H(x) = \P(\norm{\bm Z} > x) \in \RV_{-\alpha}(x)$ as $x \rightarrow \infty$. 
Therefore, from
Karamata's theorem,
\begin{align}
    \norm{\E \bm Z^{(1)}_i} 
    \leq 
    \int^{\infty}_{1/\eta^\beta}\P(\norm{\bm Z} > x)dx + 1/\eta^\beta\cdot\P( \norm{\bm Z} > 1/\eta^\beta ) \in \RV_{(\alpha - 1)\beta}(\eta)
    \label{term E Z 1, lemma LDP, small jump perturbation}
\end{align}
as $\eta \downarrow 0$.
Therefore, there exists some $\eta_0 = \eta_0(t,M,\epsilon) > 0$ such that
for any $\eta \in (0,\eta_0)$,
we have 
$
t \cdot M \cdot  \norm{\E \bm Z^{(1)}_i} < \epsilon/3,
$
and hence
for any $(\bm V_i)_{i \geq 0} \in \bm{\Gamma}_M$ and $\eta \in (0,\eta_0)$,
\begin{align*}
    \underset{ j \leq \floor{t/\eta} }{\max}
        \eta\norm{\sum_{i = 1}^j \bm V_{i-1}\E \bm Z^{(1)}_i } 
        \leq 
        \floor{t/\eta} \cdot M \cdot \eta\norm{\E \bm Z^{(1)}_i} < \epsilon/3,
\end{align*}
from which we immediately get \eqref{proof: LDP, small jump perturbation, asymptotics part 1}.

Next, for \eqref{proof: LDP, small jump perturbation, asymptotics part 2}, 
recall our convention that vectors in Euclidean spaces are understood as row vectors (unless specified otherwise),
and write
$
\bm V_t = (V_{t;l,k})_{l \in [m], k \in [d]},
$
$
\widehat{\bm Z}_t = (\widehat Z_{t;1},\cdots,\widehat Z_{t;d})^T.
$
Since
\begin{align*}
    \norm{ \sum_{i = 1}^j \bm V_{i-1}\widehat{\bm Z}_i }
    & = 
    \sqrt{
        \sum_{l = 1}^m \bigg(
            \sum_{i = 1}^j
            \sum_{k = 1}^d V_{i-1;l,k}\hat Z_{i,k}
        \bigg)^2
    },
\end{align*}
to prove \eqref{proof: LDP, small jump perturbation, asymptotics part 2},
it suffices to show that
\begin{align}
    \lim_{\eta \downarrow 0} 
        \eta^{-N} \underset{ (\bm V_i)_{i \geq 0} \in  \bm{\Gamma}_M}{\sup}
            \P\Bigg( \underset{ j \leq \floor{t/\eta} }{\max}\ 
                \eta| Y_{l,k}(j;\bm V) | > \frac{\epsilon}{3 \sqrt{md^2} } \Bigg)   &= 0
                \qquad \forall l \in [m],\ k \in [d],
    \label{proof: LDP, small jump perturbation, asymptotics part 2, R1 decomp}
\end{align}
where 
\begin{align*}
    Y_{l,k}(j;\bm V) \delequal \sum_{i = 1}^j V_{i-1;l,k}\hat Z_{i,k}.
\end{align*}
To proceed, we fix a sufficiently large $p$ satisfying
\begin{align}
    p\geq 1,\ \ p > \frac{2N}{\beta},\ \ p > \frac{2N}{1 - \beta},\ \ p > \frac{2N}{(\alpha - 1)\beta} > \frac{2N}{(2\alpha - 1)\beta}, \label{proof: LDP, small jump perturbation, choose p}
\end{align}
and some $l \in [m]$, $k \in [d]$.
Note that for  $(\bm V_i)_{i \geq 0} \in \bm{\Gamma}_M$ and $\eta >0$, 
$\{ V_{i-1;l,k}\widehat Z_{i;k}^{(1)}: i\geq 1\}$ is a martingale difference sequence.
Therefore,
$\big(Y_{l,k}(j;\bm V)\big)_{j \geq 0}$ is a martingale, and
\begin{align}
    &\E 
        \left[\rule{0cm}{0.7cm}
            \left(
                \underset{ j \leq \floor{t/\eta} }{\max} \eta\Big|Y_{l,k}(j;\bm V)\Big|
            \right)^p
        \right]\rule{0cm}{0.7cm}
    \nonumber
    \\
    & 
    \leq 
    c_1\E \left[ \left(
     \sum_{i = 1}^{ \floor{t/\eta} } \left(\eta V_{i-1;l,k}\widehat{Z}^{(1)}_{i;k}\right)^2
    \right)^{p/2} \right]
    \nonumber
    \\
    & 
    \leq 
    c_1 M^p 
    \E \left[
    \left(
    \sum_{i = 1}^{ \floor{t/\eta} } \left(\eta \widehat{Z}^{(1)}_{i;k}\right)^2
    \right)^{p/2}
    \right]
    \qquad 
    \text{due to $\norm{\bm V_s} \leq M$ for all $s \geq 0$}
    \nonumber
    \\
    & 
    \leq 
    c_1 c_2 M^p
    \E \left[\rule{0cm}{0.9cm}
        \left(
            \underset{ j \leq \floor{t/\eta} }{\max} \Big|\sum_{i = 1}^j \eta\widehat{Z}^{(1)}_{i;k}\Big|
        \right)^p
    \right]
    \leq 
    \underbrace{c_1c_2\left(\frac{p}{p-1}\right)^p}_{\delequal c^\prime} M^p 
        \E \left[\rule{0cm}{0.9cm}
            \left|
                \sum_{i = 1}^{\floor{t/\eta}} \eta\widehat{Z}^{(1)}_{i;k}
            \right|^p
        \right]
    \ \ \ 
    \label{proof: LDP, small jump perturbation, ineq 1}
\end{align}
for some $c_1,c_2>0$ that only depend on $p$ and won't vary with $(\bm V_i)_{i \geq 0}$ and $\eta$.
The first and third inequalities are from the uppper and lower bounds of Burkholder-Davis-Gundy inequality (Theorem 48, Chapter IV of \cite{protter2005stochastic}), respectively, and 
the fourth inequality is from Doob's maximal inequality.
It then follows from Bernstein's inequality that for any $\eta > 0$ and any $s \in [0,t], y \geq 1$
\begin{align}
    \P\Bigg( \bigg|\sum_{j = 1}^{ \floor{s/\eta} }\eta \widehat{Z}^{(1)}_{j;l,k} \bigg|^p > \eta^{2N}y \Bigg)
    &
    = 
    \P\Bigg( \bigg|\sum_{j = 1}^{ \floor{s/\eta} }\eta \widehat{Z}^{(1)}_{j;l,k} \bigg| > \eta^{ \frac{2N}{p} }{y^{1/p}} \Bigg)
    \nonumber
    \\
    &
    \leq
    2\exp\Bigg( - \frac{ \frac{1}{2}\eta^{\frac{4N}{p}}\sqrt[p]{y^2} }{ \frac{1}{3}\eta^{1-\beta+ \frac{2N}{p}}\sqrt[p]{y} + \frac{t}{\eta}\cdot\eta^2\cdot\E\big[(\widehat{Z}^{(1)}_{1;k})^2\big]  } \Bigg).
    \label{proof: applying berstein ineq, lemma LDP, small jump perturbation}
\end{align}
Our next goal is to show that $\frac{t}{\eta} \cdot \eta^2 \cdot \E\big[(\widehat{Z}^{(1)}_{1;k})^2\big] < \frac{1}{3}\eta^{1 - \beta + \frac{2N}{p}}$
for any $\eta > 0$ small enough.
First, due to $(a+b)^2 \leq 2a^2 + 2b^2$,
\begin{align*}
\E\big[(\widehat{Z}^{(1)}_{1;k})^2\big]
& = \E\big[\big( Z^{(1)}_{1;k} - \E Z^{(1)}_{1;k} \big)^2\big]
\leq 
2\E\big[\big( Z^{(1)}_{1;k}\big)^2\big] + 2\big[\E Z^{(1)}_{1;k}\big]^2
\leq 
2\E\bigg[\norm{\bm Z^{(1)}_1}^2\bigg] 
    + 
    2\bigg[\E \norm{\bm Z^{(1)}_1}\bigg]^2.
\end{align*}
Also, it has been shown earlier that $\E \norm{\bm Z^{(1)}_1} \in \RV_{(\alpha - 1)\beta}(\eta)$,
and hence $\Big[\E \norm{\bm Z^{(1)}_1}\Big]^2 \in \RV_{2(\alpha - 1)\beta}(\eta)$.
From our choice of $p > \frac{2N}{(2\alpha - 1)\beta}$ in \eqref{proof: LDP, small jump perturbation, choose p},
we have 
$
1 + 2(\alpha - 1)\beta > 1 - \beta + \frac{2N}{p},
$
thus implying 
$$\frac{t}{\eta} \cdot \eta^2 \cdot 2\bigg[\E \norm{\bm Z^{(1)}_1}\bigg]^2 < \frac{1}{6}\eta^{1 - \beta + \frac{2N}{p}}$$
for any $\eta > 0$ sufficiently small.
Next,
$\E\big[\norm{ \bm Z_1^{(1)}}^2\big] = \int^{\infty}_0 2x \P(\norm{ \bm Z_1^{(1)}} > x)dx = \int^{1/\eta^\beta}_0 2x \P(\norm{\bm Z_1} > x)dx.$ 
If $\alpha \in (1,2]$, then Karamata's theorem implies
$\int^{1/\eta^\beta}_0 2x \P(\norm{\bm Z_1}> x)dx \in \RV_{ -(2-\alpha)\beta }(\eta)$ as $\eta \downarrow 0$.
Given our choice of $p$ in \eqref{proof: LDP, small jump perturbation, choose p},
one can see that
$1 - (2-\alpha)\beta > 1 - \beta + \frac{2N}{p}$.
As a result, 
for any $\eta > 0$ small enough
we have
$\frac{t}{\eta}\cdot\eta^2\cdot 2\E\big[\norm{\bm Z_1^{(1)}}^2\big] < \frac{1}{6}\eta^{1-\beta + \frac{2N}{p}}$.
If $\alpha > 2$, then
$\lim_{\eta\downarrow 0}\int^{1/\eta^\beta}_0 2x \P(\norm{\bm Z_1} > x)dx =  \int^{\infty}_0 2x \P(\norm{\bm Z_1} > x)dx < \infty$.
Also,  \eqref{proof: LDP, small jump perturbation, choose p} implies that $1 - \beta + \frac{2N}{p} < 1$.
As a result, for any $\eta > 0$ small enough
we have
$\frac{t}{\eta}\cdot\eta^2\cdot 2\E\big[\norm{\bm Z_1^{(1)}}^2\big] < \frac{1}{6}\eta^{1-\beta + \frac{2N}{p}}$.
In summary, 
\begin{align}
    \frac{t}{\eta} \cdot \eta^2 \cdot \E\big[(\widehat{Z}^{(1)}_{1;k})^2\big] < \frac{1}{3}\eta^{1 - \beta + \frac{2N}{p}}
    \label{term second order moment hat Z 1, lemma LDP, small jump perturbation}
\end{align}
holds for any $\eta > 0$ small enough.
Along with \eqref{proof: applying berstein ineq, lemma LDP, small jump perturbation},
we yield that for any $\eta > 0$ small enough,
\begin{align*}
     \P\Bigg( \bigg|\sum_{j = 1}^{ \floor{s/\eta} }\eta \widehat{Z}^{(1)}_{j;l,k} \bigg|^p > \eta^{2N}y \Bigg)
     \leq 
     2\exp\Bigg( \frac{ -\frac{1}{2}{y^{1/p}} }{ \frac{2}{3}\eta^{1-\beta- \frac{2N}{p}} } \Bigg)
     \leq 
     2\exp\Big( -\frac{3}{4}{y^{1/p}} \Big)
     \qquad \forall y \geq 1,
\end{align*}
where the last inequality is due to our choice of $p$ in \eqref{proof: LDP, small jump perturbation, choose p} that $1 - \beta - \frac{2N}{p} > 0$.
Moreover, since $\int_0^\infty\exp\big( -\frac{3}{4}{y^{1/p}} \big)dy < \infty$,
one can see the existence of some $C^{(1)}_p < \infty$ such that
$ \E\Big|\sum_{j = 1}^{ \floor{t/\eta} }\eta \widehat{Z}^{(1)}_{j;l,k} \Big|^p\Big/\eta^{2N} < C^{(1)}_p $ for all $\eta > 0$ small enough.
Combining this bound, \eqref{proof: LDP, small jump perturbation, ineq 1}, and Markov inequality, 
for all $\eta > 0$ small enough,
\begin{align*}
    \P\Bigg( \underset{ j \leq \floor{t/\eta} }{\max}\ 
                \eta| Y_{l,k}(j;\bm V) | > \frac{\epsilon}{3 \sqrt{md^2} } \Bigg)
    & \leq 
    \frac{ \E\Big[ \underset{ j \leq \floor{t/\eta} }{\max} \Big|\sum_{i = 1}^j \eta Y_{l,k}(j,\bm V)\Big|^p \Big] }{ \epsilon^p/(3\sqrt{md^2})^p }
    \\
    & \leq 
    \frac{ c^\prime M^p \E\Big|\sum_{j = 1}^{ \floor{t/\eta} }\eta \widehat{Z}^{(1)}_{i;k} \Big|^p }{ \epsilon^p/(3\sqrt{md^2})^p }
    \leq 
    \frac{c' M^p\cdot C^{(1)}_p }{ \epsilon^p/(3\sqrt{md^2})^p } \cdot \eta^{2N}
\end{align*}
holds uniformly for all $(\bm  V_i)_{i \geq 0} \in \bm{\Gamma}_M$.
This proves \eqref{proof: LDP, small jump perturbation, asymptotics part 2, R1 decomp} and hence \eqref{proof: LDP, small jump perturbation, asymptotics part 2}.  

Finally, for \eqref{proof: LDP, small jump perturbation, asymptotics part 3}, recall that we have chosen $\beta$ in such a way that $\alpha\beta - 1 > 0$.
Fix a constant $J = \ceil{\frac{N}{\alpha\beta - 1}} + 1$,
and define
$I(\eta) \delequal \#\big\{ i \leq \floor{t/\eta}:\ \bm Z^{(2)}_i \neq 0 \big\}$.
Besides, fix $\delta_0 = \frac{\epsilon}{3MJ}$.
For any $\delta \in (0,\delta_0)$ and $(\bm V_i)_{i \geq 0} \in \bm{\Gamma}_M$,
note that on event $\{I(\eta) < J\}$,
we must have
$\underset{ j \leq \floor{t/\eta} }{\max}\ \eta\norm{\sum_{i = 1}^j \bm V_{i-1}\bm Z^{(2)}_i}
< \eta \cdot M \cdot J \cdot \delta_0/\eta < MJ\delta_0 < \epsilon/3$.
On the other hand,
\begin{align*}
    \P\big( I(\eta) \geq J\big)
    \leq 
    \binom{\floor{t/\eta}}{J}\cdot \Big(H\big({1}/{\eta^\beta}\big)\Big)^J
    \leq 
    (t/\eta)^J \cdot \Big(H\big({1}/{\eta^\beta}\big)\Big)^J \in \RV_{ J(\alpha\beta - 1) }(\eta)\text{ as }\eta \downarrow 0.
\end{align*}
Lastly, the choice of $J = \ceil{\frac{N}{\alpha\beta - 1}} + 1$ guarantees that $J(\alpha\beta - 1) > N$,
and hence,
$$
\lim_{\eta \downarrow 0}{ \underset{ (V_i)_{i \geq 0} \in  \bm{\Gamma}_M}{\sup}
    \P\Bigg( \underset{ j \leq \floor{t/\eta} }{\max}\ \eta\norm{\sum_{i = 1}^j \bm V_{i-1}\bm Z^{(2)}_i} > \frac{\epsilon}{3} \Bigg)  }\Big/{\eta^N} 
\leq 
\lim_{\eta \downarrow 0}{ \underset{ (V_i)_{i \geq 0} \in  \bm{\Gamma}_M}{\sup}
\P( I(\eta) \geq J )  }\Big/{\eta^N} = 0.
$$
This concludes the proof of part $(a)$.

\medskip
$(b)$
To ease notations, in this proof we write $\bm X^\eta = \bm X^{\eta|\infty}$ for the cases where $b = \infty$.
Due to Assumption \ref{assumption: boundedness of drift and diffusion coefficients},
it holds for any $\bm x\in\R^m$ and any $\eta > 0, n \geq 0$ that $\norm{\bm \sigma\big(\bm X^{\eta|b}_n(\bm x)\big)} \leq C$,
so
$\{\bm \sigma(\bm X_i^{\eta|b}(\bm x))\}_{i\geq 0} \in \bm{\Gamma}_C$.
By
strong Markov property at stopping times $\big(\tau^{>\delta}_j(\eta)\big)_{j \geq 1}$,
\begin{align*}
    \sup_{\bm x \in \R^m}\P\Bigg( \Big(\bigcap_{i = 1}^{k}A_i(\eta,b,\epsilon,\delta,\bm x)\Big)^c \Bigg)
    & \leq 
    \sum_{i = 1}^k \sup_{\bm x \in \R^m}\P\Bigg( \Big(A_i(\eta,b,\epsilon,\delta,\bm x)\Big)^c \Bigg)
    \\
    & \leq
    k \cdot  \underset{ (\bm V_i)_{i \geq 0} \in  \bm{\Gamma}_C}{\sup}
        \P\Bigg( \underset{ j \leq \floor{1/\eta} \wedge \big(\tau^{>\delta}_{1}(\eta) - 1\big) }{\max}\ 
            \eta\norm{\sum_{i = 1}^j \bm V_{i-1}\bm Z_i } > \epsilon \Bigg)
\end{align*}
where $C < \infty$ is the constant in Assumption \ref{assumption: boundedness of drift and diffusion coefficients} and the set $\bm{\Gamma}_C$ is defined in \eqref{def: Gamma M, set of bounded adapted process}.
Thanks to part $(a)$,
one can find some $\delta_0 = \delta_0(\epsilon,C,N) \in (0,\bar{\delta})$ such that
$$
\underset{ (\bm V_i)_{i \geq 0} \in  \bm{\Gamma}_C}{\sup}
        \P\Bigg( \underset{ j \leq \floor{1/\eta} \wedge \big(\tau^{>\delta}_{1}(\eta) - 1\big) }{\max}\ 
            \eta\norm{\sum_{i = 1}^j \bm V_{i-1}\bm Z_i } > \epsilon \Bigg)
= \bm{o}(\eta^N)$$
(as $\eta \downarrow 0$) for any $\delta \in (0,\delta_0)$,
which concludes the proof of part $(b)$.
\end{proof}

Next, for any $c > \delta > 0$,
we study the law of $\big(\tau^{>\delta}_j(\eta)\big)_{j \geq 1}$ and $\big(\bm W^{>\delta}_j(\eta)\big)_{j \geq 1}$ conditioned on event
\begin{align}
\linkdest{notation-set-E-delta-LDP}
    \notationdef{notation-set-E-delta-LDP}{E^\delta_{c,k}(\eta)} \delequal 
    \Big\{ 
        \tau^{>\delta}_{k}(\eta) \leq \floor{1/\eta} < \tau^{>\delta}_{k+1}(\eta);\ 
        \eta\norm{\bm W^{>\delta}_j(\eta)} > c\ \ \forall j \in [k] 
    \Big\}. \label{def: E eta delta set, LDP}
\end{align}
The intuition is that, on event $E^\delta_{c,k}(\eta)$, among the first $\floor{1/\eta}$ steps there are exactly $k$ ``large'' jumps, 
all of which has size larger than $c/\eta$. 
Next, 
for each $c > 0$,
we
consider a random vector ${\bm W^*(c)}$ in $\R^d$ with $\norm{\bm W^*(c)} > c$ almost surely, whose polar coordinates
$
\big(R^*(c),\bm \Theta^*(c)\big) \delequal \Big( \norm{\bm W^*(c)}, \frac{\bm W^*(c)}{\norm{\bm W^*(c)}}\Big)
$
admit the law
\begin{align}
\P\bigg(
    \Big(R^*(c),\bm \Theta^*(c)\Big) \in\ \cdot\ 
\bigg)
=
\big(\bar\nu_\alpha|_{(c,\infty)} \times \mathbf S\big)(\cdot).
    \label{def: prob measure Q, LDP}
\end{align}
Here, recall the definition of the measure $\nu_\alpha$ in \eqref{def: measure nu alpha} and the measure $\mathbf S$ in Assumption~\ref{assumption gradient noise heavy-tailed}, and note that $\alpha > 1$ is the heavy-tail index in Assumption~\ref{assumption gradient noise heavy-tailed}.
For any $c > 0$, we set
\begin{align*}
    \bar \nu_\alpha|_{(c,\infty)}\big(\cdot\big)
    \delequal
    c^{\alpha} \cdot \nu_\alpha\Big(\ \cdot\ \cap (c,\infty)\Big).
\end{align*}
to be the restricted and normalized (as a probability measure) version of $\nu_\alpha$ over $(c,\infty)$.
Let
$\big(\notationdef{notation-W-*_j}{\bm W^*_j(c)}\big)_{j \geq 1}$ be a sequence of iid copies of $\bm W^*(c)$.
Also, for $\notationdef{notation-U-j-t}{(U_j)_{j \geq 1}}$, a sequence of iid copies of Unif$(0,1)$ that is also independent of $\big(\bm W^*_j(c)\big)_{j \geq 1}$,
let
$\notationdef{notation-U-j-k-LDP}{U_{(1;k)} \leq U_{(2;k)}\leq \cdots \leq U_{(k;k)}}$
be the order statistics of $(U_j)_{j = 1}^{k}$.
For any random element $X$ and any Borel measureable set $A$,
let $\notationdef{notation-law-of-X}{\mathscr{L}(X)}$ be the law of $X$, and $\notationdef{notation-law-of-X-cond-on-A}{\mathscr{L}(X|A)}$ be the conditional law of $X$ given event $A$.

\begin{lemma} \label{lemma: weak convergence of cond law of large jump, LDP}
\linksinthm{lemma: weak convergence of cond law of large jump, LDP}%
Let Assumption \ref{assumption gradient noise heavy-tailed} hold.
For any $\delta >0, c \geq \delta$ and $k \in \mathbb{Z}^+$,
\begin{align}
    \lim_{\eta \downarrow 0}\frac{\P\big( E^\delta_{c,k}(\eta)\big)}{\lambda^k(\eta)} 
=  \frac{1/c^{\alpha k}}{k!},
\label{claim 1, lemma: weak convergence of cond law of large jump, LDP}
\end{align}
and
\begin{align*}
    & \mathscr{L}\Big( \eta \bm W^{>\delta}_{1}(\eta),\eta \bm W^{>\delta}_{2}(\eta),\cdots,\eta \bm W^{>\delta}_{k}(\eta),\eta \tau^{>\delta}_{1}(\eta),\eta \tau^{>\delta}_{2}(\eta),\cdots,\eta \tau^{>\delta}_{k}(\eta) \Big| E^\delta_{c,k}(\eta)\Big) \\
    \rightarrow & \mathscr{L}\Big(\bm W^*_1(c), \bm W^*_2(c),\cdots,\bm W^*_{k}(c), U_{(1;k)},U_{(2;k)},\cdots,U_{(k;k)} \Big)
    \text{ as }\eta \downarrow 0.
\end{align*}
\end{lemma}

\begin{proof}
\linksinpf{lemma: weak convergence of cond law of large jump, LDP}%
Note that $\big( \tau^{>\delta}_i(\eta) \big)_{i \geq 1}$ is independent of $\big( \bm W^{>\delta}_i(\eta) \big)_{i \geq 1}$.
Therefore,
$ \P\big( E^\delta_{c,k}(\eta) \big)  
    = \P\big(\tau^{>\delta}_{k}(\eta) \leq \floor{1/\eta} < \tau^{>\delta}_{k+1}(\eta) \big)\cdot 
    \Big(\P\big( \eta\norm{\bm W^{>\delta}_{1}(\eta)} > c \big)\Big)^{k}.$
Recall that $H(x) = \P(\norm{\bm Z} > x)$. Observe that
\begin{equation}\label{proof, prob tau decomp, lemma: weak convergence of cond law of large jump, LDP}
    \begin{split}
        \P\Big(\tau^{>\delta}_{k}(\eta) \leq \floor{1/\eta} < \tau^{>\delta}_{k+1}(\eta) \Big)
    & 
    =
   \P\Big( \#\big\{ j \leq \floor{1/\eta}:\ \eta|Z_j| > \delta  \big\} = k  \Big)
    \\
    &
    =\underbrace{\binom{\floor{1/\eta} }{ k}}_{ \delequal q_1(\eta) }
    \underbrace{\Big( 1 - H(\delta/\eta) \Big)^{ \floor{1/\eta} - k}}_{\delequal q_2(\eta)} \underbrace{\big(H(\delta/\eta)\big)^{k}}_{\delequal q_3(\eta)}.
    \end{split}
\end{equation}
For $q_1(\eta)$, note that
\begin{align}
    \lim_{\eta \downarrow 0}\frac{ q_1(\eta) }{1/\eta^{k} }
    =
    \frac{  \big( \floor{1/\eta}  \big)\big( \floor{1/\eta} - 1 \big)\cdots\big( \floor{1/\eta} - k + 1 \big)\big/ k!   }{ 1/\eta^{k} }
    =
    \frac{ 1  }{k!}.
    \label{proof: lemma weak convergence of cond law of large jump, 1, LDP}
\end{align}
Also, since $(\floor{1/\eta} - k)\cdot H(\delta/\eta) = \bm o(1)$ as $\eta\downarrow0$, we have that
$\lim_{\eta \downarrow 0}q_2(\eta) = 1$.
Lastly, note that 
\begin{align*}
\P\big( \eta\norm{\bm W^{>\delta}_{1}(\eta)} > c \big)
&
= H(c/\eta)   \Big/ H(\delta/\eta),
\end{align*}
and hence, 
\begin{align}
    \lim_{\eta \downarrow 0}\frac{q_3(\eta) \cdot \Big(\P\big( \eta\norm{\bm W^{>\delta}_{1}(\eta)} > c \big)\Big)^{k}}{ \big( H(1/\eta) \big)^{k} }
    & = 
    \lim_{\eta \downarrow 0}
    \frac{\big( H(\delta/\eta) \big)^{k} \cdot \Big( H(c/\eta)   \Big/ H(\delta/\eta)  \Big)^{k}}{ \big( H(1/\eta) \big)^{k}  }
    =  \lim_{\eta \downarrow 0} \frac{\big( H(c/\eta)  \big)^{k}}{ \big( H(1/\eta) \big)^{k}  }
    =
    1/c^{\alpha k}
    \label{proof, term q 3, lemma: weak convergence of cond law of large jump, LDP}
\end{align}
Plugging \eqref{proof: lemma weak convergence of cond law of large jump, 1, LDP} and \eqref{proof, term q 3, lemma: weak convergence of cond law of large jump, LDP} into \eqref{proof, prob tau decomp, lemma: weak convergence of cond law of large jump, LDP},
we obtain \eqref{claim 1, lemma: weak convergence of cond law of large jump, LDP}.
\elaborate{$$
\lim_{\eta \downarrow 0}
\frac{\P\big( E^\delta_{c,k}(\eta)\big) }{\lambda^k(\eta)} 
=  
\frac{q_1(\eta)\cdot q_2(\eta)\cdot q_3(\eta) \cdot \Big(\P\big( \eta\big|W^{>\delta}_{1}(\eta)\big| >c \big)\Big)^{k}}{1/\eta^k\big(H(1/\eta)\big)^k}
=\frac{1/c^{\alpha k}}{k!}.
$$
}

Next, we move onto the proof of the weak convergence.
We use 
$
\big(R^{>\delta}_1(\eta),\bm \Theta^{>\delta}_1(\eta)\big)
\delequal
\Big(
\norm{\bm W^{>\delta}_1(\eta)}, \frac{\bm W^{>\delta}_1(\eta)}{\norm{\bm W^{>\delta}_1(\eta)}}
\Big)
$
to denote the polar coordinates of $\bm W^{>\delta}_1(\eta)$.
Observe the following weak convergence:
\begin{align*}
    & \P\bigg(  
        \Big(\eta R^{>\delta}_1(\eta),\bm \Theta^{>\delta}_1(\eta)\Big)\in \ \cdot\ 
        \bigg|\ 
        \eta R^{>\delta}_1(\eta) > c
    \bigg)
    \\
    & = 
    \frac{
        \P\Big(
            \big(\eta R^{>\delta}_1(\eta),\bm \Theta^{>\delta}_1(\eta)\big)\in \ \cdot\ \cap \big((c,\infty)\cap \mathfrak N_d\big)
        \Big)
    }{
        \P\big( \eta \norm{\bm W^{>\delta}_1(\eta)}>c \big)
    }
    \\
    & = 
    \frac{
    \P\Big(
        \big(\eta R,\bm \Theta\big)\in \ \cdot\ \cap \big((c,\infty)\cap \mathfrak N_d\big)
    \Big)\Big/\P( \eta\norm{\bm Z} > \delta )
    }{
        \P( \eta \norm{\bm Z}>c )\big/\P( \eta \norm{\bm Z}>\delta )
    }
    \qquad
    \text{with $(R,\bm \Theta) = ( \norm{\bm Z}, \frac{\bm Z}{\norm{\bm Z}} )$}
    \\ 
    & = 
    \frac{
    \P\Big(
        \big(\eta R,\bm \Theta\big)\in \ \cdot\ \cap \big((c,\infty)\cap \mathfrak N_d\big)
    \Big)
    }{
        \P( \eta \norm{\bm Z}>1 )    }
        \cdot \frac{ \P( \eta \norm{\bm Z}>1 )  }{ \P( \eta \norm{\bm Z}>c )  }
    = 
    \frac{
    \P\Big(
        \big(\eta R,\bm \Theta\big)\in \ \cdot\ \cap \big((c,\infty)\cap \mathfrak N_d\big)
    \Big)
    }{
        H(\eta^{-1})    } \cdot \frac{ H(\eta^{-1})  }{ H( c\cdot\eta^{-1}  ) }
    \\ 
    &
    \Rightarrow
    (\bar \nu_\alpha|_{(c,\infty)} \times \mathbf S)(\cdot)
    \qquad 
    \text{as $\eta \downarrow 0$ by Assumption~\ref{assumption gradient noise heavy-tailed}.}
\end{align*}
As a result, we must have
$\mathscr{L}\Big( \eta \bm W^{>\delta}_{1}(\eta),\eta \bm W^{>\delta}_{2}(\eta),\cdots,\eta \bm W^{>\delta}_{k}(\eta)\Big| E^\delta_{c,k}(\eta)\Big)
    \to \mathscr{L}\Big(\bm W^*_1(c), \cdots,\bm W^*_{k}(c) \Big)$.
Moreover, 
one can easily see that, conditioned on the event $E^\delta_{c,k}(\eta)$,
the sequences $\eta \bm W^{>\delta}_{1}(\eta),\cdots,\eta \bm W^{>\delta}_{k}(\eta)$
and
$
\eta \tau^{>\delta}_{1}(\eta),\cdots,\eta \tau^{>\delta}_{k}(\eta)
$
are conditionally independent.
\elaborate{Indeed, for any $1 \leq i_1 < \cdots < i_k \leq \floor{1/\eta}$
and $c_1,\cdots,c_k > c$,
\begin{align*}
    & \frac{ \P\Big(\tau^{>\delta}_j(\eta) = i_j\text{ and }\eta |W^{>\delta}_j(\eta)| > c_j\ \forall j \in [k]\Big) }{ 
    \P\Big(
    \tau^{>\delta}_{k}(\eta) < \floor{1/\eta} < \tau^{>\delta}_{k+1}(\eta);\ \eta|W^{>\delta}_j(\eta)| > c\ \ \forall j \in [k]
    \Big)
    }
    \\
    & = 
    \frac{ \P\big(\tau^{>\delta}_j(\eta) = i_j\ \forall j \geq 1\big)\P\big(\eta |W^{>\delta}_j(\eta)| > c_j\ \forall j \in [k]\big) }{ 
    \P\big(
    \tau^{>\delta}_{k}(\eta) < \floor{1/\eta} < \tau^{>\delta}_{k+1}(\eta)\big)\P\big(\eta|W^{>\delta}_j(\eta)| > c\ \ \forall j \in [k]
    \big)
    }
    \\
    &
    \text{ due to the independence between $\big( \tau^{>\delta}_i(\eta) \big)_{i \geq 1}$ and $\big( W^{>\delta}_i(\eta) \big)_{i \geq 1}$}
    \\
    & = 
    \P\Big(\tau^{>\delta}_j(\eta) = i_j\ \forall j \geq 1\ \Big|\ \tau^{>\delta}_{k}(\eta) < \floor{1/\eta} < \tau^{>\delta}_{k+1}(\eta)\Big)
    \cdot 
    \P\Big(\eta |W^{>\delta}_j(\eta)| > c_j\ \forall j \in [k]\ \Big|\ \eta|W^{>\delta}_j(\eta)| > c\ \ \forall j \in [k]\Big)
    \\
    & = 
    \P\Big(\tau^{>\delta}_j(\eta) = i_j\ \forall j \geq 1\ \Big|\ \tau^{>\delta}_{k}(\eta) < \floor{1/\eta} < \tau^{>\delta}_{k+1}(\eta);\ \eta|W^{>\delta}_j(\eta)| > c\ \ \forall j \in [k]\Big)
    \\
    &\qquad \cdot 
    \P\Big(\eta |W^{>\delta}_j(\eta)| > c_j\ \forall j \in [k]\ \Big|\ \tau^{>\delta}_{k}(\eta) < \floor{1/\eta} < \tau^{>\delta}_{k+1}(\eta);\ \eta|W^{>\delta}_j(\eta)| > c\ \ \forall j \in [k]\Big).
\end{align*}
Again, we applied the independence between $\big( \tau^{>\delta}_i(\eta) \big)_{i \geq 1}$ and $\big( W^{>\delta}_i(\eta) \big)_{i \geq 1}$.} Therefore, as $\eta \downarrow 0$, the limit of the conditional law
$
\mathscr{L}\Big( \eta \bm W^{>\delta}_{1}(\eta),\cdots,\eta \bm W^{>\delta}_{k}(\eta)\Big| E^\delta_{c,k}(\eta)\Big)
$
is also independent from that of
$
\mathscr{L}\Big( \eta \tau^{>\delta}_{1}(\eta),\cdots,\eta \tau^{>\delta}_{k}(\eta)\Big| E^\delta_{c,k}(\eta)\Big),
$
and
it only remains to show that
$$\mathscr{L}\Big( \eta \tau^{>\delta}_{1}(\eta),\eta \tau^{>\delta}_{2}(\eta),\cdots,\eta \tau^{>\delta}_{k}(\eta)\Big| E^\delta_{c,k}(\eta)\Big) \to  \mathscr{L}\Big(U_{(1;k)},\cdots,U_{(k;k)} \Big).$$
Note that since both $\{\eta \tau^{>\delta}_i(\eta): i=1,\ldots,k\}$ and $\{U_{(i):k}: i=1,\ldots,k\}$ are sorted in an ascending order, the joint CDFs are completely characterized by $\{t_i: i=1,\ldots,k\}$'s such that $0 \leq t_1 \leq t_2 \leq \cdots \leq t_{k} \leq 1$. 
For any such $(t_1,\cdots,t_{k}) \in [0,t]^{k}$, note that
\begin{align*}
    & \P\Big( \eta \tau^{>\delta}_{1}(\eta) > t_1,\ \eta \tau^{>\delta}_{2}(\eta) > t_2,\ \cdots, \eta \tau^{>\delta}_{k}(\eta) > t_{k}    \ \Big|\ E^\delta_{c,k}(\eta) \Big)
    \\
    &= 
    \P\Big( \eta \tau^{>\delta}_{1}(\eta) > t_1,\ \eta \tau^{>\delta}_{2}(\eta) > t_2,\ \cdots, \eta \tau^{>\delta}_{k}(\eta) > t_{k}    \ \Big|\ \tau^{>\delta}_{k}(\eta) < \floor{1/\eta} < \tau^{>\delta}_{k+1}(\eta) \Big)
    \\
    &= 
    \frac{
    \P\Big( \eta \tau^{>\delta}_{1}(\eta) > t_1,\ \eta \tau^{>\delta}_{2}(\eta) > t_2,\ \cdots, \eta \tau^{>\delta}_{k}(\eta) > t_{k};\ \tau^{>\delta}_{k}(\eta) < \floor{1/\eta} < \tau^{>\delta}_{k+1}(\eta) \Big) 
    }{
    \P\Big(\tau^{>\delta}_{k}(\eta) < \floor{1/\eta} < \tau^{>\delta}_{k+1}(\eta)\Big)
    }
\end{align*}
and observe that
\begin{align*}
    &\frac{\P\Big( \eta \tau^{>\delta}_{1}(\eta) > t_1,\ \eta \tau^{>\delta}_{2}(\eta) >  t_2,\ \cdots, \eta \tau^{>\delta}_{k}(\eta) > t_{k};\ \tau^{>\delta}_{k}(\eta) < \floor{1/\eta} < \tau^{>\delta}_{k+1}(\eta) \Big)}{\P\Big(\tau^{>\delta}_{k}(\eta) < \floor{1/\eta} < \tau^{>\delta}_{k+1}(\eta)\Big)}
    \\
    &=
    \frac{ \big| \bm{E}^\eta  \big| \cdot q_2(\eta)q_3(\eta)  }{ q_1(\eta) q_2(\eta)q_3(\eta) }
    =\big| \bm{E}^\eta  \big|\Big/q_1(\eta)
\end{align*}
where 
$\bm{E}^\eta \delequal \Big\{ (s_1,\cdots,s_{k}) \in \{1,2,\cdots,\floor{1/\eta}-1\}^{k} :\ \eta s_j > t_j\ \forall j \in [k];\ s_1 < s_2 < \cdots < s_k \Big\}$.
Note that
\begin{align*}
    |\bm{E}^\eta| & =
    \sum_{ s_{k} = \floor{\frac{t_{k}}{\eta}} + 1 }^{\floor{1/\eta} - 1}\
    \sum_{ s_{k- 1} = \floor{\frac{t_{k - 1}}{\eta}} + 1 }^{s_k - 1}\
    \sum_{ s_{k- 2} = \floor{\frac{t_{k - 2}}{\eta}} + 1 }^{s_{k-1}-1}
    \cdots
    \sum_{ s_{2} = \floor{\frac{t_{2}}{\eta}} + 1 }^{s_3 - 1}\
    \sum_{ s_{1} = \floor{\frac{t_{1}}{\eta}} + 1 }^{s_2 - 1}1.
\end{align*}
Together with \eqref{proof: lemma weak convergence of cond law of large jump, 1, LDP},
we obtain
\begin{align*}
\lim_{\eta \downarrow 0} \big| \bm{E}^\eta  \big|\Big/q_1(\eta)
&= 
(k!)\cdot\lim_{\eta \downarrow 0}\frac{ \big| \bm{E}^\eta  \big|}{(1/\eta)^{k}} 
=
(k!)\int_{t_k}^1 \int_{t_{k - 1}}^{s_{k}} \int_{t_{k - 2}}^{s_{k - 1}}
\cdots
\int_{t_{2}}^{s_{3}}\int_{t_1}^{s_2} ds_1 ds_2 \cdots ds_{k}
\\
& = \P\big( U_{(i;k)} >t_i\ \forall i \in [j] \big)
\end{align*}
and conclude the proof.
\end{proof}

Next, we present useful results about mappings $h^{(k)}_{[0,T]}$ and $h^{(k)|b}_{[0,T]}$ defined in \eqref{def: perturb ode mapping h k b, 1}--\eqref{def: perturb ode mapping h k b, 4}.
These results will serve as crucial tools when establishing
Theorems~\ref{corollary: LDP 2} and \ref{theorem: LDP 1, unclipped}.
First,
recall the definitions of the sets $\mathbb{D}_{A}^{(k)}(r)$ and $\mathbb{D}^{(k)|b}_{A}(r)$ in \eqref{def: l * tilde jump number for function g, clipped SGD}, respectively.
The first two results reveal useful properties of $\mathbb{D}_{A}^{(k)}(r)$ and $\mathbb{D}^{(k)|b}_{A}(r)$
when Assumptions~\ref{assumption: lipschitz continuity of drift and diffusion coefficients} and \ref{assumption: boundedness of drift and diffusion coefficients} hold.
As their proofs mostly rely on arguments and calculations independent of those in the other sections of our analyses,
we collect the proofs of Lemmas \ref{lemma: LDP, bar epsilon and delta} and \ref{lemma: LDP, bar epsilon and delta, clipped version} in Section~\ref{sec: appendix, mapping h}.

\begin{lemma}
\label{lemma: LDP, bar epsilon and delta}
\linksinthm{lemma: LDP, bar epsilon and delta}
Let Assumptions \ref{assumption: lipschitz continuity of drift and diffusion coefficients} and \ref{assumption: boundedness of drift and diffusion coefficients} hold.
Given some compact $A \subseteq \R^m$, some $B \in \mathscr{S}_{\mathbb{D}}$, and some $k \in \mathbb N,\ r > 0$,
if $B$ is bounded away from $\mathbb{D}_A^{(k-1)}(r)$,
then there exist $\bar{\epsilon}>0$ and $\bar{\delta}>0$ such that the following claims hold:
\begin{enumerate}[(a)]
    \item
    For any $\bm x \in A$,
    \begin{align*}
        h^{(k)}\big(\bm x,(\bm w_1,\cdots,\bm w_k),\bm{t}\big) \in B^{\bar{\epsilon}}
        \qquad
        \Longrightarrow\qquad
            \norm{\bm w_j} >\bar\delta\ \forall j \in [k];
    \end{align*}

    \item 
        $\bm{d}_{J_1}\big(B^{\bar{\epsilon}},\mathbb{D}_{A}^{(k - 1)}(r)\big) > 0$.
\end{enumerate}
\end{lemma}

\begin{lemma}
\label
{lemma: LDP, bar epsilon and delta, clipped version}
\linksinthm
{lemma: LDP, bar epsilon and delta, clipped version}
Let Assumptions~\ref{assumption: lipschitz continuity of drift and diffusion coefficients} and \ref{assumption: boundedness of drift and diffusion coefficients}  hold.
Given some compact $A \subseteq \R^m$, some $B \in \mathscr{S}_{\mathbb{D}}$, and some $k \in \mathbb N$, $b,r > 0$,
if $B$ is bounded away from $\mathbb{D}^{(k-1)|b}_{A}(r)$,
then there exist $\bar{\epsilon}>0$ and $\bar{\delta}>0$ such that the following claims hold:
\begin{enumerate}[(a)]
    \item 
    for any $\bm x \in A$, $b > 0$, and any $(\bm v_1,\cdots,\bm v_k) \in \R^{m \times k}$ with $\max_{j \in [k]}\norm{\bm v_j} \leq \bar\epsilon$,
    \begin{align*}
        \bar h^{(k)|b}\big(\bm x,(\bm w_1,\cdots,\bm w_k),(\bm v_1,\cdots,\bm v_k),\bm{t}\big) \in B^{\bar{\epsilon}}
        \qquad
        \Longrightarrow\qquad
            \norm{\bm w_j} >\bar\delta\ \forall j \in [k];
    \end{align*}
    
\item $\bm{d}_{J_1}\big(B^{\bar\epsilon},\mathbb{D}_{A}^{(k - 1)|b}(r)\big) > 0$.
\end{enumerate}
\end{lemma}

The next lemma establishes a convergence result from $\mathbf C^{(k)|b}$ to $\mathbf C^{(k)}$.
Again, we collect the proof in Section~\ref{sec: appendix, mapping h}.

\begin{lemma}\label{lemma: convergence from C k b measure to C k measure}
\linksinthm{lemma: convergence from C k b measure to C k measure}
Let Assumptions \ref{assumption: lipschitz continuity of drift and diffusion coefficients} and \ref{assumption: boundedness of drift and diffusion coefficients} hold.
Let $k \in \mathbb N$, $r > 0$, and $A\subseteq \R^m$ be compact.
For any
$g \in \mathcal{C}\big( \mathbb{D}\setminus \mathbb{D}^{(k-1)}_A (r)\big)$ and $\bm x \in A$,
\begin{align*}
    \lim_{b \rightarrow \infty}\mathbf{C}^{(k)|b}   (g;\bm x) = \mathbf{C}^{(k)}(g;\bm x).
\end{align*}
\end{lemma}

In Lemma~\ref{lemma: SGD close to approximation x circ, LDP}, we show that the image of $h^{(1)}$ (resp. $h^{(1)|b}$) provides good approximations of the sample path of $\bm X^\eta_j(\bm x)$ (resp. $\bm X^{\eta|b}_j(\bm x)$) up until $\tau^{>\delta}_1(\eta)$, i.e. the arrival time of the first ``large noise''; see 
\eqref{defArrivalTime large jump},\eqref{defSize large jump} for the definition of $\tau^{>\delta}_{i}(\eta),\bm W^{>\delta}_{i}(\eta)$.

\begin{lemma} 
\label
{lemma: SGD close to approximation x circ, LDP}
\linksinthm
{lemma: SGD close to approximation x circ, LDP}%
Let Assumptions \ref{assumption: lipschitz continuity of drift and diffusion coefficients} and \ref{assumption: boundedness of drift and diffusion coefficients} hold.
Let
$D,C \in [1,\infty)$ be the constants in Assumptions \ref{assumption: lipschitz continuity of drift and diffusion coefficients} and \ref{assumption: boundedness of drift and diffusion coefficients}, respectively, and let $\notationdef{notation-rho-LDP}{\rho} \delequal \exp(D)$.

\begin{enumerate}[(a)]
    \item For any $\epsilon,\delta,\eta > 0$ and any $\bm x,\bm y \in \R^m$, it holds on the event 
    \begin{align*}
        \Bigg\{
        \max_{ i \leq \floor{1/\eta}\wedge \big( \tau^{>\delta}_1(\eta) - 1  \big) }
        \eta\norm{
            \sum_{j = 1}^i \bm \sigma\big(\bm X^\eta_{j-1}(\bm x)\big)\bm Z_j
        } \leq \epsilon
        \Bigg\}
    \end{align*}
    that
    \begin{align}
        \sup_{ t \in [0,1]:\ t <\eta \tau^{>\delta}_1(\eta) }\norm{\xi_t - \bm X^{\eta}_{\floor{ t/\eta }}(\bm x)}
        & \leq
        \rho \cdot \big( \epsilon + \norm{\bm x - \bm y} + \eta C \big),
        \label{ineq, no jump time, a, lemma: SGD close to approximation x circ, LDP}
    \end{align}
    where
   \begin{align}
    \xi = 
    \begin{cases}
     h^{(1)}\big(\bm y,\eta \bm W^{>\delta}_{1}(\eta), \eta \tau^{>\delta}_1(\eta) \big) & \text{ if }\eta \tau^{>\delta}_1(\eta) \leq 1,
     \\
     h^{(0)}(\bm y) & \text{ if }\eta \tau^{>\delta}_1(\eta) > 1.
    \end{cases}
    \nonumber
    \end{align}

    \item 
    For any  $\gamma,b > 0$, $\epsilon \in (0,1)$,
    $\delta \in (0,\frac{b}{2C})$, $\eta \in (0,\frac{b \wedge 1}{2C})$, and $\bm x,\bm y \in \R^m$, it holds on the event 
    \begin{align*}
       \Bigg\{
        \max_{ i \leq \floor{1/\eta}\wedge \big( \tau^{>\delta}_1(\eta) - 1  \big) }
        \eta\norm{
            \sum_{j = 1}^i \bm \sigma\big(\bm X^{\eta|b}_{j-1}(\bm x)\big)\bm Z_j
        } \leq \epsilon
        \Bigg\}
    \cap 
        \Big\{
            \eta \norm{\bm W^{>\delta}_1(\eta)} \leq 1/\epsilon^\gamma
        \Big\}
    \end{align*}
    that
    \begin{align}
        \sup_{ t \in [0,1]:\ t <\eta \tau^{>\delta}_1(\eta) }\norm{ \xi_t - \bm X^{\eta|b}_{\floor{ t/\eta }}(\bm x) }
        & \leq
        \rho \cdot \big( \epsilon + \norm{\bm x-\bm y} + \eta C \big),
        \label{ineq, no jump time, b, lemma: SGD close to approximation x circ, LDP}
        \\
        \sup_{ t \in [0,1]:\ t \leq \eta \tau^{>\delta}_1(\eta) }\norm{ \xi_t - \bm X^{\eta|b}_{\floor{ t/\eta }}(\bm x) }
        & \leq
   \rho D \cdot \big( \epsilon + \norm{\bm x - \bm y} + 2\eta C \big) \cdot \epsilon^{-\gamma}
    \label{ineq, with jump time, b, lemma: SGD close to approximation x circ, LDP}
    \end{align}
    where 
   \begin{align*}
    \xi = 
    \begin{cases}
     h^{(1)|b}\big(\bm y,\eta \bm W^{>\delta}_{1}(\eta), \eta \tau^{>\delta}_1(\eta) \big) & \text{ if }\eta \tau^{>\delta}_1(\eta) \leq 1,
     \\
     h^{(0)|b}(\bm y) & \text{ if }\eta \tau^{>\delta}_1(\eta) > 1.
    \end{cases}
    \end{align*}
\end{enumerate}
\end{lemma}

\begin{proof}
\linksinpf{lemma: SGD close to approximation x circ, LDP}%
\noindent$(a)$
Recall that $\bm{y}_t(\bm x)$ defined in \eqref{def ODE path y t} is the solution to ODE $d \bm{y}_t(\bm x)/dt = \bm a\big(\bm{y}_t(\bm x)\big)$
under initial condition $\bm{y}_0(\bm x) = \bm x$.
By definition of $\xi$,
we have $\xi_t = \bm{y}_t(\bm y)$ 
for any $t \in [0,1]$ with $t < \eta \tau^{>\delta}_1(\eta)$.
Also,
since $\tau^{>\delta}_1(\eta)$ only takes integer values, we know that
$
\eta \tau^{>\delta}_1(\eta) \leq 1 \Longleftrightarrow \tau^{>\delta}_1(\eta) \leq \floor{1/\eta}
$
and
$
\eta \tau^{>\delta}_1(\eta) > 1 \Longleftrightarrow \tau^{>\delta}_1(\eta) >\floor{1/\eta}.
$

Let $A \delequal\Big\{
        \max_{ i \leq \floor{1/\eta}\wedge \big( \tau^{>\delta}_1(\eta) - 1  \big) }
        \eta\norm{
            \sum_{j = 1}^i \bm \sigma\big(\bm X^\eta_{j-1}(\bm x)\big)\bm Z_j
        } \leq \epsilon
        \Big\}.$
Let $\bm{x}^\eta_\cdot(\cdot)$ be the deterministic process defined in \eqref{def: gradient descent process y}.
Applying discrete version of Gronwall's inequality (see, for example, Lemma A.3 of \cite{kruse2014strong})
we know that on event $A$,
\begin{align}
    \norm{ 
        \bm{x}^\eta_j(\bm x) - \bm X^\eta_j(\bm x)
    } \leq \epsilon\cdot\exp( \eta D \cdot \floor{1/\eta}) \leq \rho \epsilon
    \qquad \forall j \leq \floor{1/\eta} \wedge \big( \tau^{>\delta}_1(\eta) - 1 \big). 
    \label{proof, ineq gap between X and y, SGD close to approximation x circ, LDP}
\end{align}
On the other hand, 
since $\xi_t = \bm{y}_t(\bm y)$ for all $t < \eta\tau^{>\delta}_1(\eta)$,
by applying Lemma \ref{lemma Ode Gd Gap} we get
\begin{align}
    \sup_{ t \in [0,1]:\ t < \eta \tau^{>\delta}_1(\eta) }
    \norm{ \xi_t - \bm{x}^\eta_{ \floor{t/\eta} }(\bm x) }
    & \leq 
    \big(\eta C + \norm{\bm x - \bm y}\big)\cdot \rho.
    \label{proof, ineq gap between y and xi, SGD close to approximation x circ, LDP}
\end{align}
Combining \eqref{proof, ineq gap between X and y, SGD close to approximation x circ, LDP} and \eqref{proof, ineq gap between y and xi, SGD close to approximation x circ, LDP}, we get
\begin{align}
\sup_{t \in [0,1]: t < \eta\tau^{>\delta}_1(\eta)}\norm{ \xi_t - \bm X^\eta_{ \floor{t/\eta} }(\bm x) }
&\leq
\rho\cdot \Big( \epsilon + \norm{\bm x - \bm y} + \eta C \Big).
\label{proof, up to t strictly less than eta tau_1^eta, SGD close to approximation x circ, LDP}
\end{align}

\medskip
\noindent
$(b)$
Note that for any $\bm x \in \R^m$ and any $t\in[0,1]$ with $t < \eta \tau^{>\delta}_1(\eta)$,
$$h^{(0)|b}(\bm x)(t) = h^{(0)}(\bm x)(t) 
= h^{(1)|b}\big(\bm x,\eta \bm W^{>\delta}_{1}(\eta),\eta\tau^{>\delta}_1(\eta)\big)(t)= h^{(1)}\big(\bm x,\eta \bm W^{>\delta}_{1}(\eta),\eta\tau^{>\delta}_1(\eta)\big)(t) = \bm{y}_t(\bm x).$$
Also, for any $\bm w \in \R^d$ with $\norm{\bm w} \leq \delta < \frac{b}{2C}$ and any $\bm x \in \R^m$
note that 
$ \varphi_b\Big( \eta \bm a(\bm x) + \bm \sigma(\bm x) \bm w  \Big) =  \eta \bm a(\bm x) + \bm \sigma(\bm x) \bm w$
due to $\eta\norm{\bm a(\bm x)} \leq \eta C < \frac{b}{2}$
and 
$
\norm{\bm\sigma(\bm x)}\norm{\bm w} \leq C \delta  < b/2 
$
(recall our choice of $\eta C < \frac{b}{2} \wedge 1$ and $\delta < \frac{b}{2C}$).
As a result, $\bm X^\eta_j(\bm x) = \bm X^{\eta|b}_j(\bm x)$ for all $\bm x \in \R^m$ and $j < \tau^{>\delta}_1(\eta)$.
It then follows directly from 
\eqref{proof, up to t strictly less than eta tau_1^eta, SGD close to approximation x circ, LDP}
that
on event 
$
\Big\{
        \max_{ i \leq \floor{1/\eta}\wedge \big( \tau^{>\delta}_1(\eta) - 1  \big) }
        \eta\norm{
            \sum_{j = 1}^i \bm \sigma\big(\bm X^{\eta|b}_{j-1}(\bm x)\big)\bm Z_j
        } \leq \epsilon
        \Big\},
$
we have
$$
\sup_{t \in [0,1]: t < \eta\tau^{>\delta}_1(\eta)}\norm{ \xi_t - \bm X^{\eta|b}_{ \floor{t/\eta} }(\bm x) }
\leq 
\rho\cdot \big( \epsilon + \norm{\bm x - \bm y} + \eta C \big).
$$
A direct consequence is (we write $\bm{y}(u;\bm x) = \bm{y}_u(\bm x)$, $\bm{y}(s-;\bm x) = \lim_{u \uparrow s}\bm{y}_u(\bm x)$, and $\xi(t) = \xi_t$ in this proof)
\begin{align}
    \norm{
    \bm{y}( \eta\tau^{>\delta}_1(\eta)-;\bm y) - \bm X^{\eta|b}_{ \tau^{>\delta}_1(\eta) - 1 }(\bm x)
    }\leq \rho \cdot \big( \epsilon + \norm{\bm x - \bm y} + \eta C \big).
    \label{proof, ineq, lemma: SGD close to approximation x circ, LDP}
\end{align}
Therefore,
\begin{align*}
    & \norm{ \xi\big(\eta \tau^{>\delta}_1(\eta)\big) - \bm X^{\eta|b}_{ \tau^{>\delta}_1(\eta) }(\bm x) }
    \\
    & = 
    \Bigg|\Bigg|
    \bm{y}(\eta \tau^{>\delta}_1(\eta)-;\bm y) + \varphi_b\bigg(\eta \bm \sigma\Big(\bm{y}(\eta \tau^{>\delta}_1(\eta)-;\bm y)\Big)\bm W^{>\delta}_{1}(\eta)\bigg)
    \\
    &\qquad\qquad
    -\bigg[
    \bm X^{\eta|b}_{ \tau^{>\delta}_1(\eta) - 1 }(\bm x) + \varphi_b\bigg(\eta \bm a\Big(\bm X^{\eta|b}_{ \tau^{>\delta}_1(\eta) - 1 }(\bm x)\Big)
    + \eta \bm \sigma\Big(\bm X^{\eta|b}_{ \tau^{>\delta}_1(\eta) - 1 }(\bm x)\Big) \bm W^{>\delta}_{1}(\eta)\bigg)
    \bigg]
    \Bigg|\Bigg|
    \\
    & \leq 
    \norm{
    \bm{y}(\eta \tau^{>\delta}_1(\eta)-;\bm y) - \bm X^{\eta|b}_{ \tau^{>\delta}_1(\eta) - 1 }(\bm x)
    }
    \\
    &\qquad\qquad
    +
    \underbrace{
    \norm{
        \varphi_b\bigg(\eta \bm \sigma\Big(\bm{y}(\eta \tau^{>\delta}_1(\eta)-;\bm y)\Big)\bm W^{>\delta}_{1}(\eta)\bigg)
        -
        \varphi_b\bigg(\eta \bm \sigma\Big(\bm X^{\eta|b}_{ \tau^{>\delta}_1(\eta) - 1 }(\bm x)\Big) \bm W^{>\delta}_{1}(\eta)\bigg)
    }
    }_{ \delequal I_1}
    \\
    &\qquad\qquad
    + 
    \underbrace{
    \norm{
        \varphi_b\bigg(\eta \bm \sigma\Big(\bm X^{\eta|b}_{ \tau^{>\delta}_1(\eta) - 1 }(\bm x)\Big) \bm W^{>\delta}_{1}(\eta)\bigg)
        -
        \varphi_b\bigg(\eta \bm a\Big(\bm X^{\eta|b}_{ \tau^{>\delta}_1(\eta) - 1 }(\bm x)\Big)
        + \eta \bm \sigma\Big(\bm X^{\eta|b}_{ \tau^{>\delta}_1(\eta) - 1 }(\bm x)\Big) \bm W^{>\delta}_{1}(\eta)\bigg)
    }
    }_{ \delequal I_2}.
\end{align*}
First, due to $\norm{\varphi_b(\bm x) - \varphi_b(\bm y)} \leq \norm{\bm x - \bm y}$,
\begin{align*}
    I_1
    & \leq 
    \eta\norm{ \bm W^{>\delta}_1(\eta) }
    \cdot 
    \norm{
        \bm \sigma\Big(\bm{y}(\eta \tau^{>\delta}_1(\eta)-;\bm y)\Big)
        -
        \bm \sigma\Big(\bm X^{\eta|b}_{ \tau^{>\delta}_1(\eta) - 1 }(\bm x)\Big)
    }
    \\ 
    & \leq 
    \eta\norm{ \bm W^{>\delta}_1(\eta) }
    \cdot 
    D \cdot 
    \norm{
    \bm{y}(\eta \tau^{>\delta}_1(\eta)-;\bm y) - \bm X^{\eta|b}_{ \tau^{>\delta}_1(\eta) - 1 }(\bm x)
    }
    \qquad 
    \text{by Assumption~\ref{assumption: lipschitz continuity of drift and diffusion coefficients}}
    \\ 
    & \leq 
    \rho D\big( \epsilon + \norm{\bm x - \bm y} + \eta C \big)
    \cdot 
    \eta\norm{ \bm W^{>\delta}_1(\eta) }
    \qquad\text{by \eqref{proof, ineq, lemma: SGD close to approximation x circ, LDP}}
    \\
    & \leq 
    \rho D\big( \epsilon + \norm{\bm x - \bm y} + \eta C \big) \cdot \epsilon^{-\gamma}
    \qquad
    \text{ on event }\Big\{
            \eta \norm{\bm W^{>\delta}_1(\eta)} \leq 1/\epsilon^\gamma
        \Big\}.
\end{align*}
Similarly, we can get
$I_2 \leq 
    \norm{\eta \bm a\Big(\bm X^{\eta|b}_{ \tau^{>\delta}_1(\eta) - 1 }(\bm x)\Big)}
    \leq \eta C$.
In summary, on event $\Big\{
            \eta \norm{\bm W^{>\delta}_1(\eta)} \leq 1/\epsilon^\gamma
        \Big\}$,
\begin{align*}
    \sup_{t \in [0,1]:\ t \leq \eta \tau^{>\delta}_1(\eta)}
    \norm{ \xi_t - \bm X^{\eta|b}_{ \floor{t/\eta} }(\bm x) }
    & \leq  \rho D\big( \epsilon + \norm{\bm x - \bm y} + \eta C \big) \cdot \epsilon^{-\gamma} + \eta C
    \\
    & \leq 
     \rho D\big( \epsilon + \norm{\bm x - \bm y} + 2\eta C \big) \cdot \epsilon^{-\gamma}.
\end{align*}
This concludes the proof of part $(b)$.
\end{proof}

\elaborate{
I moved this part into elaboration because $\hat{\bm X}^{\eta|b;>\delta}$ is not needed in the analysis.

We introduce process $\hat{\bm X}^{\eta|b;>\delta}$ as an approximation for $\bm X^{\eta|b}$.
For any $x \in \mathbb{R}$ and $\eta > 0,\delta > 0,b > 0$, define $\notationdef{notation-low-level-approximation-X-eta-b-LDP}{\hat{X}^{\eta|b;>\delta}_t(x)}$ on $(\Omega,\mathcal{F},\P)$ as the solution to (under initial condition $\hat{X}^{\eta|b;>\delta}_0(x) = x$)
\begin{align*}
    \frac{d \hat{X}^{\eta|b;>\delta}_t(x)}{dt} & = a\big( \hat{X}^{\eta|b;>\delta}_t(x)\big)\qquad \forall t \in [0,\infty)\setminus \{ \eta\tau^{>\delta}_i(\eta):\ i \geq 1 \},
    \\
    \hat{X}^{\eta|b;>\delta}_{\eta\tau^{>\delta}_i(\eta) }(x) & = \hat{X}^{\eta|b;>\delta}_{\eta\tau^{>\delta}_i(\eta) -}(x) + \varphi_b\Big(\sigma\big(\hat{X}^{\eta|b;>\delta}_{\eta\tau^{>\delta}_i(\eta) -}(x)\big)\cdot \eta W^{>\delta}_i(\eta)\Big)\qquad \forall i \geq 1.
\end{align*}
Here $\hat{X}^{\eta|b;>\delta}_{s-}(x) \delequal \lim_{u \uparrow s}\hat{X}^{\eta|b;>\delta}_u(x)$.
Let $\notationdef{notation-low-level-approximation-X-eta-b-LDP-full-path}{\hat{\bm X}^{\eta|b;>\delta}(x)} \delequal \big\{\hat{X}^{\eta|b;>\delta}_t(x):\ t\in[0,1]\big\}$.
By definition of $\hat{\bm X}^{\eta|b;>\delta}(x)$,
it holds for any $\eta,b > 0, k \geq 0, x \in \R$ that
\begin{align*}
\tau^{>\delta}_k(\eta) < \floor{1/\eta} < \tau^{>\delta}_{k+1}(\eta)
\qquad 
\Longrightarrow
\qquad 
    \hat{\bm{X}}^{\eta|b; >\delta }(x)
    =
    h^{(k)|b}\big( x, \eta \bm W^{>\delta}(\eta), \eta \bm \tau^{>\delta}(\eta)\big)
\end{align*}
with $\bm W^{>\delta}(\eta) =  (W^{>\delta}_{1}(\eta), \cdots,W^{>\delta}_k(\eta))$
and 
$\bm \tau^{>\delta}(\eta) =  (\tau^{>\delta}_{1}(\eta), \cdots,\tau^{>\delta}_k(\eta))$.
Furthermore, by applying Lemma \ref{lemma: SGD close to approximation x circ, LDP} inductively, we can reveal the conditions under which $\hat{\bm{X}}^{\eta|b; >\delta }(x)$ is close to $\bm{X}^{\eta|b}(x)$.
}
By applying Lemma \ref{lemma: SGD close to approximation x circ, LDP} inductively,
the next result establishes the conditions under which the image of the mapping $h^{(k)|b}$ approximates the path of $\bm X^{\eta|b}_j(\bm x)$.
\linkdest{location, proof of lemma: SGD close to approximation x breve, LDP clipped}
\begin{lemma}
\label{lemma: SGD close to approximation x breve, LDP clipped}
\linksinthm{lemma: SGD close to approximation x breve, LDP clipped}
Let Assumptions~\ref{assumption: lipschitz continuity of drift and diffusion coefficients} and \ref{assumption: boundedness of drift and diffusion coefficients} hold.
Let $A_i(\eta,b,\epsilon,\delta,\bm x)$ be defined as in \eqref{def: event A i concentration of small jumps, 1}.
For any $k \in \mathbb N$, $\bm x \in \mathbb{R}^m$, $b > 0$, $\epsilon \in (0,1)$,
$\delta \in (0,\frac{b}{2C})$, and $\eta \in (0,\frac{b \wedge \epsilon}{2C})$,
it holds on event
$$\bigg(\bigcap_{i = 1}^{k+1} A_i(\eta,b,\epsilon,\delta,\bm x) \bigg)\cap 
\Big\{ \tau^{>\delta}_k(\eta) \leq \floor{1/\eta} <  \tau^{>\delta}_{k+1}(\eta) \Big\}
\cap 
\Big\{
\eta \norm{\bm W^{>\delta}_i(\eta)} \leq 1/\epsilon^{ \frac{1}{2k}  }\ \forall i \in [k]
\Big\}
$$
that
$$\underset{ t \in [0,1] }{\sup}
\norm{ \xi_t - \bm X^{\eta|b}_{\floor{t/\eta}}(\bm x) } < 
(2\rho D)^{k+1} \sqrt{\epsilon},
$$
where
$$
\xi \delequal h^{(k)|b}\Big(\bm x,\big(\eta \bm W^{>\delta}_1(\eta),\cdots,\eta \bm W^{>\delta}_k(\eta)\big),\big(\eta \tau^{>\delta}_1(\eta),\cdots,\eta \tau^{>\delta}_k(\eta)\big)\Big),
$$
$\rho = \exp(D) \geq 1$,
$D \in[1,\infty)$ is the Lipschitz coefficient in Assumption~\ref{assumption: lipschitz continuity of drift and diffusion coefficients},
and $C \geq 1$ is the constant in Assumption~\ref{assumption: boundedness of drift and diffusion coefficients}.
\end{lemma}

\begin{proof}
\linksinpf{lemma: SGD close to approximation x breve, LDP clipped}%
It is straightforward to see the claim is an immeidate corollary of
\eqref{ineq, with jump time, b, lemma: SGD close to approximation x circ, LDP} in Lemma~\ref{lemma: SGD close to approximation x circ, LDP} when applied inductively
(in particular, with $\gamma = \frac{1}{2k}$, and note that due to our choice of $\eta$, we have $2\eta C < \epsilon$).
To avoid repetition, we omit the details.
\elaborate{First of all, on $A_1(\eta,b,\epsilon,\delta,x)$,
one can apply \eqref{ineq, no jump time, b, lemma: SGD close to approximation x circ, LDP}
of Lemma \ref{lemma: SGD close to approximation x circ, LDP} and obtain
\begin{align*}
    \sup_{t \in [0,1]:\ t < \eta\tau^{>\delta}_1(\eta)}
    \Big| \xi_t - X^{\eta|b}_{\floor{t/\eta}}(x) \Big|
    & = 
   \sup_{t \in [0,1]:\ t < \eta\tau^{>\delta}_1(\eta)}
    \Big| \bm{y}_t(x) - X^\eta_{ \floor{t/\eta} }(x) \Big|
    \leq \rho \cdot (\epsilon + \eta C) < 2\rho \epsilon,
\end{align*}
where we applied our choice of $\eta C < \epsilon/2$.
In case that $k = 0$,
we can already conclude the proof.
Henceforth in the proof, we focus on the case where $k \geq 1$.
Now
we can instead apply \eqref{ineq, with jump time, b, lemma: SGD close to approximation x circ, LDP} of Lemma \ref{lemma: SGD close to approximation x circ, LDP} to get
\begin{align*}
    \sup_{t \in [0,\eta\tau^{>\delta}_1(\eta)]}
   \Big| \xi_t - X^{\eta|b}_{\floor{t/\eta}}(x) \Big|
    & \leq
    \rho \cdot \Big(1 + \frac{bD}{c}\Big)\big( \epsilon + 2\eta C \big)
    \leq 
    3\rho\cdot \Big(1 + \frac{bD}{c}\Big)\epsilon
\end{align*}
due to
our choice of $2\eta C < \epsilon$.
To proceed with an inductive argument,
suppose that for some $j = 1,2,\cdots,k-1$ we can show that 
\begin{align*}
    \sup_{t \in [0,1\wedge \eta\tau^{>\delta}_j(\eta)]}
     \Big|\xi_t - X^{\eta|b}_{\floor{t/\eta}}(x) \Big|
    & \leq 
    \underbrace{
    \Big[ 3\rho\cdot(1 + \frac{bD}{c}) \Big]^{j}\epsilon
    }_{\delequal R_j}.
\end{align*}
To highlight the timestamp in the ODE $\bm y_t(y)$ we write
$\bm{y}(t;y) = \bm{y}_t(y)$ in this proof.
Note that for any $t \in \big[\eta\tau^{>\delta}_j(\eta), \eta\tau^{>\delta}_{j+1}(\eta)\big)$,
we have 
$\xi_t = \bm{y}\Big( t - \eta\tau^{>\delta}_j(\eta);\xi_{\eta\tau^{>\delta}_j(\eta)} \Big)$.
Therefore, by applying \eqref{ineq, with jump time, b, lemma: SGD close to approximation x circ, LDP} of Lemma \ref{lemma: SGD close to approximation x circ, LDP} again,
we obtain
\begin{align*}
    \sup_{t \in \big[\eta\tau^{>\delta}_j(\eta), \eta\tau^{>\delta}_{j+1}(\eta)\big]}
    \Big| \xi_t - X^{\eta|b}_{\floor{t/\eta}}(x) \Big|
    & \leq 
     \rho \cdot \Big(1 + \frac{bD}{c}\Big)
    \cdot (\epsilon + R_j + 2\eta C)
    \\
    & \leq 
   \rho \cdot \Big(1 + \frac{bD}{c}\Big)\cdot (2\epsilon + R_j)
    \qquad 
    \text{due to $2\eta C < \epsilon$}
    \\
    & \leq
   3\rho \cdot \Big(1 + \frac{bD}{c}\Big)R_j = R_{j + 1}
    \qquad\text{ due to }R_j  > \epsilon.
\end{align*}
Arguing inductively, we yield
$
\sup_{t \in [0 ,\eta\tau^{>\delta}_{k}(\eta)]}
      \big| \xi_t - X^{\eta|b}_{\floor{t/\eta}}(x) \big|
     \leq R_k = \big[ 3\rho\cdot(1 + \frac{bD}{c}) \big]^{k}\epsilon.
$
Lastly, due to \eqref{ineq, no jump time, a, lemma: SGD close to approximation x circ, LDP} of Lemma \ref{lemma: SGD close to approximation x circ, LDP}
and the fact that $\eta \tau^{>\delta}_{k+1}(\eta) > 1$,
\begin{align*}
     \sup_{t \in [\eta\tau^{>\delta}_{k}(\eta),1]}
     \Big| \xi_t - X^{\eta|b}_{\floor{t/\eta}}(x) \Big|
     & \leq \rho \cdot (\epsilon + R_k + \eta C)
     \leq \rho \cdot (2\epsilon + R_k)
     \\
     & \leq 
     \rho \cdot 3 R_k 
     < 
    \Big[ 3\rho\cdot(1 + \frac{bD}{c}) \Big]^{k}\cdot 3\rho\epsilon
\end{align*}
This concludes the proof.}e
\end{proof}

To conclude, Lemma~\ref{lemma: sequential compactness for limiting measures, LD of SGD}
provides tools for
verifying the sequential compactness condition \eqref{assumption in portmanteau, uniform M convergence}
for measures $\mathbf C^{(k)}(\ \cdot\ ; \bm x)$ and $\mathbf C^{(k)|b}(\ \cdot\ ;\bm x)$
when we restrict $\bm x$ over a compact set $A$.

\begin{lemma}\label{lemma: sequential compactness for limiting measures, LD of SGD}
\linksinthm{lemma: sequential compactness for limiting measures, LD of SGD}
Let Assumptions~\ref{assumption: lipschitz continuity of drift and diffusion coefficients} and \ref{assumption: boundedness of drift and diffusion coefficients} hold.
Let $T,r > 0$ and $k \geq 1$. Let $A \subseteq \R^m$ be compact.
\begin{enumerate}[(a)]
    \item 
        For any $\bm x_n \in A$ and $\bm x^* \in A$ such that $\lim_{n \to \infty}\bm x_n = \bm x^*$,
        \begin{align*}
            \lim_{n \to \infty}\mathbf C^{(k)}(f;\bm x_{n})
            =
            \mathbf C^{(k)}(f;\bm x^*)
            \qquad 
            \forall f \in \mathcal C\big(\mathbb D[0,T]\setminus \mathbb D^{(k-1)}_A[0,T](r)\big).
        \end{align*}

    \item 
        Let $b > 0$.
        For any $\bm x_n \in A$ and $\bm x^* \in A$ such that $\lim_{n \to \infty}\bm x_n = \bm x^*$,
        \begin{align*}
            \lim_{n \to \infty}\mathbf C^{(k)|b}(f;\bm x_{n})
            =
            \mathbf C^{(k)|b}(f;\bm x^*)
            \qquad 
            \forall f \in \mathcal C\big(\mathbb D[0,T]\setminus \mathbb D^{(k-1)|b}_A[0,T](r)\big).
        \end{align*}
\end{enumerate}
\end{lemma}

\begin{proof}
\linksinpf{lemma: sequential compactness for limiting measures, LD of SGD}
For convenience we consider the case $T = 1$, but the proof can easily extend for arbitrary $T > 0$.

\noindent
$(a)$
Pick some $f \in \mathcal C\big(\D \setminus \D^{(k-1)}_A(r)\big)$,
and let $\phi(\bm x) \delequal \mathbf{C}^{(k)}(f;\bm x)$.
We argue that $\phi(\cdot)$ is a continuous function using Dominated Convergence theorem.
First,
from the continuity of $f$ and $h^{(k)}$ (see Lemma \ref{lemma: continuity of h k b mapping}),
for any sequence $\bm y_n \in \R^m$ with $\lim_{n\to \infty} \bm y_n = \bm y \in \R^m$, we have
\begin{align*}
    \lim_{m \to \infty} f\Big( h^{(k)}(\bm y_m,\textbf W,\bm{t})  \Big)=  f\Big( h^{(k)}(\bm y,\textbf W,\bm{t})  \Big)
    \qquad 
    \forall \textbf W \in \R^{d \times k},\ \bm t \in (0,1)^{k\uparrow}.
\end{align*}
Next, by applying
Lemma~\ref{lemma: LDP, bar epsilon and delta} onto $B = \text{supp}(f)$, which is bounded away from $\D_A^{(k-1)}(r)$,
we find $\bar \delta > 0$ such that
$
h^{(k)}\big(\bm x,(\bm w_1,\cdots,\bm w_k),\bm{t}\big) \in B \Longrightarrow \norm{\bm w_j} > \bar \delta\ \forall j \in [k].
$
As a result,
$$
\Big|f\Big( h^{(k)}\big(\bm x,(\bm w_1,\cdots,\bm w_k),\bm{t}\big)  \Big)\Big| \leq \norm{f} \cdot \mathbbm{I}\big\{\norm{\bm w_j} > \bar \delta\ \forall j \in [k]\big\}.
$$
Also, note that
$
\int \mathbbm{I}\big\{\norm{\bm w_j} > \bar \delta\ \forall j \in [k]\big\} \big((\nu_\alpha \times \mathbf S)\circ \Phi\big)^k(d \textbf W) \times \mathcal{L}^{k\uparrow}_1(d\bm{t}) \leq 1/\bar\delta^{k\alpha} < \infty.
$
This allows us to apply  Dominated Convergence theorem and obtain
\elaborate{
From the continuity of $h^{(k)}$ (see Lemma \ref{lemma: continuity of h k b mapping}) and the boundedness and continuity of $f$,
one can see that $\phi:\R \to [0,\infty)$ is continuous and bounded.
To see why, first note that $\norm{\phi} \leq \norm{f} < \infty$.
Next, applying Lemma \ref{lemma: LDP, bar epsilon and delta} onto set $B = \text{supp}(g)$, which is bounded away from $\D_A^{(k-1)}$,
we can find some constant $\bar \delta > 0$ such that
\begin{align*}
    h^{(k)}(x,\bm w,\bm{t}) \in B\qquad \Longrightarrow\qquad |w_j| > \bar \delta\ \forall j \in [k].
\end{align*}
As a result,
$
\Big|f\Big( h^{(k)}(x,\bm w,\bm{t})  \Big)\Big| \leq \norm{f} \cdot \mathbbm{I}\big(|w_j| > \bar \delta\ \forall j \in [k]\big).
$
In particular, note that 
$
\int \mathbbm{I}\big(|w_j| > \bar \delta\ \forall j \in [k]\big) \nu_\alpha^k(d \bm w) \times \mathcal{L}^{k\uparrow}_1(d\bm{t}) \leq 1/\bar\delta^{k\alpha} < \infty
$
so the dominating function $\norm{f} \cdot \mathbbm{I}(|w_j| > \bar \delta\ \forall j \in [k])$ is integrable.
We also observe that for any $\bm{w} \in \R^{k},\bm{t} \in (0,1)^{k\uparrow}, y \in A$, and any sequence $y_m \in \R$ such that $y_m \rightarrow y$ as $m\rightarrow\infty$, it follows from the continuity of $f$ and $h^{(k)}$ that
\begin{align*}
    f\Big( h^{(k)}(y_m,\bm w,\bm{t})  \Big) \rightarrow  f\Big( h^{(k)}(y,\bm w,\bm{t})  \Big)\qquad\text{ as }m\rightarrow\infty.
\end{align*}
Using Dominated Convergence Theorem, we see that $\phi(x)$ is continuous and bounded.
}
\begin{align*}
    \lim_{n \to \infty} \phi(\bm x_{n}) = \lim_{n \to \infty} \mathbf C^{(k)}(f;\bm x_{n}) = \mathbf C^{(k)}(f;\bm x^*) =  \phi(\bm x^*),
\end{align*}
which the proof of part $(a)$.

\medskip
\noindent
$(b)$ The proof is almost identical. The only differences are that we apply Lemma \ref{lemma: continuity of h k b mapping clipped} (resp. Lemma \ref{lemma: LDP, bar epsilon and delta, clipped version})
instead of Lemma \ref{lemma: continuity of h k b mapping} (resp. Lemma \ref{lemma: LDP, bar epsilon and delta})
so we omit the details.
\end{proof}

\subsection{Proofs of Theorems~\ref{corollary: LDP 2} and \ref{theorem: LDP 1, unclipped}}
\label{subsec: LDP clipped, proof of main results}

In the proofs of Theorems~\ref{corollary: LDP 2} and \ref{theorem: LDP 1, unclipped} below,
without loss of generality we focus on the case where $T = 1$.
But we note that the proof for the cases with arbitrary $T > 0$ is nearly identical.
Recall that, to simplify notations, we write $\bm X^\eta(\bm x) = \bm X^{\eta}_{[0,1]}(\bm x) = \{ \bm X^{\eta}_{\floor{t/\eta}}(\bm x):\ t \in [0,1] \}$,
and
$\bm X^{\eta|b}(\bm x) = \bm X^{\eta|b}_{[0,1]}(\bm x) = \{ \bm X^{\eta|b}_{\floor{t/\eta}}(\bm x):\ t \in [0,1] \}$.

\subsubsection{Proof of Theorem~\ref{theorem: LDP 1, unclipped}}
\label{subsubsec: proof of LDP unclipped}

Recall the notion of uniform $\mathbb{M}$-convergence introduced in Definition \ref{def: uniform M convergence}.
At first glance, the uniform version of $\mathbb{M}$-convergence stated in Theorem~\ref{corollary: LDP 2} and \ref{theorem: LDP 1, unclipped}
is stronger than the standard $\mathbb{M}$-convergence introduced in \cite{lindskog2014regularly}.
Nevertheless,
under the conditions stated in Theorem~\ref{theorem: LDP 1, unclipped} or \ref{corollary: LDP 2} regarding the initial values of $\bm{X}^\eta$ or $\bm{X}^{\eta|b}$,
we can show that it suffices to prove the standard notion of $\mathbb{M}$-convergence.
In particular, the proofs to Theorem~\ref{corollary: LDP 2} and \ref{theorem: LDP 1, unclipped} hinge on the following key proposition for $\bm{X}^{\eta|b}$.

\begin{proposition}
\label{proposition: standard M convergence, LDP clipped}
\linksinthm{proposition: standard M convergence, LDP clipped}
Let $\eta_n$ be a sequence of strictly positive real numbers with $\lim_{n \rightarrow \infty}\eta_n = 0$.
Let compact set $A \subseteq \R^m$ and $\bm x_n,\bm x^* \in A$ be such that $\lim_{n \rightarrow \infty}\bm x_n = \bm x^*$.
Under Assumptions \ref{assumption gradient noise heavy-tailed} and \ref{assumption: lipschitz continuity of drift and diffusion coefficients},
it holds for all $k \in \mathbb N$ and $b,r > 0$ that
\begin{align*}
    { \P\big( \bm{X}^{\eta_n|b}(\bm x_n) \in\ \cdot\ \big)  }\big/{ \lambda^k(\eta_n) } \rightarrow \mathbf{C}^{(k)|b}   \big(\ \cdot\ ; \bm x^*\big)\text{ in }\mathbb{M}\big(\mathbb{D}\setminus \mathbb{D}^{(k-1)|b}_{A}(r) \big)\text{ as }n \rightarrow \infty.
\end{align*}
\end{proposition}

As the first application of Proposition \ref{proposition: standard M convergence, LDP clipped},
in Section~\ref{subsubsec: proof of LDP unclipped}
we prepare a similar result for the unclipped dynamics $\bm{X}^\eta$ defined in \eqref{def: X eta b j x, unclipped SGD} and \eqref{def: scaled SGD, LDP},
which will be the key tool in our proof of Theorem~\ref{theorem: LDP 1, unclipped}.

\begin{proposition}
\label{proposition: standard M convergence, LDP unclipped}
\linksinthm{proposition: standard M convergence, LDP unclipped}
Let $\eta_n$ be a sequence of strictly positive real numbers with $\lim_{n \rightarrow \infty}\eta_n = 0$.
Let compact set $A \subseteq \R^m$ and $\bm x_n,\bm x^* \in A$ be such that $\lim_{n \rightarrow \infty}\bm x_n = \bm x^*$.
Under Assumptions \ref{assumption gradient noise heavy-tailed}, \ref{assumption: lipschitz continuity of drift and diffusion coefficients}, and \ref{assumption: boundedness of drift and diffusion coefficients},
it holds for all $k \in \mathbb N$ and $r > 0$ that
\begin{align*}
    { \P\big( \bm{X}^{\eta_n}(\bm x_n) \in\ \cdot\ \big)  }\big/{ \lambda^k(\eta_n)  }
    \rightarrow \mathbf{C}^{(k)}\big(\ \cdot\ ; \bm x^*\big)\text{ in }\mathbb{M}\big(\mathbb{D}\setminus \mathbb{D}^{(k-1)}_A(r) \big)\text{ as }n \rightarrow \infty.
\end{align*}
\end{proposition}

\begin{proof}
\linksinpf{proposition: standard M convergence, LDP unclipped}
Fix some $k = 0,1,2,\cdots$, $r >0$, and some $g \in \mathcal{C}\big( \mathbb{D}\setminus \mathbb{D}^{(k-1)}_A (r)\big)$.
By virtue of the Portmanteau theorem for $\mathbb{M}$-convergence (see theorem 2.1 of \cite{lindskog2014regularly}),
it suffices to show that 
\begin{align*}
    \lim_{n \rightarrow \infty}{ \E\big[g\big( \bm{X}^{\eta_n}(\bm x_n)\big)\big]  }\big/{ \lambda^k(\eta_n)  }
    =
    \mathbf{C}^{(k)}(g;\bm x^*).
\end{align*}
To this end, we let $B \delequal \text{supp}(g)$ and observe that for any $n \geq 1$ and any $\delta,b > 0$,
\begin{align*}
    & \E\Big[g\big( \bm{X}^{\eta_n}(\bm x_n)\big)\Big] 
    \\
    & = 
    \E\Big[g( \bm{X}^{\eta_n}(\bm x_n))\mathbbm{I}\Big\{\bm{X}^{\eta_n}(\bm x_n) \in B\Big\}\Big]
    \\
    & = 
    \E\Big[ { g( \bm{X}^{\eta_n}(\bm x_n)) \mathbbm{I}\big\{ \tau^{>\delta}_{k+1}(\eta_n) < \floor{1/\eta_n};\ \bm{X}^{\eta_n}(\bm x_n) \in B \big\} }\Big]
    \\
    &\qquad
    +
    \E\Big[{g( \bm{X}^{\eta_n}(\bm x_n)) \mathbbm{I}\Big\{ \tau^{>\delta}_{k}(\eta_n) > \floor{1/\eta_n};\ \bm{X}^{\eta_n}(\bm x_n) \in B\Big\}}\Big]
    \\
    &\qquad
    + 
    \E\bigg[
        {g( \bm{X}^{\eta_n}(\bm x_n)) \mathbbm{I}\bigg\{ \tau^{>\delta}_{k}(\eta_n) \leq \floor{1/\eta_n} < \tau^{>\delta}_{k+1}(\eta_n);\ 
        {\eta_n}\norm{\bm W^{>\delta}_{j}(\eta_n)} > \frac{b}{2C}\ \text{ for some }j \in [k];\ \bm{X}^{\eta_n}(\bm x_n) \in B
        \bigg\}}
    \bigg]
    \\
    & \qquad + 
    \E\bigg[
        \underbrace{g( \bm{X}^{\eta_n}(\bm x_n)) \mathbbm{I}\bigg\{ \tau^{>\delta}_{k}(\eta_n) \leq \floor{1/\eta_n} < \tau^{>\delta}_{k+1}(\eta_n);\ 
        {\eta_n}\norm{\bm W^{>\delta}_{j}(\eta_n)} \leq \frac{b}{2C}\ \forall j \in [k];\ \bm{X}^{\eta_n}(\bm x_n) \in B
        \bigg\}}_{\delequal I_*(n,b,\delta) }
    \bigg],
\end{align*}
where $C \geq 1$ is the constant in Assumption \ref{assumption: boundedness of drift and diffusion coefficients}
such that $\norm{\bm a(\bm x)} \vee \norm{\bm \sigma(\bm x)} \leq C$ for any $\bm x$,
and $\tau^{>\delta}_j(\eta)$'s, ${\bm W^{>\delta}_j(\eta)}$'s are defined in \eqref{defArrivalTime large jump} and \eqref{defSize large jump}.
Now, we focus on the term $I_*(n,b,\delta)$.
For any $n$ large enough, we have $\eta_n \cdot \sup_{\bm x \in \R^m} \norm{\bm a(\bm x)} \leq  \eta_n C \leq b/2$.
As a result,for such $n$ and any $\delta \in (0,\frac{b}{2C})$, on the event
\begin{align*}
    \widetilde A(n,b,\delta) \delequal 
    \bigg\{ \tau^{>\delta}_{k}(\eta_n) \leq \floor{1/\eta_n} < \tau^{>\delta}_{k+1}(\eta_n);\ 
        {\eta_n}\norm{\bm W^{>\delta}_{j}(\eta_n)} \leq \frac{b}{2C}\ \forall j \in [k];\ \bm{X}^{\eta_n}(\bm x_n) \in B
        \bigg\},
\end{align*}
the norm of the step-size (before truncation)
$\eta \bm a\big(\bm X^{\eta|b}_{j - 1}(\bm x)\big) + \eta \bm \sigma\big(\bm X^{\eta|b}_{j - 1}(\bm x)\big)\bm Z_j$ of $\bm X_j^{\eta|b}$ is less than $b$ for each $j \leq \floor{1/\eta_n}$,
and hence $\bm{X}^{\eta_n}(\bm x_n) = \bm{X}^{\eta_n|b}(\bm x_n)$.
\elaborate{
For any $n$ large enough, we have $ \eta_n \cdot \sup_{x \in \R} |a(x)| \leq  \eta_n C \leq b/2$.
For such large $n$, the following claims hold on event
$
\big\{ \tau^{>\delta}_{k}(\eta_n) <\floor{1/\eta_n} < \tau^{>\delta}_{k+1}(\eta_n);\ 
    {\eta_n}|W^{>\delta}_{j}(\eta_n)| \leq \frac{b}{2C}\ \forall j \in [k]
    \big\}
$:
\begin{itemize}
    \item For $i = 1,2,\cdots,\floor{1/\eta_n}$ such that $i = \tau^{>\delta}_{j}(\eta_n)$ for some $j \in [k]$,
    we have
    \begin{align*}
        \big|
        \eta_n a\big(X^{\eta_n}_{i-1}(x_n)\big) + \eta_n \sigma\big(X^{\eta_n}_{i-1}(x_n)\big)Z_i
        \big|
        & = 
        \big|
        \eta_n a\big(X^{\eta_n}_{i-1}(x_n)\big) + \eta_n \sigma\big(X^{\eta_n}_{i-1}(x_n)\big)W^{>\delta}_{j}(\eta_n)
        \big|
        \\
        & \leq 
        \eta_n\big|
        a\big(X^{\eta_n}_{i-1}(x_n)\big) \big|
        +
        \eta_n\cdot \sigma\big(X^{\eta_n}_{i-1}(x_n)\big)  \cdot |W^{>\delta}_{j}(\eta_n)|
        \\
        & \leq \eta_n C + C\cdot \eta_n \cdot |W^{>\delta}_{j}(\eta_n)|
        \\
        & \leq \frac{b}{2} + C \cdot \frac{b}{2C} = b;
    \end{align*}
    
    \item Similarly, for $i = 1,2,\cdots,\floor{1/\eta_n}$ such that $i \neq \tau^{>\delta}_{j}(\eta_n)$ for any $j \in [k]$,
    \begin{align*}
        \big|
        \eta_n a\big(X^{\eta_n}_{i-1}(x_n)\big) + \eta_n \sigma\big(X^{\eta_n}_{i-1}(x_n)\big)Z_i
        \big|
        & \leq 
        \eta_n\big|
        a\big(X^{\eta_n}_{i-1}(x_n)\big) \big|
        +
        \eta_n\cdot \sigma\big(X^{\eta_n}_{i-1}(x_n)\big)  \cdot |Z_i|
        \\
        & \leq \eta_n C + C\cdot \eta_n \cdot |Z_i|
        \\
        & \leq \frac{b}{2} + C \cdot \delta \leq \frac{b}{2} + \frac{b}{2} = b
    \end{align*}
    due to our choice of $\delta \in (0,\frac{b}{2C})$.
\end{itemize}
}
This observation leads to the following upper bound:
Given any $b > 0$ and $\delta \in (0,\frac{b}{2C})$, it holds for any $n$ large enough that
\begin{align*}
     \E\Big[g\big( \bm{X}^{\eta_n}(\bm x_n)\big)\Big]
     & \leq 
     \norm{g}\underbrace{\P\big( \tau^{>\delta}_{k+1}(\eta_n) \leq \floor{1/\eta_n} \big)}_{ \delequal p_1(n,\delta) }
     \\
     &
     +
     \norm{g}\underbrace{\P\big( \tau^{>\delta}_{k}(\eta_n) > \floor{1/\eta_n}
     ;\ \bm{X}^{\eta_n}(\bm x_n) \in B
     \big)}_{ \delequal p_2(n,\delta) }
     \\
     & + \norm{g}\underbrace{\P\bigg( \tau^{>\delta}_{k}(\eta_n) \leq \floor{1/\eta_n} < \tau^{>\delta}_{k+1}(\eta_n);\ 
    {\eta_n}\norm{ \bm W^{>\delta}_{j}(\eta_n)} > \frac{b}{2C}\text{ for some }j \in [k]
    \bigg)}_{\delequal p_3(n,b,\delta)}
    \\
    & + 
    \E\Big[{g\big( \bm{X}^{\eta_n|b}(\bm x_n)\big)}\Big].
\end{align*}
Similarly, given any $n$ large enough, any $b > 0$ and any $\delta \in (0,\frac{b}{2C})$, we have the lower bound
\begin{align*}
     \E\Big[g\big( \bm{X}^{\eta_n}(\bm x_n)\big)\Big]
     & \geq \E\big[I_*(n,b,\delta)\big]
     \\
     & = 
      \E\Big[g\big( \bm{X}^{\eta_n|b}(\bm x_n)\big) \mathbbm{I}\Big(\widetilde{A}(n,b,\delta)\Big)\Big]
      \ \ \ 
      \text{ due to }\bm X^{\eta_n}(\bm x_n) = \bm X^{\eta_n|b}(\bm x_n)\text{ on }\widetilde{A}(n,b,\delta)
    \\
    & \geq
    \E\Big[{g\big( \bm{X}^{\eta_n|b}(\bm x_n)\big)}\Big]
    -\norm{g}\P\Big(\big(\widetilde{A}(n,b,\delta)\big)^\complement\Big)
    \\
    & \geq 
    \E\Big[{g\big( \bm{X}^{\eta_n|b}(\bm x_n)\big)}\Big]
    -
    \norm{g}\cdot\big[p_1(n,\delta) + p_2(n,\delta) +  p_3(n,b,\delta)\big].
\end{align*}
We claim that there exists some $\delta > 0$ such that
\begin{align}
    \lim_{n \rightarrow \infty} p_1(n,\delta)\big/\lambda^k(\eta_n)  & = 0,
    \label{goal 1, proposition: standard M convergence, LDP unclipped}
    \\
    \lim_{n \rightarrow \infty} p_2(n,\delta)\big/ \lambda^k(\eta_n)  & = 0.
    \label{goal 3, proposition: standard M convergence, LDP unclipped}
\end{align}
Furthermore, we claim that for any $b > 0$,
\begin{align}
    \limsup_{n \rightarrow \infty} p_3(n,b,\delta)\big/ \lambda^k(\eta_n)  & \leq \psi_\delta(b) \delequal \frac{k}{ \delta^{ \alpha k } } \cdot \big(\frac{\delta}{2C}\big)^\alpha \cdot \frac{1}{b^{\alpha}}.
    \label{goal 4, proposition: standard M convergence, LDP unclipped}
\end{align}
Note that $\lim_{b \rightarrow \infty}\psi_\delta(b) = 0$.
Lastly, by Lemma~\ref{lemma: convergence from C k b measure to C k measure},
\begin{align}
    \lim_{b \rightarrow \infty}\mathbf{C}^{(k)|b}   (g;\bm x^*) = \mathbf{C}^{(k)}(g;\bm x^*).
    \label{goal 5, proposition: standard M convergence, LDP unclipped}
\end{align}
Then by combining \eqref{goal 1, proposition: standard M convergence, LDP unclipped}--\eqref{goal 4, proposition: standard M convergence, LDP unclipped} with the upper and lower bounds above for $\E\big[g( \bm{X}^{\eta_n}(x_n))\big]$,
we see that for any $b$ large enough (such that $\frac{b}{2C} > \delta$),
\begin{align}
     & \lim_{n \rightarrow \infty}\frac{ \E\big[g( \bm{X}^{\eta_n|b}(\bm x_n))\big]  }{ \lambda^k(\eta_n)  }
     - \norm{g}\psi_\delta(b)
     \leq 
     \lim_{n \rightarrow \infty}\frac{ \E\big[g( \bm{X}^{\eta_n}(\bm x_n))\big]  }{ \lambda^k(\eta_n)  }
     \leq 
     \lim_{n \rightarrow \infty}\frac{ \E\big[g( \bm{X}^{\eta_n|b}(\bm x_n))\big]  }{ \lambda^k(\eta_n)  }
     + \norm{g}\psi_\delta(b),
     \nonumber
     \\
     \Longrightarrow & 
     -\norm{g}\psi_\delta(b) + \mathbf{C}^{(k)|b}   (g;\bm x^*)
     \leq 
     \lim_{n \rightarrow \infty}\frac{ \E\big[g( \bm{X}^{\eta_n}(\bm x_n))\big]  }{ \lambda^k(\eta_n)  }
     \leq 
     \norm{g}\psi_\delta(b) + \mathbf{C}^{(k)|b}   (g;\bm x^*).
     \nonumber
\end{align}
In the last line of the display, we applied Proposition \ref{proposition: standard M convergence, LDP clipped}.
Letting $b$ tend to $\infty$ and applying the limit \eqref{goal 5, proposition: standard M convergence, LDP unclipped}, we conclude the proof.
Now, it only remains to prove \eqref{goal 1, proposition: standard M convergence, LDP unclipped}
\eqref{goal 3, proposition: standard M convergence, LDP unclipped}
\eqref{goal 4, proposition: standard M convergence, LDP unclipped}.

\medskip
\noindent\textbf{Proof of Claim }\eqref{goal 1, proposition: standard M convergence, LDP unclipped}:

We show that this claim holds for any $\delta > 0$.
Applying \eqref{property: large jump time probability}, we see that 
$p_1(n,\delta) \leq \big( H(\frac{\delta}{\eta_n})\big/\eta_n \big)^{k+1}$
holds
for any $\delta > 0$ and any $n \geq 1$.
Due to the regularly varying nature of $H(\cdot)$,
we then yield
$
\limsup_{n \rightarrow \infty}\frac{p_1(n,\delta)}{ \lambda^{k+1}(\eta_n)}
\leq 1/\delta^{\alpha(k+1)} < \infty.
$
Therefore,
\begin{align*}
     \limsup_{n \rightarrow \infty} \frac{p_1(n,\delta)}{\lambda^k(\eta_n)}
     \leq 
     \limsup_{n \rightarrow \infty} \frac{p_1(n,\delta)}{\lambda^{k+1}(\eta_n)}
     \cdot \lim_{n \rightarrow \infty}{\lambda(\eta_n)}
     \leq \frac{1}{\delta^{\alpha (k+1)}}\cdot \lim_{n \rightarrow \infty}\frac{H(1/\eta_n)}{\eta_n} = 0
\end{align*}
due to $\frac{H(1/\eta)}{\eta} = \lambda(\eta) \in \RV_{\alpha - 1}(\eta)$ as $\eta \downarrow 0$ and $\alpha > 1$.



\medskip
\noindent\textbf{Proof of Claim }\eqref{goal 3, proposition: standard M convergence, LDP unclipped}:

We claim the existence of some $\epsilon > 0$ such that
\begin{align}
    \Big\{\tau^{>\delta}_k(\eta) > \floor{1/\eta}
     ;\ \bm{X}^{\eta}(\bm x) \in B\Big\} \cap \bigg( \bigcap_{ i = 1 }^{k+1}A_i(\eta,\infty,\epsilon,\delta,\bm x) \bigg)
      = \emptyset\quad
     \forall \bm x \in A,\ \delta> 0,\ \eta \in (0,\frac{\epsilon}{C\rho})
     \label{subgoal for goal 2, proposition: standard M convergence, LDP unclipped}
\end{align}
where $D,C \in [1,\infty)$ are the constants in Assumptions \ref{assumption: lipschitz continuity of drift and diffusion coefficients} and \ref{assumption: boundedness of drift and diffusion coefficients} respectively,
$\rho \delequal \exp(D)$,
and event $A_i(\eta,b,\epsilon,\delta,\bm x)$ is defined in \eqref{def: event A i concentration of small jumps, 1}.
Then for any $\delta > 0$,
\begin{align*}
    \limsup_{n \rightarrow \infty}{p_2(n,\delta)}\big/{ \lambda^k(\eta_n) } 
     & \leq 
     \limsup_{n \rightarrow \infty}\sup_{x \in A}{
        \P\Bigg( 
            \bigg( \bigcap_{ i = 1 }^{k+1}A_i(\eta_n,\infty,\epsilon,\delta,\bm x) \bigg)^\complement
        \Bigg)}\Big/{ \lambda^k(\eta_n) }.
\end{align*}
Applying Lemma \ref{lemma LDP, small jump perturbation} $(b)$
with $N > k(\alpha - 1)$, we conclude that 
claim \eqref{goal 3, proposition: standard M convergence, LDP unclipped} holds for all $\delta > 0$ small enough.
Now, it only remains to find $\epsilon > 0$ that satisfies condition \eqref{subgoal for goal 2, proposition: standard M convergence, LDP unclipped}.
To this end, we first 
recall that the set $B = \text{supp}(g)$ is bounded away from $\mathbb{D}^{(k-1)}_A(r)$.
By Lemma~\ref{lemma: LDP, bar epsilon and delta}, there is $\bar{\epsilon} > 0$ such that
$
\bm{d}_{J_1}\big(B^{\bar\epsilon},\mathbb{D}^{(k- 1)}_A(r)\big) > \bar\epsilon.
$
W.l.o.g.\ we pick $\bar\epsilon$ small enough such that $\bar\epsilon \in (0,r)$.
Next, we show that \eqref{subgoal for goal 2, proposition: standard M convergence, LDP unclipped} holds for any $\epsilon > 0$ small enough with
$
    (\rho + 1)\epsilon < \bar{\epsilon}.
$
To see why, we fix some $\bm x \in A$, $\delta > 0$ and $\eta \in (0,\frac{\epsilon}{C\rho})$.
Define a process ${\breve{\bm X}^{\eta,\delta}(\bm x)} \delequal \big\{\notationdef{notation-breve-X-eta-delta-t}{\breve{\bm X}^{\eta,\delta}_t(\bm x)}:\ t \in [0,1]\big\}$ as the solution to
(under initial condition $ \breve{\bm X}^{\eta,\delta}_0(\bm x) = \bm x$)
\begin{align*}
    \frac{d \breve{\bm X}^{\eta,\delta}_t(\bm x)}{dt} & = \bm a\big( \breve{\bm X}^{\eta,\delta}_t(\bm x) \big)\qquad \forall t \geq 0,\ t \notin \{ \eta\tau^{>\delta}_j(\eta):\ j \geq 1 \},
    \\
     \breve{\bm X}^{\eta,\delta}_{ \eta \tau^{>\delta}_i(\eta)}(\bm x) & = \bm X^\eta_{ \tau^{>\delta}_i(\eta) }(\bm x)\qquad \forall j \geq 1.
\end{align*}
On event $\big(\cap_{i = 1}^{k + 1} A_i(\eta,\infty,\epsilon,\delta,\bm .x)\big) \cap \{ \tau^{>\delta}_{k}(\eta) > \floor{1/\eta} \}$,
observe that
\begin{align}
& \bm{d}_{J_1}\Big(\breve{\bm X}^{\eta,\delta}(\bm x), \bm{X}^\eta(\bm x)\Big)\nonumber
    \\
    & \leq \sup_{ t \in \big[0,\eta\tau^{>\delta}_1(\eta)\big) \cup 
    \big[\eta\tau^{>\delta}_1(\eta),\eta\tau^{>\delta}_{2}(\eta)\big) \cup \cdots 
    \cup 
    \big[\eta\tau^{>\delta}_{k}(\eta),\eta\tau^{>\delta}_{k+1}(\eta)\big)
    }
    \norm{ \breve{\bm X}^{\eta,\delta}_t(\bm x) - \bm X^\eta_{ \floor{t/\eta} }(\bm x)  }
    \nonumber
    \\
    & \leq \rho\cdot \big( \epsilon + \eta C \big) \leq \rho \epsilon + \epsilon  < \bar\epsilon 
    \qquad \text{because of \eqref{ineq, no jump time, a, lemma: SGD close to approximation x circ, LDP} of Lemma \ref{lemma: SGD close to approximation x circ, LDP}}.
    \nonumber
\end{align}
In the last line of the display above, we applied $\eta < \frac{\epsilon}{C\rho}$ and our choice of $(\rho + 1)\epsilon < \bar\epsilon$.
\elaborate{
\begin{align}
    & \bm{d}_{J_1}\big(\breve{\bm X}^{\eta,\delta}(x), \bm{X}^\eta(x)\big)\nonumber
    \\
    & \leq 
    \sup_{ t \in [0,1] }\Big| \breve{X}^{\eta,\delta}_t(x) - X^\eta_{ \floor{t/\eta} }(x)  \Big|
    \nonumber
    \\
    & \leq \sup_{ t \in \big[0,\eta\tau^{>\delta}_1(\eta)\big) \cup 
    \big[\eta\tau^{>\delta}_1(\eta),\eta\tau^{>\delta}_{2}(\eta)\big) \cup \cdots 
    \cup 
    \big[\eta\tau^{>\delta}_{k}(\eta),\eta\tau^{>\delta}_{k+1}(\eta)\big)
    }
    \Big| \breve{X}^{\eta,\delta}_t(x) - X^\eta_{ \floor{t/\eta} }(x)  \Big|
    \ \ \ \text{by definition 
    of }\breve{\bm X}^{\eta,\delta}
    \nonumber
    \\
    & \leq \rho\cdot \big( \epsilon + \eta C \big) 
    \ \ \ \text{because of \eqref{ineq, no jump time, a, lemma: SGD close to approximation x circ, LDP} of Lemma \ref{lemma: SGD close to approximation x circ, LDP}}
    \nonumber
    \\
    & \leq (\rho + 1)\epsilon < \bar{\epsilon}\ \ \ \text{due to }\eta < \frac{\epsilon}{C\rho}\text{ and our choice of }\epsilon.
    \nonumber
\end{align}
}
However, from the display above, we also learned that on $\{\tau^{>\delta}_k(\eta) > \floor{1/\eta}\}$,
we have $\breve{\bm X}^{\eta,\delta}(x) \in \mathbb{D}^{(k-1)}_A(\bar\epsilon) \subseteq \mathbb{D}^{(k-1)}_A(r)$;
recall that we picked $\bar\epsilon\in (0,r)$.
As a result, on event $\big(\cap_{i = 1}^{k + 1} A_i(\eta,\infty,\epsilon,\delta,\bm x)\big) \cap \{ \tau^{>\delta}_{k}(\eta) > \floor{1/\eta} \}$
we must have $\bm{d}_{J_1}\big( \mathbb{D}^{(k-1)}_A(r) ,\bm{X}^\eta(\bm x)\big) < \bar{\epsilon}$,
and hence
$\bm{X}^\eta(\bm x) \notin B$ due to the fact that $\bm{d}_{J_1}\big(B^{\bar\epsilon},\mathbb{D}^{(k- 1)}_A(r)\big) > \bar\epsilon$.
This verifies \eqref{subgoal for goal 2, proposition: standard M convergence, LDP unclipped}.

\medskip
\noindent\textbf{Proof of Claim }\eqref{goal 4, proposition: standard M convergence, LDP unclipped}:
 
 Due to the independence between $\big(\tau^{>\delta}_i(\eta) - \tau^\eta_{j-1}(\delta)\big)_{j \geq 1}$ and $\big(\bm W^{>\delta}_i(\eta)\big)_{j \geq 1}$,
\begin{align*}
p_3(n,b,\delta) & = 
    \P\Big( \tau^{>\delta}_{k}(\eta_n) < \floor{1/\eta_n} < \tau^{>\delta}_{k+1}(\eta_n)\Big)
    \cdot
    \P\bigg(
    {\eta_n}\norm{\bm W^{>\delta}_{j}(\eta_n)} > \frac{b}{2C}\text{ for some }j \in [k]
    \bigg)
    \\
    & \leq 
    \P\Big( \tau^{>\delta}_{k}(\eta_n) \leq \floor{1/\eta_n} \Big)
    \cdot \sum_{j = 1}^k
    \P\bigg(
    {\eta_n}\norm{\bm W^{>\delta}_{j}(\eta_n)} > \frac{b}{2C}
    \bigg)
    \\
    & \leq 
    \Big(\frac{H(\delta/\eta_n)}{\eta_n}\Big)^k \cdot k \cdot \frac{ H\big( \frac{b}{2C} \cdot \frac{1}{\eta_n} \big)   }{ H\big( \delta \cdot \frac{1}{\eta_n} \big) }.
\end{align*}
Due to $H(x) \in \RV_{-\alpha}(x)$ as $x\to\infty$ (see Assumption~\ref{assumption gradient noise heavy-tailed}), we conclude that
$
\limsup_{n \rightarrow \infty}\frac{p_4(n,b,\delta)}{ \lambda^k(\eta_n) }
    \leq \frac{k}{ \delta^{ \alpha k } } \cdot \big(\frac{\delta}{2C}\big)^\alpha \cdot \frac{1}{b^{\alpha}} = \psi_\delta(b).
$

\end{proof}

With Proposition \ref{proposition: standard M convergence, LDP unclipped} in our arsenal, 
we prove Theorem~\ref{theorem: LDP 1, unclipped}.

\begin{proof}[Proof of Theorem~\ref{theorem: LDP 1, unclipped}]
\linksinpf{theorem: LDP 1, unclipped}%
For simplicity of notations we focus on the case where $T = 1$, but the proof below can be easily generalized for arbitrary $T > 0$.

We first prove the uniform $\M$-convergence.
Specifically,
we proceed with a proof by contradiction.
Fix some $r > 0$ and $k \in \mathbb N$, and
suppose that there is some $f \in \mathcal{C}\big(\mathbb{D}\setminus\mathbb{D}^{(k-1)}_A(r)\big)$,
some sequence $\eta_n > 0$ with limit $\lim_{n \rightarrow \infty}\eta_n = 0$, some sequence $\bm x_n \in A$, and $\epsilon > 0$ such that
$$
\big| \mu^{(k)}_n(f) - \mathbf{C}^{(k)}(f;\bm x_n) \big| > \epsilon \ \forall n \geq 1
\qquad 
\text{ with }\mu^{(k)}_n(\cdot) 
\delequal { \P\big( \bm{X}^{\eta_n}(\bm x_n) \in\ \cdot\ \big)  }\big/{ \lambda^k(\eta_n)  }.
$$
Since $A \subseteq \R^m$ is compact,
by picking a proper subsequence we can assume w.l.o.g.\ that
$
\lim_{n \rightarrow \infty}\bm x_n = \bm x^*
$
for some $\bm x^* \in A$.
This allows us to apply Proposition \ref{proposition: standard M convergence, LDP unclipped}
and yield $\lim_{n \rightarrow \infty}\big| \mu^{(k)}_n(f) - \mathbf{C}^{(k)}(f;\bm x^*) \big| = 0.$
On the other hand,
using part $(a)$ of Lemma~\ref{lemma: sequential compactness for limiting measures, LD of SGD}, we get
$
\lim_{n \rightarrow \infty}\big| 
    \mathbf{C}^{(k)}(f;\bm x_n) - \mathbf{C}^{(k)}(f;\bm x^*)
    \big| = 0.
$
Therefore, we arrive at the contradiction
\begin{align*}
    \lim_{n \rightarrow \infty}
 \big| \mu^{(k)}_n(f) - \mathbf{C}^{(k)}(f;\bm x_n) \big|
 \leq 
 \lim_{n \rightarrow \infty}
 \big| \mu^{(k)}_n(f) - \mathbf{C}^{(k)}(f;\bm x^*)\big|
 +
 \lim_{n\rightarrow \infty}
 \big|\mathbf{C}^{(k)}(f;\bm x^*)-
 \mathbf{C}^{(k)}(f;\bm x_n) \big| = 0
\end{align*}
and conclude the proof of the uniform $\M$-convergence claim.


Next, we prove the uniform sample-path large deviations stated in \eqref{claim, uniform sample path LD, theorem: LDP 1, unclipped}.
Part $(a)$ of Lemma \ref{lemma: sequential compactness for limiting measures, LD of SGD} verifies the compactness condition 
\eqref{assumption in portmanteau, uniform M convergence}
for the family of measures $\{\mathbf C^{(k)}(\ \cdot\ ; \bm x):\ \bm x \in A\}$.
In light of the Portmanteau theorem for uniform $\M$-convergence (i.e., Theorem \ref{theorem: portmanteau, uniform M convergence}),
most claims follow directly from the uniform $\mathbb M$-convergence established above, and
it only remains to verify that  $\sup_{\bm x \in A}\mathbf{C}^{(k)}\big( B^-; \bm x \big) < \infty$.
To do so,
note that $B^-$ is bounded away from $\mathbb{D}_{A}^{(k - 1)}(r)$.
This allows us to apply Lemma \ref{lemma: LDP, bar epsilon and delta} and find $\bar{\epsilon} > 0$, $\bar{\delta} > 0$ such that,
for any $\bm x \in A$ and $\bm t \in (0,1]^{k\uparrow}$,
    $$
    h^{(k)} \big(\bm x,(\bm w_1,\cdots,\bm w_k),\bm{t}\big) \in B^{\bar{\epsilon}} \Longrightarrow \norm{\bm w_j} > \bar\delta\ \forall j \in [k].
    $$
Then by the definition of $\mathbf C^{(k)} = \mathbf C^{(k)|\infty}$ in \eqref{def: measure mu k b t},
\begin{align*}
    \sup_{\bm x \in A}
    \mathbf{C}^{(k)}(B^-;\bm x)
    & =\sup_{\bm x \in A}
    \int\mathbbm{I}\Big\{ h^{(k)} \big( \bm x,(\bm w_1,\cdots,\bm w_k),\bm t  \big) \in B^- \Big\}
    \big((\nu_\alpha \times \mathbf S)\circ \Phi\big)^k(d \textbf W) \times \mathcal{L}^{k\uparrow}_1(d\bm{t}q
    \\
    & \leq 
    \int\mathbbm{I}\Big\{ \norm{\bm w_j} > \bar{\delta}\ \forall j \in [k]   \Big\}
    \big((\nu_\alpha \times \mathbf S)\circ \Phi\big)^k(d \textbf W) \times \mathcal{L}^{k\uparrow}_1(d\bm{t})
    \leq 
    1/\bar\delta^{k\alpha} < \infty.
\end{align*}
This concludes the proof.
\end{proof}

\subsubsection{Proof of Theorem~\ref{corollary: LDP 2}}

Aside from Proposition \ref{proposition: standard M convergence, LDP clipped}, another key tool in our proof of Theorem~\ref{corollary: LDP 2}
is the following ``truncated'' version of the drift and diffusion coefficients $\bm a(\cdot),\bm \sigma(\cdot)$.
Given any $M \geq 1$,
let 
\begin{align}
    \notationdef{a-M}{\bm a_M}(\bm x) \delequal
    \begin{cases}
     \bm a\Big( M \cdot \frac{\bm x}{\norm{\bm x}}  \Big) & \text{ if }\norm{\bm x} > M, \\
     \bm a(\bm x) & \text{ otherwise.}
    \end{cases}
    \ \ \ \ \ \ \ \ \ \ \ 
    \notationdef{sigma-M}{\bm \sigma_M}(\bm x) \delequal
    \begin{cases}
     \bm \sigma\Big( M \cdot \frac{\bm x}{\norm{\bm x}}  \Big) & \text{ if }\norm{\bm x} > M, \\
     \bm \sigma(\bm x) & \text{ otherwise.}
    \end{cases}
    \label{def: a sigma truncated at M, LDP}
\end{align}
That is, we project $\bm x$ onto the closed ball $\{\bm x \in \R^m:\ \norm{\bm x} \leq M\}$.
For any $\bm a(\cdot),\bm \sigma(\cdot)$ satisfying Assumption~\ref{assumption: lipschitz continuity of drift and diffusion coefficients},
one can see that
$\bm a_M(\cdot),\bm \sigma_M(\cdot)$ will satisfy Assumptions~\ref{assumption: lipschitz continuity of drift and diffusion coefficients} and
\ref{assumption: boundedness of drift and diffusion coefficients}.
Similarly, recall the definition of the mapping $\bar h^{(k)|b}$ in \eqref{def: perturb ode mapping h k b, 1}-\eqref{def: perturb ode mapping h k b, 3}.
We also consider its ``truncated'' counterpart by defining
the mapping $\notationdef{notation-mapping-bar-h-k-t-b-M-LDP}{\bar h^{(k)|b}_{M\downarrow}}: \mathbb{R}^m\times \mathbb{R}^{d \times k} \times \R^{m \times k} \times (0,1]^{k\uparrow} \to \mathbb{D}$ as follows.
Given any
$\bm x \in \mathbb{R}^m$,
$\textbf W = (\bm w_1,\cdots,\bm w_k) \in \mathbb{R}^{d \times k},\ \textbf V = (\bm v_1,\cdots,\bm v_j) \in \R^{m \times k},\ \bm{t} = (t_1,\cdots,t_k)\in (0,1]^{k\uparrow}$, let $\xi = \bar h^{(k)|b}_{M\downarrow}(\bm x,\textbf W,\textbf V,\bm{t})$ be the solution to
\begin{align}
    \xi_0 & = \bm x; \label{def: perturb ode mapping h k b, truncated at M, 1}
    \\
    \frac{d\xi_t}{d t} & = \bm a_M(\xi_t)\qquad\forall t \in [0,1],\ t \neq t_1,t_2,\cdots,t_k; \label{def: perturb ode mapping h k b, truncated at M, 2}
    \\
    \xi_t & = \xi_{t-} + \bm v_j + \varphi_b\big( \bm \sigma_M(\xi_{t-} + \bm v_j) \bm w_j\big)\qquad \text{ if }t = t_j\text{ for some }j\in[k]. \label{def: perturb ode mapping h k b, truncated at M, 3}
\end{align}
Define mapping $h^{(k)|b}_{M\downarrow}:\R^m \times \R^{d \times k} \times (0,1]^{k\uparrow} \to \D$ by
\begin{align}
    \notationdef{mapping-h-k-b-M-down}{h^{(k)|b}_{M\downarrow}}\big(\bm x, (\bm w_1,\cdots,\bm w_k),\bm t\big)
    \delequal 
    \bar h^{(k)|b}_{M\downarrow}\big(\bm x, (\bm w_1,\cdots,\bm w_k),(\bm 0,\cdots,\bm 0),\bm t\big).
    \label{def: perturb ode mapping h k b, truncated at M, 4}
\end{align}
Also, recall that $\bar B_r(\bm x)$ is the closed ball with radius $r$ centered at $\bm x$, and set
\begin{align}
    \notationdef{notation-D-A-k-t-truncation-b-M-LDP}{\mathbb{D}_{A;M\downarrow}^{(k)|b}(r)} \delequal 
\bar h^{(k)|b}_{M\downarrow}\Big( A \times \mathbb{R}^{m \times k} \times \big(\bar B_r(\bm 0)\big)^k \times (0,1]^{k\uparrow} \Big).
\label{def D A k t truncation set}
\end{align}
In short, the difference between $\bar h^{(k)|b}_{M\downarrow}$ and $\bar h^{(k)|b}$
is that, when constructing $\bar h^{(k)|b}_{M\downarrow}$,
we use the truncated drift and diffusion coefficients $\bm a_M(\cdot)$ and $\bm \sigma_M(\cdot)$.

The main idea for our proof of Theorem~\ref{corollary: LDP 2} is as follows.
For large enough $M > 0$, one can show that it is very unlikely for the truncated dynamics $\bm X^{\eta|b}(\bm x)$ to exit from the the ball $\bar B_r(\bm 0) = \{\bm y:\ \norm{\bm y} \leq M\}$.
Therefore, it suffices to study the $\M$-convergence and large deviation limits of a modified version of $\bm X^{\eta|b}(\bm x)$, where we use $\bm a_M$ and $\bm \sigma_M$ for the drift and diffusion coefficients, instead of $\bm a$ and $\bm \sigma$.
Since $\bm a_M$ and $\bm \sigma_M$ automatically satisfy the boundedness condition in Assumption~\ref{assumption: boundedness of drift and diffusion coefficients}, we essentially reduce the problem to a simpler one,
whose proof is almost identical to that of Theorem~\ref{theorem: LDP 1, unclipped} and builds upon the technical tools developed in Section~\ref{subsec: proof of lemmas, LD of SGD proof} again.

\begin{proof}[Proof of Theorem~\ref{corollary: LDP 2}]
\linksinpf{corollary: LDP 2}%
First, we argue that the proof is almost identical to that of Theorem~\ref{theorem: LDP 1, unclipped} if Assumption~\ref{assumption: boundedness of drift and diffusion coefficients} also holds.
In particular,
the proof-by-contradiction approach in Theorem~\ref{theorem: LDP 1, unclipped}
can be applied here to
establish the uniform $\M$-convergence.
The only difference is that we apply Proposition \ref{proposition: standard M convergence, LDP clipped} (resp., part $(b)$ of Lemma \ref{lemma: sequential compactness for limiting measures, LD of SGD})
instead of Proposition \ref{proposition: standard M convergence, LDP unclipped} (resp., part $(a)$ of Lemma \ref{lemma: sequential compactness for limiting measures, LD of SGD}).
Similarly, the proof to the uniform sample-path large deviations stated in \eqref{claim, uniform sample path LD, corollary: LDP 2}
is almost identical to that of \eqref{claim, uniform sample path LD, theorem: LDP 1, unclipped} in Theorem \ref{theorem: LDP 1, unclipped}.
The only difference is that we apply part $(b)$ of  Lemma~\ref{lemma: sequential compactness for limiting measures, LD of SGD}
(resp., Lemma~\ref{lemma: LDP, bar epsilon and delta, clipped version})
instead of part $(a)$ of  Lemma \ref{lemma: sequential compactness for limiting measures, LD of SGD}
(resp., Lemma \ref{lemma: LDP, bar epsilon and delta}).
To avoid repetition we omit the details.

In the remainder of this proof, we discuss how to extend the proof and cover the case where Assumption~\ref{assumption: boundedness of drift and diffusion coefficients} is dropped.
To prove the uniform $\mathbb M$-convergence claim,
we proceed again with a proof by contradiction.
Fix some $b, r > 0, k \in\mathbb N$, and
suppose that there are some $g \in \mathcal{C}\big(\mathbb{D}\setminus\mathbb{D}^{(k-1)}_A(r)\big)$,
some sequence $\eta_n > 0$ with limit $\lim_{n \rightarrow \infty}\eta_n = 0$, some sequence $\bm x_n \in A$, and $\epsilon > 0$ such that
\begin{align}
    \big| \mu^{(k)}_n(g) - \mathbf{C}^{(k)|b}(g;\bm x_n) \big| > \epsilon \ \forall n \geq 1
\qquad 
\text{ with }\mu^{(k)}_n(\cdot) 
\delequal { \P\big( \bm{X}^{\eta_n|b}(\bm x_n) \in\ \cdot\ \big)  }\big/{ \lambda^k(\eta_n)  }.
\label{proof, target contradiction claim, corollary: LDP 2}
\end{align}
By the compactness of $A$, we can pick a sub-sequence if needed and w.l.o.g.\ assume that $\lim_{n \to \infty}\bm x_n = \bm x^*$ for some $\bm x^* \in A$.
Next,
let $B \delequal \text{supp}(g)$ and note that $B$ is bounded away from $\mathbb{D}^{(k-1)|b}_{A}(r)$.
Applying Corollary \ref{corollary: existence of M 0 bar delta bar epsilon, clipped case, LDP},
we can fix some $M_0$ such that the following claim holds for any $M \geq M_0:$
 for any $\xi = \bar h^{(k)|b}_{M \downarrow}(\bm x,\textbf W,\textbf V,\bm{t})$ with $\bm{t} = (t_1,\cdots,t_{k}) \in (0,1]^{k\uparrow}$, $\textbf W = (\bm w_1,\cdots,\bm w_{k}) \in \mathbb{R}^{d \times k}$,
  $\textbf V = (\bm v_1,\cdots,\bm v_k) \in \R^{m \times k}$ with $\max_{j \in [d]}\norm{\bm v_j} \leq r$,
 and $\bm x \in A$,
\begin{align}
     \xi = \bar h^{(k)|b} (\bm x,\textbf W,\textbf V,\bm{t}) =  \bar h^{(k)|b}_{M \downarrow}(\bm x,\textbf W,\textbf V,\bm{t})
     \qquad\text{and}\qquad
     \sup_{t \in [0,1]}\norm{\xi_t} \leq M_0.
     \label{property: choice of M 0, new, proposition: standard M convergence, LDP clipped}
\end{align}
Here, recall that the mappings $\bar h^{(k)|b}_{M\downarrow}$ and $h^{(k)|b}_{M\downarrow}$ are defined in \eqref{def: perturb ode mapping h k b, truncated at M, 1}--\eqref{def: perturb ode mapping h k b, truncated at M, 4}.
Now, we fix some $M \geq M_0 + 1$ and recall the definitions of $\bm a_M,\ \bm \sigma_M$ in \eqref{def: a sigma truncated at M, LDP}.
Define the stochastic processes $\widetilde{\bm X}^{\eta|b}(\bm x) \delequal \big\{\widetilde{\bm X}^{\eta|b}_{\floor{t/\eta}}(\bm x):\ t\in [0,1]\big\}$
by
\begin{align}
 {\widetilde{\bm X}^{\eta|b}_j(\bm x)} = \widetilde{\bm X}^{\eta|b}_{j - 1}(\bm x) +  \varphi_b\Big(\eta \bm a_M\big(\widetilde{\bm X}^{\eta|b}_{j - 1}(\bm x)\big) + \eta \bm \sigma_M\big(\widetilde{\bm X}^{\eta|b}_{j - 1}(\bm x)\big)\bm Z_j\Big)\ \ \forall j \geq 1
 \label{proof, def tilde X eta b j x, corollary: LDP 2}
\end{align}
under initial condition $\widetilde{\bm X}^{\eta|b}_0(\bm x) = \bm x$.
In particular, by comparing the definition of $\widetilde{\bm X}^{\eta|b}_j(\bm x)$
with that of ${\bm X}^{\eta|b}_j(\bm x)$ in \eqref{def: X eta b j x, clipped SGD},
one can see that (for any $\bm x \in \R^m, \eta > 0$)
\begin{align}
    \sup_{t \in [0,1]}\norm{\widetilde{\bm X}^{\eta|b}_{\floor{t/\eta}}(\bm x)} > M & \ \ \Longleftrightarrow\ \ 
     \sup_{t \in [0,1]}\norm{{\bm X}^{\eta|b}_{\floor{t/\eta}}(\bm x)} > M,
     \label{property: Y and tilde Y, 1, proposition: standard M convergence, LDP clipped}
     \\
     \sup_{t \in [0,1]}\norm{{\bm X}^{\eta|b}_{\floor{t/\eta}}(\bm x)} \leq M 
     &\ \  \Longrightarrow\ \ 
     {\bm X}^{\eta|b}(\bm x) = \widetilde{\bm X}^{\eta|b}(\bm x).
     \label{property: Y and tilde Y, 2, proposition: standard M convergence, LDP clipped}
\end{align}

Now, we observe a few facts.
First, define measure 
\begin{align*}
    \widetilde{\mathbf{C}}^{(k)|b}(\ \cdot \ ;\bm x) \delequal &
   \int \mathbbm{I}\Big\{ h^{(k)|b}_{M\downarrow}\big( \bm x,\textbf W,\bm{t}  \big) \in\ \cdot\  \Big\} 
   \big((\nu_\alpha \times \mathbf S)\circ \Phi\big)^k(d \textbf W) \times\mathcal{L}^{k\uparrow}_{ {T} }(d\bm t).
\end{align*}
Due to \eqref{property: choice of M 0, new, proposition: standard M convergence, LDP clipped},
we must have
\begin{align}
    \widetilde{\mathbf{C}}^{(k)|b}(\ \cdot \ ;\bm x) = {\mathbf{C}}^{(k)|b}(\ \cdot \ ;\bm x)
    \qquad \forall \bm x \in A.
    \label{proof, tilde C k b = C k b, corollary: LDP 2}
\end{align}
Next, recall that 
 $\bm a_M$ and $\bm \sigma_M$ satisfy Assumption~\ref{assumption: boundedness of drift and diffusion coefficients}.
Then as has been established at the beginning of the proof, we have the following uniform $\M$-convergence for $\widetilde{\bm X}^{\eta|b}(x)$:
\begin{align}
     \lambda^{-k}(\eta) \P\big( \widetilde{\bm{X}}^{\eta|b}(\bm x) \in\ \cdot\ \big) 
    \rightarrow 
    \widetilde{\mathbf{C}}^{(k)|b}(\ \cdot\ ; \bm x) = {\mathbf{C}}^{(k)|b}(\ \cdot\ ; \bm x)
\quad
\text{in
$
\mathbb{M}\Big( \mathbb{D}\setminus \mathbb{D}^{(k-1)|b}_{A}(r) \Big)
$
uniformly in {$\bm x$ on $A$} 
}
\label{proof, uniform M convergence for M truncated path, corollary: LDP 2}
\end{align}
as $\eta \downarrow 0$.
By Definition~\ref{def: uniform M convergence}, for the function $g \in \mathcal{C}\big(\mathbb{D}\setminus\mathbb{D}^{(k-1)}_A(r)\big)$ fixed above,
we now have
\begin{align}
    \lim_{n \to \infty}
    \big| \tilde \mu^{(k)}_n(g) - \mathbf{C}^{(k)|b}(g;\bm x_n) \big| = 0 
    \qquad
\text{ with }\tilde\mu^{(k)}_n(\cdot) 
\delequal { \P\big( \widetilde{\bm{X}}^{\eta_n|b}(\bm x_n) \in\ \cdot\ \big)  }\big/{ \lambda^k(\eta_n)  }.
\label{proof, intermediate result for contradiction, proposition: standard M convergence, LDP clipped}
\end{align}
On the other hand,
for any $n \geq 1$ (recall that $B = \text{supp}(g)$)
\begin{equation}\label{proof, decomp of E g, proposition: standard M convergence, LDP clipped}
    \begin{split}
         \E\Big[ g\big( \bm{X}^{\eta_n|b}(\bm x_n) \big) \Big]
    & =  \E\Bigg[ g\big( \bm{X}^{\eta_n|b}(\bm x_n) \big) \mathbbm{I}\Big\{  \bm{X}^{\eta_n|b}(\bm x_n) \in B;\  \sup_{t \in [0,1]}\norm{{\bm X}^{\eta_n|b}_{\floor{t/\eta}}(\bm x_n)} \leq M  \Big\}  \Bigg]
    \\
    &\qquad + 
    \E\Bigg[g\big( \bm{X}^{\eta_n|b}(\bm x_n) \big) \mathbbm{I}\Big\{  \bm{X}^{\eta_n|b}(\bm x_n) \in B;\  \sup_{t \in [0,1]}\norm{{\bm X}^{\eta_n|b}_{\floor{t/\eta}}(\bm x_n)} > M  \Big\}  \Bigg].
    \end{split}
\end{equation}
The following bound then follows immediately from  \eqref{property: Y and tilde Y, 1, proposition: standard M convergence, LDP clipped} and \eqref{property: Y and tilde Y, 2, proposition:  standard M convergence, LDP clipped}:
\begin{align}
    \bigg| \E\Big[ g\big( \bm{X}^{\eta_n|b}(\bm x_n) \big) \Big] - 
            \E\Big[ g\big( \widetilde{\bm{X}}^{\eta_n|b}(\bm x_n) \big) \Big]
    \bigg|
    & \leq 
    \norm{g}\P\Bigg( \sup_{t \in [0,1]}\norm{\widetilde{\bm X}^{\eta_n|b}_{\floor{t/\eta}}(\bm x_n)} > M  \Bigg).
    \label{proof, consequence, decomp of E g, proposition: standard M convergence, LDP clipped}
\end{align}
Furthermore, we claim that 
\begin{align}
    &\lim_{n\rightarrow \infty}{\lambda^{-k}(\eta_n)}{ \P\Bigg( \sup_{t \in [0,1]}\norm{\widetilde{\bm X}^{\eta_n|b}_{\floor{t/\eta}}(\bm x_n)} > M  \Bigg) }
     = 0.
      \label{goal new 2, proposition: standard M convergence, LDP clipped}
\end{align}
Then observe that
\begin{align*}
    & \limsup_{n \to \infty}
        \Big| \mu^{(k)}_n(g) - \mathbf{C}^{(k)|b}(g;\bm x_n) \Big|
    \\
    & \leq 
    \limsup_{n \to \infty}\Big| \mu^{(k)}_n(g) - \tilde \mu^{(k)}_n(g)\Big|
    +
    \limsup_{n \to \infty}\Big| \tilde \mu^{(k)}_n(g) -  \mathbf{C}^{(k)|b}(g;\bm x_n) \Big|
    \\ 
    & \leq \limsup_{n \to \infty}
        {\lambda^{-k}(\eta_n)}{ \P\Bigg( \sup_{t \in [0,1]}\norm{\widetilde{\bm X}^{\eta_n|b}_{\floor{t/\eta}}(\bm x_n)} > M  \Bigg) }
    + 
    0
    \qquad \text{ due to \eqref{proof, consequence, decomp of E g, proposition: standard M convergence, LDP clipped} and \eqref{proof, intermediate result for contradiction, proposition: standard M convergence, LDP clipped} }
    \\ 
    & = 0
    \qquad\text{due to \eqref{goal new 2, proposition: standard M convergence, LDP clipped}}.
\end{align*}
In summary, we end up with a clear contradiction to \eqref{proof, target contradiction claim, corollary: LDP 2}, thus allowing us to conclude the proof.
Now, it only remains to prove claim~\eqref{goal new 2, proposition: standard M convergence, LDP clipped}.

\medskip
\noindent\textbf{Proof of Claim }\eqref{goal new 2, proposition: standard M convergence, LDP clipped}:

Let ${E} \delequal \{ \xi \in \mathbb{D}:\ \sup_{t \in [0,1]}\norm{\xi_t} > M \}$.
Suppose we can show that $E$ is bounded away from $\mathbb{D}^{(k)|b}_{A}(r)$,
then by applying the uniform $\M$-convergence established above in \eqref{proof, uniform M convergence for M truncated path, corollary: LDP 2} for $\widetilde{\bm X}^{\eta|b}(\bm x)$,
we get
$
 \limsup_{n \rightarrow \infty}
    { \P\Big( \widetilde{\bm X}^{\eta_n|b}(\bm x_n) \in E \Big) }\Big/{\lambda^{k+1}(\eta_n)}
    < \infty,
$
which then implies \eqref{goal new 2, proposition: standard M convergence, LDP clipped}.
To see why $E$ is bounded away from $\mathbb{D}^{(k)|b}_{A}(r)$,
note that by \eqref{property: choice of M 0, new, proposition: standard M convergence, LDP clipped},
$$
\xi \in \mathbb{D}^{(k)|b}_{A}(r) \Longrightarrow \sup_{t \in [0,1]}\norm{\xi_t} \leq M_0 \leq M - 1
$$
due to our choice of $M \geq M_0 + 1$ at the beginning.
Therefore, we yield $\bm{d}_{J_1}\big(\mathbb{D}^{(k)|b}_{A}(r),E\big) \geq 1$ and conclude the proof.
\end{proof}

\subsubsection{Proof of Proposition~\ref{proposition: standard M convergence, LDP clipped}}
\label{subsubsec: proof, proposition: standard M convergence, LDP clipped}

As has been demonstrated earlier, Proposition \ref{proposition: standard M convergence, LDP clipped} lays the foundation for the sample path large deviations of heavy-tailed stochastic difference equations.
In Section~\ref{subsubsec: proof, proposition: standard M convergence, LDP clipped}, we provide
the proof of Proposition~\ref{proposition: standard M convergence, LDP clipped}.
Analogous to the proof of Theorem~\ref{corollary: LDP 2} above, we show that it suffices to prove the seemingly more restrictive results stated below in Proposition~\ref{proposition: standard M convergence, LDP clipped, stronger boundedness assumption}, where we impose the 
the boundedness condition in Assumption~\ref{assumption: boundedness of drift and diffusion coefficients}.

\begin{proposition}
\label{proposition: standard M convergence, LDP clipped, stronger boundedness assumption}
\linksinthm{proposition: standard M convergence, LDP clipped, stronger boundedness assumption}
Let $\eta_n$ be a sequence of strictly positive real numbers with $\lim_{n \rightarrow \infty}\eta_n = 0$.
Let compact set $A \subseteq \R^m$ and $\bm x_n,\bm x^* \in A$ be such that $\lim_{n \rightarrow \infty}\bm x_n = \bm x^*$.
Under Assumptions~\ref{assumption gradient noise heavy-tailed}, \ref{assumption: lipschitz continuity of drift and diffusion coefficients}, and \ref{assumption: boundedness of drift and diffusion coefficients},
it holds for all $k \in \mathbb N$ and $b,r > 0$ that
\begin{align*}
    { \P\big( \bm{X}^{\eta_n|b}(\bm x_n) \in\ \cdot\ \big)  }\big/{ \lambda^k(\eta_n) } \rightarrow \mathbf{C}^{(k)|b}   \big(\ \cdot\ ; \bm x^*\big)\text{ in }\mathbb{M}\big(\mathbb{D}\setminus \mathbb{D}^{(k-1)|b}_{A}(r) \big)\text{ as }n \rightarrow \infty.
\end{align*}
\end{proposition}

\begin{proof}[Proof of Proposition \ref{proposition: standard M convergence, LDP clipped}]
\linksinpf{proposition: standard M convergence, LDP clipped}
The proof is almost identical to the second half of the proof for Theorem~\ref{corollary: LDP 2}.
Specifically, we 
fix some $M \geq M_0 + 1$ with $M_0$ specified in \eqref{property: choice of M 0, new, proposition: standard M convergence, LDP clipped},
and we arbitrarily pick some $g \in \mathcal{C}\big(\mathbb{D}\setminus\mathbb{D}^{(k-1)}_A(r)\big)$.
Besides, define
the stochastic processes $\widetilde{\bm X}^{\eta|b}(\bm x) \delequal \big\{\widetilde{\bm X}^{\eta|b}_{\floor{t/\eta}}(\bm x):\ t\in [0,1]\big\}$
by
\eqref{proof, def tilde X eta b j x, corollary: LDP 2}.
By repeating the arguments in the proof for Theorem~\ref{corollary: LDP 2},
we yield \eqref{proof, tilde C k b = C k b, corollary: LDP 2} and \eqref{proof, consequence, decomp of E g, proposition: standard M convergence, LDP clipped} again.
Next, by applying Proposition~\ref{proposition: standard M convergence, LDP clipped, stronger boundedness assumption} onto $\widetilde{\bm X}^{\eta|b}(\bm x)$,
we again obtain \eqref{proof, intermediate result for contradiction, proposition: standard M convergence, LDP clipped} and \eqref{goal new 2, proposition: standard M convergence, LDP clipped} (in particular, for the claim \eqref{goal new 2, proposition: standard M convergence, LDP clipped}, note that at the end of the proof for Theorem~\ref{corollary: LDP 2} we have already shown that $\{ \xi \in \mathbb{D}:\ \sup_{t \in [0,1]}\norm{\xi_t} > M \}$ is bounded away from $\D^{(k)|b}_A(r)$).
Now, for 
$
\mu^{(k)}_n(\cdot) 
\delequal { \P\big( \bm{X}^{\eta_n|b}(\bm x_n) \in\ \cdot\ \big)  }\big/{ \lambda^k(\eta_n)  },
$
observe that
\begin{align*}
    & \lim_{n \to \infty}
        \Big| \mu^{(k)}_n(g) - \mathbf{C}^{(k)|b}(g;\bm x_n) \Big|
    \\ 
    & \leq 
    \limsup_{n \to \infty}\Big| \mu^{(k)}_n(g) - \tilde \mu^{(k)}_n(g)\Big|
    +
    \limsup_{n \to \infty}\Big| \tilde \mu^{(k)}_n(g) -  \mathbf{C}^{(k)|b}(g;\bm x_n) \Big|
    \\
    & \leq \limsup_{n \to \infty}
        {\lambda^{-k}(\eta_n)}{ \P\Bigg( \sup_{t \in [0,1]}\norm{\widetilde{\bm X}^{\eta_n|b}_{\floor{t/\eta}}(\bm x_n)} > M  \Bigg) }
    + 
    0
    \qquad \text{ due to \eqref{proof, consequence, decomp of E g, proposition: standard M convergence, LDP clipped} and \eqref{proof, intermediate result for contradiction, proposition: standard M convergence, LDP clipped} }
    \\
    & = 0
    \qquad\text{due to \eqref{goal new 2, proposition: standard M convergence, LDP clipped}}.
\end{align*}
By the Portmanteau theorem for $\M$-convergence (see theorem 2.1 of \cite{lindskog2014regularly}) and the arbitrariness of the function $g \in \mathcal{C}\big(\mathbb{D}\setminus\mathbb{D}^{(k-1)}_A(r)\big)$, we conclude the proof.
\end{proof}

The rest of Section~\ref{subsubsec: proof, proposition: standard M convergence, LDP clipped} is devoted to establishing Proposition~\ref{proposition: standard M convergence, LDP clipped, stronger boundedness assumption}.
In light of Lemma \ref{lemma: asymptotic equivalence when bounded away, equivalence of M convergence},
one approach to Proposition \ref{proposition: standard M convergence, LDP clipped, stronger boundedness assumption}
is to construct some stochastic process that is asymptotically equivalent to $\bm{X}^{\eta|b}$ and, as $\eta \downarrow 0$,
converges to the limiting measure ${\mathbf{C}}^{(k)|b}$ stated in  Proposition~\ref{proposition: standard M convergence, LDP clipped, stronger boundedness assumption}.
Specifically,
recall the definitions of $\tau^{>\delta}_i(\eta)$ and $\bm W^{>\delta}_i(\eta)$ in \eqref{defArrivalTime large jump}--\eqref{defSize large jump}.
Given $\eta,b,\delta>0$ and $\bm x \in \R^m$,
we define
$\notationdef{notation-hat-X-clip-b-top-j-jumps}{\hat{\bm{X}}^{\eta|b; >\delta }(\bm x)} \delequal \big\{
\hat{\bm{X}}^{\eta|b; >\delta }_t(\bm x):\ t \in [0,1]
\big\}$
as the solution to
\begin{align}
    \frac{d\hat{\bm{X}}^{\eta|b; >\delta}_t(\bm x)}{dt}
    &= 
    \bm a\big( \hat{\bm{X}}^{\eta|b; >\delta }_t(\bm x) \big)
    \ \ \ \ \forall t \in [0,1],\ t \notin \big\{\eta \tau^{>\delta}_i( \eta ):\ i \geq 1 \big\},
     \label{def: hat X truncated b, j top jumps, 1}
    \\
    \hat{\bm{X}}^{\eta|b; >\delta }_t (\bm x)
    &= 
    \hat{\bm{X}}^{\eta|b; >\delta }_{t-}(\bm x)
    +
    \varphi_b\Big(\eta \bm \sigma\big( \hat{\bm{X}}^{\eta|b; >\delta }_{t-}(\bm x) \big)\bm W^{>\delta}_i(\eta)\Big)
    \ \ \ \ \text{if } t = \eta \tau^{>\delta}_i( \eta )\text{ for some }i \geq 1
    \label{def: hat X truncated b, j top jumps, 2}
\end{align}
with initial condition $\hat{\bm{X}}^{\eta|b; >\delta }_0(\bm x) = \bm x$.
By definition of the mapping $h^{(k)|b}$ in \eqref{def: perturb ode mapping h k b, 1}--\eqref{def: perturb ode mapping h k b, 4}, we have
the following property:
for any $\eta,b,\delta > 0$, $j \geq 0$, and $\bm x\in \R^m$,
\begin{align}
\text{on event }
\Big\{\tau^{>\delta}_j(\eta) \leq \floor{1/\eta} < \tau^{>\delta}_{j+1}(\eta)\Big\},
\text{ we have }
    \hat{\bm{X}}^{\eta|b; >\delta }(\bm x)
    =
     h^{(j)|b}\big( \bm x, \eta \textbf W^{>\delta}(\eta), \eta \bm \tau^{>\delta}(\eta)\big)
      \label{property: equivalence between breve X and hat X, clipped at b}
\end{align}
with $\textbf W^{>\delta}(\eta) =  (\bm W^{>\delta}_{1}(\eta), \cdots,\bm W^{>\delta}_j(\eta))$
and 
$\bm \tau^{>\delta}(\eta) =  (\tau^{>\delta}_{1}(\eta), \cdots,\tau^{>\delta}_j(\eta))$.

\ 

We first state two results that allow us to apply Lemma~\ref{lemma: asymptotic equivalence when bounded away, equivalence of M convergence}.
\begin{proposition}\label{proposition: asymptotic equivalence, clipped}
\linksinthm{proposition: asymptotic equivalence, clipped}
Let $\eta_n$ be a sequence of strictly positive real numbers with $\lim_{n \rightarrow \infty}\eta_n = 0$.
Let compact set $A \subseteq \R^m$ and $\bm x_n,\bm x^* \in A$ be such that $\lim_{n \rightarrow \infty}\bm x_n = \bm x^*$.
Under Assumptions  \ref{assumption gradient noise heavy-tailed}, \ref{assumption: lipschitz continuity of drift and diffusion coefficients}, \ref{assumption: boundedness of drift and diffusion coefficients}, 
it holds for all $k \in \mathbb N$ and $b,r > 0$ that
$\bm{X}^{\eta_n|b}(\bm x_n)$
is asymptotically equivalent to $
\hat{\bm{X}}^{\eta_n|b;>\delta}(\bm x_n)
$
in $\mathbb M\big(\D \setminus \D^{(k)|b}_A(r)\big)$
w.r.t.\ ${\lambda^{k}(\eta_n)}$ as $\delta \downarrow 0$.
\end{proposition}

\begin{proposition}\label{proposition: uniform weak convergence, clipped}
\linksinthm{proposition: uniform weak convergence, clipped}
Let $\eta_n$ be a sequence of strictly positive real numbers with $\lim_{n \rightarrow \infty}\eta_n = 0$,
 $A \subseteq \R^m$ be compact, and $\bm x_n,\bm x^* \in A$ be such that $\lim_{n \rightarrow \infty}\bm x_n = \bm x^*$.
Let Assumptions  \ref{assumption gradient noise heavy-tailed}, \ref{assumption: lipschitz continuity of drift and diffusion coefficients}, and \ref{assumption: boundedness of drift and diffusion coefficients} hold.
Let $k \geq 0$ and $b,r,\delta > 0$.
For any $g \in \mathcal{C}\big({ \mathbb{D}\setminus  \mathbb{D}^{(k-1)|b}_{A}(r) }\big)$,
\begin{align*}
    \lim_{n \to \infty}{ \E\Big[ g\big( \hat{\bm{X}}^{\eta_n|b; (k)}(\bm x_n)\big) \Big] }\Big/{\lambda^{k}(\eta_n)} = \mathbf{C}^{(k)|b} (g\ ;{\bm x^*})
    \qquad\forall
    \text{$\delta > 0$ small enough,}
\end{align*}
where $\mathbf{C}^{(k)|b}   $ is the measure defined in \eqref{def: measure mu k b t},
and $
{\mathcal{C}({ \mathbb{S}\setminus \mathbb{C} })}
$
is the set of all real-valued, non-negative, bounded and continuous functions with support bounded away from $\mathbb{C}$.
\end{proposition}

\begin{proof}[Proof of Proposition \ref{proposition: standard M convergence, LDP clipped, stronger boundedness assumption}]
\linksinpf{proposition: standard M convergence, LDP clipped}%
In the context of Lemma~\ref{lemma: asymptotic equivalence when bounded away, equivalence of M convergence}
and
under the choice of 
\begin{align*}
    (\mathbb S,\bm d) = (\D,\dj{}),
    \quad 
    \mathbb C = \D^{(k-1)|b}_A(r),
    \quad 
    X_n = \bm X^{\eta_n|b}(\bm x_n),
    \quad 
    Y^\delta_n = \hat{\bm X}^{\eta_n|b; > \delta}(\bm x_n),
\end{align*}
Proposition~\ref{proposition: asymptotic equivalence, clipped} verifies condition (i),
while Proposition~\ref{proposition: uniform weak convergence, clipped} (together with Urysohn's Lemma)
verifies condition (ii).
Applying Lemma~\ref{lemma: asymptotic equivalence when bounded away, equivalence of M convergence}, we conclude the proof.
\end{proof}

Now, it only remains to prove Propositions~\ref{proposition: asymptotic equivalence, clipped} and \ref{proposition: uniform weak convergence, clipped}.

\begin{proof}[Proof of Proposition \ref{proposition: asymptotic equivalence, clipped}]
\linksinpf{proposition: asymptotic equivalence, clipped}%
Fix 
some $b,r > 0, k \in \mathbb N$,
and some sequence of strictly positive real numbers $\eta_n$ with $\lim_{n \rightarrow \infty}\eta_n = 0$.
Also, fix a compact set $A \subseteq \R^m$ and $\bm x_n,\bm x^* \in A$ such that $\lim_{n \rightarrow\infty}\bm x_n = \bm x^*$.
Besides, we arbitrarily pick  some $\Delta > 0$ and some $B \in \mathscr{S}_{\mathbb{D}}$ that is bounded away from $\mathbb{D}^{(k - 1)|b}_{A}(r)$.
By Definition~\ref{def: uniform asymptotic equivalence}, our goal is to show that (for all $\delta > 0$ small enough)
\begin{align}
    \lim_{n \rightarrow \infty}
    \P\Big( \bm{d}_{J_1}\big( \bm{X}^{\eta_n|b}(\bm x_n), \hat{\bm X}^{ \eta_n|b;>\delta }(\bm x_n) \big)
    \mathbbm{I}\big\{\bm{X}^{\eta_n|b}(\bm x_n)\text{ or }\hat{\bm X}^{ \eta_n|b;>\delta }(\bm x_n) \in B  \big\} > \Delta \Big)\Big/\lambda^k(\eta_n) = 0.
    \label{goal: asymptotic equivalence claim, proposition: asymptotic equivalence, clipped}
\end{align}
By Lemma~\ref{lemma: LDP, bar epsilon and delta, clipped version},
there are some $\bar{\epsilon} \in (0,r)$ and $\bar{\delta} > 0$
such that 
\begin{itemize}
    \item for any $\bm x \in A$ and $b > 0$,
    and any $(\bm v_1,\cdots,\bm v_k) \in \R^{m \times k}$ with $\max_{j \in [k]}\norm{\bm v_j} \leq \bar\epsilon$,
        \begin{align}
           \bar h^{(k)|b}\big(\bm x,(\bm w_1,\cdots,\bm w_k),(\bm v_1,\cdots,\bm v_k),\bm{t}\big) \in B^{\bar{\epsilon}}
            \qquad
            \Longrightarrow
            \qquad
            \norm{\bm w_i} > \bar\delta\ \forall i \in [k];
            \label{choice of bar delta, proof, proposition: asymptotic equivalence, clipped}
        \end{align} 

    \item furthermore,
        \begin{align}
            \bm{d}_{J_1}\big(B^{\bar\epsilon},\mathbb{D}^{(k- 1)|b}_{A}(r)\big) > \bar\epsilon.
            \label{choice of bar epsilon, proof, proposition: asymptotic equivalence, clipped}
\end{align}
\end{itemize}
Henceforth in this proof, we only consider $\delta \in (0,\bar\delta)$.
Meanwhile,
given $\eta,\delta,\epsilon > 0$ and $\bm x \in A$,
let
\begin{align*}
    \notationdef{notation-set-B-0}{B_0} &\delequal 
    \Big\{ \bm{X}^{\eta|b}(\bm x)\in B\text{ or }\hat{\bm X}^{\eta|b;>\delta}(\bm x) \in B;\   
    \bm{d}_{J_1}\big(\bm{X}^{\eta|b}(\bm x),\hat{\bm X}^{\eta|b;>\delta}(\bm x)\big) > \Delta
    \Big\},
    \\
    \notationdef{notation-B1}
    {B_1} & \delequal 
    \Big\{ \tau^{>\delta}_{k+1}(\eta) > \floor{1/\eta}\Big\},
    \\
    \notationdef{notation-B2}
    {B_2} & \delequal 
    \Big\{ \tau^{>\delta}_{k}(\eta) \leq \floor{1/\eta} \Big\},
    \\
    \notationdef{notation-B3}
    {B_3} & \delequal 
    \Big\{\eta \norm{ \bm W^{>\delta}_{i}(\eta) } > \bar{\delta}\ \text{for all }i \in [k] \Big\},
    \\
    \notationdef{notation-B4}
    {B_4} & \delequal 
    \Big\{\eta \norm{ \bm W^{>\delta}_{i}(\eta) } \leq 1/\epsilon^{ \frac{1}{2k}  } \ \text{for all }i \in [k] \Big\}.
\end{align*}
We have the following decomposition of events:
\begin{align}
B_0 
&=
(B_0\cap B_1^\complement)
\cup
(B_0\cap B_1 \cap B_2^\complement)
\cup
(B_0\cap B_1 \cap B_2 \cap B_3^\complement)
\nonumber
\\
&\qquad
\cup
(B_0\cap B_1 \cap B_2 \cap B_3 \cap B_4^\complement)
\cup 
(B_0 \cap B_1\cap B_2 \cap B_3 \cap B_4).
\label{proof, decomp of event B 0, proposition: asymptotic equivalence, clipped}
\end{align}
To proceed, 
let
$\rho = \exp(D)$ and
$D \in[1,\infty)$ is the Lipschitz coefficient in Assumption \ref{assumption: lipschitz continuity of drift and diffusion coefficients}.
For any $\epsilon > 0$ small enough such that
\begin{align}
    (2\rho D)^{k+1} \sqrt{\epsilon} < \Delta,\qquad 2\rho\epsilon < \bar\epsilon,\qquad \epsilon \in (0,1),
    \label{proof, choice of epsilon, proposition: standard M convergence, LDP clipped, stronger boundedness assumption}
\end{align}
we claim that
\begin{align}
    \lim_{\eta \downarrow 0 }\sup_{\bm x \in A}
    \P\Big( B_0\cap B_1^\complement \Big)
    \Big/ \lambda^k(\eta) & = 0,
    \label{goal, event B 1, clipped, proposition: asymptotic equivalence, clipped}
    \\
    \lim_{\eta \downarrow 0}\sup_{\bm x \in A}
    \P\Big( B_0\cap B_1 \cap B_2^\complement \Big)
    \Big/ \lambda^k(\eta) & = 0,
    \label{goal, event B 2, clipped, proposition: asymptotic equivalence, clipped}
    \\
    \lim_{\eta \downarrow 0}\sup_{\bm x \in A}
    \P\Big( B_0\cap B_1 \cap B_2 \cap B_3^\complement \Big)
    \Big/ \lambda^k(\eta) & = 0,
    \label{goal, event B 3, clipped, proposition: asymptotic equivalence, clipped}
    \\
    \limsup_{\eta \downarrow 0}\sup_{\bm x \in A}
    \P\Big( B_0\cap B_1 \cap B_2 \cap B_3 \cap B_4^\complement \Big)
    \Big/ \lambda^k(\eta) & \leq \bar\delta^{-k\alpha} \cdot \epsilon^{\frac{\alpha}{2k}},
    \label{goal, event B 4, clipped, proposition: asymptotic equivalence, clipped}
    \\
    \lim_{\eta \downarrow 0}\sup_{\bm x \in A}
    \P\Big( B_0\cap B_1 \cap B_2 \cap B_3 \cap B_4\Big)
    \Big/ \lambda^k(\eta) & = 0,
    \label{goal, event B 5, clipped, proposition: asymptotic equivalence, clipped}
\end{align}
if we pick $\delta > 0$ sufficiently small.
Under such $\delta$,
by the decomposition of event $B_0$ in \eqref{proof, decomp of event B 0, proposition: asymptotic equivalence, clipped}, we yield 
\begin{align*}
    \limsup_{\eta \downarrow 0 }\sup_{\bm x \in A}
    \P\big( B_0 \big)
    \big/ \lambda^k(\eta)
    \leq 
    \bar\delta^{-k\alpha} \cdot \epsilon^{\frac{\alpha}{2k}}
\end{align*}
for all $\epsilon > 0$ small enough.
Note that $\bar\delta > 0$ is the constant fixed in \eqref{choice of bar delta, proof, proposition: asymptotic equivalence, clipped}.
Driving $\epsilon \downarrow 0$, we conclude the proof of 
\eqref{goal: asymptotic equivalence claim, proposition: asymptotic equivalence, clipped}.
The remainder of this proof is devoted to claims \eqref{goal, event B 1, clipped, proposition: asymptotic equivalence, clipped}--\eqref{goal, event B 5, clipped, proposition: asymptotic equivalence, clipped}.

\medskip
\noindent\textbf{Proof of }\eqref{goal, event B 1, clipped, proposition: asymptotic equivalence, clipped}:

For any $\delta > 0$, \eqref{property: large jump time probability} implies that
$\sup_{x\in A}\P(B_0\cap B_1^\complement) \leq \P(B_1^\complement) \leq \big( \eta^{-1}H(\delta\eta^{-1}) \big)^{k+1} 
= \bm{O}\big( \lambda^{k+1}(\eta) \big)
= \bm{o}\big(\lambda^k(\eta)\big)$.

\medskip
\noindent\textbf{Proof of }\eqref{goal, event B 2, clipped, proposition: asymptotic equivalence, clipped}:

It suffices to show that (for all $\delta > 0$ small enough)
\begin{align*}
    \lim_{\eta\downarrow 0}\sup_{\bm x \in A}
    \P\Big(
    \underbrace{B_0 \cap \big\{ \tau^{>\delta}_k(\eta) > \floor{1/\eta}\big\}}_{\delequal \widetilde{B}}
    \Big)\Big/\lambda^k(\eta) = 0
\end{align*}
In particular,
we only consider $\delta \in (0,\bar{\delta} \wedge \frac{b}{2})$ with $\bar{\delta}$ characterized in \eqref{choice of bar delta, proof, proposition: asymptotic equivalence, clipped}
and $C \geq 1$ being the constant in Assumption~\ref{assumption: boundedness of drift and diffusion coefficients}.
By property \eqref{property: equivalence between breve X and hat X, clipped at b},
it holds 
on event $\{\tau^{>\delta}_k(\eta) > \floor{1/\eta}\}$ 
that $\hat{\bm X}^{\eta_n|b;>\delta}(\bm x) \in \D^{(k-1)}_A(0) \subseteq \D^{(k-1)}_A(r)$.
In light of \eqref{choice of bar epsilon, proof, proposition: asymptotic equivalence, clipped},
we must have $\hat{\bm{X}}^{\eta|b; >\delta }(\bm x) \notin B^{\bar{\epsilon}}$ on event $\{\tau^{>\delta}_k(\eta) > \floor{1/\eta}\}$,
and hence
\begin{align*}
    \widetilde{B} \subseteq
    \{ \bm{X}^{\eta|b}(\bm x)\in B \}
    \cap
    \{ \tau^{>\delta}_{k}(\eta) > \floor{1/\eta} \}
    \qquad
    \forall \bm x \in A.
\end{align*}
Furthermore, let event $A_i(\eta,b,\epsilon,\delta,\bm x)$ be defined as in \eqref{def: event A i concentration of small jumps, 1}.
We claim that
\begin{align}
    \{ \bm{X}^{\eta|b}(x)\in B \}
    \cap
    \{ \tau^{>\delta}_{k}(\eta) > \floor{1/\eta} \} \cap \big(\cap_{i = 1}^{k} A_i(\eta,b,\epsilon,\delta,\bm x)\big) = \emptyset
    \label{proof, subgoal of claim 2, proposition: asymptotic equivalence, clipped}
\end{align}
holds for all  $\eta > 0$ small enough with $\eta < \min\{\frac{b \wedge 1}{2C},\frac{\epsilon}{C}\}$, all $\delta \in (0,\frac{b}{2C})$, and all $\bm x \in A$.
Then
\begin{align*}
    \lim_{\eta \downarrow 0}\sup_{\bm x \in A}
    \P\big( \widetilde{B}\big)
    \Big/ \lambda^k(\eta)
    & \leq 
    \lim_{\eta \downarrow 0}\sup_{\bm x \in A}
    \P\Bigg( \Big( \bigcap_{i = 1}^{k}A_i(\eta,b,\epsilon,\delta,\bm x) \Big)^c \Bigg)
    \Bigg/ \lambda^k(\eta).
\end{align*}
To conclude the proof, one only need to apply Lemma~\ref{lemma LDP, small jump perturbation} $(b)$
with some $N > k(\alpha - 1)$ (recall that $\lambda(\eta) \in \RV_{ k(\alpha - 1) }(\eta)$ as $\eta \downarrow 0$).

Now, we prove claim~\eqref{proof, subgoal of claim 2, proposition: asymptotic equivalence, clipped}
for any $\eta \in \big( 0, \min\{\frac{b \wedge 1}{2C},\frac{\epsilon}{C}\}\big)$,
$\delta \in (0,\frac{b}{2C})$, and $\bm x \in A$.
Define the stochastic process 
$\bm{\breve X}^{\eta|b;\delta}(\bm x) \delequal \big\{ \notationdef{notation-breve-X-eta-b-delta-t}{\breve{\bm X}^{\eta|b;\delta}_t(\bm x)}:\ t \in [0,1]\big\}$
as the solution to
\begin{align}
    \frac{d \breve{\bm X}^{\eta|b;\delta}_t(\bm x)}{dt} & = \bm a\big( \breve{\bm X}^{\eta|b;\delta}_t(\bm x) \big)\qquad \forall t \in [0,\infty)\setminus \{ \eta\tau^{>\delta}_j(\eta):\ j \geq 1 \},
    \label{def: x eta b M circ approximation, LDP, 2}
    \\
     \breve{\bm X}^{\eta|b;\delta}_{\eta\tau^{>\delta}_j(\eta)}(\bm x) & = \bm X^{\eta|b}_{\tau^{>\delta}_j(\eta)}(\bm x)\qquad \forall j \geq 1,
     \label{def: x eta b M circ approximation, LDP, 3}
\end{align}
under the initial condition $\breve{\bm X}^{\eta|b;\delta}_0(\bm x)  = \bm x$.
For any $j \geq 1$, observe that on event 
$\big(\cap_{i = 1}^{j} A_i(\eta,b,\epsilon,\delta,\bm x)\big) \cap \{ \tau^{>\delta}_{j}(\eta) > \floor{1/\eta} \}$,
\begin{align}
    & \bm{d}_{J_1}\Big(\breve{\bm X}^{\eta|b;\delta}(\bm x), \bm{X}^{\eta|b}(\bm x)\Big)
    \nonumber 
    \\ 
    & \leq \sup_{t \in [0,1]}\norm{ \breve{\bm X}^{\eta|b;\delta}_t(\bm x) - \bm X^{\eta|b}_{ \floor{t/\eta} }(\bm x) }
    \nonumber
    \\
    & \leq \sup_{ t \in \big[0,\eta\tau^{>\delta}_1(\eta)\big) \cup 
    \big[\eta\tau^{>\delta}_1(\eta),\eta\tau^{>\delta}_{2}(\eta)\big) \cup \cdots 
    \cup 
    \big[\eta\tau^{>\delta}_{j-1}(\eta),\eta\tau^{>\delta}_{j}(\eta)\big)
    }
   \norm{ \breve{\bm X}^{\eta|b;\delta}_t(\bm x) - \bm X^{\eta|b}_{ \floor{t/\eta} }(\bm x) }
   \nonumber
   \\
    & \leq \rho\cdot \big( \epsilon + \eta C \big) 
    \leq 2\rho\epsilon < \bar{\epsilon}
    \qquad 
\text{by \eqref{ineq, no jump time, b, lemma: SGD close to approximation x circ, LDP} of Lemma \ref{lemma: SGD close to approximation x circ, LDP}}.
    \label{proof, ineq for mathring x, proposition: asymptotic equivalence, clipped}
\end{align}
\elaborate{
For any $\eta > 0$ small enough with $\eta < \min\{\frac{b \wedge 1}{2C},\frac{\epsilon}{C}\}$,
on event $\big(\cap_{i = 1}^{k+1} A_i(\eta,b,\epsilon,\delta,x)\big) \cap \{ \tau^{>\delta}_{k + 1}(\eta) > \floor{1/\eta} \}$
that
\begin{align}
    \bm{d}_{J_1}(\breve{\bm X}^{\eta|b;\delta}(\bm x), \bm{X}^{\eta|b}(\bm x))
    & \leq 
    \sup_{ t \in [0,1] }\Big| \breve{X}^{\eta|b;\delta}_t(x) - X^{\eta|b}_{ \floor{t/\eta} }(x) \Big|
    \nonumber
    \\
    & = \sup_{ t \in \big[0,\eta\tau^{>\delta}_1(\eta)\big) \cup 
    \big[\eta\tau^{>\delta}_1(\eta),\eta\tau^{>\delta}_{2}(\eta)\big) \cup \cdots 
    \cup 
    \big[\eta\tau^{>\delta}_{k}(\eta),\eta\tau^{>\delta}_{k+1}(\eta)\big)
    }
   \Big| \breve{X}^{\eta|b;\delta}_t(x) - X^{\eta|b}_{ \floor{t/\eta} }(x) \Big|
   \nonumber
   \\
   &\ \ \ \ \text{due to }\tau^{>\delta}_{k+1}(\eta) > \floor{1/\eta}\text{ and the definition of }\breve{X}^{\eta|b;\delta}_t(x)
    \nonumber
    \\
    & \leq \rho\cdot \big( \epsilon + \eta C \big) 
    \ \ \ \text{due to \eqref{ineq, no jump time, b, lemma: SGD close to approximation x circ, LDP} of Lemma~\ref{lemma: SGD close to approximation x circ, LDP}}
    \nonumber
    \\
    & \leq 2\rho\epsilon < \bar{\epsilon}
    \nonumber
\end{align}
}
In the last line of the display above, note that $(i)$ our choices of $\eta < \frac{b \wedge 1}{2C}$ and $\delta < \frac{b}{2C}$ allow us to apply part $(b)$ of Lemma~\ref{lemma: SGD close to approximation x circ, LDP}, and $(ii)$ the inequalities then follow from the choice of $\eta < \frac{\epsilon}{C}$ above and the choice of $2\rho\epsilon < \bar\epsilon$  in \eqref{proof, choice of epsilon, proposition: standard M convergence, LDP clipped, stronger boundedness assumption}.
Moreover, recall that we have fixed $\bar\epsilon < r$ at the beginning of the proof,
and note that \eqref{proof, ineq for mathring x, proposition: asymptotic equivalence, clipped} confirms (under the choice of $j = k$)
that on event 
\begin{align*}
    \bm{\breve X}^{\eta|b;\delta}(\bm x) \in \D^{(k-1)|b}_A(\bar\epsilon) \subseteq \D^{(k-1)|b}_A(r)
    \qquad
    \text{and}
    \qquad
    \bm{d}_{J_1}\Big(\breve{\bm X}^{\eta|b;\delta}(\bm x), \bm{X}^{\eta|b}(\bm x)\Big) < \bar\epsilon.
\end{align*}
In light of \eqref{choice of bar epsilon, proof, proposition: asymptotic equivalence, clipped}, this implies that on event $\big(\cap_{i = 1}^{k} A_i(\eta,b,\epsilon,\delta,\bm x)\big) \cap \{ \tau^{>\delta}_{k}(\eta) > \floor{1/\eta} \}$,
we must have ${\bm X}^{\eta|b}(\bm x) \notin B^{\bar\epsilon}$,
thus concluding the proof of claim~\eqref{proof, subgoal of claim 2, proposition: asymptotic equivalence, clipped}.

\medskip
\noindent\textbf{Proof of }\eqref{goal, event B 3, clipped, proposition: asymptotic equivalence, clipped}:

On event $B_1 \cap B_2 = \{\tau^{>\delta}_{k}(\eta) \leq \floor{1/\eta} < \tau^{>\delta}_{k+1}(\eta)\}$,
recall that \eqref{property: equivalence between breve X and hat X, clipped at b} holds.
Furthermore, on $B_3^\complement$,
there is some $i \in [k]$ such that $\eta \norm{\bm W^{>\delta}_{i}(\eta)} \leq \bar{\delta}$.
Combining \eqref{property: equivalence between breve X and hat X, clipped at b} with
the choice of $\bar{\delta}$ in  \eqref{choice of bar delta, proof, proposition: asymptotic equivalence, clipped},
we get that for all $\bm x \in A$,
it holds on event $B_1 \cap B_2 \cap B_3^\complement$ that $\hat{\bm{X}}^{\eta|b; >\delta }(\bm x) \notin B$, and hence
\begin{align*}
    & B_0 \cap B_1 \cap B_2 \cap B_3^\complement 
    \\ 
    & \qquad \subseteq \{ \bm{X}^{\eta|b}(\bm x) \in B \} \cap 
     \Big\{\tau^{>\delta}_k(\eta) \leq \floor{1/\eta} < \tau^{>\delta}_{k+1}(\eta);\ \eta \norm{\bm W^{>\delta}_{i}(\eta)} \leq \bar{\delta}\ \text{for some }i \in [k] \Big\}.
\end{align*}
Furthermore, we claim that
for all $\bm x \in A$, $\delta \in (0,\bar{\delta}\wedge\frac{b}{2C})$ and $\eta \in \big(0,\min\{\frac{b \wedge 1}{2C},\bar \delta\}\big)$,
\begin{equation}   \label{proof: goal 1, lemma: atypical 3, large jumps being too small, LDP clipped}
    \begin{split}
        & \{ \bm{X}^{\eta|b}(x) \in B \} \cap 
     \Big\{\tau^{>\delta}_k(\eta) \leq \floor{1/\eta} < \tau^{>\delta}_{k+1}(\eta);\ \eta \norm{\bm W^{>\delta}_{i}(\eta)} \leq \bar{\delta}\ \text{for some }i \in [k] \Big\}
     \\ 
     & \qquad \qquad \qquad \qquad \qquad \qquad \qquad \qquad \qquad \qquad \qquad \qquad 
     \cap \Bigg( \bigcap_{i = 1}^{k+1}A_i(\eta,b,\epsilon,\delta,\bm x) \Bigg) = \emptyset.
    \end{split}
\end{equation}
Then for any $\delta \in (0,\bar{\delta}\wedge\frac{b}{2})$,
\begin{align*}
    \lim_{\eta \downarrow 0}\sup_{\bm x \in A}
    \P\Big( B_0\cap B_1 \cap B_2 \cap B_3^\complement \Big)
    \Big/ \lambda^k(\eta)
    & \leq 
    \lim_{\eta \downarrow 0}\sup_{\bm x \in A}
    \P\Bigg( \Big( \bigcap_{i = 1}^{k+1}A_i(\eta,b,\epsilon,\delta,\bm x) \Big)^c \Bigg)
    \Bigg/ \lambda^k(\eta).
\end{align*}
Applying Lemma~\ref{lemma LDP, small jump perturbation} $(b)$
with some $N > k(\alpha - 1)$, we conclude the proof of \eqref{goal, event B 3, clipped, proposition: asymptotic equivalence, clipped}.

Now, it remains to prove the claim~\eqref{proof: goal 1, lemma: atypical 3, large jumps being too small, LDP clipped}
for any $\bm x \in A$, $\delta \in (0,\bar{\delta}\wedge\frac{b}{2C})$ and $\eta \in \big(0,\min\{\frac{b \wedge 1}{2C},\bar \delta\}\big)$.
First, on this event, there exists some $J \in [k]$ such that $\eta\norm{ \bm W^{>\delta}_J(\eta)} \leq \bar{\delta}$.
Next, recall the definition of the process $\breve{\bm X}^{\eta|b;\delta}_t(\bm x)$ in \eqref{def: x eta b M circ approximation, LDP, 2}--\eqref{def: x eta b M circ approximation, LDP, 3}.
Applying \eqref{proof, ineq for mathring x, proposition: asymptotic equivalence, clipped} with $j = k+1$, we get that
$\big(\cap_{i = 1}^{k+1} A_i(\eta,b,\epsilon,\delta,\bm x)\big) \cap \{ \tau^{>\delta}_{k+1}(\eta) > \floor{1/\eta} \}$,
\begin{align}
\bm{d}_{J_1}\Big(\breve{\bm X}^{\eta|b;\delta}(\bm x), \bm{X}^{\eta|b}(\bm x)\Big)
\leq 
   \sup_{ t \in [0,1] }\norm{ \breve{\bm X}^{\eta|b;\delta}_t(\bm x) - \bm X^{\eta|b}_{ \floor{t/\eta} }(\bm x) }
   < 2\rho\epsilon < \bar{\epsilon}.
    \label{proof: ineq 1, lemma: atypical 3, large jumps being too small, LDP, clipped}
\end{align}
This further confirms that, on the said event,
there exists some $\textbf V = (\bm v_1,\cdots,\bm v_k) \in \R^{m \times k}$ with $\norm{\bm v_j} \leq \bar\epsilon < r$ (recall that we have fixed $\bar\epsilon < r$ at the beginning of the proof) such that
\begin{align*}
    \breve{\bm X}^{\eta|b;\delta}(\bm x)
    =
    \bar h^{(k)|b}\Big(
        \bm x, \big( \eta \bm W^{>\delta}_1(\eta),\cdots,\eta \bm W^{>\delta}_k(\eta)\big), \textbf V,
        \big(\eta  \tau^{>\delta}_1(\eta),\cdots,\eta  \tau^{>\delta}_k(\eta)\big)
    \Big),
\end{align*}
where the mapping $\bar h^{(k)|b}$ is defined in \eqref{def: perturb ode mapping h k b, 1}--\eqref{def: perturb ode mapping h k b, 3}.
Due to $\eta\norm{ \bm W^{>\delta}_J(\eta)} \leq \bar{\delta}$,
it follows from \eqref{choice of bar delta, proof, proposition: asymptotic equivalence, clipped} that
$
\breve{\bm X}^{\eta|b;\delta}(\bm x) \notin B^{\bar\epsilon}.
$
Then by \eqref{choice of bar epsilon, proof, proposition: asymptotic equivalence, clipped} and \eqref{proof: ineq 1, lemma: atypical 3, large jumps being too small, LDP, clipped},
we must have $\bm X^{\eta|b}(\bm x) \notin B$ on the event
$
\big\{\tau^{>\delta}_k(\eta) \leq \floor{1/\eta} < \tau^{>\delta}_{k+1}(\eta);\ \eta \norm{\bm W^{>\delta}_{i}(\eta)} \leq \bar{\delta}\ \text{for some }i \in [k] \big\}
\cap 
\big( \bigcap_{i = 1}^{k+1}A_i(\eta,b,\epsilon,\delta,\bm x) \big),
$
thus verifying claim~\eqref{proof: goal 1, lemma: atypical 3, large jumps being too small, LDP clipped}.

\medskip
\noindent\textbf{Proof of }\eqref{goal, event B 4, clipped, proposition: asymptotic equivalence, clipped}:

Recall that $H(x) = \P(\norm{\bm Z} > x)$.
Due to
\begin{align*}
    & B_0\cap B_1 \cap B_2 \cap B_3 \cap B_4^\complement
    \\ 
    &
    \subseteq 
    \Big\{
        \tau^{>\delta}_k(\eta) \leq \floor{1/\eta} < \tau^{>\delta}_{k+1}(\eta)
    \Big\}
    \cap 
    \Big\{
        \eta\norm{\bm W^{>\delta}_i(\eta)} > \bar\delta\ \forall i \in [k];\ 
        \eta\norm{\bm W^{>\delta}_i(\eta)} > 1/\epsilon^{\frac{1}{2k}}\text{ for some }i \in [k]
    \Big\}.
\end{align*}
and the independence between $\big(\tau^{>\delta}_i(\eta)\big)_{i \in [k]}$ and $\big(\bm W^{>\delta}_i(\eta)\big)_{i \in [k]}$,
we get
\begin{align*}
    & \limsup_{\eta \downarrow 0}\sup_{\bm x \in A}
        \frac{
            \P\big( B_0\cap B_1 \cap B_2 \cap B_3 \cap B_4^\complement \big)
        }{
            \lambda^k(\eta)
        }
    \\ 
    & \leq 
    \lim_{\eta \downarrow 0}
        \frac{1}{ \lambda^k(\eta) }
        \cdot 
        \Big( \eta^{-1}H(\delta \eta^{-1}) \Big)^k
        \cdot k \cdot 
        \Bigg(
            \frac{ H(\bar\delta \eta^{-1}) }{ H(\delta \eta^{-1}) }
        \Bigg)^{k - 1}
        \cdot 
            \frac{ H(\epsilon^{ - \frac{1}{2k} } \eta^{-1}) }{ H(\delta \eta^{-1}) }
    \qquad\text{ by \eqref{property: large jump time probability}}
    \\ 
    & = 
    \lim_{\eta \downarrow 0}
    \frac{1}{ \lambda^k(\eta) }
        \cdot 
        \Big( \eta^{-1}H(\eta^{-1}) \Big)^k
    \cdot k \cdot 
        \Bigg(
            \frac{ H(\bar\delta \eta^{-1}) }{ H(\eta^{-1}) }
        \Bigg)^{k - 1}
        \cdot 
            \frac{ H(\epsilon^{ - \frac{1}{2k} } \eta^{-1}) }{ H(\eta^{-1}) }
    \\ 
    & = k \cdot \lim_{\eta \downarrow 0}
    \Bigg(
            \frac{ H(\bar\delta \eta^{-1}) }{ H(\eta^{-1}) }
        \Bigg)^{k - 1}
        \cdot 
            \frac{ H(\epsilon^{ - \frac{1}{2k} } \eta^{-1}) }{ H(\eta^{-1}) }
    \qquad \text{recall that }\lambda(\eta) = \eta^{-1}H(\eta^{-1})
    \\ 
    & = \bar\delta^{-k\alpha} \cdot \epsilon^{\frac{\alpha}{2k}}
    \qquad
    \text{due to $H(x) \in \RV_{-\alpha}(x)$ as $x \to \infty$; see Assumption~\ref{assumption gradient noise heavy-tailed}.}
\end{align*}

\medskip
\noindent\textbf{Proof of }\eqref{goal, event B 5, clipped, proposition: asymptotic equivalence, clipped}:

We only consider $\delta \in (0, \frac{b}{2C})$.
On event $B_1 \cap B_2 = \{\tau^{>\delta}_{k}(\eta) \leq \floor{1/\eta} < \tau^{>\delta}_{k+1}(\eta)\}$,
$ \hat{\bm{X}}^{\eta|b; >\delta }(\bm x)$ admits the expression in \eqref{property: equivalence between breve X and hat X, clipped at b}.
\elaborate{
it follows from \eqref{property: equivalence between breve X and hat X, clipped at b} that
$
    \hat{\bm{X}}^{\eta|b; >\delta }(x)
    =
    h^{(k)|b}\big( x, \eta W^{>\delta}_{1}(\eta), \cdots,\eta W^{>\delta}_k(\eta), 
     \eta \tau^{>\delta}_1(\eta), \cdots,\eta \tau^{>\delta}_k(\eta)\big).
$
}
Then by applying Lemma \ref{lemma: SGD close to approximation x breve, LDP clipped}
we yield that for any $\bm x \in A$ and any $\eta \in (0,\frac{\epsilon \wedge b}{2C})$, the inequality
$$
\bm{d}_{J_1}\Big(\hat{\bm{X}}^{\eta|b; >\delta }(\bm x), \bm{X}^{\eta|b}(\bm x)  \Big)
\leq 
\sup_{ t \in [0,1] }
\norm{ \hat{\bm{X}}^{\eta|b; >\delta }_t(\bm x) -  \bm{X}^{\eta|b}_{\floor{t/\eta}}(\bm x) }
< 
(2\rho D)^{k+1} \sqrt{\epsilon},
$$
holds on event
$
\Big(\bigcap_{i = 1}^{k+1} A_i(\eta,b,\epsilon,\delta,\bm x) \Big).
$
Due to our choice of
$(2\rho D)^{k+1} \sqrt{\epsilon} < \Delta$ in \eqref{proof, choice of epsilon, proposition: standard M convergence, LDP clipped, stronger boundedness assumption}, we get
$
 \Big(\bigcap_{i = 1}^{k+1} A_i(\eta,b,\epsilon,\delta,\bm x) \Big) \cap B_1 \cap B_2 \cap B_3 \cap B_4 \cap B_0 = \emptyset.
$
\elaborate{
\begin{align*}
    & \big(\bigcap_{i = 1}^{k+1} A_i(\eta,b,\epsilon,\delta,x) \big) \cap B_1 \cap B_2 \cap B_3 \cap B_0
    \\
    & \subseteq \{\bm{d}_{J_1}\big(\hat{\bm{X}}^{\eta|b; (k) }(x), \bm{X}^{\eta|b}(x)  \big) < \Delta\} \cap \{\bm{d}_{J_1}\big(\hat{\bm{X}}^{\eta|b; (k) }(x), \bm{X}^{\eta|b}(x)  \big) > \Delta\} = \emptyset.
\end{align*}
}
Therefore,
\begin{align*}
    \limsup_{\eta\downarrow 0}\sup_{\bm x \in A}
    {\P\Big( B_1 \cap B_2 \cap B_3  \cap B_0 \Big)}\Big/{\lambda^k(\eta)}
    \leq 
    \limsup_{\eta\downarrow 0}\sup_{\bm x \in A}
    {\P\Bigg( \Big(\bigcap_{i = 1}^{k+1} A_i(\eta,b,\epsilon,\delta,
    \bm x) \Big)^c\Bigg)}\Bigg/{\lambda^k(\eta)}.
\end{align*}
Again, by applying Lemma~\ref{lemma LDP, small jump perturbation} $(b)$
with some $N > k(\alpha - 1)$, we conclude the proof.
\end{proof}


Recall that
${\big(\bm W^*_j(c)}\big)_{j \geq 1}$ is a sequence of iid copies of $\bm W^*(c)$ defined in \eqref{def: prob measure Q, LDP},
and $\big(U_{(j:k)}\big)_{j \in [k]}$ are the order statistics of $k$ samples of Unif$(0,1)$.
In order to prove Proposition \ref{proposition: uniform weak convergence, clipped},
we prepare a lemma regarding a weak convergence on events
$
E^\delta_{c,k}(\eta) = \big\{ \tau^{>\delta}_{k}(\eta) \leq  \floor{1/\eta} < \tau^{>\delta}_{k+1}(\eta);\ 
\eta\norm{ \bm W^{>\delta}_j(\eta)} > c\ \ \forall j \in [k] \big\}
$
defined
in \eqref{def: E eta delta set, LDP}.

\begin{lemma}
\label
{lemma: weak convergence, expectation wrt approximation, LDP, preparation}
\linksinthm
{lemma: weak convergence, expectation wrt approximation, LDP, preparation}%
Let Assumption \ref{assumption gradient noise heavy-tailed}
hold.
Let $A \subseteq \R^m$ be a compact set.
Let bounded function $\Psi:\mathbb{R}^m\times \mathbb{R}^{d \times k} \times (0,1]^{k\uparrow} \to \mathbb{R} $ be continuous on $\mathbb{R}^m \times \mathbb{R}^{d \times k} \times (0,1)^{k\uparrow}$.
For any $\delta > 0,\ c > \delta$ and $k \in \mathbb N$, 
\begin{align*}
    \lim_{\eta \downarrow 0}\sup_{\bm x \in A}
    \left|\rule{0cm}{0.8cm}
        \frac{ \E \Big[ 
            \Psi\Big(
                \bm x, \big(\eta \bm W^{>\delta}_{1}(\eta),\cdots,\eta \bm W^{>\delta}_{k}(\eta)\big),\big(\eta \tau^{>\delta}_{1}(\eta),\cdots,\eta \tau^{>\delta}_{k}(\eta)\big)  \Big)
            \mathbbm{I}_{E^\delta_{c,k}(\eta)}  \Big] }{ \lambda^k(\eta) } - \frac{(1/c^{\alpha k})\psi_{c,k}(\bm x) }{k!} 
    \right|\rule{0cm}{0.8cm}
    = 0
\end{align*}
where
$ \psi_{c,k}(\bm x) \delequal \E\Big[\Psi\Big(\bm x, \big(\bm W^*_1(c),\cdots,\bm W^*_{k}(c)\big),\big( U_{(1;k)},\cdots,U_{(k;k)}\big)\Big)\Big]$. 
\end{lemma}

\begin{proof}
\linksinpf
{lemma: weak convergence, expectation wrt approximation, LDP, preparation}%
Fix some $\delta > 0, c > \delta$ and $k \in \mathbb N$.
We proceed with a proof by contradiction.
Suppose there exist some $\epsilon > 0$,
some sequence $\bm x_n \in A$, and some sequence $\eta_n \downarrow 0$ such that
\begin{align}
    \Big| { { \lambda^{-k}(\eta_n) } 
    \E \Big[ \Psi\big( \bm x_n, 
    \eta_n \textbf{W}^{\eta_n},\eta_n \bm{\tau}^{\eta_n}
    \big)\mathbbm{I}_{E^\delta_{c,k}(\eta_n)}  \Big] }
    - { (1/{ k! } )\cdot c^{-\alpha k}\cdot \psi_{c,k}(\bm x_n) } \Big| > \epsilon\ \ \ \forall n \geq 1
    \label{proof by contradiction, assumption, lemma: weak convergence, expectation wrt approximation, LDP, preparation}
\end{align}
where
$
\textbf W^{\eta}\delequal(\bm W^{>\delta}_1(\eta),\cdots, \bm W^{>\delta}_k(\eta)),\
    \bm \tau^{\eta} \delequal
(\tau^{>\delta}_1(\eta),\cdots,\tau^{>\delta}_k(\eta)).
$
Since $A$ is compact, 
by picking a sub-sequence if needed we can assume w.l.o.g.\ that 
 $\bm x_{n} \to \bm x^*$ for some $\bm x^*\in A$.
Now, observe that
\begin{align*}
    & \lim_{n \rightarrow \infty}
    { \lambda^{-k}(\eta_n) }{ \E \bigg[ \Psi\big( \bm x_n, 
    \eta_n \textbf{W}^{\eta_n},\eta_n \bm{\tau}^{\eta_n}
    \big)\mathbbm{I}_{E^\delta_{c,k}(\eta_n)}  \bigg] }
    \\
    & = 
    \bigg[ \lim_{n \rightarrow \infty} { \lambda^{-k}(\eta_n) }{\P\Big( E^\delta_{c,k}(\eta_n) \Big)} \bigg]
    \cdot \lim_{n\rightarrow\infty}
    \E\bigg[
   \Psi\big( \bm x_n, 
    \eta_n \textbf{W}^{\eta_n},
    \eta_n\bm{\tau}^{\eta_n}
    \big)\Big| E^\delta_{c,k}(\eta_n)
   \bigg]
   \\
   & = 
    (1/{ k! } )\cdot c^{-\alpha k} \cdot \psi_{c,k}(\bm x^*)
    \qquad 
    \text{by Lemma \ref{lemma: weak convergence of cond law of large jump, LDP}, $\bm x_n \to \bm x^*$, and continuous mapping theorem}.
\end{align*}
However, by Bounded Convergence theorem, we see that $\psi_{c,k}$ is also continuous, and hence $\psi_{c,k}(\bm x_n) \rightarrow \phi_{c,k}(\bm x^*)$.
This leads to a contradiction with \eqref{proof by contradiction, assumption, lemma: weak convergence, expectation wrt approximation, LDP, preparation} and allows us to conclude the proof.
\elaborate{
This leads to the contradiction 
\begin{align*}
    \lim_{n \rightarrow \infty}\Bigg| \frac{ \E \Big[ \Phi\big( x_n, 
    \bm{W}^{\eta_n},\bm{\tau}^{\eta_n}
    \big)\mathbbm{I}_{E^\delta_{c,k}(\eta_n)}  \Big] }{ \lambda^k(\eta_n) } - \frac{(1/c^{\alpha k})\phi_{c,k}(x_n) }{k!} \Bigg| = 0
\end{align*}
}
\end{proof}

We are now ready to prove Proposition \ref{proposition: uniform weak convergence, clipped}.

\begin{proof}[Proof of Proposition \ref{proposition: uniform weak convergence, clipped}]
\linksinpf{proposition: uniform weak convergence, clipped}%
Fix some $b,r > 0$, $k \in \mathbb N$, and $g \in \mathcal{C}\big({ \mathbb{D}\setminus  \mathbb{D}^{(k-1)|b}_{A}(r) }\big)$;
i.e. $g:\mathbb{D} \to [0,\infty)$ is non-negative, continuous, and bounded, whose support $B \delequal \text{supp}(g)$ bounded away from $\mathbb{D}^{(k - 1)|b}_{A}(r)$.
By Lemma~\ref{lemma: LDP, bar epsilon and delta, clipped version},
we can fix some $\bar{\epsilon} \in (0,r)$ and $\bar{\delta} > 0$ such that
\begin{itemize}
    \item 
        for any $\bm x \in A$ and $b > 0$,
         \begin{align}
            h^{(k)|b}\big(\bm x,(\bm w_1,\cdots,\bm w_k),\bm{t}\big) \in B^{\bar{\epsilon}} \Longrightarrow
            \norm{\bm w_j} > \bar{\delta}\ \ \forall j \in [k];
            \label{choice of bar delta, proposition: uniform weak convergence, clipped}
        \end{align}

    \item
        $\bm{d}_{J_1}\big(B^{\bar\epsilon},\mathbb{D}^{(k- 1)|b}_{A}(r)\big) > \bar\epsilon.$
\end{itemize}
Fix some $\delta \in (0,\bar{\delta} \wedge \frac{b}{2})$,
and observe that for any $\eta > 0$ and $\bm x \in A$,
\begin{align*}
    & g\Big( \hat{\bm{X}}^{\eta|b;>\delta}(\bm x)\Big)
    \\ 
    & = 
    \underbrace{g\Big( \hat{\bm{X}}^{\eta|b;>\delta}(\bm x)\Big)\mathbbm{I}\Big\{ \tau^{>\delta}_{k+1}(\eta) \leq \floor{1/\eta} \Big\} }_{ \delequal I_1(\eta,\bm x) }
    +
    \underbrace{ g\Big( \hat{\bm{X}}^{\eta|b;>\delta}(\bm x)\Big)\mathbbm{I}\Big\{ \tau^{>\delta}_{k}(\eta) >\floor{1/\eta} \Big\} }_{ \delequal I_2(\eta,\bm x) }
    \\
    & 
    \qquad
    + 
    \underbrace{ g\Big( \hat{\bm{X}}^{\eta|b;>\delta}(\bm x)\Big)\mathbbm{I}\Big\{ 
        \tau^{>\delta}_{k}(\eta) \leq \floor{1/\eta} < \tau^{>\delta}_{k+1}(\eta);\ \eta\norm{\bm W^{>\delta}_j(\eta)} \leq \bar{\delta}\text{ for some }j \in [k] 
        \Big\} }_{ \delequal I_3(\eta,\bm x) }
    \\
    & \qquad
    + 
    \underbrace{ g\Big( \hat{\bm{X}}^{\eta|b;>\delta}(\bm x)\Big)\mathbbm{I}\Big( E^\delta_{\bar{\delta},k}(\eta)\Big) }_{ \delequal I_4(\eta,\bm x)}.
\end{align*}

For $I_1(\eta,\bm x)$, it follows from \eqref{property: large jump time probability} that
$
    \sup_{\bm x \in \R^m}\E[I_1(\eta,\bm x)]
    \leq \norm{g}\cdot\Big[ \frac{1}{\eta}\cdot H(\delta/\eta)      \Big]^{k + 1}.
$
Therefore,
$
\lim_{\eta \downarrow 0}\sup_{\bm x \in A}\E[I_1(\eta,\bm x)]\Big/\big( \eta^{-1}H(\eta^{-1}) \big)^k
\leq 
\frac{\norm{g}}{\delta^{\alpha(k + 1)}} \cdot
\lim_{n \rightarrow \infty}\frac{H(1/\eta)}{\eta} = 0
$
due to $H(x) \in \RV_{-\alpha}(x)$ and $\alpha > 1$.

Next, for term $I_2(\eta,\bm x)$,
it has been shown in the proof of \eqref{goal, event B 2, clipped, proposition: asymptotic equivalence, clipped} for Proposition~\ref{proposition: asymptotic equivalence, clipped} that, for all $\delta \in (0,\bar{\delta} \wedge \frac{b}{2})$ and $\bm x \in A$,
it holds on event $\{\tau^{>\delta}_k(\eta) > \floor{1/\eta}\}$ that $\hat{\bm{X}}^{\eta|b;>\delta }(\bm x) \notin B^{\bar{\epsilon}}$,
and hence $I_2(\eta,\bm x) = 0$.


For the term $I_3(\eta,\bm x)$, 
on event $\{\tau^{>\delta}_k(\eta) \leq \floor{1/\eta} < \tau^{>\delta}_{k+1}(\eta) \}$ the process $\hat{\bm X}^{\eta|b;>\delta}(\bm x)$
admits the expression in \eqref{property: equivalence between breve X and hat X, clipped at b}.
In particular,
since there is some $i \in [k]$ such that $\eta \norm{\bm W^{>\delta}_{i}(\eta)} \leq \bar{\delta}$,
by \eqref{choice of bar delta, proposition: uniform weak convergence, clipped} we must have 
$
\hat{\bm{X}}^{\eta|b; (k) }(\bm x) \notin B,
$
and hence $I_3(\eta,\bm x) = 0$.


Lastly, for the term $I_4(\eta,\bm x)$,
on event $E^\delta_{\bar{\delta},k}(\eta)$
the process $\hat{\bm X}^{\eta|b;(k)}(\bm x)$
would again admit the expression in \eqref{property: equivalence between breve X and hat X, clipped at b}.
As a result, for any $\eta > 0$ and $\bm x \in A$,
we have 
$$
\E[I_4(\eta,\bm x)]
    = \E \bigg[ 
    \Psi\big( \bm x, 
    \eta \textbf{W}^{\eta},\eta \bm{\tau}^{\eta}
    \big)\mathbbm{I}_{E^\delta_{\bar\delta,k}(\eta)}  
    \bigg],
$$
where
$
\textbf W^{\eta}\delequal(\bm W^{>\delta}_1(\eta),\cdots, \bm W^{>\delta}_k(\eta)),\
    \bm \tau^{\eta} \delequal
(\tau^{>\delta}_1(\eta),\cdots,\tau^{>\delta}_k(\eta)),
$
and
$
\Psi(\bm x,\textbf W,\bm t)
    \delequal
    g\big( h^{(k)|b}(\bm x,\textbf W,\bm t) \big).
$
\elaborate{
\begin{align*}
    \E[I_5(\eta,x)]
    & = \E\Big[ \Phi\big( x, \eta W^{>\delta}_{1}(\eta),\cdots,\eta W^{>\delta}_{k}(\eta),\eta \tau^{>\delta}_{1}(\eta),\cdots,\eta \tau^{>\delta}_{k}(\eta)  \big)\mathbbm{I}\big(E^\delta_{\bar{\delta},k}(\eta)\big) \Big]
\end{align*}
where
$\Phi: \mathbb{R}\times \mathbb{R}^{k} \times (0,1)^{k\uparrow} \to \mathbb{R}$ is defined as
\begin{align}
    \Phi(x_0,w_1,w_2,\cdots,w_{k},t_1,t_2,\cdots,t_{k})
    \delequal
    g\big( h^{(k)|b}(x_0,w_1,w_2,\cdots,w_{k},t_1,t_2,\cdots,t_{k}) \big).
    \nonumber
\end{align}
}
Besides, let $\psi(\bm x) \delequal \E\Big[\Psi\Big(\bm x, \big(\bm W^*_1(c),\cdots,\bm W^*_{k}(c)\big),\big( U_{(1;k)},\cdots,U_{(k;k)}\big)\Big)\Big]$.
First, the continuity of mapping $\Psi$ 
on 
$
\mathbb{R}^m\times \mathbb{R}^{d \times k} \times (0,1)^{k\uparrow}
$
follows directly from the continuity of $g$ and $h^{(k)|b}$
(see Lemma \ref{lemma: continuity of h k b mapping clipped}).
Besides, $\norm{\Psi} \leq \norm{g} < \infty$, so $\Psi(\cdot)$ is also bounded.
By Bounded Convergence Theorem, one can see that $\psi(\cdot)$ is also continuous.
Also, $\norm{\psi} \leq \norm{\Psi} \leq \norm{g} < \infty$.
By Lemma \ref{lemma: weak convergence, expectation wrt approximation, LDP, preparation},
\begin{align*}
    \lim_{\eta \downarrow 0}\sup_{\bm x \in A}
    \Bigg| 
    {\lambda^{-k}(\eta)}\E \bigg[ 
    \Psi\big( \bm x, 
    \eta \textbf{W}^{\eta},\eta \bm{\tau}^{\eta}
    \big)\mathbbm{I}_{E^\delta_{\bar\delta,k}(\eta)}  
    \bigg]
    -
    (1/{ k! } )\cdot {\bar\delta}^{-\alpha k}\cdot \psi(\bm x) 
    \Bigg|
 = 0.
\end{align*}
Meanwhile, 
due to the continuity of $\psi(\cdot)$, for any $\bm x_n,\ \bm x^* \in A$ with $\lim_{n \rightarrow\infty}\bm x_n = \bm x^*$,
we have $\lim_{n \rightarrow \infty}\psi(\bm x_n) = \psi(\bm x^*)$.
To conclude the proof, we only need to show that 
\begin{align}
    \frac{(1/{\bar\delta}^{\alpha k}) \psi(\bm x^*)  }{k !} = \mathbf{C}^{(k)|b}   (g;\bm x^*).
    \label{proof, goal, limiting measure, proposition: uniform weak convergence, clipped}
\end{align}
To do so, recall the law of $\bm W^*(c)$ in \eqref{def: prob measure Q, LDP}.
By definition of 
$\psi(\cdot)$,
\begin{align*}
    \psi(\bm x^*) & =
    \int g\Big( h^{(k)|b}\big(\bm x^*,(w_1\bm \theta_1,\cdots,w_{k}\bm \theta_k),(t_1,\cdots,t_{k}) \big) \Big)
    \mathbbm{I}\Big\{w_j > \bar{\delta}\ \forall j \in [k]\Big\}
    \\
    &\ \ \ \ \ \ \ \ \ \ \ \ \ \ \ \ \ \qquad
    \P\Big( U_{(1;k)} \in dt_1,\cdots,U_{(k;k)} \in dt_{k} \Big) \times 
    \Bigg( \bigtimes_{j = 1}^{k} \bigg(\bar{\delta}^\alpha \cdot  \nu_\alpha(d w_j) \times \mathbf S(d \bm \theta_j)\bigg) \Bigg).
\end{align*}
By \eqref{choice of bar delta, proposition: uniform weak convergence, clipped},
 we have
\begin{align*}
   & g\Big( h^{(k)|b}\big(\bm x^*,(w_1\bm \theta_1,\cdots,w_{k}\bm \theta_k),(t_1,\cdots,t_{k}) \big) \Big)
   \\
   & = 
    g\Big( h^{(k)|b}\big(\bm x^*,(w_1\bm \theta_1,\cdots,w_{k}\bm \theta_k),(t_1,\cdots,t_{k}) \big) \Big)
    \mathbbm{I}\Big\{w_j > \bar{\delta}\ \forall j \in [k]\Big\}.
\end{align*} 
Besides,
$\P\Big( U_{(1;k)} \in dt_1,\cdots,U_{(k;k)}\in dt_{k} \Big) = k! \cdot \mathbbm{I}\{0 < t_1 < t_2 < \cdots <t_{k} < 1\}
    \mathcal{L}^{k\uparrow}_1(dt_1,\cdots,dt_{k})$ 
    where $\mathcal{L}^{k\uparrow}_1$ is the Lebesgue measure restricted on $(0,1)^{k\uparrow}$.
\elaborate{
We make a few observations.
\begin{itemize}
    \item From \eqref{choice of bar delta, proposition: uniform weak convergence, clipped},
    we must have $g\big( h^{(k)}(x^*,w_1,\cdots,w_{k},t_1,\cdots,t_{k}) \big)= 0$ if there is some $j\in[k]$ with $|w_j| \leq \bar{\delta}$;
    \item Regarding the density of $\big(U_{(i;k)}\big)_{i = 1}^{k}$,
    the definition of order statistics and the law of the uniform random variables imply that 
    $\P\Big( U_{(1;k)} \in dt_1,\cdots,U_{(k;k)}\in dt_{k} \Big) = k! \cdot \mathbbm{I}\{0 < t_1 < t_2 < \cdots <t_{k} < 1\}
    \mathcal{L}^{k\uparrow}_1(dt_1,\cdots,dt_{k})$ 
    where $\mathcal{L}^{k\uparrow}_1$ is the Lebesgue measure restricted on $(0,1)^{k\uparrow}$.
\end{itemize}
}
As a result,
\begin{align*}
    & \psi(\bm x^*) 
    \\
    & = k! \cdot \bar{\delta}^{\alpha k} 
    \int g\Big( h^{(k)|b}\big(\bm x^*, (w_1\bm \theta_1,\cdots,w_k\bm\theta_k),\bm t\big) \Big)
    \Bigg(\bigtimes_{j = 1}^k\bigg(\nu_\alpha(d w_j)\times \mathbf S(d\bm \theta_j)\bigg)\Bigg)
     \times \mathcal{L}^{k\uparrow}_1(d\bm{t})
    \\
    & 
    = 
    k! \cdot {\bar{\delta}^{\alpha k}}\cdot  \mathbf{C}^{(k)|b }\big(g;\bm x^*\big)
\end{align*}
by the definition of $\mathbf{C}^{(k)|b }$ in \eqref{def: measure mu k b t},
thus verifying \eqref{proof, goal, limiting measure, proposition: uniform weak convergence, clipped}.
\end{proof}

\section{Metastability Analysis}
\label{sec: first exit time simple version, proof}

In this section, we collect the proofs for Section~\ref{sec: first exit time simple version}.
Specifically, Section~\ref{subsec: Exit time analysis framework} develops the general framework for first exit analysis of Markov processes by establishing Theorem~\ref{thm: exit time analysis framework}.
Section~\ref{sec: proof of proposition: first exit time} then applies the framework in the context of heavy-tailed stochastic difference equations and proves Theorem~\ref{theorem: first exit time, unclipped}.

\subsection{Proof of Theorem~\ref{thm: exit time analysis framework}}
\label{subsec: Exit time analysis framework}

Our proof of Theorem~\ref{thm: exit time analysis framework} hinges on the following proposition.

\begin{proposition}\label{prop: exit time analysis main proposition}
\linksinthm{prop: exit time analysis main proposition}
Suppose that Condition~\ref{condition E2} holds.
\begin{enumerate}[$(i)$]
    \item 
        If $C(\cdot)$ is a probability measure supported on $I^\complement$ (i.e., $C(I^\complement) = 1$),
        then for each measurable set $B\subseteq \mathbb S$ and $t\geq 0$, there exists $\delta_{t,B}(\epsilon)$ such that
        \begin{align*}
            C(B^\circ)\cdot e^{-t} - \delta_{t,B}(\epsilon)
            &
            \leq
            \liminf_{\eta\downarrow 0} \inf_{x\in A(\epsilon)}\P\big(\gamma(\eta) \tau_{I(\epsilon)^\complement}^\eta(x) > t;\; V^\eta_{\tau_\epsilon}(x) \in B\big)
            \\
            &
            \leq
            \limsup_{\eta\downarrow 0} \sup_{x\in A(\epsilon)}\P\big(\gamma(\eta) \tau_{I(\epsilon)^\complement}^\eta(x) > t;\; V^\eta_{\tau_\epsilon}(x) \in B\big)
            \leq C(B^-)\cdot e^{-t} + \delta_{t,B}(\epsilon)
        \end{align*}
        for all sufficiently small $\epsilon>0$, where $\delta_{t,B}(\epsilon) \to 0$ as $\epsilon \to 0$.

    \item 
        If $C(I^\complement) = 0$ (i.e., $C(\cdot)$ is trivially zero), then for each $t > 0$, there exists $\delta_{t}(\epsilon)$ such that
        \begin{align*}
            \limsup_{\eta \downarrow 0}
            \sup_{x\in A(\epsilon)}\P\big(\gamma(\eta) \tau_{I(\epsilon)^\complement}^\eta(x) \leq t\big)
            \leq \delta_{t}(\epsilon)
        \end{align*}
        for all $\epsilon > 0$ sufficiently small, where $\delta_t(\epsilon) \to 0$ as $\epsilon \to 0$.
\end{enumerate}

\end{proposition}

\begin{proof}
\linksinpf{prop: exit time analysis main proposition}
Fix some measurable $B \subseteq \S$ and $t \geq 0$.
Henceforth in the proof, given any choice of $0 < r < R$, we only consider
$\epsilon \in (0,\epsilon_B)$ and $T$ sufficiently large such that Condition~\ref{condition E2} holds with $T$ replaced with $\frac{1-r}{2}T$, $\frac{2-r}{2}T$, $rT$, and $RT$.
Let
\[
    \rho^{\eta}_{i}(x)
    \delequal \inf\Big\{j\geq \rho^{\eta}_{i-1}(x) + \lfloor rT/\eta \rfloor: V_j^\eta(x) \in A(\epsilon)\Big\}
\]
where 
\(\rho_0^\eta(x) = 0.\)
One can interpret these as the $i$\textsuperscript{th} asymptotic regeneration times after cooling period $rT/\eta$.
We start with the following two observations: For any $0 < r < R$,
\begin{align}
    \P\Big(\tau_{I(\epsilon)^\complement}^\eta(y) \in \big(RT/\eta,\, \rho_1^\eta(y)\big]\Big)
    &
    \leq 
    \P\Big(\tau_{I(\epsilon)^\complement}^\eta(y) \wedge \rho_1^\eta(y) > RT/\eta\Big)
    \nonumber\\
    &
    \leq
    \P\Big(V_j^\eta(y)\in I(\epsilon)\setminus A(\epsilon)\quad \forall j \in \big[\lfloor r T/\eta \rfloor,\,RT/\eta \big] \Big)
    \nonumber\\
    &
    \leq
    \sup_{z\in I(\epsilon)\setminus A(\epsilon)}
        \P\Big(\tau^\eta_{(I(\epsilon)\setminus A(\epsilon))^c}(z) > \frac{R-r}{2} T/\eta \Big)
    \nonumber\\
    &
    =
    \gamma(\eta)T/\eta \cdot \lo(1),
    \label{eq: exit time analysis: almost nothing happens between T over eta and rho one}
\end{align}
where  the last equality is from \eqref{eq:E3} of Condition~\ref{condition E2}, and
\begin{align}
    &
    \sup_{y\in A(\epsilon)}
    \P\Big(
        V_{\tau_\epsilon}^\eta(y) \in B
        ;\;
        \tau_{I(\epsilon)^\complement}^\eta(y) \leq \rho_1^\eta(y)
    \Big)
    \nonumber\\
    &
    \leq 
    \sup_{y\in A(\epsilon)}
    \P\Big(
        V_{\tau_\epsilon}^\eta(y) \in B
        ;\;
        \tau_{I(\epsilon)^\complement}^\eta(y) \leq RT/\eta
    \Big)
    +
    \sup_{y\in A(\epsilon)}
    \P\Big(
        \tau_{I(\epsilon)^\complement}^\eta(y) \in \big(RT/\eta,\, \rho_1^\eta(y)\big]
    \Big)
    \nonumber\\
    &
    \leq 
    \sup_{y\in A(\epsilon)}
    \P\Big(
        V_{\tau_\epsilon}^\eta(y) \in B
        ;\;
        \tau_{I(\epsilon)^\complement}^\eta(y) \leq RT/\eta
    \Big)
    +
    \gamma(\eta)T/\eta \cdot \lo(1)
    \nonumber\\
    &
    \leq\big(C(B^-)+ \delta_B(\epsilon,RT) + \lo(1)\big) \cdot \gamma(\eta) RT/\eta,
    \label{eq: exit time analysis: upper bound of the probability that tau is before rho}
\end{align}
where the second inequaility is from \eqref{eq: exit time analysis: almost nothing happens between T over eta and rho one} and the last equality is from \eqref{eq: exit time condition upper bound} of Condition~\ref{condition E2}.

\medskip
\noindent
\textbf{Proof of Case $(i)$.}

We work with different choices of $R$ and $r$ for the lower and upper bounds.
For the lower bound, we work with $R>r>1$ and set
\(K = \left\lceil \frac{t/\gamma(\eta)}{T/\eta} \right\rceil\).
Note that for $\eta \in \big(0, (r-1)T\big)$,
we have
$\lfloor r T/\eta \rfloor \geq T/\eta$ and hence $\rho_K^\eta(x) \geq K \lfloor r T/\eta\rfloor \geq t/\gamma(\eta)$.
Note also that from the Markov property conditioning on $\mathcal F_{\rho_j^\eta(x)}$,
\begin{align}
    &
    \inf_{x \in A(\epsilon)}\P\big(\gamma(\eta) \tau_{I(\epsilon)^\complement}^\eta(x) > t;\; V^\eta_{\tau_\epsilon}(x) \in B\big)
    \nonumber\\
    &
    \geq
    \inf_{x \in A(\epsilon)}
        \P(\tau_{I(\epsilon)^\complement}^\eta (x) > \rho_K^\eta(x);\; V^\eta_{\tau_\epsilon}(x) \in B)
    =
    \inf_{x \in A(\epsilon)}
    \sum_{j=K}^\infty     
        \P\Big(
            \tau_{I(\epsilon)^\complement}^\eta(x) \in \big(\rho_j^\eta(x),\, \rho_{j+1}^\eta(x) \big];\; V^\eta_{\tau_\epsilon}(x) \in B
        \Big)         
    \nonumber\\
    &
    \geq
    \inf_{x \in A(\epsilon)}
    \sum_{j=K}^\infty     
        \P\Big(
            \tau_{I(\epsilon)^\complement}^\eta(x) \in \big(\rho_j^\eta(x),\, \rho_{j}^\eta(x) + T/\eta \big];\; V^\eta_{\tau_\epsilon}(x) \in B
        \Big)         
    \nonumber\\
    &
    \geq
    \inf_{x \in A(\epsilon)}
    \sum_{j=K}^\infty     
        \inf_{y\in A(\epsilon)}
        \P\Big(
            \tau_{I(\epsilon)^\complement}^\eta(y) \leq T/\eta;\; V^\eta_{\tau_\epsilon}(y) \in B
        \Big) 
        \cdot
        \P\Big(\tau_{I(\epsilon)^\complement}^\eta(x) > \rho_j^\eta(x)\Big).
    \nonumber\\
    &
    \geq
    \inf_{y\in A(\epsilon)}
    \P\Big(
        \tau_{I(\epsilon)^\complement}^\eta(y) \leq T/\eta;\; V^\eta_{\tau_\epsilon}(y) \in B
    \Big) 
    \cdot
    \sum_{j=K}^\infty \inf_{x \in A(\epsilon)}  
        \P\Big(\tau_{I(\epsilon)^\complement}^\eta(x) > \rho_j^\eta(x)\Big).
        \label{eq: exit time and location analysis with epsilon lower bound}
\end{align}
\elaborate{
\begin{align*}
    &
    \P\Big(
        \tau_{I(\epsilon)^\complement}^\eta(x) \in \big(\rho_j^\eta(x),\, \rho_j^\eta(x) + T/\eta\big];\; V^\eta_{\tau_\epsilon}(x) \in B
    \Big) 
    \\
    &
    =
    \E\bigg[
        \P\Big(
            \tau_{I(\epsilon)^\complement}^\eta(x) \in \big(\rho_j^\eta(x),\, \rho_j^\eta(x) + T/\eta\big];\; V^\eta_{\tau_\epsilon}(x) \in B
        \Big|
            \mathcal F_{\rho_j^\eta(x)}
        \Big) 
    \bigg]
    \\    
    &
    =
    \E\bigg[
        \P\Big(
            \tau_{I(\epsilon)^\complement}^\eta(x) \leq \rho_j^\eta(x) + T/\eta;\; V^\eta_{\tau_\epsilon}(x) \in B
        \Big|
            \mathcal F_{\rho_j^\eta(x)}
        \Big) 
        \cdot
        \I\big\{\tau_{I(\epsilon)^\complement}^\eta(x) > \rho_j^\eta(x)\big\}
    \bigg]
    \\
    &
    \geq
        \E\bigg[
            \inf_{y\in A(\epsilon)}
            \P\Big(
                \tau_{I(\epsilon)^\complement}^\eta(y) \leq T/\eta;\; V^\eta_{\tau_\epsilon}(y) \in B
            \Big) 
            \cdot
            \I\big\{\tau_{I(\epsilon)^\complement}^\eta(x) > \rho_j^\eta(x)\big\}
        \bigg]
    \\
    &
    =
    \inf_{y\in A(\epsilon)}
    \P\Big(
        \tau_{I(\epsilon)^\complement}^\eta(y) \leq T/\eta;\; V^\eta_{\tau_\epsilon}(y) \in B
    \Big) 
    \cdot
    \P\big(\tau_{I(\epsilon)^\complement}^\eta(x) > \rho_j^\eta(x)\big).
\end{align*}
}%
From the Markov property conditioning on $\mathcal F_{\rho_j^\eta(x)}$,
the second term can be bounded as follows:
\begin{align}
    &
    \sum_{j=K}^\infty \inf_{x \in A(\epsilon)}  
    \P\Big(\tau_{I(\epsilon)^\complement}^\eta(x) > \rho_j^\eta(x)\Big)
    \nonumber\\
    &
    \geq 
    \sum_{j=0}^\infty
    \bigg(
        \inf_{y \in A(\epsilon)}
            \P\Big(\tau_{I(\epsilon)^\complement}^\eta(y) > \rho_1^\eta(y)\Big)
    \bigg)^{K+j}
    =
    \sum_{j=0}^\infty
        \bigg(
            1 -\sup_{y \in A(\epsilon)}
                \P\Big(\tau_{I(\epsilon)^\complement}^\eta(y) \leq \rho_1^\eta(y) \Big)
        \bigg)^{K+j}
    \nonumber\\
    &
    =
    \frac1{\sup_{y \in A(\epsilon)}\P\Big(\tau_{I(\epsilon)^\complement}^\eta(y) \leq \rho_1^\eta(y) \Big)}
    \cdot
    \bigg(
        1 - \sup_{y \in A(\epsilon)}\P\Big(\tau_{I(\epsilon)^\complement}^\eta(y) \leq \rho_1^\eta(y) \Big)
    \bigg)^{ \ceil{\frac{t/\gamma(\eta)}{T/\eta}}  }   
    \nonumber\\
    &
    \geq
    \frac1{\big(1+\delta_{\S}(\epsilon, RT) + \lo(1)\big) \cdot \gamma(\eta) RT/\eta}
    \cdot
    \bigg(
        1 - \big(1+\delta_{\S}(\epsilon, RT) + \lo(1)\big) \cdot \gamma(\eta) RT/\eta
    \bigg)^{\frac{t/\gamma(\eta)}{T/\eta} +1}.    
    \label{eq:exit time and location tail of tau upper bound}    
\end{align}
where the last inequality is from \eqref{eq: exit time analysis: upper bound of the probability that tau is before rho} with $B=\mathbb S$.
From \eqref{eq: exit time and location analysis with epsilon lower bound},
\eqref{eq:exit time and location tail of tau upper bound}, and    
\eqref{eq: exit time condition lower bound} of Condition~\ref{condition E2}, we have
\begin{align*}
    &
    \liminf_{\eta\downarrow 0} 
    \inf_{x\in A(\epsilon)}
        \P\big(\gamma(\eta) \tau_{I(\epsilon)^\complement}^\eta(x) > t;\; V^\eta_{\tau_\epsilon}(x) \in B\big) 
    \\
    &
    \geq
    \liminf_{\eta\downarrow 0}
    \frac{C(B^\circ) - \delta_B(\epsilon, T)+ \lo(1)}{\big(1+\delta_{\S}(\epsilon, RT)+ \lo(1)\big)\cdot R}
    \cdot 
        \bigg(
            1 - \big(1+\delta_{\S}(\epsilon, RT) + \lo(1)\big) \cdot \gamma(\eta) RT/\eta
        \bigg)^{\frac{ R\cdot t}{\gamma(\eta)RT/\eta} +1}.    
    \\
    &
    \geq
    \frac{C(B^\circ)-\delta_B(\epsilon,T)}{1+\delta_{\S}(\epsilon,RT)}\cdot\exp\Big(-\big(1+\delta_{\S}(\epsilon,RT)\big) \cdot R\cdot t\Big).
\end{align*}
By taking limit $T\to\infty$ and then considering an $R$ arbitrarily close to 1, it is straightforward to check that the desired lower bound holds.

Moving on to the upper bound, we set $R=1$ and fix an arbitrary $r\in (0,1)$.
Set
\(k = \left\lfloor \frac{t/\gamma(\eta)}{T/\eta} \right\rfloor\)
and note that
\begin{align*}
    \sup_{x \in A(\epsilon)}\P\big(\gamma(\eta) \tau_{I(\epsilon)^\complement}^\eta(x) > t;\; V^\eta_{\tau_\epsilon}(x) \in B\big)
    &
    =\sup_{x \in A(\epsilon)}
        \P\big(\tau_{I(\epsilon)^\complement}^\eta(x) > t/\gamma(\eta);\; V^\eta_{\tau_\epsilon}(x) \in B\big)
    \\
    &
    =
        \underbrace{ \sup_{x \in A(\epsilon)}
            \P\big(\tau_{I(\epsilon)^\complement}^\eta(x) > t/\gamma(\eta)\geq \rho_k^\eta(x);\; V^\eta_{\tau_\epsilon}(x) \in B\big)
        }_{\mathrm{(I)}}
        \\&
        \quad
        +
        \underbrace{ \sup_{x \in A(\epsilon)}
            \P\big(\tau_{I(\epsilon)^\complement}^\eta(x) > t/\gamma(\eta);\; \rho_k^\eta(x) > t/\gamma(\eta);\; V^\eta_{\tau_\epsilon}(x) \in B\big)
        }_{\mathrm{(II)}}
\end{align*}
We first show that (II) vanishes as $\eta \to 0$. 
Our proof hinges on the following claim:
\begin{align*}
    \big\{
        \tau_{I(\epsilon)^\complement}^\eta (x) > t/ \gamma(\eta)
        ;\;
        \rho_k^\eta(x) > t/\gamma(\eta)
    \big\}
    \ \subseteq\ 
    \bigcup_{j=1}^k\big\{\tau_{I(\epsilon)^\complement}^\eta(x) \wedge \rho_j^\eta(x) - \rho_{j-1}^\eta(x) \geq T/\eta\big\}
\end{align*}
Proof of the claim: 
Suppose that \(\tau_{I(\epsilon)^\complement}^\eta (x) > t/ \gamma(\eta)\) and \(\rho_k^\eta(x) > t/\gamma(\eta)\).
Let $k^* \delequal \max\{j\geq 1: \rho_j^\eta(x) \leq t/\gamma(\eta)\}$. 
Note that $k^* < k$.
We consider two cases separately: (i)
\(
\rho_{k^*}^\eta(x)/k^* > (t/\gamma(\eta) - T/\eta)/k^*
\)
and
(ii) 
\(
\rho_{k^*}^\eta(x) \leq t/\gamma(\eta) - T/\eta.
\)
In case of (i), since $\rho_{k^*}^\eta(x)/k^*$ is the average of \(\{\rho_j^\eta(x) - \rho_{j-1}^\eta(x): j=1,\ldots,k^*\}\), there exists $j^*\leq k^*$ such that
\[
    \rho_{j^*}^\eta(x) - \rho_{j^*-1}^\eta(x) > \frac{t/\gamma(\eta) - T/\eta}{k^*} \geq \frac{kT/\eta - T/\eta}{k-1} = T/\eta
\]
Note that since $\rho_{j^*}^\eta(x) \leq \rho_{k^*}^\eta(x) \leq t/\gamma(\eta) \leq \tau_{I(\epsilon)^\complement}^\eta(x)$, this proves the claim for case (i). 
For case (ii), note that
\[
    \rho_{k^*+1}^\eta(x) \wedge \tau_{I(\epsilon)^\complement}^\eta(x) - \rho_{k^*}^\eta(x)
    \geq t / \gamma(\eta)  - (t/\gamma(\eta) - T/\eta)   = T/\eta,
\]
which proves the claim.

Now, with the claim in hand, we have that
\begin{align*}
    \mathrm{\text{(II)}}
    &
    \leq  
    \sum_{j=1}^k \sup_{x \in A(\epsilon)}
        \P\big(\tau_{I(\epsilon)^\complement}^\eta(x) \wedge \rho_j^\eta(x) - \rho_{j-1}^\eta (x) \geq T/\eta\big)
    \\
    &
    =
    \sum_{j=1}^k \sup_{x \in A(\epsilon)}
        \E\Big[
            \P\big(
                \tau_{I(\epsilon)^\complement}^\eta(x) \wedge \rho_j^\eta(x) - \rho_{j-1}^\eta (x) \geq T/\eta
            \big|
                \mathcal F_{\rho_{j-1}^\eta(x)}
            \big)
        \Big]
    \\
    &
    \leq
    \sum_{j=1}^k \sup_{y \in A(\epsilon)}
        \P\big(
            \tau_{I(\epsilon)^\complement}^\eta(y) \wedge \rho_1^\eta(y) \geq T/\eta        
        \big)
    \\
    &
    \leq 
    \frac{t}{\gamma(\eta)T/\eta}
    \cdot 
    \gamma(\eta) T/\eta \cdot \lo(1) 
    = \lo (1)
\end{align*}
for sufficiently large $T$'s, where the last inequality is from the definition of $k$ and \eqref{eq: exit time analysis: almost nothing happens between T over eta and rho one}.
We are now left with bounding (I) from above.
\begin{align*}
    \mathrm{(I)}
    &
    =  \sup_{x \in A(\epsilon)}
    \P
    \big(
        \tau_{I(\epsilon)^\complement}^\eta(x) > t/\gamma(\eta)\geq \rho_K^\eta(x);\; V^\eta_{\tau_\epsilon}(x) \in B
    \big)
    \leq  \sup_{x \in A(\epsilon)}
    \P
    \big(
        \tau_{I(\epsilon)^\complement}^\eta(x) 
        >
        \rho_K^\eta(x);\; V^\eta_{\tau_\epsilon}(x) \in B
    \big)
    \\&
    = 
    \sum_{j=k}^\infty \sup_{x \in A(\epsilon)}
        \P
        \Big(
            \tau_{I(\epsilon)^\complement}^\eta(x) \in \big(\rho_j^\eta(x),\,\rho_{j+1}^\eta(x)\big] ;\; 
            V^\eta_{\tau_\epsilon}(x) \in B
        \Big)
    \\
    &=
    \sum_{j=k}^\infty     \sup_{x \in A(\epsilon)}
    \E\bigg[
        \E\Big[
            \I\big\{V_{\tau_\epsilon}^\eta(x) \in B\big\}
            \cdot
            \I\big\{\tau_{I(\epsilon)^\complement}^\eta(x) \leq \rho_{j+1}^\eta(x)\big\}
        \Big|
            \mathcal F_{\rho_j^\eta(x)}
        \Big]
        \cdot
        \I\big\{\tau_{I(\epsilon)^\complement}^\eta(x) > \rho_j^\eta(x) \big\}
    \bigg]
    \\
    &\leq
    \sum_{j=k}^\infty     \sup_{x \in A(\epsilon)}
    \E\bigg[
        \sup_{y\in A(\epsilon)}
        \P\Big(
            V_{\tau_\epsilon}^\eta(y) \in B
            ;\;
            \tau_{I(\epsilon)^\complement}^\eta(y) \leq \rho_1^\eta(y)
        \Big)
        \cdot
        \I\big\{\tau_{I(\epsilon)^\complement}^\eta(x) > \rho_j^\eta(x) \big\}
    \bigg]    
    \\
    &=
    \sup_{y\in A(\epsilon)}
    \P\Big(
        V_{\tau_\epsilon}^\eta(y) \in B
        ;\;
        \tau_{I(\epsilon)^\complement}^\eta(y) \leq \rho_1^\eta(y)
    \Big)
    \cdot
    \sum_{j=k}^\infty     \sup_{x \in A(\epsilon)}
    \P\Big(\tau_{I(\epsilon)^\complement}^\eta(x) > \rho_j^\eta(x)\Big)
\end{align*}
The first term can be bounded via \eqref{eq: exit time analysis: upper bound of the probability that tau is before rho} with $R=1$:
\begin{align*}
    &
    \sup_{y\in A(\epsilon)}
    \P\Big(
        V_{\tau_\epsilon}^\eta(y) \in B
        ;\;
        \tau_{I(\epsilon)^\complement}^\eta(y) \leq \rho_1^\eta(y)
    \Big)
    \\
    &
    \leq
    \big(C(B^-)+\delta_B(\epsilon, T)+ \lo(1)\big) \cdot \gamma(\eta) T/\eta + \frac{1-r}2\cdot \gamma(\eta)T/\eta \cdot \lo(1)
\end{align*}
whereas the second term is bounded via \eqref{eq: exit time condition lower bound} of Condition~\ref{condition E2} as follows: 
\begin{align}
    &
    \sum_{j=k}^\infty \sup_{x \in A(\epsilon)}
    \P\Big(\tau_{I(\epsilon)^\complement}^\eta(x) > \rho_{j}^{\eta}(x)\Big)
    \nonumber\\
    &
    \leq 
    \sum_{j=0}^\infty
    \bigg(
        \sup_{y \in A(\epsilon)}
            \P\Big(\tau_{I(\epsilon)^\complement}^\eta(y) > \lfloor rT/\eta\rfloor\Big)
    \bigg)^{k+j}
    =
    \sum_{j=0}^\infty
        \bigg(
            1 -\inf_{y \in A(\epsilon)}
                \P\Big(\tau_{I(\epsilon)^\complement}^\eta(y) \leq rT/\eta \Big)
        \bigg)^{k+j}
    \nonumber\\
    &\leq
    \frac1{\inf_{y \in A(\epsilon)}\P\Big(\tau_{I(\epsilon)^\complement}^\eta(y) \leq rT/\eta \Big)}
    \cdot
    \bigg(
        1 - \inf_{y \in A(\epsilon)}\P\Big(\tau_{I(\epsilon)^\complement}^\eta(y) \leq rT/\eta \Big)
    \bigg)^{\frac{t/\gamma(\eta)}{T/\eta} -1}
    \nonumber\\
    &
    = 
    \frac{1}
    {r\cdot\big(1-\delta_B(\epsilon, rT)+\lo(1)\big)\cdot \gamma(\eta)T/\eta} 
    \cdot 
    \Big(
    1- r\cdot\big(1-\delta_B(\epsilon, rT)+\lo(1)\big)\cdot \gamma(\eta)T/\eta
    \Big)^{\frac{t}{\gamma(\eta)T/\eta}-1}
    \label{eq:exit time and location tail of tau lower bound}
    \nonumber
\end{align}
Therefore,
\begin{align*}
    \limsup_{\eta\downarrow 0}
    \sup_{x\in A(\epsilon)}
        \P\big(\gamma(\eta) \tau_{I(\epsilon)^\complement}^\eta(x) > t;\; V^\eta_{\tau_\epsilon}(x) \in B\big)
    &
    \leq
    \frac{C(B^-)+\delta_B(\epsilon,T)}{r\cdot (1-\delta_B(\epsilon,rT))} \cdot \exp\Big(-r \cdot\big(1-\delta_B(\epsilon, rT)\big)\cdot  t\Big).
\end{align*}
Again, taking $T\to\infty$ and considering $r$ arbitrarily close to 1, we can check that the desired upper bound holds. 

\medskip
\noindent
\textbf{Proof of Case $(ii)$.}

We work with $R=1$ and set
\(K = \left\lceil \frac{t/\gamma(\eta)}{T/\eta} \right\rceil\).
Again, for $\eta \in \big(0, (r-1)T\big)$,
we have
$\lfloor r T/\eta \rfloor \geq T/\eta$ and hence $\rho_K^\eta(x) \geq K \lfloor r T/\eta\rfloor \geq t/\gamma(\eta)$.
By the Markov property conditioning on $\mathcal F_{\rho_j^\eta(x)}$,
\begin{align*}
    &
    \sup_{x \in A(\epsilon)}\P\big(\gamma(\eta) \tau_{I(\epsilon)^\complement}^\eta(x) \leq t\big)
    \\ 
    & \leq 
    \sup_{x \in A(\epsilon)}\P\Big( \tau_{I(\epsilon)^\complement}^\eta(x) \leq \rho^\eta_K(x) \Big)
    =
    \sup_{x \in A(\epsilon)}\sum_{j = 1}^K\P
    \Big(\tau_{I(\epsilon)^\complement}^\eta(x) \in \big(\rho^\eta_{j-1}(x), \rho^\eta_j(x)\big] \Big)
    \\ 
    & \leq 
    \sum_{j = 1}^K
    \sup_{y \in A(\epsilon)}
    \bigg(
        1 -
        \P\Big(
        \tau_{I(\epsilon)^\complement}^\eta(y) \leq \rho_1^\eta(y)
    \Big)
    \bigg)^{j - 1}
    \cdot 
    \sup_{y \in A(\epsilon)}\P\Big(
        \tau_{I(\epsilon)^\complement}^\eta(y) \leq \rho_1^\eta(y)
    \Big)
    \\ 
    & \leq 
    K \cdot \sup_{y \in A(\epsilon)}\P\Big(
        \tau_{I(\epsilon)^\complement}^\eta(y) \leq \rho_1^\eta(y)
    \Big)
    \leq 
    K \cdot \big(\delta_{I^\complement}(\epsilon,T) + \lo(1)\big) \cdot \gamma(\eta) T/\eta
    \\ 
    &\qquad
    \text{by \eqref{eq: exit time analysis: upper bound of the probability that tau is before rho} (with $B = I^\complement$) and the running assumption of Case $(ii)$ that $C(\cdot)\equiv 0$}
    \\ 
    & \leq 
    \frac{2t/\gamma(\eta)}{T/\eta} \cdot 
     \big(\delta_{I^\complement}(\epsilon,T) + \lo(1)\big) \cdot \gamma(\eta) T/\eta
     \qquad
     \text{ for all $\eta$ small enough under }K = \ceil{\frac{t/\gamma(\eta)}{T/\eta}}
     \\
     & = 
     2t \cdot   \big(\delta_{I^\complement}(\epsilon,T) + \lo(1)\big).
\end{align*}
Lastly, by Condition~\ref{condition E2} (specifically, $\lim_{\epsilon \downarrow 0}\lim_{T \uparrow \infty}\delta_{I^\complement}(\epsilon,T) = 0$ in Definition~\ref{def: asymptotic atom}),
we verify the upper bounds in Case $(ii)$ and conclude the proof.
\end{proof}

Now, we are ready to prove Theorem~\ref{thm: exit time analysis framework}.
\begin{proof}[Proof of Theorem~\ref{thm: exit time analysis framework}]
\linksinpf{thm: exit time analysis framework} 
We focus on the proof of Case $(i)$ since the proof of Case $(ii)$ is almost identical, with the only key difference being that we apply part $(ii)$ of Proposition~\ref{prop: exit time analysis main proposition} instead of part $(i)$.

We first claim that for any $\epsilon, \epsilon' > 0$, $t\geq 0$, and measurable $B\subseteq \mathbb S$,
\begin{equation}\label{exit time and location analysis framework main theorem proof claim}
\begin{aligned}
C(B^\circ)\cdot e^{-t} - \delta_{t,B}(\epsilon)
&
\leq
\liminf_{\eta\downarrow 0} 
    \inf_{x\in I(\epsilon')}
        \P\Big(\gamma(\eta)\cdot\tau_{I(\epsilon)^\complement}^\eta (x) > t,\, V_{\tau_\epsilon}^\eta(x) \in B \Big)
\\
&
\leq
\limsup_{\eta\downarrow 0} 
    \sup_{x\in I(\epsilon')}
        \P\Big(\gamma(\eta)\cdot\tau_{I(\epsilon)^\complement}^\eta (x) > t,\, V_{\tau_\epsilon}^\eta(x) \in B \Big)
\leq 
C(B^-)\cdot e^{-t} + \delta_{t,B}(\epsilon)
\end{aligned}
\end{equation}
where $\delta_{t,B}(\epsilon)$ is characterized in part $(i)$ of Proposition \ref{prop: exit time analysis main proposition} such that $\delta_{t,B}(\epsilon) \to 0$ as $\epsilon \to 0$.
Now, note that for any measurable $B \subseteq I^\complement$,
\begin{align}
    &
    \P\Big(\gamma(\eta)\cdot \tau^\eta_{I^\complement}(x) > t,\, V_\tau^\eta (x) \in B\Big)
    \nonumber
    \\
    &
    =
    \underbrace{
        \P\Big(\gamma(\eta)\cdot \tau^\eta_{I^\complement}(x) > t,\, V_\tau^\eta (x) \in B,\, V^\eta_{\tau_\epsilon}(x) \in I\Big)
    }_{\mathrm{\text{(I)}}}
    +
    \underbrace{
        \P\Big(\gamma(\eta)\cdot \tau^\eta_{I^\complement}(x) > t,\, V_\tau^\eta (x) \in B,\, V^\eta_{\tau_\epsilon}(x) \notin I\Big)
    }_{\mathrm{\text{(II)}}}
    \nonumber
\end{align}
and since 
\[
    \mathrm{\text{(I)}}
    \leq 
    \P\Big(V^\eta_{\tau_\epsilon}(x) \in I\Big)
    \qquad\text{and}\qquad
    \mathrm{\text{(II)}}
    =\P\Big(\gamma(\eta)\cdot \tau^\eta_\epsilon(x) > t,\, V^\eta_{\tau_\epsilon}(x) \in B\setminus I\Big),
\]
we have that
\begin{align*}
    \liminf_{\eta\downarrow 0}
        \inf_{x\in I(\epsilon')}
            \P\Big(\gamma(\eta)\cdot \tau^\eta_{I^\complement}(x) > t,\, V_\tau^\eta (x) \in B\Big)    
    &
    \geq
    \liminf_{\eta\downarrow 0}
        \inf_{x\in I(\epsilon')}
            \P\Big(\gamma(\eta)\cdot \tau^\eta_\epsilon(x) > t,\, V^\eta_{\tau_\epsilon}(x) \in B\setminus I\Big)
    \\
    &
    \geq
    C\big((B\setminus I)^\circ\big) \cdot e^{-t} - \delta_{t,B\setminus I}(\epsilon)
    \\[2pt]
    &
    =
    C(B^\circ)\cdot e^{-t} - \delta_{t,B\setminus I}(\epsilon)
\end{align*}
due to $B \subseteq I^\complement$, and
\begin{align*}
    &
    \limsup_{\eta\downarrow 0}
        \sup_{x\in I(\epsilon')}
            \P\Big(\gamma(\eta)\cdot \tau^\eta_{I^\complement}(x) > t,\, V_\tau^\eta (x) \in B\Big)
    \\
    &
    \leq
    \limsup_{\eta\downarrow 0}
        \sup_{x\in I(\epsilon')}
            \P\Big(\gamma(\eta)\cdot \tau^\eta_\epsilon(x) > t,\, V^\eta_{\tau_\epsilon}(x) \in B\setminus I\Big) 
    + 
    \limsup_{\eta\downarrow 0}
        \sup_{x\in I(\epsilon')}
            \P\Big(V^\eta_{\tau_\epsilon}(x) \in I\Big)
    \\
    &
    \leq
    C\big((B\setminus I)^-\big)\cdot e^{-t} + \delta_{t,B\setminus I}(\epsilon) + C (I^-) + \delta_{0,I}(\epsilon)
    \\[3pt]
    &
    = 
    C(B^-)\cdot e^{-t} + \delta_{t,B\setminus I}(\epsilon) + \delta_{0,I}(\epsilon).
\end{align*}
Taking $\epsilon\to0$, we arrive at the desired lower and upper bounds of the theorem.
Now we are left with the proof of the claim~\eqref{exit time and location analysis framework main theorem proof claim} is true. 
Note that for any $x \in I$,
\begin{align}
    &
    \P\Big(\gamma(\eta)\cdot \tau^\eta_\epsilon(x) > t,\, V^\eta_{\tau_\epsilon}(x) \in B\Big)
    \nonumber
    \\
    &
    =
    \E
    \bigg[
        \P\Big(\gamma(\eta)\cdot \tau^\eta_\epsilon(x) > t,\, V^\eta_{\tau_\epsilon}(x) \in B\Big| \mathcal F_{\tau_{A(\epsilon)}^\eta(x)}\Big)
        \cdot 
        \Big( \I\big\{\tau_{A(\epsilon)}^\eta (x) \leq T/\eta\big\} + \I\big\{\tau_{A(\epsilon)}^\eta (x) > T/\eta\big\}\Big)
    \bigg]
    \label{eq:exit time and location analysis framework: proof of the main theorem: decomposition in the proof of the claim}
\end{align}
Fix an arbitrary $s>0$, and note that from the Markov property,
\begin{align*}
    &
    \P\Big(\gamma(\eta)\cdot \tau^\eta_\epsilon(x) > t,\, V^\eta_{\tau_\epsilon}(x) \in B\Big)
    \\
    &
    \leq
    \E\bigg[
        \sup_{y \in A(\epsilon)}
            \P
            \Big(
                \tau^\eta_\epsilon(y) > t/\gamma(\eta) - T/\eta,\, V^\eta_{\tau_\epsilon}(y) \in B
            \Big)
            \cdot 
            \I\big\{\tau_{A(\epsilon)}^\eta (x) \leq T/\eta\big\}
    \bigg] 
    + 
    \P\Big(\tau_{A(\epsilon)}^\eta (x) > T/\eta\Big)
    \\
    &
    \leq
    \sup_{y \in A(\epsilon)}
        \P
        \Big(
            \gamma(\eta)\cdot \tau^\eta_\epsilon(y) > t-s,\, V^\eta_{\tau_\epsilon}(y) \in B
        \Big)
    + 
    \P\Big(\tau_{A(\epsilon)}^\eta (x) > T/\eta\Big)
\end{align*}
for sufficiently small $\eta$'s;
here, we applied $\gamma(\eta)/\eta \to 0$ as $\eta \downarrow 0$ in the last inequality.
In light of part $(i)$ of Proposition~\ref{prop: exit time analysis main proposition},
by taking  $T \to \infty$ we yield
\begin{align*}
    \limsup_{\eta\downarrow 0}
        \sup_{x\in I(\epsilon')}
            \P\Big(\gamma(\eta)\cdot \tau^\eta_\epsilon(x) > t,\, V^\eta_{\tau_\epsilon}(x) \in B\Big)
    &
    \leq
    C(B^-)\cdot e^{-(t-s)} + \delta_{t,B}(\epsilon)
\end{align*}
Considering an arbitrarily small $s>0$, we get the upper bound of the claim~\eqref{exit time and location analysis framework main theorem proof claim}. 
For the lower bound, again from \eqref{eq:exit time and location analysis framework: proof of the main theorem: decomposition in the proof of the claim} and the Markov property,
\begin{align*}
    &
    \liminf_{\eta\downarrow 0}
        \inf_{x\in I(\epsilon')}
            \P\Big(\gamma(\eta)\cdot \tau^\eta_\epsilon(x) > t,\, V^\eta_{\tau_\epsilon}(x) \in B\Big)
    \\
    &
    \geq 
    \liminf_{\eta\downarrow 0}
        \inf_{x\in I(\epsilon')}
            \E\bigg[
                \inf_{y \in A(\epsilon)}
                    \P
                    \Big(
                        \tau^\eta_\epsilon(y) > t/\gamma(\eta),\, V^\eta_{\tau_\epsilon}(y) \in B
                    \Big)
                \cdot 
                \I\big\{\tau_{A(\epsilon)}^\eta (x) \leq T/\eta\big\}
            \bigg] 
    \\
    &
    \geq 
    \liminf_{\eta\downarrow 0}       
        \inf_{y \in A(\epsilon)}
            \P
            \Big(
                \gamma(\eta)\cdot \tau^\eta_\epsilon(y) > t,\, V^\eta_{\tau_\epsilon}(y) \in B
            \Big)
        \cdot 
        \inf_{x\in I(\epsilon')}
            \P\big(\tau_{A(\epsilon)}^\eta (x) \leq T/\eta\big)
    \\
    &
    \geq
    C(B^\circ) - \delta_{t,B}(\epsilon),
\end{align*}
which is the desired lower bound of the claim~\eqref{exit time and location analysis framework main theorem proof claim}. 
This concludes the proof. 
\end{proof}

\subsection{Proof of Theorem~\ref{theorem: first exit time, unclipped}}
\label{sec: proof of proposition: first exit time}

In this section, we apply the framework developed in Section~\ref{subsec: framework, first exit time analysis}
and prove Theorem~\ref{theorem: first exit time, unclipped}.
Throughout this section,
we impose Assumptions \ref{assumption gradient noise heavy-tailed}, \ref{assumption: lipschitz continuity of drift and diffusion coefficients}, and \ref{assumption: shape of f, first exit analysis}.
Besides, we fix a few useful constants.
Recall the definition of the discretized width metric $\mathcal J^I_b$ defined in \eqref{def: first exit time, J *}.
To prove Theorem~\ref{theorem: first exit time, unclipped},
in this section
we fix some $b > 0$ such that the conditions in  Theorem~\ref{theorem: first exit time, unclipped} hold.
This allows us to fix some $\widecheck \epsilon > 0$ small enough such that 
\begin{align}
    \bar B_{\widecheck{\epsilon}}(\bm 0) \subseteq I,
    \qquad
    \bm a(\bm x)\bm x < 0\ \forall \bm x \in \bar B_{\widecheck \epsilon}(\bm 0)\setminus\{\bm 0\},
    \quad
    \inf\Big\{ 
        \norm{\bm x - \bm y}:\ 
        \bm x \in I^\complement,\ \bm y \in \mathcal G^{( \mathcal J^I_b - 1 )|b}(\widecheck \epsilon)
    \Big\}
    > 0.
    \label{constant bar epsilon, first exit time analysis}
\end{align}
Here, $\bar B_r(\bm x) = \{ \bm x:\ \norm{\bm x} \leq r  \}$ is the closed ball with radius $r$ centered at $\bm x$.
An direct implication of the first condition in \eqref{constant bar epsilon, first exit time analysis} is the following positive invariance property under the gradient field $\bm a(\cdot)$:
for any $r \in (0,\widecheck \epsilon]$, 
\begin{align}
    \bm y_t(\bm x) \in \bar B_{ r }(\bm 0)\qquad \forall \bm x \in \bar B_{ r }(\bm 0).
    \label{property: contraction of the ODE around the origin}
\end{align}
Next, for any $\epsilon \in (0,\widecheck \epsilon)$, let
\begin{align}
    \notationdef{notation-domain-check-I-epsilon}{\widecheck I(\epsilon)} \delequal 
    \Big\{
        \bm x \in I:\  \norm{ \bm y_{1/\epsilon}(\bm x) } < \widecheck \epsilon
    \Big\}
    \label{def: covering sets I epsilon, first exit time}
\end{align}
with the ODE $\bm y_t(x)$ defined in \eqref{def ODE path y t}.
By Gronwall's inequality, it is easy to see that $\widecheck I(\epsilon)$ is an open set.
Meanwhile, by Assumption \ref{assumption: shape of f, first exit analysis},
given any $\bm x \in I$ we must have $\bm x \in \widecheck I(\epsilon)$ for all $\epsilon > 0$ small enough.
As a result, the collection of open sets $\{ \widecheck I(\epsilon):\ \epsilon \in (0,\widecheck \epsilon) \}$
provides a covering for $I$:
\begin{align*}
    \bigcup_{\epsilon \in (0,\widecheck \epsilon)}\widecheck I(\epsilon) = I.
\end{align*}

Next, recall that we use 
$
{I_\epsilon} = \{ \bm y\in \R^n:\  \norm{\bm x - \bm y} < \epsilon\ \Longrightarrow\ \bm x \in I  \}
$
to denote the $\epsilon$-shrinkage of the set $I$.
Given any $\epsilon > 0$,
note that $I_\epsilon$ is an open set and, by definition, its closure $I_\epsilon^-$ is still bounded away from $I^\complement$,
i.e.,
$
\norm{\bm x-\bm y} \geq \epsilon
$
for all $\bm x\in I^-_\epsilon$, $\bm y \in I^\complement$.
Then from the continuity of $\bm a(\cdot)$ (see Assumption~\ref{assumption: lipschitz continuity of drift and diffusion coefficients}), the boundedness of set $I$ and hence $I^-_\epsilon \subseteq I$,
as well as property \eqref{property: contraction of the ODE around the origin},
we know that
given any $\epsilon > 0$, the claim
\begin{align*}
    \norm{\bm y_T(\bm x)} < \widecheck \epsilon\qquad \forall \bm x \in I^-_\epsilon
\end{align*}
holds for all $T > 0$ large enough.
This confirms that
given $\epsilon > 0$, it holds for all $\epsilon^\prime > 0$ small enough that 
\begin{align}
    I^-_\epsilon \subseteq \widecheck I(\epsilon^\prime).
    \label{property, I epsilon and I _ epsilon }
\end{align}

As a direct consequence of the discussion above,
we highlight another important property of the sets $\mathcal G^{(k)|b}(\epsilon)$ defined in \eqref{def: set G k b epsilon}.
For any $k \in \mathbb N$, $b > 0$, and $\epsilon \geq 0$, let
\begin{align}
    \notationdef{notation-extended-coverage-set-bar-G-k-b-epsilon}{\bar{\mathcal G}^{(k)|b}(\epsilon)}
    \delequal
    \Big\{
        \bm y_t(\bm x):\ \bm x \in \mathcal G^{(k)|b}(\epsilon),\ t \geq 0
    \Big\},
    \label{def: bar G k b epsilon, extended k jump coverage set}
\end{align}
where $\bm y_\cdot(\bm x)$ is the ODE defined in \eqref{def ODE path y t}.
First, due to \eqref{property, I epsilon and I _ epsilon } and the fact that $\mathcal G^{(\mathcal J^I_b - 1)|b}(\widecheck{\epsilon})$ is bounded away from $I^\complement$ (see \eqref{constant bar epsilon, first exit time analysis}),
given any $\epsilon \in (0,\widecheck{\epsilon}]$, it holds for all $\epsilon^\prime > 0$ small enough that 
$
\mathcal G^{(\mathcal J^I_b - 1)|b}(\epsilon) \subseteq \widecheck I(\epsilon^\prime).
$
Furthermore, we claim that $\bar{\mathcal G}^{(\mathcal J^I_b - 1)|b}(\widecheck\epsilon)$ is also bounded away from $I^\complement$, i.e.,
\begin{align}
    \inf\Big\{
        \norm{\bm x - \bm z}:\ \bm x \in \bar{\mathcal G}^{(\mathcal J^I_b - 1)|b}(\widecheck\epsilon),\ \bm z \in I^\complement
    \Big\}
    > 0.
    \label{property, bar G k b epsilon bounded away from I complement}
\end{align}
Again, this can be argued with a proof by contradiction.
Suppose there exist sequences $\bm x_n^\prime \in \bar{\mathcal G}^{(\mathcal J^I_b - 1)|b}(\widecheck\epsilon)$ and $\bm z_n \notin I$
such that $\norm{\bm x_n^\prime - \bm z_n} \leq 1/n$.
By definition of $\bar{\mathcal G}^{(\mathcal J^I_b - 1)|b}(\widecheck\epsilon)$,
there exist sequences $\bm x_n \in {\mathcal G}^{(\mathcal J^I_b - 1)|b}(\widecheck\epsilon)$ and $t_n \geq 0$ such that
$
\bm x^\prime_n = \bm y_{t_n}(\bm x_n)
$
for all $n\geq 1$.
Furthermore, recall that we have
$
\mathcal G^{(\mathcal J^I_b - 1)|b}(\widecheck \epsilon) \subseteq\widecheck I(\epsilon)
$
for $\epsilon > 0$ small enough.
On the other hand,
by the definition of $\widecheck I(\epsilon)$ in \eqref{def: covering sets I epsilon, first exit time}
and the property \eqref{property: contraction of the ODE around the origin},
it holds for all $n \geq 1$ that 
$
\bm y_t(\bm x_n) \in \bar B_{\widecheck{\epsilon}}(\bm 0)\ \forall t \geq 1/\epsilon.
$
Since $\bm z_n \notin I$ and $\bar B_{\widecheck{\epsilon}}(\bm 0) \subseteq I$ (see \eqref{constant bar epsilon, first exit time analysis}),
we must have $t_n < 1/\widecheck\epsilon$ for all $n$.
Together with the boundedness of $I$,
by picking a sub-sequence if necessary, we can w.l.o.g.\ assume that $\bm x_n \to \bm x^*$ for some $\bm x^* \in \big(\mathcal{G}^{(\mathcal J^I_b - 1)|b}\big)^- \subset I$
and $t_n \to t^*$ for some $t^* \in [0,1/\widecheck\epsilon]$.
Since $\bm x^* \in I$,
by Assumption~\ref{assumption: shape of f, first exit analysis}
we must have $\bm y_{t^*}(\bm x^*) \in I$.
By the continuity of the flow (specifically, using Gronwall's inequality) and the fact that $I$ is an open set,
we have $\bm z_n = \bm y_{t_n}(\bm x_n) \in I$ for all $n$ large enough.
This contradicts our choice that $\bm z_n \notin I$ for all $n$,
thus establishing \eqref{property, bar G k b epsilon bounded away from I complement}.
Now, by \eqref{constant bar epsilon, first exit time analysis}, \eqref{property: contraction of the ODE around the origin}, and \eqref{property, bar G k b epsilon bounded away from I complement},
we can fix some $\bar \epsilon > 0$ small enough such that the following claims hold:
\begin{align}
    &\bar B_{\bar{\epsilon}}(\bm 0) \subseteq I_{\bar\epsilon},
    \label{constant bar epsilon, new, 1, first exit time analysis}
    \\ 
    &r \in (0,\bar\epsilon],\ \bm x \in \bar B_r(\bm 0)
    \quad \Longrightarrow \quad 
    \bm y_t(\bm x) \in \bar B_r(\bm 0)\ \forall t \geq 0,
    \label{constant bar epsilon, new, 2, first exit time analysis}
    \\ 
     &\inf\Big\{
        \norm{\bm x - \bm z}:\ \bm x \in \bar{\mathcal G}^{(\mathcal J^I_b - 1)|b}(2\bar\epsilon),\ \bm z \notin I_{\bar\epsilon}
    \Big\}
     > \bar\epsilon.
    \label{constant bar epsilon, new, 3, first exit time analysis}
\end{align}

Moving on, let
\begin{align*}
    \notationdef{notation-hitting-time-t-x-epsilon}{\bm t_{\bm x}(\epsilon)} \delequal \inf\Big\{ t \geq 0:\ \bm y_t(\bm x) \in \bar B_{\epsilon}(\bm 0) \Big\}
\end{align*}
be the hitting time of the closed ball $\bar B_{\epsilon}(\bm 0)$ for the ODE $\bm y_t(\bm x)$,
and let
\begin{align}
    \notationdef{notation-t-epsilon-ode-return-time}{ \bm{t}(\epsilon) } \delequal
    \sup\Big\{ \bm t_{\bm x}(\epsilon):\ \bm x \in I_\epsilon^-   \Big\}
    \label{def: t epsilon function, first exit analysis}
\end{align}
be the upper bound for the hitting times $\bm t_{\bm x}(\epsilon)$ over $\bm x \in I_\epsilon^-$.
Again, from the continuity of $\bm a(\cdot)$, the contraction of $\bm y_t(\bm x)$ around the origin (see Assumption~\ref{assumption: shape of f, first exit analysis} and its implication \eqref{constant bar epsilon, new, 2, first exit time analysis}), 
and the boundedness of $I$ and hence $I_\epsilon^-$,
we have $\bm t(\epsilon) < \infty$ for any $\epsilon > 0$.
Besides, by definition of $\bm{t}(\cdot)$, we have
\begin{align}
    \bm{y}_{ t }(\bm x) \in \bar B_{\epsilon}(\bm 0) \qquad \forall \bm x \in I_\epsilon^-,\ t \geq \bm{t}(\epsilon).
    \label{property: t epsilon function, first exit analysis}
\end{align}
Furthermore, by repeating the arguments for \eqref{property, bar G k b epsilon bounded away from I complement},
one can show that (for all $\epsilon > 0$)
\begin{align}
    \inf\Big\{
        \norm{ \bm y_t(\bm x) - \bm z  }:\ 
        \bm x \in I_\epsilon^-,\ t \geq 0,\ \bm z \notin I
    \Big\} > 0.
    \label{property: coverage of I_epsilon, bounded away from I complement}
\end{align}
Specifically, for the constant $\bar\epsilon > 0$ fixed in \eqref{constant bar epsilon, new, 1, first exit time analysis}--\eqref{constant bar epsilon, new, 3, first exit time analysis},
by \eqref{property: coverage of I_epsilon, bounded away from I complement} we can find some $\bar c \in (0,1)$ such that
\begin{align}
    \Big\{
        \bm y_t(\bm x):\ \bm x \in I^-_{\bar\epsilon},\ t \geq 0
    \Big\}
    \subseteq I_{\bar c\bar\epsilon}.
    \label{constant bar c for bar epsilon, first exit time}
\end{align}

Recall that we use $E^-$ and $E^\circ$ to denote the closure and interior of any Borel set $E$.
In our analysis below, we make use of the following inequality in Lemma~\ref{lemma: limiting measure, with exit location B, first exit analysis}.
We collect its proof in Section~\ref{subsec: lemma for measure check C}, together with the proofs of other useful properties regarding measures $\widecheck{ \mathbf C }^{(k)|b}$.
\begin{lemma}
\label{lemma: limiting measure, with exit location B, first exit analysis}
\linksinthm{lemma: limiting measure, with exit location B, first exit analysis}
Let $\bar t,\bar\delta \in (0,\infty)$ be the constants characterized in part $(b)$ of Lemma~\ref{lemma: choose key parameters, first exit time analysis}.
Given $\Delta \in (0, {\bar\epsilon})$,
there exists $\epsilon_0 = \epsilon_0(\Delta) > 0$ such that for any $\epsilon \in (0,\epsilon_0]$, $T \geq \bar t$,
and Borel measurable $B \subseteq (I_\epsilon)^c$,
\begin{align*}
    (T-\bar t)\cdot \Big( \widecheck{\mathbf{C}}^{ (\mathcal{J}^I_b)|b }(B_{\Delta}) - \widecheck{\bm c}(\epsilon_0)\Big) 
    & \leq 
    \inf_{\bm x:\ \norm{\bm x} \leq \epsilon  }\mathbf{C}^{ (\mathcal{J}^I_b)|b}_{[0,T]}
    \bigg( 
    \Big(\widecheck E(\epsilon,B,T)\Big)^\circ;\ \bm x
    \bigg)
    \\ 
    &
    \leq 
    \sup_{\bm x:\ \norm{\bm x} \leq \epsilon}
    \mathbf{C}^{ (\mathcal{J}^I_b)|b}_{[0,T]}
    \bigg( 
    \Big(\widecheck E(\epsilon,B,T)\Big)^-;\ \bm x
    \bigg)
    \leq 
    T\cdot \Big( \widecheck{\mathbf{C}}^{ (\mathcal{J}^I_b)|b }(B^{\Delta})+ \widecheck{\bm c}(\epsilon_0)\Big)
\end{align*}
where 
\begin{align}
    \notationdef{notation-set-check-E-epsilon-B-T}{\widecheck{E}(\epsilon,B,T)} 
    & \delequal 
    \Big\{ \xi \in \mathbb{D}[0,T]:\ \exists t \leq T\ s.t.\ \xi_t \in B\text{ and }\xi_s \in I(\epsilon)\ \forall s \in [0,t) \Big\},
    \label{def: set check E epsilon B T, measure check C k b}
    \\
    \notationdef{notation-error-function-check-c-epsilon}{\widecheck{\bm c}(\epsilon)} 
    & \delequal 
    \mathcal J^I_b \cdot (\bar t)^{ \mathcal J^I_b - 1 } \cdot (\bar\delta)^{ -\alpha \cdot (\mathcal J^I_b - 1) }
    \cdot 
    \epsilon^{ \frac{\alpha}{2\mathcal J^I_b}  }.
    \label{def: check c error function, first exit time proof}
\end{align}
\end{lemma}

To see how we apply the framework developed in Section \ref{subsec: framework, first exit time analysis},
let us specialize Condition \ref{condition E2}
to a setting
where $\S = \R$, $A(\epsilon)  = \{ \bm x \in \R^m:\ \norm{\bm x} < \epsilon \}$,
and the covering
$I(\epsilon) = I_\epsilon$.
Let $V^\eta_j(x) = \bm X^{\eta|b}_j(\bm x)$.
Meanwhile, for
$
C^I_b = \widecheck{\mathbf C}^{(\mathcal{J}^I_b)|b}\big( I^\complement\big),
$
it is shown in Lemma~\ref{lemma: exit rate strictly positive, first exit analysis} that $C^I_b < \infty$.
Now, recall that $H(\cdot) = \P(\norm{\bm Z_1} > \cdot)$ and $\lambda(\eta) = \eta^{-1}H(\eta^{-1})$.
Recall that in Theorem~\ref{theorem: first exit time, unclipped},
we consider two cases: $(i)$ $C^I_b \in (0,\infty)$, and $(ii)$ $C^I_b = 0$.
We first discuss our choices in Case $(i)$.
When $C^I_b > 0$, we set
\begin{align}
    C(\ \cdot\ ) \delequal \frac{ \widecheck{\mathbf C }^{ (\mathcal J^I_b)|b }(\ \cdot \ \setminus I )  }{C^I_b},
    \qquad 
    \gamma(\eta) \delequal C^I_b \cdot \eta \cdot \big(\lambda(\eta)\big)^{\mathcal J^I_b}.
    \label{def: measure C and scale gamma when applying the exit time framework}
\end{align}
The regularity conditions in Theorem~\ref{theorem: first exit time, unclipped} dictate that
$
\widecheck{\mathbf C}^{(\mathcal J^I_b)|b}(\partial I) = 0,
$
and hence
$C(\partial I) = 0$.
Besides, note that $C(\cdot)$ is a probability measure and
$
\gamma(\eta)T/\eta =  C^I_b  T \cdot \big(\lambda(\eta)\big)^{\mathcal J^I_b}.
$
Besides, this corresponds to Case $(i)$ for the location measure in the definition of asymptotic atoms; see the discussion before Definition~\ref{def: asymptotic atom}.

The application of the framework developed in Section \ref{subsec: framework, first exit time analysis}
(specifically, Theorem~\ref{thm: exit time analysis framework})
hinges on the verification of \eqref{eq: exit time condition lower bound}--\eqref{eq:E4}.
We start by verifying \eqref{eq: exit time condition lower bound} and \eqref{eq: exit time condition upper bound}.
First, given any Borel measurable $B \subseteq \R$, we specify the choice of function $\delta_B(\epsilon,T)$ in Condition \ref{condition E2}.
From the continuity of measures,
we get
$\lim_{\Delta \downarrow 0}\widecheck{\mathbf{C}}^{ (\mathcal{J}^I_b)|b }\Big( (B^{\Delta}\cap I^\complement ) \setminus (B^- \cap I^\complement)\Big) = 0$
and
$\lim_{\Delta \downarrow 0} \widecheck{\mathbf{C}}^{ (\mathcal{J}^I_b)|b }\Big( (B^\circ \cap I^\complement) \setminus (B_{\Delta} \cap I^\complement ) \Big) = 0$.
This allows us to fix a sequence $(\Delta^{(n)})_{n \geq 1}$
such that $\Delta^{(n+1)} \in (0,\Delta^{(n)}/2)$ and
\begin{align}
    \widecheck{\mathbf{C}}^{ (\mathcal{J}^I_b)|b }\Big( (B^{\Delta^{(n)}}\cap I^\complement ) \setminus (B^- \cap I^\complement)\Big) \vee 
   \widecheck{\mathbf{C}}^{ (\mathcal{J}^I_b)|b }\Big( (B^\circ \cap I^\complement) \setminus (B_{\Delta^{(n)}} \cap I^\complement ) \Big)
   \leq 1/2^n
   \label{choice of Delta n, first exit time proof}
    \end{align}
for each $n \geq 1$.
Next,
recall the definition of set $\widecheck E(\epsilon,B,T)$ in Lemma \ref{lemma: limiting measure, with exit location B, first exit analysis},
and
let $\widetilde B(\epsilon) \delequal B \setminus I_\epsilon$.
Using Lemma \ref{lemma: limiting measure, with exit location B, first exit analysis},
we are able to fix another sequence of strictly decreasing positive real numbers $(\epsilon^{(n)})_{n \geq 1}$
such that 
$\epsilon^{(n)} \in (0,\bar\epsilon]\ \forall n \geq 1$
and for any $n \geq 1$, $\epsilon \in (0,\epsilon^{(n)}]$,
we have
\begin{align}
 \sup_{\bm x:\ \norm{\bm x} \leq \epsilon}
 \mathbf{C}^{ (\mathcal{J}^I_b)|b}_{[0,T]}
    \bigg(
    \Big(\widecheck E\big(\epsilon,\widetilde B(\epsilon),T\big)\Big)^-;\ \bm x
    \bigg)
    & \leq 
    T\cdot 
    \bigg( 
        \widecheck{\mathbf{C}}^{ (\mathcal{J}^I_b)|b }\Big( \big(B \setminus I_\epsilon\big)^{\Delta^{(n)}} \Big)
        + \widecheck{\bm c}(\epsilon^{(n)})
    \bigg),
    \label{measure C bound, 2, first exit time analysis}
    \\
    \inf_{\bm x:\ \norm{\bm x} \leq \epsilon}
    \mathbf{C}^{ (\mathcal{J}^I_b)|b}_{[0,T]}
    \bigg(
    \Big(\widecheck E\big(\epsilon,\widetilde B(\epsilon),T\big)\Big)^\circ;\ \bm x
    \bigg)
    & \geq 
    (T - \bar t)\cdot 
    \bigg( 
        \widecheck{\mathbf{C}}^{ (\mathcal{J}^I_b)|b }\Big( \big(B \setminus I_\epsilon\big)_{\Delta^{(n)}} \Big) - \widecheck{\bm c}(\epsilon^{(n)})
    \bigg).
    \label{measure C bound, 1, first exit time analysis}
\end{align}
Besides, note that
given any $\epsilon \in (0,\epsilon^{(1)}]$,
there uniquely exists some $n = n_\epsilon \geq 1$ such that $\epsilon \in (\epsilon^{(n+1)},\epsilon^{(n)}]$.
This allows us to set
\begin{align}
            & \widecheck{\delta}_B(\epsilon,T) \label{def: bm delta B, first exit analysis}
        \\ 
    & =
    T \cdot \widecheck{ \mathbf C}^{ (\mathcal J^I_b)|b }\Big( (B^{ \Delta^{(n)} } \cap I^\complement) \symbol{92} (B^- \cap I^\complement) \Big)
    \vee 
     \widecheck{ \mathbf C}^{ (\mathcal J^I_b)|b }\Big( (B^\circ\cap I^\complement) \symbol{92} (B_{ \Delta^{(n)} } \cap I^\complement) \Big)
     \vee 
     \widecheck{ \mathbf C}^{ (\mathcal J^I_b)|b }\Big( (\partial I)^{ \epsilon + \Delta^{(n)} }\Big)
     \nonumber
     \\
     &\qquad
     +
     T \cdot \widecheck{\bm c}(\epsilon^{(n)})
     + \bar t \cdot  \widecheck{ \mathbf C}^{ (\mathcal J^I_b)|b }\big(B^\circ \setminus I\big),
     \nonumber
\end{align}
where $\widecheck{\bm c}(\cdot)$ is defined in \eqref{def: check c error function, first exit time proof}.
Also, let $ \delta_B(\epsilon,T) \delequal {\widecheck{\delta}_B(\epsilon,T)}/({ C^I_b \cdot T })$.
By \eqref{choice of Delta n, first exit time proof}
and $\widecheck{\mathbf C}^{(\mathcal J^I_b)|b}( B \setminus I) \leq \widecheck{\mathbf C}^{(\mathcal J^I_b)|b}(I^\complement) < \infty$,
we get
\begin{align*}
   \lim_{T \to \infty}\delta_B(\epsilon,T)
   \leq 
   \frac{1}{C^I_b} \cdot \Big[ \widecheck{\bm c}(\epsilon^{(n)}) +  \frac{1}{2^{ n } } \vee \widecheck{ \mathbf C}^{ (\mathcal J^I_b)|b }\Big( (\partial I)^{ \epsilon + \Delta^{(n)} }\Big) \Big],
\end{align*}
where $n$ is the unique positive integer satisfying $\epsilon \in (\epsilon^{(n+1)},\epsilon^{(n)}]$.
Moreover, as $\epsilon \downarrow 0$ we get $n_\epsilon \to \infty$.
Since $\partial I$ is closed, we get $\cap_{r > 0}(\partial I)^r = \partial I$,
which then implies $\lim_{r \downarrow 0}\widecheck{ \mathbf C}^{ (\mathcal J^I_b)|b }\Big( (\partial I)^{ r }\Big)
=
\widecheck{ \mathbf C}^{ (\mathcal J^I_b)|b }( \partial I) = 0
$
due to continuity of measures.
Also, by definition of $\widecheck{\bm c}$ in \eqref{def: check c error function, first exit time proof},
we have $\lim_{\epsilon \downarrow 0}\widecheck{\bm c}(\epsilon) = 0$.
In summary, we have verified that $\lim_{\epsilon \downarrow 0}\lim_{T \to \infty}\delta_B(\epsilon,T) = 0$.

Next, in case that $C^I_b = 0$, we set 
\begin{align*}
    C(\cdot) \equiv 0,\qquad
    \gamma(\eta) \delequal \eta \big(\lambda(\eta)\big)^{\mathcal J^I_b},
    \qquad
    \delta_B(\epsilon,T) \delequal \widecheck{\delta}_{B}(\epsilon,T)/T.
\end{align*}
The  calculations above again verify that $\lim_{\epsilon \downarrow 0}\lim_{T \to \infty}\delta_B(\epsilon,T) = 0$.

Now, we are ready to verify conditions \eqref{eq: exit time condition lower bound} and \eqref{eq: exit time condition upper bound}.
Specifically,
we
introduce stopping time
\begin{align}
    \notationdef{notation-tau-eta-b-epsilon-exit-time}{\tau^{\eta|b}_\epsilon(\bm x)} & \delequal \min\big\{ j \geq 0:\ \bm X^{\eta|b}_j(\bm x) \notin I_\epsilon \big\}.
    \label{def: epsilon relaxed first exit time}
\end{align}

\begin{lemma}[Verifying conditions \eqref{eq: exit time condition lower bound} and \eqref{eq: exit time condition upper bound}]
\label{lemma: exit prob one cycle, with exit location B, first exit analysis}
\linksinthm{lemma: exit prob one cycle, with exit location B, first exit analysis}
Let $\bar t$ be characterized as in Lemma~\ref{lemma: limiting measure, with exit location B, first exit analysis}.
Given any measurable $B \subseteq \R$,
any $\epsilon \in (0,\bar\epsilon]$ small enough,
and any $T > \bar t$,
\begin{align*}
 C(B^\circ) - \delta_B(\epsilon,T)
& \leq 
    \liminf_{\eta \downarrow 0}\inf_{\bm x:\ \norm{\bm x} \leq \epsilon }
    \frac{ 
        \P\Big(
            \tau^{\eta|b}_\epsilon(\bm x) \leq T/\eta;\ \bm X^{\eta|b}_{\tau^{\eta|b}_\epsilon(\bm x)}(\bm x) \in B
        \Big) }{\gamma(\eta) T/\eta }
    \\
    & \leq 
     \limsup_{\eta \downarrow 0}\sup_{\bm x:\ \norm{\bm x} \leq \epsilon}
    \frac{ \P\Big(
        \tau^{\eta|b}_\epsilon(\bm x) \leq T/\eta;\ \bm X^{\eta|b}_{\tau^{\eta|b}_\epsilon(\bm x)}(\bm x) \in B
    \Big) }{\gamma(\eta)T/\eta }
    \leq 
   C(B^-)
    +
    \delta_B(\epsilon,T).
\end{align*}
\end{lemma}

\begin{proof}
\linksinpf{lemma: exit prob one cycle, with exit location B, first exit analysis}
Recall that 
\begin{enumerate}[$(i)$]
    \item
        in case that $C^I_b \in (0,\infty)$, we have
        $
        \gamma(\eta)T/\eta =  C^I_b  T \cdot \big(\lambda(\eta)\big)^{\mathcal J^I_b},
        $
        $
        C(\cdot) = \widecheck{\mathbf C}^{ (\mathcal J^I_b)|b }(\ \cdot\ \setminus I)/C^I_b,
        $
        and
        $\delta_B(\epsilon,T) = {\widecheck{\delta}_B(\epsilon,T)}/({ C^I_b \cdot T })$;

    \item 
        in case that $C^I_b = 0$, we have
        $
        \gamma(\eta)T/\eta =  T \cdot \big(\lambda(\eta)\big)^{\mathcal J^I_b},
        $
        $
        C(\cdot) \equiv 0,
        $
        and
        $\delta_B(\epsilon,T) = {\widecheck{\delta}_B(\epsilon,T)}/T$.
\end{enumerate}
In both cases,
by rearranging the terms, it suffices to show that
\begin{align}
 \limsup_{\eta \downarrow 0}\sup_{\bm x:\ \norm{\bm x} \leq \epsilon}
    \frac{ \P\Big(\tau^{\eta|b}_\epsilon(\bm x) \leq T/\eta;\ \bm X^{\eta|b}_{\tau^{\eta|b}_\epsilon(\bm x)}(\bm x) \in B\Big) }{\big(\lambda(\eta)\big)^{\mathcal J^I_b} }
    & \leq 
   T \cdot \widecheck{ \mathbf C }^{ (\mathcal J^I_b)|b }(B^- \setminus I)
    +
   \widecheck{\delta}_B(\epsilon,T),
   \label{proof, goal upper bound, lemma: exit prob one cycle, with exit location B, first exit analysis}
   \\
    \liminf_{\eta \downarrow 0}\inf_{\bm x:\ \norm{\bm x} \leq \epsilon}
    \frac{ \P\Big(\tau^{\eta|b}_\epsilon(\bm x) \leq T/\eta;\ \bm X^{\eta|b}_{\tau^{\eta|b}_\epsilon(\bm x)}(\bm x) \in B\Big) }{\big(\lambda(\eta)\big)^{\mathcal J^I_b} }
    & \geq 
   T \cdot \widecheck{ \mathbf C }^{ (\mathcal J^I_b)|b }(B^\circ \setminus I ) - \widecheck{\delta}_B(\epsilon,T).
   \label{proof, goal lower bound, lemma: exit prob one cycle, with exit location B, first exit analysis}
\end{align}

Recall the definition of set $\widecheck E(\epsilon,\cdot,T)$ in \eqref{def: set check E epsilon B T, measure check C k b}.
Let $\widetilde B(\epsilon) \delequal B \setminus I_\epsilon$.
Note that
$$
\Big\{
    \tau^{\eta|b}_\epsilon(\bm x) \leq T/\eta;\ \bm X^{\eta|b}_{\tau^{\eta|b}_\epsilon(\bm x)}(\bm x) \in B
\Big\}
=
\Big\{
    \tau^{\eta|b}_\epsilon(\bm x) \leq T/\eta;\ \bm X^{\eta|b}_{\tau^{\eta|b}_\epsilon(\bm x)}(\bm x) \in \widetilde B(\epsilon)
\Big\}
=
 \Big\{ 
    \bm{X}^{\eta|b}_{[0,T]}(\bm x) \in \widecheck E\big(\epsilon, \widetilde B(\epsilon) ,T\big)
\Big\}.
$$
For any $\epsilon \in (0,\bar \epsilon)$ and $\xi \in \widecheck E(\epsilon,\widetilde B(\epsilon),T)$, 
there exists $t \in [0,T]$ such that $\xi_t \notin I(\epsilon)$.
On the other hand,
recall that we use $\bar B_\epsilon(\bm 0)$ to denote the closed ball with radius $\epsilon$ centered at the origin.
By part $(a)$ of Lemma~\ref{lemma: choose key parameters, first exit time analysis},
given $\epsilon \in (0,\bar\epsilon]$,
it holds for all $\xi \in \mathbb{D}^{ (\mathcal{J}^I_b - 1)|b }_{ \bar B_\epsilon(\bm 0) }[0,T](\epsilon)$
that $\xi_t \in I_{2\bar\epsilon}^-\ \forall t \in [0,T]$.
Therefore, the claim
\begin{align}
    \dj{[0,T]}
    \bigg( \widecheck E\big(\epsilon,\widetilde B(\epsilon),T\big),\ 
    \mathbb{D}^{ (\mathcal{J}^I_b - 1)|b }_{ \bar B_\epsilon(\bm 0) }[0,T](\epsilon)
    \bigg) \geq \bar\epsilon
    \nonumber
\end{align}
for all $\epsilon \in (0,\bar\epsilon]$.
Next, recall the 
strictly decreasing positive real number sequence $(\epsilon^{(n)})_{n \geq 1}$
specified in \eqref{measure C bound, 2, first exit time analysis}--\eqref{measure C bound, 1, first exit time analysis}.
For all $\epsilon > 0$ small enough we have $\epsilon \in (0,\epsilon^{(1)}]$,
so for such $\epsilon$ we can set $n = n_\epsilon$ as the unique positive integer such that $\epsilon \in (\epsilon^{(n+1)},\epsilon^{(n)}]$.
It then follows from Theorem~\ref{corollary: LDP 2} that
\begin{align}
     \limsup_{\eta \downarrow 0}\sup_{ \bm x:\ \norm{\bm x} \leq \epsilon }
         \frac{ 
    \P\Big(
        \tau^{\eta|b}_\epsilon(\bm x) \leq T/\eta;\ \bm X^{\eta|b}_{\tau^{\eta|b}_\epsilon(\bm x)}(\bm x) \in B
    \Big)   
    }{ \big(\lambda(\eta)\big)^{ \mathcal{J}^I_b }} 
    &  \leq 
    \sup_{ \bm x:\ \norm{\bm x} \leq \epsilon }
    \mathbf{C}^{ (\mathcal{J}^I_b)|b}_{[0,T]}\bigg( \Big(\widecheck E\big(\epsilon, \widetilde B(\epsilon) ,T\big)\Big)^-;\bm x\bigg)
    \nonumber
    \\ 
    & \leq
    T\cdot
    \bigg( 
        \widecheck{\mathbf{C}}^{ (\mathcal{J}^I_b)|b }\Big( (B \setminus I_\epsilon)^{\Delta^{(n)}} \Big) + \widecheck{\bm c}(\epsilon^{(n)})
    \bigg),
    \label{proof, ineq for upper bound, lemma: exit prob one cycle, with exit location B, first exit analysis}
\end{align}
where we applied property \eqref{measure C bound, 2, first exit time analysis} in the last inequality.
Furthermore,
\begin{align*}
    \widecheck{\mathbf{C}}^{ (\mathcal{J}^I_b)|b }\Big( (B \setminus I_\epsilon)^{\Delta^{(n)}} \Big)
    & \leq 
    \widecheck{\mathbf{C}}^{ (\mathcal{J}^I_b)|b }\Big( B^{\Delta^{(n)}} \cup (I^\complement_\epsilon)^{\Delta^{(n)}} \Big)
    \qquad 
    \text{due to }(E\cup F)^\Delta \subseteq E^\Delta \cup F^\Delta
    \\ 
    & = 
    \widecheck{\mathbf{C}}^{ (\mathcal{J}^I_b)|b }\Big( B^{\Delta^{(n)}} \cup (I^\complement_\epsilon)^{\Delta^{(n)}} \cap I^\complement \Big)
    +
    \widecheck{\mathbf{C}}^{ (\mathcal{J}^I_b)|b }\Big( B^{\Delta^{(n)}} \cup (I^\complement_\epsilon)^{\Delta^{(n)}} \cap I \Big)
    \\
    & \leq 
    \widecheck{\mathbf{C}}^{ (\mathcal{J}^I_b)|b }\Big( B^{\Delta^{(n)}}\setminus I\Big)
    +
    \widecheck{\mathbf{C}}^{ (\mathcal{J}^I_b)|b }\Big( (I^\complement_\epsilon)^{\Delta^{(n)}} \cap I \Big)
    \\ 
    & \leq 
    \widecheck{\mathbf{C}}^{ (\mathcal{J}^I_b)|b }\Big( B^{\Delta^{(n)}}\setminus I\Big)
    +
    \widecheck{\mathbf{C}}^{ (\mathcal{J}^I_b)|b }\Big( (\partial I)^{ \epsilon + \Delta^{(n)}} \Big)
    \\
    & \leq 
    \widecheck{\mathbf{C}}^{ (\mathcal{J}^I_b)|b }\Big( B^-\setminus I\Big)
    +
    \widecheck{\mathbf{C}}^{ (\mathcal{J}^I_b)|b }\Big( (B^{\Delta^{(n)}}\cap I^\complement ) \setminus (B^- \cap I^\complement)\Big)
    +
    \widecheck{\mathbf{C}}^{ (\mathcal{J}^I_b)|b }\Big( (\partial I)^{ \epsilon + \Delta^{(n)}} \Big)
\end{align*}
By definition of $\widecheck{\delta}_B$ in \eqref{def: bm delta B, first exit analysis}
and the choice of $C(\cdot)$ in \eqref{def: measure C and scale gamma when applying the exit time framework},
we can 
plug this bound back into \eqref{proof, ineq for upper bound, lemma: exit prob one cycle, with exit location B, first exit analysis}
and yield the upper bound \eqref{proof, goal upper bound, lemma: exit prob one cycle, with exit location B, first exit analysis}.
Similarly, by Theorem~\ref{corollary: LDP 2} and the property \eqref{measure C bound, 1, first exit time analysis}, we obtain (for all $\epsilon$ small enough)
\begin{align}
    \liminf_{\eta \downarrow 0}\inf_{ \bm x:\ \norm{\bm x} \leq \epsilon }
         \frac{ 
    \P\Big(
        \tau^{\eta|b}_\epsilon(\bm x) \leq T/\eta;\ \bm X^{\eta|b}_{\tau^{\eta|b}_\epsilon(\bm x)}(\bm x) \in B
    \Big)   
    }{ \big(\lambda(\eta)\big)^{ \mathcal{J}^I_b }} 
    &  \geq 
    \inf_{ \bm x:\ \norm{\bm x} \leq \epsilon }
    \mathbf{C}^{ (\mathcal{J}^I_b)|b}_{[0,T]}\bigg( \Big(\widecheck E\big(\epsilon, \widetilde B(\epsilon) ,T\big)\Big)^\circ;\bm x\bigg)
    \nonumber
    \\ 
    & \geq
    (T - \bar t)\cdot 
    \bigg( 
        \widecheck{\mathbf{C}}^{ (\mathcal{J}^I_b)|b }\Big( (B \setminus I_\epsilon)_{\Delta^{(n)}} \Big) - \widecheck{\bm c}(\epsilon^{(n)})
    \bigg).
    \label{proof, ineq for lower bound, lemma: exit prob one cycle, with exit location B, first exit analysis}
\end{align}
Furthermore, from the preliminary bound $(E\cap F)_\Delta \supseteq E_\Delta \cap F_\Delta $
we get
\begin{align*}
  \widecheck{\mathbf{C}}^{ (\mathcal{J}^I_b)|b }\Big( (B \setminus I_\epsilon)_{\Delta^{(n)}} \Big)
  & \geq 
  \widecheck{\mathbf{C}}^{ (\mathcal{J}^I_b)|b }\Big( (B \setminus I)_{\Delta^{(n)}} \Big)
  \geq 
  \widecheck{\mathbf{C}}^{ (\mathcal{J}^I_b)|b }\Big( B_{\Delta^{(n)}} \cap I^\complement_{\Delta^{(n)}} \Big).
\end{align*}
Together with the fact that
$
B_\Delta \setminus I 
= B_\Delta \cap I^\complement
\subseteq
\Big(B_\Delta \cap (I^\complement)_\Delta\Big) 
\cup 
\Big( I^\complement \setminus (I^\complement)_\Delta \Big),
$
we yield
\begin{align*}
    \widecheck{\mathbf{C}}^{ (\mathcal{J}^I_b)|b }\Big( (B \setminus I_\epsilon)_{\Delta^{(n)}} \Big)
    & \geq 
    \widecheck{\mathbf{C}}^{ (\mathcal{J}^I_b)|b }\Big( B_{\Delta^{(n)}} \setminus I \Big)
    -
    \widecheck{\mathbf{C}}^{ (\mathcal{J}^I_b)|b }\Big( I^\complement \setminus I^\complement_{\Delta^{(n)}} \Big)
    \\ 
    & \geq 
    \widecheck{\mathbf{C}}^{ (\mathcal{J}^I_b)|b }\Big( B_{\Delta^{(n)}} \setminus I \Big)
    -
    \widecheck{\mathbf{C}}^{ (\mathcal{J}^I_b)|b }\Big( (\partial I)^{\Delta^{(n)}} \Big)
    \\
    & \geq 
    \widecheck{\mathbf{C}}^{ (\mathcal{J}^I_b)|b }\Big( B^\circ \setminus I \Big)
    -
    \widecheck{\mathbf{C}}^{ (\mathcal{J}^I_b)|b }\Big( (B^\circ \cap I^\complement) \setminus (B_{\Delta^{(n)}} \cap I^\complement ) \Big)
    -
    \widecheck{\mathbf{C}}^{ (\mathcal{J}^I_b)|b }\Big( (\partial I)^{\Delta^{(n)}} \Big).
\end{align*}
Plugging this bound back into \eqref{proof, ineq for lower bound, lemma: exit prob one cycle, with exit location B, first exit analysis},
we establish the lower bound \eqref{proof, goal lower bound, lemma: exit prob one cycle, with exit location B, first exit analysis} and conclude the proof.
\end{proof}

The next two results verify conditions \eqref{eq:E3} and \eqref{eq:E4}.
Let
\begin{align}
     \notationdef{notation-R-eta-b-epsilon-return-time}{R^{\eta|b}_\epsilon(\bm x)} & \delequal 
     \min\bigg\{ j \geq 0:\ \norm{\bm X^{\eta|b}_j(\bm x)} < \epsilon \bigg\}
    \label{def: return time R in first cycle, first exit analysis}
\end{align}
be the first time $\bm X^{\eta|b}_j(\bm x)$ returns to the $\epsilon$-neighborhood of the origin.
Under our choice of $A(\epsilon) = \{\bm x \in \R^m:\ \norm{\bm x} < \epsilon\}$
and $I(\epsilon) = I_\epsilon$,
the event $\{\tau^{\eta}_{(I(\epsilon)\setminus A(\epsilon))^\complement}(x) > T/\eta\}$ in condition \eqref{eq:E3}
means that $\bm X^{\eta|b}_j(\bm x) \in I_\epsilon \symbol{92} \{ \bm x:\norm{\bm x} < \epsilon  \}$ for all $j \leq T/\eta$.
Also, recall
the definition of $\bm t(\cdot)$ in \eqref{def: t epsilon function, first exit analysis}
and that
$
\gamma(\eta)T/\eta =  C^I_b  T \cdot \big(\lambda(\eta)\big)^{\mathcal J^I_b}.
$
Therefore, to verify condition \eqref{eq:E3}, it suffices to prove the following result.

\begin{lemma}[Verifying condition \eqref{eq:E3}]
\label{lemma: fixed cycle exit or return taking too long, first exit analysis}
\linksinthm{lemma: fixed cycle exit or return taking too long, first exit analysis}
 Given $k \geq 1$ and $\epsilon \in (0,\bar\epsilon)$,
 it holds for all $T \geq k \cdot \bm{t}(\epsilon/2)$ that
 \begin{align*}
     \lim_{\eta \downarrow 0}\ \sup_{ \bm x \in I_\epsilon }
     \ \, \frac1{\lambda^{k-1}(\eta)}\P\Big( \bm X^{\eta|b}_j(\bm x) \in I_\epsilon\setminus \{ \bm x:\norm{\bm x} < \epsilon  \}\quad \forall j \leq T/\eta \Big)
      = 0.
 \end{align*}
\end{lemma}

\begin{proof}
\linksinpf{lemma: fixed cycle exit or return taking too long, first exit analysis}
In this proof, we write $\xi(t) = \xi_t$ for any $\xi \in \D[0,T]$,
and set $B_\epsilon(\bm 0) = \{\bm x \in \R^m:\ \norm{\bm x} < \epsilon\}$.
Note that
$\big\{
\bm X^{\eta|b}_j(\bm x) \in I_\epsilon \setminus B_\epsilon(\bm 0)\ \ \forall j \leq T/\eta
\big\}
= \big\{ \bm{X}^{\eta|b}_{[0,T]}(x) \in E(\epsilon)\big\}
$
where
\begin{align*}
    E(\epsilon) \delequal 
    \Big\{ \xi \in \mathbb{D}[0,T]:\ \xi(t) \in I_\epsilon\setminus B_\epsilon(\bm 0) \ \ \forall t \in[0,T]\Big\}.
\end{align*}
Recall the definition of $\mathbb{D}^{(k)|b}_{A}[0,T](\epsilon)$ in \eqref{def: l * tilde jump number for function g, clipped SGD}.
We claim that $E(\epsilon)$ is bounded away from $\mathbb{D}^{ (k - 1)|b }_{ I_\epsilon^- }[0,T](\epsilon)$.
This allows us to apply Theorem \ref{corollary: LDP 2} and conclude that
\begin{align*}
    \sup_{ x \in I_\epsilon }
    \P\Big( \bm{X}^{\eta|b}_{[0,T]}(\bm x) \in E(\epsilon)\Big)
    =
    \bm{O}\big( \lambda^{k}(\eta) \big)
    =
    \bm{o}\big(\lambda^{k - 1}(\eta)\big)\ \ \ \text{as }\eta \downarrow 0.
\end{align*}
Now, it only remains to verify that $E(\epsilon)$ is bounded away from $\mathbb{D}^{ (k - 1)|b }_{ I_\epsilon^- }[0,T](\epsilon)$,
which can be established if we show 
the existence of some $\delta > 0$ such that
\begin{align}
    \dj{[0,T]}(\xi,\xi^\prime) \geq \delta > 0
    \qquad
    \forall \xi \in \mathbb{D}^{ (k - 1)|b }_{ I_\epsilon^- }[0,T](\epsilon),\ \xi^\prime \in E(\epsilon).
    \label{proof, ineq, lemma: fixed cycle exit or return taking too long, first exit analysis}
\end{align}
First, 
by definition of $E(\epsilon)$, we have $\xi^\prime_t \in I_\epsilon\ \forall t \in [0,T]$
for any $\xi^\prime \in E(\epsilon)$.
Note that
$
\inf\{ \norm{\bm x - \bm y}:\ \bm x \in I_\epsilon, \bm y \notin I_{\epsilon/2} \} \geq \epsilon/2.
$
Therefore, 
if $\xi_t \notin I_{\epsilon/2}$
for some $t \leq T$,
we must have
$
\dj{[0,T]}(\xi,\xi^\prime) \geq \epsilon/2 > 0.
$
Now suppose that $\xi_t \in I_{\epsilon/2}$
for all $t \leq T$.
Due to 
$
\xi \in \mathbb{D}^{ (k - 1)|b }_{  I_\epsilon^- }[0,T](\epsilon),
$
there is some $\bm x \in I_\epsilon^-$, $\textbf W \in \R^{ d \times (k - 1)}$, 
$\textbf V \in \big(\bar{B}_{\epsilon}(\bm 0)\big)^{k-1}$,
and $(t_1,\cdots,t_{k - 1}) \in (0,T]^{ k - 1 \uparrow}$ such that
$
\xi = \bar h^{ (k - 1)|b }_{[0,T]}\big(\bm x,\textbf W,\textbf V,(t_1,\cdots,t_{k - 1})\big).
$
With the convention that $t_0 = 0$ and $t_{k} = T$,
we have 
\begin{align}
    \xi(t) = \bm{y}_{ t - t_{j-1} }\big(\xi(t_{j-1})\big)\qquad 
    \forall t \in [t_{j-1},t_j).
    \label{proof, property of path xi, lemma: fixed cycle exit or return taking too long, first exit analysis}
\end{align}
for each $j \in [k]$.
Here, $\bm{y}_\cdot(x)$ is the ODE defined in \eqref{def ODE path y t}.
Due to the assumption $T \geq k \cdot \bm{t}(\epsilon/2)$, 
there must be some $j \in [k]$ such that $t_j - t_{j - 1} \geq \bm{t}({\epsilon}/{2})$.
However, due to the running assumption that $\xi(t) \in I_{\epsilon/2}\ \forall t \in [0,T]$, we have $\xi(t_{j-1}) \in I_{\epsilon/2}$.
Combining \eqref{proof, property of path xi, lemma: fixed cycle exit or return taking too long, first exit analysis} along with property \eqref{property: t epsilon function, first exit analysis}, we get
$
\lim_{t \uparrow t_j}\xi(t) \in \bar B_{\epsilon/2}(\bm 0) \subset B_\epsilon(\bm 0).
$
On the other hand, by definition of $E(\epsilon)$, we have $\xi^\prime(t) \notin B_\epsilon(\bm 0)$
for all $t \in [0,T]$, 
which implies 
$
\dj{[0,T]}(\xi,\xi^\prime) \geq \frac{\epsilon}{2}.
$
This concludes the proof.
\end{proof}

Lastly, we establish condition \eqref{eq:E4}.
Note that the first visit time $\tau^{\eta}_{A(\epsilon)}(x)$ therein coincides with $R^{\eta|b}_\epsilon(x)$ defined in \eqref{def: return time R in first cycle, first exit analysis}
under our choice of $A(\epsilon) = \{ \bm x \in \R^m:\ \norm{\bm x} < \epsilon \}$.

\begin{lemma}[Verifying condition \eqref{eq:E4}]
\label{lemma: cycle, efficient return}
\linksinthm{lemma: cycle, efficient return}
    Let $\bm{t}(\cdot)$ be defined as in \eqref{def: t epsilon function, first exit analysis} and
    \begin{align*}
        E(\eta,\epsilon,\bm x)
        \delequal 
        \Bigg\{ R^{\eta|b}_\epsilon(\bm x) \leq \frac{\bm t(\epsilon/2)}{\eta};\ \bm X^{\eta|b}_j(\bm x) \in I\ \forall j \leq R^{\eta|b}_\epsilon(\bm x) \Bigg\}.
    \end{align*}
    It holds for all $\epsilon \in (0,\bar\epsilon)$ that
    $
     \lim_{\eta\downarrow 0}\sup_{\bm x \in I_\epsilon^-}
        \P\Big( \big(E(\eta,\epsilon,\bm x)\big)^\complement\Big) = 0.
    $
\end{lemma}

\begin{proof}
\linksinpf{lemma: cycle, efficient return}
In this proof, we write $B_\epsilon(\bm 0) = \{\bm x \in \R^m:\ \norm{\bm x} < \epsilon\}$ and $I(\epsilon) = I_\epsilon$.
    Note that
    $
    \big(E(\eta,\epsilon,\bm x)\big)^\complement
    \subseteq 
    \big\{ \bm{X}^{\eta|b}_{ [0,\bm{t}(\epsilon/2)] }(\bm x) \in E^*_1(\epsilon) \cup E^*_2(\epsilon) \big\},
    $
    where 
\begin{align*}
    E^*_1(\epsilon) & \delequal \bigg\{ \xi \in \mathbb{D}[0,\bm{t}(\epsilon/2)]:\ \xi(t) \notin B_\epsilon(\bm 0)\ \forall t \in [0,\bm{t}(\epsilon/2)] \bigg\},
    \\
    E^*_2(\epsilon) & \delequal \bigg\{ \xi \in \mathbb{D}[0,\bm{t}(\epsilon/2)]:\ \exists 0 \leq s \leq t \leq \bm{t}(\epsilon/2)\ s.t.\ \xi(t) \in B_\epsilon(\bm 0),\ \xi(s) \notin I \bigg\}.
\end{align*}
Recall the definition of $\mathbb{D}^{(k)|b}_{A}[0,T](\epsilon)$ in \eqref{def: l * tilde jump number for function g, clipped SGD}.
We claim that both $E^*_1(\epsilon)$ and $E^*_2(\epsilon)$ are bounded away from 
\begin{align*}
    \mathbb{D}^{(0)|b}_{ (I(\epsilon))^- }[0,\bm{t}(\epsilon/2)] = 
    \bigg\{ 
        \big\{ \bm{y}_t(\bm x):\ t \in [0,\bm{t}(\epsilon/2)] \big\}:\ \bm x \in  \big(I(\epsilon)\big)^-
    \bigg\}.
\end{align*}
To see why, note that
$
\inf\{ \norm{\bm x - \bm y}:\ \bm x \in I(\epsilon), \bm y \notin I(\epsilon/2) \} \geq \epsilon/2.
$
Meanwhile, properties \eqref{constant bar epsilon, new, 2, first exit time analysis} and \eqref{property: t epsilon function, first exit analysis}
imply that
 $\bm{y}_{ \bm{t}(\epsilon/2) }(\bm x) \in \bar B_{\epsilon/2}(\bm 0)$ for all $\bm x \in \big(I(\epsilon)\big)^-$.
Therefore,
\begin{align}
    \dj{[0,\bm{t}(\epsilon/2)]}
    \bigg( 
        \mathbb{D}^{(0)|b}_{ (I(\epsilon))^- }[0,\bm{t}(\epsilon/2)],\ E^*_1(\epsilon)
    \bigg) & \geq \frac{\epsilon}{2} > 0,
    \label{proof, bound 1, lemma: cycle, efficient return}
\end{align}
Meanwhile, by property \eqref{property: coverage of I_epsilon, bounded away from I complement},
we immediately get
\begin{align}
    \dj{[0,\bm{t}(\epsilon/2)]}
    \bigg( 
        \mathbb{D}^{(0)|b}_{ (I(\epsilon))^- }[0,\bm{t}(\epsilon/2)],\ E^*_2(\epsilon)
    \bigg) & \geq \delta > 0.
     \label{proof, bound 2, lemma: cycle, efficient return}
\end{align}
This allows us to apply Theorem \ref{corollary: LDP 2} and obtain 
$$
\sup_{\bm x \in (I(\epsilon))^-}
\P\bigg( \big(E(\eta,\epsilon,\bm x)\big)^\complement\bigg) \leq 
\sup_{\bm x \in (I(\epsilon))^-}
\P\bigg(\bm{X}^{\eta|b}_{ [0,\bm{t}(\epsilon/2)] }(\bm x) \in E^*_1(\epsilon) \cup E^*_2(\epsilon)\bigg)
= \bm{O}\big(\lambda(\eta)\big)
$$
as $\eta \downarrow 0$.
To conclude the proof, one only needs to note that $\lambda(\eta) \in \RV_{\alpha - 1}(\eta)$ (with $\alpha > 1$) and hence $\lim_{\eta \downarrow 0}\lambda(\eta) = 0$.
\end{proof}

We conclude this section with the proof of Theorem \ref{theorem: first exit time, unclipped}.

\begin{proof}[Proof of Theorem \ref{theorem: first exit time, unclipped}]
\linksinpf{theorem: first exit time, unclipped}
First, it is established in Lemma~\ref{lemma: exit rate strictly positive, first exit analysis} that $C^I_b < \infty$.
Next, since Lemmas~\ref{lemma: exit prob one cycle, with exit location B, first exit analysis}--\ref{lemma: cycle, efficient return}
verify Condition~\ref{condition E2},
Theorem~\ref{theorem: first exit time, unclipped}
follows immediately from Theorem~\ref{thm: exit time analysis framework}.
\end{proof}

\section*{Acknowledgements}
The authors gratefully acknowledge the support of the National Science Foundation (NSF) under CMMI-2146530.

\bibliographystyle{abbrv} 
\bibliography{bib_appendix} 

\newpage
\appendix

\section{Results under General Scaling}
\label{sec: appendix, SGD, general scaling, results}

Below, we present results analogous to those in Section~\ref{sec: main results} under a general scaling.
Specifically, throughout this section we define $(\bm X^\eta_j(\bm x))_{j \geq 0}$ 
and 
$(\bm X^{\eta|b}_j(\bm x))_{j \geq 0}$ by the recursions in \eqref{eq:gneral-scaling}
with $\gamma \in (\frac{1}{2 \wedge \alpha}, \infty)$,
where $\alpha > 1$ is the heavy-tailed index in Assumption~\ref{assumption gradient noise heavy-tailed}.
Let 
\begin{align*}
    \lambda(\eta;\gamma) = \eta^{-1}H(\eta^{-\gamma}).
\end{align*}
We adopt the notations
$
\mathbf{C}^{(k)|b}_{[0,T]},
$
$
\mathbb{D}^{(k)|b}_A[0,T](\epsilon),
$
$
{\bm{X}^{\eta|b}_{[0,T]}(\bm x)},
$
etc., introduced
in Section~\ref{sec: main results}.
First, we present the sample path large deviations under general scaling.
\begin{theorem}
Let  Assumptions \ref{assumption gradient noise heavy-tailed} and \ref{assumption: lipschitz continuity of drift and diffusion coefficients} hold.
Let $\gamma \in (\frac{1}{2 \wedge \alpha}, \infty)$.
\begin{enumerate}[(a)]
    \item 
        For any $k \in \mathbb N$, any $b,T,\epsilon > 0$, and any compact $A\subseteq \R^m$,
$$
 \lambda^{-k}(\eta;\gamma) \P\big( \bm{X}^{\eta|b}_{[0,T]}(\bm x) \in\ \cdot\ \big) 
    \rightarrow 
    \mathbf{C}^{(k)|b}_{[0,T]}   (\ \cdot\ ; \bm x)
    \quad
    \text{
        in
        $
        \mathbb{M}\Big( \mathbb{D}[0,T]\setminus \mathbb{D}^{(k-1)|b}_{A}[0,T](\epsilon) \Big)
        $
        uniformly in {$\bm x$ on} A
    }
$$
 as $\eta \downarrow 0$.
Furthermore,
for any
$B \in \mathscr{S}_{\mathbb{D}[0,T]}$
that is 
bounded away from ${ \mathbb{D}}_{A}^{(k - 1)|b}[0,T](\epsilon)$
for some (and hence all) $\epsilon > 0$ small enough,
\begin{equation}\nonumber
    \begin{split}
        \inf_{\bm x \in A}
    \mathbf{C}^{(k)|b}_{[0,T]}\big( B^\circ; \bm x \big)
& \leq  \liminf_{\eta \downarrow 0}\frac{ \inf_{\bm x \in A}\P\big({\bm{X}}^{\eta|b}_{[0,T]}(\bm x) \in B \big) }{  \lambda^k(\eta;\gamma)  } 
\\
   & \leq  \limsup_{\eta \downarrow 0}\frac{ \sup_{\bm x \in A}\P\big({\bm{X}}^{\eta|b}_{[0,T]}(\bm x) \in B \big) }{  \lambda^k(\eta;\gamma)  } 
    \leq 
    \sup_{\bm x \in A}
    \mathbf{C}^{(k)|b}_{[0,T]}\big( B^-; \bm x \big)
    <
    \infty.
    \end{split}
\end{equation}

\item Furthermore, suppose that Assumption~\ref{assumption: boundedness of drift and diffusion coefficients} holds.
For any $k \in \mathbb N$, $T,\epsilon >0$, and any compact $A \subseteq \R^m$ that
$$
\lambda^{-k}(\eta;\gamma) \P\big( \bm{X}^{\eta}_{[0,T]}(\bm x) \in\ \cdot\ \big) 
    \rightarrow 
    \mathbf{C}^{(k)}_{[0,T]}(\ \cdot\ ;\bm x)
    \quad
    \text{
        in $\mathbb{M}\Big(\mathbb{D}[0,T]\setminus \mathbb{D}^{(k-1)}_A[0,T](\epsilon)\Big)$
        uniformly in $\bm x$ on A
    }
$$
 as $\eta \downarrow 0$.
Furthermore, 
for any
$B \in \mathscr{S}_{\mathbb{D}[0,T]}$ that is
 bounded away from ${ \mathbb{D}}_{A}^{(k - 1)}[0,T](\epsilon)$
 for some (and hence all) $\epsilon > 0$ small enough,
\begin{equation} \nonumber
    \begin{split}
        \inf_{\bm x \in A}
    \mathbf{C}^{(k)}_{[0,T]}( B^\circ;\bm x)
& \leq  \liminf_{\eta \downarrow 0}\frac{ \inf_{\bm x \in A}\P\big({\bm{X}}^{\eta}_{[0,T]}(\bm x) \in B \big) }{  \lambda^k(\eta;\gamma)  } 
\\
   & \leq  \limsup_{\eta \downarrow 0}\frac{ \sup_{\bm x \in A}\P\big({\bm{X}}^{\eta}_{[0,T]}(\bm x) \in B \big) }{  \lambda^k(\eta;\gamma)  } 
    \leq 
    \sup_{\bm x \in A}
    \mathbf{C}^{(k)}_{[0,T]}( B^-; \bm x)
    <
    \infty.
    \end{split}
\end{equation}
\end{enumerate}
\end{theorem}

The corresponding conditional limit theorem is identical to Corollary~\ref{corollary: conditional limit, SGD}, under the condition that $\gamma \in (\frac{1}{2 \wedge \alpha}, \infty)$, so we skip the details.
Lastly, we present the metastability analysis.
Let $I \subseteq \R^m$ be an open set such that $\bm 0 \in \bm I$ and Assumption~\ref{assumption: shape of f, first exit analysis} holds.
Let the first exit times $\tau^\eta(\bm x)$ and $\tau^{\eta|b}(\bm x)$ be defined as in \eqref{def: first exit time for heavy tailed SGD}.
We adopt the notations $\mathcal J^I_b$, $\mathcal G^{(k)|b}(\epsilon)$, $\widecheck{\mathbf C}^{k|b}$, etc.,
introduced
in Section~\ref{sec: first exit time simple version}.

\begin{theorem}
    Let Assumptions \ref{assumption gradient noise heavy-tailed}, \ref{assumption: lipschitz continuity of drift and diffusion coefficients}, and \ref{assumption: shape of f, first exit analysis} hold.
    Let $\gamma \in (\frac{1}{2 \wedge \alpha}, \infty)$.
\begin{enumerate}[(a)]
    \item 
         Let $b > 0$.
         Suppose that $\mathcal J^I_b < \infty$,
         $I^c$ is bounded away from $\mathcal G^{(\mathcal J^I_b - 1)|b}(\epsilon)$ for some (and hence all) $\epsilon > 0$ small enough,
        and
        $
        \widecheck{\mathbf C}^{( \mathcal J^I_b )|b}(\partial I) = 0.
        $
        Then 
        ${C^I_b} \delequal  \widecheck{ \mathbf{C} }^{ (\mathcal{J}^I_b)|b }(I^\complement) < \infty$.
        Furthermore, if $C^I_b \in (0,\infty)$,
        then
        for any $\epsilon > 0$, $t \geq 0$, and measurable set $B \subseteq I^c$,
        \begin{align*}
            \limsup_{\eta\downarrow 0}\sup_{\bm x \in I_\epsilon}
            \P\bigg(
                C^I_b \eta\cdot \lambda^{ \mathcal{J}^I_b }(\eta;\gamma)\tau^{\eta|b}(\bm x) > t
             ;\ \bm X^{\eta|b}_{ \tau^{\eta|b}(\bm x)}(\bm x) \in B
             \bigg)
             & \leq \frac{ \widecheck{\mathbf{C}}^{ (\mathcal{J}^I_b)|b }(B^-) }{ C^I_b }\cdot\exp(-t),
             \\
             \liminf_{\eta\downarrow 0}\inf_{\bm x \in I_\epsilon}
            \P\bigg(
                C^I_b \eta\cdot \lambda^{ \mathcal{J}^I_b }(\eta;\gamma)\tau^{\eta|b}(\bm x) > t
             ;\ \bm X^{\eta|b}_{ \tau^{\eta|b}(\bm x)}(\bm x) \in B
             \bigg)
             & \geq \frac{ \widecheck{\mathbf{C}}^{ (\mathcal{J}^I_b)|b }(B^\circ) }{ C^I_b }\cdot\exp(-t).
        \end{align*}
        Otherwise, we have $C^I_b = 0$, and
        \begin{align*}
            \limsup_{\eta\downarrow 0}\sup_{\bm x \in I_\epsilon}
        \P\bigg(
            \eta\cdot \lambda^{ \mathcal J^I_b }(\eta;\gamma)\tau^{\eta|b}(\bm x) \leq t
         \bigg) = 0
         \qquad
         \forall \epsilon > 0,\ t \geq 0.
        \end{align*}

    \item 
        Suppose that $\widecheck{\mathbf C}(\partial I) = 0$.
        Then ${C^I_\infty} \delequal  \widecheck{ \mathbf{C} }(I^\complement) < \infty$.
        Furthermore, if $C^I_\infty > 0$,
        then
        for any $t \geq 0$ and measurable set $B \subseteq I^c$,
    \begin{align*}
        \limsup_{\eta\downarrow 0}\sup_{\bm x \in I_\epsilon}
        \P\bigg(
            C^I_\infty \eta\cdot \lambda(\eta;\gamma)\tau^\eta(\bm x) > t;\ \bm X^\eta_{ \tau^\eta(\bm x)}(\bm x) \in B
        \bigg)
        & \leq
        \frac{ \widecheck{\mathbf{C}}(B^-) }{ C^I_\infty }\cdot\exp(-t),
        \\ 
        \liminf_{\eta\downarrow 0}\inf_{\bm x \in I_\epsilon}
        \P\bigg(
            C^I_\infty \eta\cdot \lambda(\eta;\gamma)\tau^\eta(\bm x) > t;\ \bm X^\eta_{ \tau^\eta(\bm x)}(\bm x) \in B
        \bigg)
        & \geq
        \frac{ \widecheck{\mathbf{C}}(B^\circ) }{ C^I_\infty  }\cdot\exp(-t).
    \end{align*}
     Otherwise, we have $C^I_\infty = 0$, and 
    \begin{align*}
        \limsup_{\eta\downarrow 0}\sup_{\bm x \in I_\epsilon}
        \P\bigg(
            \eta\cdot \lambda(\eta;\gamma)\tau^\eta(\bm x) \leq t
        \bigg) = 0
        \qquad
        \forall \epsilon > 0,\ t \geq 0.
    \end{align*}
    
\end{enumerate}
\end{theorem}
The proofs for results in this section will be almost identical to those presented in the main paper.
We omit the details to avoid repetition.

\section{Results for L\'evy-Driven Stochastic Differential Equations}
\label{sec: appendix, SDE reults}

In this section, we collect the results for stochastic differential equations driven by L\'evy processes with regularly varying increments. Specifically,
any multidimensional L\'evy process $\notationdef{notation-levy-process}{\bm{L}} = \{ \bm L_t:\ t \geq 0\}$
can be characterized by its generating triplet $(\bm c_{\bm L},\bm \Sigma_{\bm L},\nu)$ where $\bm c_{\bm L} \in \mathbb{R}^m$ is the drift parameter, the positive semi-definite matrix $\bm \Sigma_{\bm L} \in \R^{m \times m}$ is the magnitude of the Brownian motion term in $\bm L_t$, 
and $\nu$ is the L\'evy measure characterizing the intensity of jumps in $\bm L_t$. 
More precisely, 
we have the following L\'evy–Itô decomposition
\begin{align}
    \bm L_t \distequal \bm c_{\bm L}t + \bm \Sigma_{\bm L}^{1/2} 
    \bm B_t+ 
    \int_{\norm{\bm x} \leq 1}\bm x\big[ N([0,t]\times d\bm x) - t\nu(dx) \big] 
    + 
    \int_{\norm{\bm x} > 1}\bm xN( [0,t]\times d\bm x ) 
    \label{prelim: levy ito decomp}
\end{align}
where $\bm B_t$ is a standard Brownian motion in $\R^m$, 
the measure $\nu$ satisfies $\int (\norm{\bm x}^2\wedge 1) \nu(d\bm x) < \infty$,
and $N$ is a Poisson random measure independent of $\bm B_t$ with intensity measure $\mathcal L_\infty\times \nu$.
See Chapter~4 of \cite{sato1999levy} for details.
We impose the following heavy-tailed assumption on the increments of $\bm L_t$.

\begin{assumption}\label{assumption: heavy-tailed levy process}
$\E \bm L_1 = \bm 0$. Besides, there exist $\alpha > 1$ 
and 
a probability measure $\mathbf S(\cdot)$ on the unit sphere of $\R^d$ such that
\begin{itemize}
    \item 
         $\notationdef{notation-H-L}{H_L(x)} \delequal \nu\Big( \big\{\bm y \in \R^d:\ \norm{\bm y} > x  \big\} \Big) \in \RV_{-\alpha}(x)$ as $x \to \infty$,

    \item As $r \to \infty$,
        \begin{align*}
            \frac{ \big(\nu \circ \Phi^{-1}_r\big)(\cdot) }{ H_L(r)} \rightarrow \nu_\alpha \times \mathbf S
            \qquad 
            \text{in $\mathbb M\Big( 
            \big([0,\infty) \times \mathfrak N_d \big)
            \setminus
            \big( \{0\} \times \mathfrak N_d \big)
            \Big)$},
        \end{align*}
        where 
        $\mathfrak N_d$ is the unit sphere of $\R^d$,
        the measure $\nu_\alpha$ is defined in \eqref{def: measure nu alpha},
        and
        \begin{align*}
            \big(\nu \circ \Phi^{-1}_r\big)(\cdot)
            \delequal
            \nu\Big(
                \Phi^{-1}(r\cdot\ ,\ \cdot)
            \Big),
        \end{align*}
        i.e. $ \big(\nu \circ \Phi^{-1}_r\big)(A\times B) = 
        \nu\big(
        \Phi^{-1}(rA,B)
        \big)
        $
        for all Borel sets $A \subseteq (0,\infty)$ and $B \subseteq \mathfrak N_d$.
\end{itemize}
\end{assumption}
Consider a filtered probability space
$\big(\Omega,\mathcal{F},\mathbb{F} = (\mathcal{F}_{t})_{t \geq 0},\P\big)$
satisfying the usual hypotheses stated in Chapter I, \cite{protter2005stochastic} and supporting the L\'evy process $\bm L$, where
$\mathcal{F}_{0} = \{\emptyset,\Omega\}$ and $\mathcal{F}_{t}$ is the $\sigma$-algebra generated by $\{\bm L_s:s\in[0,t]\}$.
For $\eta \in (0,1]$
and $\beta \geq 0$, define the scaled process
\begin{align}
    {\bar{\bm L}^\eta} \delequal \big\{ \bar{\bm L}^\eta_t = \eta \bm L_{t/\eta^\beta}:\ t \geq 0\big\},
    \label{def: scaled levy process}
\end{align}
and
let $\bm Y^\eta_t(\bm x)$ be the solution to SDE
\begin{align}
\bm Y^\eta_0(\bm x) = \bm x,\qquad 
    d\notationdef{notation-Y-eta-SDE}{\bm Y^\eta_t(\bm x)} 
    =
    \bm a\big(\bm Y^\eta_{t-}(\bm x)\big)dt
        +
    \bm \sigma\big(\bm Y^\eta_{t-}(\bm x)\big) d\bar{\bm L}^\eta_t.
    \label{defSDE, initial condition x}
\end{align}
Henceforth in Section~\ref{sec: appendix, SDE reults}, we consider $\beta \in [0,2\wedge \alpha)$
where $\alpha > 1$ is the tail index in Assumption~\ref{assumption: heavy-tailed levy process}.
Below, we present the sample-path large deviations and metastability analysis of $\bm Y^\eta_t(\bm x)$.

\subsection{Sample Path Large Deviations}
\label{subsec: LD, SDE}

Recall the definitions of the mapping $h^{(k)}_{[0,T]} = h^{(k)|\infty}_{[0,T]}$ in \eqref{def: perturb ode mapping h k b, 1}--\eqref{def: perturb ode mapping h k b, 3}
as well as the measure $\mathbf{C}^{(k)}_{[0,T]} = \mathbf{C}^{(k)|\infty}_{[0,T]}$ in \eqref{def: measure mu k b t}.
Also, recall uniform $\mathbb{M}$-convergence introduced in Definition \ref{def: uniform M convergence}.
Define
$\notationdef{notation-bm-Y-0T-eta-SDE}{\bm Y^\eta_{[0,T]}(\bm x)} = \{\bm Y^\eta_t(\bm x):\ t\in[0,T]\}$ as a random element in $\D[0,T]$.
In case that $T = 1$, we suppress $[0,1]$ and write $\notationdef{notation-bm-Y-eta-SDE}{\bm Y^\eta(\bm x)}$.
The next result characterizes the sample-path large deviations for $\bm Y^\eta_{[0,T]}(x)$  by establishing $\mathbb{M}$-convergence that is uniform in the initial condition $x$.
The proofs are almost identical to those of $\bm X^\eta_j(\bm x)$ and hence omitted to avoid repetition.
Recall that ${H_L(x)} = \nu\big( \{\bm y \in \R^d:\ \norm{\bm y} > x  \} \big)$.
Let
\begin{align*}
   \notationdef{notation-scale-function-lambda-L}{\lambda_L(\eta;\beta)} \delequal \eta^{-\beta} H_L(\eta^{-1})
\end{align*}
and
$
\lambda^k_L(\eta;\beta) = \big(\lambda_L(\eta;\beta)\big)^k,
$
where $\beta \in [0,2\wedge \alpha)$ determines the time scaling in \eqref{def: scaled levy process}.

\begin{theorem}
\label{theorem: LDP, SDE, uniform M convergence}
\linksinthm{theorem: LDP, SDE, uniform M convergence}
Under Assumptions \ref{assumption: lipschitz continuity of drift and diffusion coefficients},  \ref{assumption: boundedness of drift and diffusion coefficients}, and \ref{assumption: heavy-tailed levy process},
it holds for
 any $\beta \in [0,2\wedge \alpha)$,
$T,\epsilon >0,\ k \in \mathbb N$, and any compact set $A\subseteq \R^m$ that
$$
\lambda^{-k}_L(\eta;\beta)\P\big(\bm{Y}^\eta_{[0,T]}(\bm x)\in\ \cdot\ \big)
    \rightarrow \mathbf{C}^{(k)}_{[0,T]}(\ \cdot\ ;\bm x)
    \quad
    \text{
    in
        $
        \mathbb{M}\Big(\mathbb{D}[0,T]\setminus \mathbb{D}^{(k-1)}_A[0,T](\epsilon)\Big)
        $
    uniformly in $\bm x$ on A
    }
$$
 as $\eta \to 0$.
Furthermore,
given
$B \in \mathscr{S}_{\mathbb{D}[0,T]}$
that is bounded away from ${ \mathbb{D}}_{A}^{(k - 1)}[0,T](\epsilon)$
for some (and hence all) $\epsilon > 0$ small enough,
\begin{align*}
\inf_{\bm x \in A}
    \mathbf{C}^{(k)}_{[0,T]}\big( B^\circ; \bm x \big)
& \leq  \liminf_{\eta \downarrow 0}\frac{ \inf_{\bm x \in A}\P\big({\bm{Y}}^{\eta}_{[0,T]}(\bm x) \in B \big) }{  \lambda^k_L(\eta;\beta)  } 
\\
   & \leq  \limsup_{\eta \downarrow 0}\frac{ \sup_{\bm x \in A}\P\big({\bm{Y}}^{\eta}_{[0,T]}(\bm x) \in B \big) }{  \lambda^k_L(\eta;\beta)  } 
    \leq 
    \sup_{\bm x \in A}
    \mathbf{C}^{(k)}_{[0,T]}\big( B^-; \bm x \big)
    <
    \infty.
\end{align*}
\end{theorem}

Analogous to the truncated dynamics $\bm X^{\eta|b}_j(\bm x)$,
we introduce a truncated variation ${\bm Y}^{\eta|b}_t(\bm x)$
where all jumps are truncated under the threshold value $b$.
More generally, we consider
a sequence of stochastic processes $\big(\bm{Y}^{\eta|b;(k)}_t(\bm x)\big)_{k \geq 0}$.
First, for any $\bm x \in \R^m$ and $t \geq 0$, let 
\begin{align}
    d{\bm Y}^{\eta|b;(0)}_t(\bm x) = \bm a\big({\bm Y}^{\eta|b;(0)}_{t-}(\bm x)\big)dt + 
    \bm \sigma\big(\bm{Y}^{\eta|b;(0)}_{t-}(\bm x)\big)d\bar{\bm L}_t
    \label{def: Y eta b 0 f g, SDE clipped}
\end{align}
under initial condition ${\bm Y}^{\eta|b;(0)}_0(\bm x) = \bm x$.
Next, we define
\begin{align}
    \tau^{\eta|b;(1)}_{Y}(\bm x) & \delequal 
    \inf\Big\{ t > 0: \norm{\bm\sigma\big(\bm {Y}^{\eta|b;(0)}_{t-}(\bm x)\big) \Delta \bar{\bm L}^\eta_t } 
    > b \Big\},
    \label{def, tau and W, discont in Y, eta b, 1}
    \\
    \bm W^{\eta|b;(1)}_{Y}(\bm x) & \delequal \Delta \bm{Y}^{\eta|b;(0)}_{ \tau^{\eta|b;(1)}_{Y}(\bm x)}(\bm x)
    \label{def, tau and W, discont in Y, eta b, 2}
\end{align}
as the arrival time and size of the first jump in $\bm{Y}^{\eta|b;(0)}_t(\bm x)$ with $L_2$ norm larger than $b$.
Furthermore, we define (for any $k \geq 1$)
\begin{align}
    {
    \bm{Y}^{\eta|b;(k)}_{ \tau^{\eta|b;(k)}_{Y}(\bm x) }(\bm x)
    } & \delequal 
    \bm{Y}^{\eta|b;(k)}_{ \tau^{\eta|b;(k)}_{Y}(\bm x)- }(\bm x)
    +
    \varphi_b\Big(
    \bm W^{\eta|b;(k)}_{Y}(\bm x)
    \Big),
    \label{def: objects for defining Y eta b, clipped SDE, 1}
    \\
    d{\bm Y^{\eta|b;(k)}_{ t}(\bm x)}& \delequal 
    \bm a\big(\bm{Y}^{\eta|b;(k)}_{ t-}(\bm x)\big) dt 
    + 
    \bm\sigma\big(\bm{Y}^{\eta|b;(k)}_{ t-}(\bm x)\big)d \bar{\bm L}^\eta_t
    \qquad \forall t > \tau^{\eta|b;(k)}_{Y}(\bm x),
    \label{def: objects for defining Y eta b, clipped SDE, 2}
    \\
    {
    \tau^{\eta|b;(k+1)}_{Y}(\bm x)
    } & \delequal 
    \min\Big\{ t > \tau^{\eta|b;(k)}_{Y}(\bm x): \norm{\sigma\big(\bm{Y}^{\eta|b;(k)}_{t-}(\bm x)\big) \Delta{\bar{\bm L}}^\eta_t } > b \Big\},
    \label{def: objects for defining Y eta b, clipped SDE, 3}
    \\
    {\bm W^{\eta|b;(k+1)}_{Y}(\bm x)} & \delequal 
    \Delta\bm{Y}^{\eta|b;(k)}_{ \tau^{\eta|b;(k + 1)}_{Y}(\bm x)}(\bm x)
    \label{def: objects for defining Y eta b, clipped SDE, 4}
\end{align}
Lastly, for any $t \geq 0,\ b > 0,\ k \in \mathbb N$ and $\bm x \in \R^m$, we define
(under convention $\tau^{\eta|b;(0)}_{Y}(\bm x) = 0$)
\begin{align}
    \notationdef{notation-Y-eta-b-SDE}{\bm Y^{\eta|b}_t(\bm x)}
    \delequal 
    \sum_{k \geq 0}
    \bm Y^{\eta|b;(k)}_{t}(\bm x)
    \cdot 
    \mathbbm{I}
    \Big\{
    t \in \Big[  
    \tau^{\eta|b;(k)}_{Y}(\bm x), 
     \tau^{\eta|b;(k+1)}_{Y}(\bm x) 
    \Big)
    \Big\}
    \label{defSDE, o.g., initial condition x, clipped}
\end{align}
and let $\notationdef{notation-bm-Y-0T-eta-b-SDE}{\bm{Y}^{\eta|b}_{[0,T]}(x)} \delequal \big\{\bm Y^{\eta|b}_t(\bm x):\ t \in [0,T]\big\}$.
By definition, for any $t \geq 0,\ b > 0,\ k \in \mathbb N$ and $\bm x \in \R^m$,
\begin{align}
    {\bm Y^{\eta|b}_t(\bm x)}  = \bm Y^{\eta|b;(k)}_{t}(\bm x)
    \qquad \Longleftrightarrow \qquad 
    t \in \Big[  
    \tau^{\eta|b;(k)}_{Y}(\bm x), 
     \tau^{\eta|b;(k+1)}_{Y}(\bm x) 
    \Big).
    \label{defSDE, initial condition x, clipped}
\end{align}
In case that $T = 1$, we suppress $[0,1]$ and write $\notationdef{notation-bm-Y-eta-b-SDE}{\bm Y^{\eta|b}(\bm x)}$.
The next theorem presents the sample path large deviations for $\bm Y^{\eta|b}_t(\bm x)$.
Once again, the proof is omitted as it closely resembles that of $\bm X^{\eta|b}_j(\bm x)$.

\begin{theorem}
\label{theorem: LDP, SDE, uniform M convergence, clipped}
\linksinthm{theorem: LDP, SDE, uniform M convergence, clipped}
Under Assumptions~\ref{assumption: lipschitz continuity of drift and diffusion coefficients} and \ref{assumption: heavy-tailed levy process},
it holds for 
any $\beta \in [0,2\wedge \alpha)$, any $b,T,\epsilon > 0,\ k\in \mathbb N$, and any compact set $A\subseteq \R^m$ that
$$
\lambda^{-k}_L(\eta;\beta)\P\big(\bm{Y}^{\eta|b}_{[0,T]}(\bm x)\in\ \cdot\ \big)
    \rightarrow \mathbf{C}^{(k)|b}_{[0,T]}(\ \cdot\ ;\bm x)
    \text{
    in
    $
    \mathbb{M}\Big(\mathbb{D}[0,T]\setminus \mathbb{D}^{(k-1)|b}_{A}[0,T](\epsilon)\Big)
    $
    uniformly in $\bm x$ on A
    }
$$
 as $\eta \to 0$.
Furthermore,
given
$B \in \mathscr{S}_{\mathbb{D}[0,T]}$ that is bounded away from ${ \mathbb{D}}_{A}^{(k - 1)|b}[0,T](\epsilon)$
for some (and hence all) $\epsilon > 0$ small enough,
\begin{align*}
\inf_{\bm x \in A}
    \mathbf{C}^{(k)|b}_{[0,T]}\big( B^\circ; \bm x \big)
& \leq  \liminf_{\eta \downarrow 0}\frac{ \inf_{\bm x \in A}\P\big({\bm{Y}}^{\eta|b}_{[0,T]}(\bm x) \in B \big) }{  \lambda^k_L(\eta;\beta)  } 
\\
   & \leq  \limsup_{\eta \downarrow 0}\frac{ \sup_{\bm x \in A}\P\big({\bm{Y}}^{\eta|b}_{[0,T]}(\bm x) \in B \big) }{  \lambda^k_L(\eta;\beta)  } 
    \leq 
    \sup_{\bm x \in A}
    \mathbf{C}^{(k)|b}_{[0,T]}\big( B^-; \bm x \big)
    <
    \infty.
\end{align*}
\end{theorem}

Analogous to Corollary~\ref{corollary: conditional limit, SGD}, we present the conditional limit theorem for 
$\bm Y^{\eta}_{[0,1]}(\bm x)$ and $\bm Y^{\eta|b}_{[0,1]}(\bm x)$.

\begin{corollary}\label{corollary: conditional limit, SDE}
Let Assumptions \ref{assumption: lipschitz continuity of drift and diffusion coefficients} and \ref{assumption: heavy-tailed levy process} hold.
Let $\beta \in [0,2\wedge \alpha)$.
\begin{enumerate}[(i)]
 \item Given $b > 0$, $k \in \mathbb N$, $\bm x \in \R^m$, and measurable $B \subseteq \mathbb D$, suppose that $B$ is bounded away from $\mathbb D^{(k-1)|b}_{\{\bm x\}}(\epsilon)$ 
 for some (and hence all) $\epsilon > 0$ small enough,
 and
 $
    \mathbf{C}^{(k)|b}(B^\circ;\bm x) = \mathbf{C}^{(k)|b}(B^-;\bm x) > 0.
    $
    Then
    \begin{align*}
        \P\big(\bm Y^{\eta|b}_{[0,1]}(\bm x)\in \cdot\,|\,\bm Y^{\eta|b}_{[0,1]}(\bm x)\in B\big) \Rightarrow
    \frac{\mathbf C^{(k)|b}(\,\cdot\cap B;x)}{\mathbf C^{(k)|b}(B;x)}
    \qquad\text{ as $\eta \downarrow 0$.}
    \end{align*}


    \item Furthermore, suppose that Assumption \ref{assumption: boundedness of drift and diffusion coefficients} holds.
    Given $k\in \mathbb N$, $\bm x \in \R^m$, and measurable $B \subseteq \mathbb D$, suppose that $B$ is bounded away from $\mathbb D^{(k-1)}_{\{\bm x\}}(\epsilon)$
    for some (and hence all) $\epsilon > 0$ small enough,
    and 
    $
    \mathbf{C}^{(k)}(B^\circ;\bm  x) = \mathbf{C}^{(k)}(B^-;\bm x) > 0.
    $
    Then
    \begin{align*}
        \P\big(\bm Y^{\eta}_{[0,1]}(x)\in \cdot\,|\,\bm Y^{\eta}_{[0,1]}(x)\in B\big) \Rightarrow
    \frac{\mathbf C^{(k)}(\,\cdot\cap B;x)}{\mathbf C^{(k)}(B;x)}
    \qquad\text{ as $\eta \downarrow 0$.}
    \end{align*}
\end{enumerate}
\end{corollary}

\subsection{Metastability Analysis}
\label{subsec: first exit time, SDE}

Consider an
open set $I \subseteq \R^m$ such that $\bm 0 \in \bm I$ and Assumption~\ref{assumption: shape of f, first exit analysis} holds.
Define stopping times
\begin{align*}
    {\tau^\eta_Y(\bm x)} \delequal \inf\big\{t \geq 0:\ \bm Y^\eta_t(\bm x) \notin I\big\},\qquad
    {\tau^{\eta|b}_Y(\bm x)} \delequal \inf\big\{t \geq 0:\ \bm Y^{\eta|b}_t(\bm x) \notin I \big\}
\end{align*}
as the first exit times of $\bm Y^\eta_t(\bm x)$ and $\bm Y^{\eta|b}_t(\bm x)$ from $I$, respectively.
The following result characterizes the asymptotic law of the first exit times and exit locations,
using the measures ${\widecheck{ \mathbf C }^{(k)|b}(\cdot )}$ defined in \eqref{def: measure check C k b}
and
${\widecheck{\mathbf C}(\cdot)}$ defined in \eqref{def: measure check C}.
We omit the proof
due to its similarity to that of Theorem \ref{theorem: first exit time, unclipped}.

\begin{theorem}
\label{theorem: first exit time, unclipped, SDE}
\linksinthm{theorem: first exit time, unclipped, SDE}
    Let Assumptions \ref{assumption: lipschitz continuity of drift and diffusion coefficients},
    \ref{assumption: shape of f, first exit analysis}, and \ref{assumption: heavy-tailed levy process} hold.
    Let $\beta \in [0,2\wedge \alpha)$.

    \begin{enumerate}[(a)]

\item 
    Let $b > 0$.
         Suppose that $\mathcal J^I_b < \infty$, $I^c$ is bounded away from $\mathcal G^{(\mathcal J^I_b - 1)|b}(\epsilon)$ for some (and hence all) $\epsilon > 0$ small enough,
        and
        $
        \widecheck{\mathbf C}^{( \mathcal J^I_b )|b}(\partial I) = 0.
        $
        Then 
        ${C^I_b} \delequal  \widecheck{ \mathbf{C} }^{ (\mathcal{J}^I_b)|b }(I^\complement) < \infty$.
        Furthermore, if $C^I_b \in (0,\infty)$,
        then
        for any $\epsilon > 0$, $t \geq 0$, and measurable set $B \subseteq I^c$,
\begin{align*}
    \limsup_{\eta\downarrow 0}\sup_{\bm x \in I_\epsilon}
    \P\bigg(C^I_b \lambda^{ \mathcal{J}^I_b }_L(\eta;\beta)\tau^{\eta|b}_Y(\bm x) > t
     ;\ \bm Y^{\eta|b}_{ \tau^{\eta|b}_Y(\bm x)}(\bm x) \in B\bigg)
     & \leq \frac{ \widecheck{\mathbf{C}}^{ (\mathcal{J}^I_b)|b }(B^-) }{ C^I_b }\cdot\exp(-t),
     \\
     \liminf_{\eta\downarrow 0}\inf_{\bm x \in I_\epsilon}
    \P\bigg(C^I_b \lambda^{ \mathcal{J}^I_b }_L(\eta;\beta)\tau^{\eta|b}_Y(\bm x) > t
     ;\ \bm Y^{\eta|b}_{ \tau^{\eta|b}_Y(\bm x)}(\bm x) \in B\bigg)
     & \geq \frac{ \widecheck{\mathbf{C}}^{ (\mathcal{J}^I_b)|b }(B^\circ) }{ C^I_b }\cdot\exp(-t).
\end{align*}
    Otherwise, we have $C^I_b = 0$, and
        \begin{align*}
            \limsup_{\eta\downarrow 0}\sup_{\bm x \in I_\epsilon}
        \P\bigg(
            \lambda^{ \mathcal J^I_b }_L(\eta;\gamma)\tau^{\eta|b}_Y(\bm x) \leq t
         \bigg) = 0
         \qquad
         \forall \epsilon > 0,\ t \geq 0.
        \end{align*}

    \item
        Suppose that $\widecheck{\mathbf C}(\partial I) = 0$.
        Then ${C^I_\infty} \delequal  \widecheck{ \mathbf{C} }(I^\complement) < \infty$.
        Furthermore, if $C^I_\infty > 0$,
        then
        for any $t \geq 0$ and measurable set $B \subseteq I^c$,
    \begin{align*}
        \limsup_{\eta\downarrow 0}\sup_{\bm x \in I_\epsilon}
        \P\bigg(
            C^I_\infty \lambda_L(\eta;\beta)\tau^\eta_Y(\bm x) > t;\ \bm Y^\eta_{ \tau^\eta_Y(\bm x)}(\bm x) \in B
        \bigg)
        & \leq
        \frac{ \widecheck{\mathbf{C}}(B^-) }{ C^I_\infty }\cdot\exp(-t),
        \\ 
        \liminf_{\eta\downarrow 0}\inf_{\bm x \in I_\epsilon}
        \P\bigg(
            C^I_\infty \lambda_L(\eta;\beta)\tau^\eta_Y(\bm x) > t;\ \bm Y^\eta_{ \tau^\eta_Y(\bm x)}(\bm x) \in B
        \bigg)
        & \geq
        \frac{ \widecheck{\mathbf{C}}(B^\circ) }{ C^I_\infty }\cdot\exp(-t).
    \end{align*}
     Otherwise, we have $C^I_\infty = 0$, and 
    \begin{align*}
        \limsup_{\eta\downarrow 0}\sup_{\bm x \in I_\epsilon}
        \P\bigg(
            \lambda(_L\eta;\gamma)\tau^\eta_Y(\bm x) \leq t
        \bigg) = 0
        \qquad
        \forall \epsilon > 0,\ t \geq 0.
    \end{align*}

    \end{enumerate}
\end{theorem}

\section{Properties of Mappings $\bar h^{(k)}_{[0,T]}$ and $\bar h^{(k)|b}_{[0,T]}$} 
\label{sec: appendix, mapping h}
In this section, we collect a few useful results about the mappings $\bar h^{(k)}_{[0,T]}$ and $\bar h^{(k)|b}_{[0,T]}$ defined in \eqref{def: perturb ode mapping h k b, 1}--\eqref{def: perturb ode mapping h k b, 3}, and provide the proof of Lemmas \ref{lemma: LDP, bar epsilon and delta}, \ref{lemma: LDP, bar epsilon and delta, clipped version}, and \ref{lemma: convergence from C k b measure to C k measure}.

For any $\xi \in \mathbb{D}$, let $\norm{\xi} \delequal \sup_{t \in [0,1]}\norm{\xi(t)}$.
Also, recall the definition of  $\mathbb{D}_{A}^{(k)|b}(r)$ in \eqref{def: l * tilde jump number for function g, clipped SGD}.
Lemma~\ref{lemma: boundedness of k jump set under truncation, LDP clipped} shows that $\norm \xi$ is uniformly bounded for all $\xi \in \mathbb{D}_{A}^{(k)|b}(r)$.

\begin{lemma} \label{lemma: boundedness of k jump set under truncation, LDP clipped}
\linksinthm{lemma: boundedness of k jump set under truncation, LDP clipped}
Let Assumption~\ref{assumption: lipschitz continuity of drift and diffusion coefficients} 
hold.
Given $k \in \mathbb N$, $b,r > 0$,  and a compact set $A \subseteq \R^m$,
there exists $M = M(k,b,r, A) < \infty$ such that
$\norm{\xi}\leq M\ \forall \xi \in \mathbb{D}_{A}^{(k)|b}(r).$
\end{lemma}

\begin{proof}
\linksinpf{lemma: boundedness of k jump set under truncation, LDP clipped}%
Fix some $\bm x_0 \in A$, and 
let $\xi^*(t) = \bm{y}_t(\bm x_0)$; see \eqref{def ODE path y t}.
Let $N = r + \sup_{ \bm x,\bm y \in A  }\norm{\bm x - \bm y}\vee b$ and $\rho = \exp(D) \geq 1$
where $D\in [1,\infty)$ is the Lipschitz coefficient in Assumption~\ref{assumption: lipschitz continuity of drift and diffusion coefficients}.

By arbitrarily picking an element from $\D_{A}^{(k)|b}(r)$, we get some $\xi = \bar h^{(k)|b}(\bm x,\textbf W,\textbf V, \bm{t})$ with
$\bm x \in A,\ \textbf W = (\bm w_1,\cdots,\bm w_k) \in \R^{d \times k},\ \textbf V = (\bm v_1,\cdots,\bm v_k) \in \R^{m\times k},\ \bm{t} = (t_1,\cdots,t_k) \in (0,1]^{k \uparrow}$.
By Assumption~\ref{assumption: lipschitz continuity of drift and diffusion coefficients}
and Gronwall's inequality, we get
$
\sup_{t \in [0,t_1)}\norm{\xi^*(t) - \xi(t)} \leq  \norm{ \bm x - \bm x_0} \cdot \exp(Dt_1) \leq \rho N.
$
Since $\xi^*(t)$ is continuous, and $\norm{\xi(t_1)- \xi(t_1-)}\leq b + r$ (see the definition of $\varphi_b$ in \eqref{def: perturb ode mapping h k b, 3}), we get 
$\sup_{t \in [0,t_1]}\norm{\xi^*(t) - \xi(t)} \leq \rho N + b + r \leq 2\rho N.$

Next, we proceed by induction.
Adopt the convention that $t_{k+1} = 1$, and suppose that for some $j = 1,2,\cdots,k$,
\begin{align*}
    \sup_{t \in [0,t_j]}\norm{\xi^*(t) - \xi(t)} \leq \underbrace{(2\rho)^j N}_{\delequal M_j}.
\end{align*}
Then from Gronwall's inequality again, we get $\norm{\xi^*(t) - \xi(t)} \leq \rho A_j$ for any $t \in [t_j, t_{j+1})$.
Due to the continuity of $\xi^*$ and the truncation threshold $b$ and the upper bound $\norm{\bm v_j} \leq r$ at step \eqref{def: perturb ode mapping h k b, 3},
we have 
\begin{align*}
    \norm{\xi(t_{j+1}) - \xi^*(t_{j+1})} & \leq \rho M_j + b + r \leq 2 \rho M_j \leq M_{j + 1}.
\end{align*}
Therefore, $\sup_{t \in [0,t_{j+1}]}\norm{\xi^*(t) - \xi(t)} \leq  M_{j+1}.$
By induction, we can conclude the proof with $M = M_{k+1}  + \|\xi^*\| = (2\rho)^{k+1}N + \|\xi^*\|$.
\end{proof}

Recall the definitions of $\bm a_M,\bm \sigma_M$ in \eqref{def: a sigma truncated at M, LDP}, the mapping $\bar h^{(k)|b}_{M\downarrow}$ in \eqref{def: perturb ode mapping h k b, truncated at M, 1}--\eqref{def: perturb ode mapping h k b, truncated at M, 3},
and sets ${\mathbb{D}_{A;M\downarrow}^{(k)|b}(r)}$ in \eqref{def D A k t truncation set}.
Next, we present a corollary of the boundedness of $\mathbb{D}^{(k)|b}_{A}(r)$ established in Lemma \ref{lemma: boundedness of k jump set under truncation, LDP clipped}.

 \begin{corollary}
 \label{corollary: existence of M 0 bar delta bar epsilon, clipped case, LDP}
 \linksinthm{corollary: existence of M 0 bar delta bar epsilon, clipped case, LDP}
Let Assumption~\ref{assumption: lipschitz continuity of drift and diffusion coefficients} hold.
Let $b,r > 0$, $k \in \mathbb N$.
Let $A \subseteq \R^m$ be compact.
There exists $M_0 \in (0,\infty)$ such that
for any $M \geq M_0$,
\begin{itemize}
    \item 
        $\sup_{t \leq 1}\norm{\xi_t} \leq M_0
        \ \forall \xi \in \mathbb{D}_{A}^{(k)|b}(r)  \cup  \mathbb{D}_{A;M\downarrow}^{(k)|b}(r)$;
    
    \item 
        For any
        $\bm{t} = (t_1,\cdots,t_{k}) \in (0,1]^{k\uparrow}$, 
        $\textbf W = (\bm w_1,\cdots,\bm w_{k}) \in \mathbb{R}^{d \times k}$,
        $\textbf V = (\bm v_1,\cdots,\bm v_k) \in \R^{m \times k}$ with $\max_{j \in [d]}\norm{\bm v_j} \leq r$,
        and $\bm x \in A$,
    \begin{align*}
          \bar h^{(k)|b}(\bm x,\textbf W,\textbf V,\bm{t}) =  \bar h^{(k)|b}_{M \downarrow}(\bm x,\textbf W,\textbf V,\bm{t}).
     \end{align*}
\end{itemize}
 \end{corollary}

\begin{proof} 
\linksinpf{corollary: existence of M 0 bar delta bar epsilon, clipped case, LDP}
The claims follow immediately from Lemma~\ref{lemma: boundedness of k jump set under truncation, LDP clipped},
as well as the fact that 
$\bar h^{(k)|b}(\bm x,\textbf W, \textbf V, \bm t)\in\mathbb D^{(k)|b}_{A}(r)$
and
$\xi = \bar h^{(k)|b}_{M\downarrow}(\bm x,\textbf W,\textbf V, \bm t)\in\mathbb D^{(k)|b}_{A;M\downarrow}(r)$.
\elaborate{
then $a_M(\xi_t) = a(\xi_t)$ and $\sigma_M(\xi(s-)) = \sigma(\xi(s-))$ for all $s\in[0,1]$, and hence, 
$$\xi = h^{(k)|b} (x_0,\bm{w},\bm{t}) = h^{(k)|b}_{M \downarrow}(x_0,\bm{w},\bm{t}),$$
which, in turn, implies that $\mathbb{D}_{A;M\downarrow}^{(k)|b} = \mathbb{D}_{A}^{(k)|b} $ since the choice of $x_0$, $\bm w$, and $\bm t$ was arbitrary.
}
\end{proof}

Next, we study the continuity of mappings $\bar h^{(k)}_{[0,T]}$ and $\bar h^{(k)|b}_{[0,T]}$.

\begin{lemma}
\label
{lemma: continuity of h k b mapping clipped}
\linksinthm
{lemma: continuity of h k b mapping clipped}%
Let Assumption~\ref{assumption: lipschitz continuity of drift and diffusion coefficients} hold.
Given any $b,T > 0$ and any $k \in \mathbb N$, the mapping
$\bar h^{(k)|b}_{[0,T]}$ is continuous on $\R^m \times \R^{d\times k} \times \R^{m \times k} \times (0,T)^{k\uparrow}$.
\end{lemma}

\begin{proof}
\linksinpf
{lemma: continuity of h k b mapping clipped}%
To ease notations we focus on the case where $T = 1$, but the proof is identical for any $T > 0$.
Arbitrarily pick some $b > 0$ and $k \in \mathbb N$,
some
$\bm x^* \in \R^m$, $\textbf W^* = (\bm w^*_1,\cdots,\bm w_k^*) \in \R^{d\times k}$,
$\textbf V^* = (\bm v^*_1,\cdots,\bm v^*_k) \in \R^{m \times k}$,
and $\bm{t}^* = (t^*_1,\cdots,t^*_k) \in (0,1)^{k\uparrow}$.
Let $\xi^* = \bar h^{(k)|b}(\bm x^*,\textbf W^*,\textbf V^*,\bm{t}^*)$.
Also, fix some $\epsilon \in (0,1)$.
It suffices to show the existence of some $\delta \in (0,1)$ such that 
$
\bm{d}_{J_1}(\xi^*,\xi^\prime) < \epsilon
$
for all $\xi^\prime = \bar h^{(k)|b}(\bm x^\prime,\textbf W^\prime,\textbf V^\prime,\bm{t}^\prime)$
with $\bm x^\prime \in \R^m,\ \textbf W^\prime = (\bm w^\prime_1,\cdots,\bm w^\prime_k) \in \R^{d \times k},\
\textbf V^\prime = (\bm v^\prime_1,\cdots,\bm v^\prime_k) \in \R^{m \times k},\ 
\bm{t}^\prime = (t^\prime_1,\cdots,t^\prime_k) \in (0,1)^{k \uparrow}$
satisfying
\begin{align}
    \norm{ \bm x^* - \bm x^\prime } < \delta, \qquad 
    \norm{ \bm w^\prime_j - \bm w^*_j} \vee \norm{\bm v^*_j - \bm v^\prime_j} \vee |t^\prime_j - t^*_j| < \delta\ \forall j \in [k].
    \label{condition, arguments for xi prime, lemma: continuity of h k b mapping clipped}
\end{align}
\elaborate{
\begin{itemize}
    \item $|x^*- x^\prime| < \delta$;
    \item $|w^\prime_j - w^*_j| < \delta\ \ \forall j \in [k];$
    \item $|t^\prime_j - t^*_j| < \delta\ \ \forall j \in [k].$
\end{itemize}
}
We start by setting some constants and notations.
First,
by Corollary \ref{corollary: existence of M 0 bar delta bar epsilon, clipped case, LDP},
it follows for any $M \in (0,\infty)$ large enough that
\begin{align}
    \norm{\xi^*} + 1 < M
    \qquad\text{ and }\qquad
    \norm{\xi^\prime} + 1 < M\quad \forall \xi^\prime = \bar h^{(k)|b}(\bm x^\prime,\textbf W^\prime,\textbf 
 V^\prime,\bm{t}^\prime)\text{ satisfying \eqref{condition, arguments for xi prime, lemma: continuity of h k b mapping clipped}.}
    \label{proof, property of M, bound on xi and xi prime, lemma: continuity of h k b mapping clipped}
\end{align}
By picking an even larger $M$ if necessary,
we can ensure that $M \geq 1 + \max_{j \in [k]}\norm{\bm w^*_j}$.
In this proof, we write $\bm a^*=\bm a_M$, $\bm \sigma^* = \bm \sigma_M$ (see \eqref{def: a sigma truncated at M, LDP} for definitions).
\elaborate{
\begin{align*}
    a^*(x) \delequal
    \begin{cases}
     a(M) & \text{ if }x > M, \\
     a(-M) & \text{ if }x < -M, \\
     a(x) & \text{ otherwise.}
    \end{cases}
    \ \ \ \ \ \ \ \ \ \ \ 
    \sigma^*(x) \delequal
    \begin{cases}
     \sigma(M) & \text{ if }x > M, \\
     \sigma(-M) & \text{ if }x < -M, \\
     \sigma(x) & \text{ otherwise.}
    \end{cases}
\end{align*}
}
Fix the constant
$$C^* \delequal \sup_{\bm y: \norm{\bm y}\leq M}\norm{\bm a(y)}\vee\norm{\bm \sigma(y)} \vee 1 < \infty.$$
We also write
$
h^* = \bar h^{(k)|b}_{M\downarrow}
$
in this proof;
see \eqref{def: perturb ode mapping h k b, truncated at M, 1}--\eqref{def: perturb ode mapping h k b, truncated at M, 3} for definitions.
The choice of $M$ ensures that $\xi^* = h^*(\bm x^*,\textbf W^*,\textbf V^*,\bm{t}^*)$
and, under condition~\eqref{condition, arguments for xi prime, lemma: continuity of h k b mapping clipped}, $\xi^\prime = h^*(\bm x^\prime, \textbf W^\prime, \textbf V^\prime, \bm t^\prime)$.
\elaborate{
Given any 
$x_0 \in \mathbb{R}$,
$\bm{w} = (w_1,\cdots,w_k) \in \mathbb{R}^{k}$, and any $\bm{t} = (t_1,\cdots,t_k)\in (0,1)^{k\uparrow}$, let $\xi = h^*(x_0,\bm{w},\bm{t})$ be the solution to
\begin{align*}
    \xi(0) & = x_0; 
    \\
    \frac{d\xi(s)}{d s} & = a^*\big( \xi(s) \big),\ \ \ \forall s \in [0,t],\ s \neq t_1,\cdots,t_k;
    \\
    \xi(s) & = \xi(s-) + \varphi_b\Big(\sigma^*\big( \xi(s-) \big)\cdot w_j\Big)\ \ \ \text{ if }s = t_j\text{ for some }j\in[k].
\end{align*}
}

Let $\rho \delequal \exp(D) \geq 1$ where $D\in [1,\infty)$ is the Lipschitz coefficient in Assumption \ref{assumption: lipschitz continuity of drift and diffusion coefficients}.
Let $R_0 = 1$,
\begin{align}
    R_{j} \delequal (7C^* + \rho R_{j-1} + 1)(DM + 1) + C^*\qquad \forall j \geq 1.
    \label{proof, def constant R j, lemma: continuity of h k b mapping clipped}
\end{align}
We pick some $\tilde{\delta} > 0$ small enough such that
\begin{align}
    2\tilde{\delta} < 1 \wedge \epsilon;\qquad
    R_{k+1}\tilde{\delta}< \epsilon.
    \label{choice of delta, proof, lemma: continuity of h k b mapping}
\end{align}
Also, by picking $\delta > 0$ small enough,
it is guaranteed that (under convention $t^*_0 = t^\prime_0 =  0,\ t^*_{k+1} = t^\prime_{k+1} = 1$)
\begin{align}
    \delta < \widetilde{\delta} \vee 1;
    \qquad
    \max_{j \in [k]}\Bigg|\frac{t^*_{j+1} - t^*_j}{t^\prime_{j+1} - t^\prime_j} - 1\Bigg| < \widetilde{\delta}
    \ \forall \bm t^\prime = (t^\prime_1,\cdots,t^\prime_k) \in (0,1)^{k\uparrow},\ 
    \max_{j \in [k]}|t^\prime_j - t^*_j| < \delta.
    \label{choice of delta, 2, proof, lemma: continuity of h k b mapping}
\end{align}
\elaborate{
for the fixed $\bm{t}^* = (t^*_1,\cdots,t^*_k) \in (0,1)^{k \uparrow}$,
let $t_\Delta = \min\{ t^*_{j + 1} - t^*_j:\ j = 1,2,\cdots,k \} > 0$
with the convention that 
$t^*_0 = t^\prime_0 =  0,\ t^*_{k+1} = t^\prime_{k+1} = 1$.
For
and $\delta \in (0,t_\Delta)$
and
any $(t^\prime_1,\cdots,t^\prime_k) \in (0,1)^{k \uparrow}$
satisfying 
$|t^\prime_j - t^*_j| < \delta\ \forall j \in[k]$,
we have 
$
\frac{t_\Delta}{ t_\Delta + \delta }
\leq
\frac{t^*_{j+1} - t^*_j}{t^\prime_{j+1} - t^\prime_j} \leq \frac{t_\Delta}{ t_\Delta - \delta }
$
for any $j \in [k]$.
Therefore, we can pick $\delta > 0$ small enough to ensure that
\begin{align}
    \delta < \widetilde{\delta};
    \qquad
    \Big|\frac{t^*_{j+1} - t^*_j}{t^\prime_{j+1} - t^\prime_j} - 1\Big| < \widetilde{\delta}
    \ \forall j \in [k].
   \nonumber
\end{align}
}
Now it only remains to show that, under the current the choice of $\delta$, the bound $\bm{d}_{J_1}(\xi,\xi^\prime) < \epsilon$ follows from condition \eqref{condition, arguments for xi prime, lemma: continuity of h k b mapping clipped}.
To do so, we fix some $\xi^\prime$ satisfying condition \eqref{condition, arguments for xi prime, lemma: continuity of h k b mapping clipped}.
Define
$\lambda: [0,1] \to [0,1]$ as
\begin{align*}
    \lambda(u) = 
    \begin{cases}
     0 & \text{ if }u = 0\\ 
     t^*_j +  \frac{t^*_{j+1} - t^*_j}{t^\prime_{j+1} - t^\prime_j}\cdot(u - t^\prime_j) & \text{ if }u \in (t^\prime_{j},t^\prime_{j+1}]\text{ for some }j = 0,1,\cdots,k.
    \end{cases}
\end{align*}
For any $u \in (0,1)$, let $j \in \{0,1,\cdots,k\}$ be such that $u \in (t^\prime_{j},t^\prime_{j+1}]$. Observe that
\begin{align}
    |\lambda(u) - u| & = 
    \bigg| t^*_j +  \frac{t^*_{j+1} - t^*_j}{t^\prime_{j+1} - t^\prime_j}\cdot(u - t^\prime_j) - u  \bigg|
    =
     \bigg| t^*_j +  \frac{t^*_{j+1} - t^*_j}{t^\prime_{j+1} - t^\prime_j}\cdot v - (v + t^\prime_j)  \bigg|\ \ \ \text{with $v \delequal u - t^\prime_j$}
     \nonumber
     \\
     & \leq 
     | t^*_j - t^\prime_j | + 
     \bigg|
      \frac{t^*_{j+1} - t^*_j}{t^\prime_{j+1} - t^\prime_j} - 1
     \bigg|\cdot v
     \nonumber
     \\
     & \leq \widetilde{\delta} + \widetilde{\delta}\cdot 1 < \epsilon.
     \label{ineq 1, proof, lemma: continuity of h k b mapping}
\end{align}
In summary, we have shown that $\sup_{ u \in [0,1] }|\lambda(u) - u| < \epsilon$.
Moving on, we prove that $$\sup_{u \in [0,1]}\norm{\xi^*\big(\lambda(u)\big) - \xi^\prime(u)} < \epsilon$$
using an inductive argument.
First,
Assumption \ref{assumption: lipschitz continuity of drift and diffusion coefficients}
allows us to apply
Gronwall's inequality and get
$
\sup_{ u \in (0,t^*_1 \wedge t^\prime_1) }\norm{ \xi^*(u) - \xi^\prime(u)} \leq \exp\big(D\cdot(t^*_1 \wedge t^\prime_1) \big)\norm{ \bm x^* - \bm x^\prime}\leq \rho \delta.
$
As a result, for any $u \in (0,t^*_1 \wedge t^\prime_1)$,
\begin{align*}
   \norm{\xi^*\big( \lambda(u) \big ) - \xi^\prime(u)}
   & = 
   \norm{
   \xi^*\bigg( \frac{t^*_1}{t^\prime_1} \cdot u \bigg) - \xi^\prime(u)
   }
   \leq 
   \norm{
   \xi^*\bigg( \frac{t^*_1}{t^\prime_1} \cdot u \bigg) - \xi^*(u)
   }
   +
   \norm{\xi^\prime(u) - \xi^*(u)}
   \\
   & \leq 
   \norm{
   \xi^*\bigg( \frac{t^*_1}{t^\prime_1} \cdot u \bigg) - \xi^*(u)
   }
   + \rho\delta
   \\
   & \leq 
   \sup_{\bm y \in \R^m}\norm{\bm a^*(\bm y)} \cdot \bigg| \frac{t^*_1}{t^\prime_1} - 1 \bigg|\cdot u + \rho \delta
   \qquad\text{by \eqref{proof, property of M, bound on xi and xi prime, lemma: continuity of h k b mapping clipped}}
   \\
   & \leq 
   C^* \widetilde{\delta} + \rho \widetilde{\delta} = (C^* + \rho)\widetilde{\delta}
   \qquad \text{due to \eqref{choice of delta, 2, proof, lemma: continuity of h k b mapping}}.
\end{align*}
In case that $t^\prime_1 \leq t^*_1$, we get 
$\sup_{ u \in (0,t^\prime_1) } \norm{\xi^*\big( \lambda(u) \big ) - \xi^\prime(u)} < (C^* + \rho)\widetilde{\delta}$
directly.
In case that $t^*_1 < t^\prime_1$,
due to $\xi^\prime = h^*(\bm x^\prime,\textbf W^\prime,\textbf V^\prime,\bm{t}^\prime)$ 
as well as the bounds in \eqref{choice of delta, 2, proof, lemma: continuity of h k b mapping}\eqref{ineq 1, proof, lemma: continuity of h k b mapping},
for any $u \in [t^*_1,t^\prime_1)$ we have
\begin{align*}
    \norm{\xi^\prime(u) - \xi^\prime(t^*_1) } & \leq \sup_{\bm y \in \R^m}\norm{\bm a^*(\bm y) } \cdot |u - t^*_1| < C^*\widetilde{\delta};
    \\
    \norm{ \xi^*\big(\lambda(u)\big) - \xi^*\big(\lambda(t^*_1)\big) }
    & \leq 
     \sup_{\bm y \in \R^m}\norm{ \bm a^*(\bm y)} \cdot \big| \lambda(u) - \lambda(t^*_1)\big|
      < 5C^*\widetilde{\delta}.
\end{align*}
As a result,
$\sup_{ u \in (0,t^\prime_1) } \norm{ \xi^*\big( \lambda(u) \big ) - \xi^\prime(u) } < (7C^* + \rho)\widetilde{\delta}.$
In addition,
due to $\norm{\varphi_b(\bm x) - \varphi_b(\bm y)} \leq \norm{\bm x-\bm y}$,
\begin{align*}
    & \norm{\xi^*\big( \lambda(t^\prime_1) \big ) - \xi^\prime(t^\prime_1)}
    \\
    & =
    \norm{ 
        \xi^*\big( \lambda(t^\prime_1 -) \big ) + \bm v^*_1 +\varphi_b\bigg(\bm \sigma^*\Big(\xi^*\big( \lambda(t^\prime_1 -) \big ) + \bm v^*_1\Big)\bm w^*_1\bigg) 
        - \xi^\prime(t^\prime_1-) - \bm v^\prime_1 -  \varphi_b\bigg(\bm \sigma^*\Big(\xi^\prime(t^\prime_1-) + \bm v^\prime_1\Big)\bm w^\prime_1\bigg)
    }
    \\
    & \leq 
    \norm{ \xi^*\big( \lambda(t^\prime_1 -) \big ) - \xi^\prime(t^\prime_1-)}
    +
    \norm{\bm v^*_1 - \bm v^\prime_1}
    +
    \norm{
        \bm \sigma^*\Big(\xi^*\big( \lambda(t^\prime_1 -) \big ) + \bm v^*_1\Big)\bm w^*_1 -\bm \sigma^*\Big(\xi^\prime(t^\prime_1-) + \bm v^\prime_1\Big)\bm w^\prime_1
    }
    \\
    & \leq 
    \norm{ \xi^*\big( \lambda(t^\prime_1 -) \big ) - \xi^\prime(t^\prime_1-) }
    +
    \norm{\bm v^*_1 - \bm v^\prime_1}
    +
    \norm{
        \bm \sigma^*\Big(\xi^*\big( \lambda(t^\prime_1 -)\big) + \bm v^*_1\Big) - \bm \sigma^*\Big(\xi^\prime(t^\prime_1-) + \bm v^\prime_1\Big)
    }\cdot \norm{ \bm w^*_1 }
    \\ 
    &\qquad\qquad\qquad\qquad\qquad\qquad\qquad\qquad\qquad\qquad
    +
    \norm{\bm \sigma^*\Big(\xi^\prime(t^\prime_1-) + \bm v^\prime_1\Big)}\cdot \norm{ \bm w^\prime_1 - \bm w^*_1}
    \\
    & \leq 
    \norm{ \xi^*\big( \lambda(t^\prime_1 -) \big ) - \xi^\prime(t^\prime_1-) }
    + \delta 
    +
    \norm{ \bm \sigma^*\Big(\xi^*\big( \lambda(t^\prime_1 -)\big) + \bm v^*_1\Big) - \bm \sigma^*\Big(\xi^\prime(t^\prime_1-) + \bm v^\prime_1\Big)}\cdot M
    + C^*\delta
    \\
    & \leq 
    (7C^* + \rho)\widetilde{\delta} + \delta 
    +
    M \cdot D \cdot 
    \Big(
        \norm{\xi^*\big( \lambda(t^\prime_1 -)\big) -\xi^\prime(t^\prime_1-) }
        +
        \norm{\bm v^*_1 - \bm v^\prime_1}
    \Big)
    + C^*\delta
    \qquad \text{due to Assumption \ref{assumption: lipschitz continuity of drift and diffusion coefficients}}
    \\
    & = 
    (7C^* + \rho)\widetilde{\delta} + \delta 
    + 
    DM\big((7C^* + \rho)\widetilde{\delta} + \delta \big) + C^*\delta
    \\
    & \leq \big[ (7C^* + \rho + 1)(DM + 1) + C^* \big]\widetilde{\delta}
    \qquad \text{by our choice of }\delta < \widetilde{\delta}\text{ in \eqref{choice of delta, proof, lemma: continuity of h k b mapping}\eqref{choice of delta, 2, proof, lemma: continuity of h k b mapping}.}
\end{align*}
In summary, we yield
$\sup_{ u \in [0,t^\prime_1] } \norm{ \xi^*\big( \lambda(u) \big ) - \xi^\prime(u) } \leq \big[ (7C^* + \rho + 1)(DM + 1) + C^* \big]\widetilde{\delta} = R_1 \widetilde{\delta};
$
see definitions in \eqref{proof, def constant R j, lemma: continuity of h k b mapping clipped}.
Now, suppose that for some $j = 1,2,\cdots,k$,
we have 
$
\sup_{ u \in [0,t^\prime_j] } \norm{ \xi^*\big( \lambda(u) \big ) - \xi^\prime(u) } \leq
    R_j\widetilde{\delta}.
$
By repeating the calculations above, one can obtain that
$
 \sup_{ u \in [0,t^\prime_{j+1}] } \norm{ \xi^*\big( \lambda(u) \big ) - \xi^\prime(u) } \leq R_{j+1}\widetilde{\delta}.
$
To conclude, note that $R_{k+1}\widetilde{\delta} < \epsilon$ by our choice of parameters in \eqref{choice of delta, proof, lemma: continuity of h k b mapping}.
\elaborate{
For any $v \in \big[0,(t^\prime_{j+1} \wedge t^*_{j+1})  - t^\prime_j\big)$,
\begin{align*}
   \big| \xi^*\big( \lambda(t^\prime_j + v ) \big ) - \xi^\prime(t^\prime_j + v ) \big|
   & \leq 
   \bigg| 
   \xi^*\big( \lambda(t^\prime_j + v ) \big ) - \xi^*(t^\prime_j + v)
   \bigg|
   +
   \big|\xi^*(t^\prime_j + v) - \xi^\prime(t^\prime_j + v)\big|
   \\
& \leq 
\bigg| 
   \xi^*\big( \lambda(t^\prime_j + v ) \big ) - \xi^*(t^\prime_j + v)
   \bigg|
   +
 \rho R_j \widetilde{\delta}
 \qquad \text{Using Gronwall's inequality}
   \\
   & \leq 
   \sup_{x \in \R}|a^*(x)| \cdot |\lambda(t^\prime_j + v ) - (t^\prime_j + v)|
   +
   \rho R_j \widetilde{\delta}
   \\
   & \leq 
   2C^*\widetilde{\delta} + \rho R_j\widetilde{\delta}\qquad\ \text{due to \eqref{ineq 1, proof, lemma: continuity of h k b mapping}.}
\end{align*}
Again,
in case that $t^\prime_{j+1} \leq t^*_{j+1}$, we already get 
$\sup_{ u \in (0,t^\prime_{j+1}) } \big| \xi^*\big( \lambda(u) \big ) - \xi^\prime(u) \big| < \big(5C + \rho R_j\big)\widetilde{\delta}.$
In case that $t^*_{j+1} < t^\prime_{j+1}$,
note that for any $u \in [t^*_{j+1},t^\prime_{j+1})$,
one can apply properties \eqref{choice of delta, 2, proof, lemma: continuity of h k b mapping}\eqref{ineq 1, proof, lemma: continuity of h k b mapping} to yield
\begin{align*}
    \big| \xi^\prime(u) - \xi^\prime(t^*_{j+1}) \big| & \leq \sup_{x \in \R}|a^*(x)| \cdot |u - t^*_{j+1}| < C^*\widetilde{\delta};
    \\
    \big| \xi^*\big(\lambda(u)\big) - \xi^*\big(\lambda(t^*_{j+1})\big) \big|
    & \leq 
     \sup_{x \in \R}|a^*(x)| \cdot \big| \lambda(u) - \lambda(t^*_{j+1})\big|
     < 2C^*\widetilde{\delta}.
\end{align*}
In summary, we get 
$\sup_{ u \in (0,t^\prime_{j+1}) } \big| \xi^*\big( \lambda(u) \big ) - \xi^\prime(u) \big| < \big(5C^* + \rho R_j\big)\widetilde{\delta}.$
Lastly, in case that $j = k+1$ (so $t^\prime_j = t^\prime_{k+1} = t_j =  t_{k+1} = 1$),
we have 
$
\Big| \xi^*( 1) - \xi^\prime(1) \Big|
\leq
\limsup_{ t \uparrow 1}\Big| \xi^*\big( \lambda(t) \big ) - \xi^\prime(t) \Big| \leq 
\big(5C^* + \rho R_j\big)\widetilde{\delta} \leq R_{j+1}\widetilde{\delta}.
$
In case that $j \leq k$,
using $|\varphi_b(x) - \varphi_b(y)| \leq |x-y|$,
\begin{align*} 
    & \Big| \xi^*\big( \lambda(t^\prime_{j+1}) \big ) - \xi^\prime(t^\prime_{j+1}) \Big|
    \\
    & =
    \bigg| \xi^*\big( \lambda(t^\prime_{j+1} -) \big ) + \varphi_b\bigg(\sigma^*\Big(\xi^*\big( \lambda(t^\prime_{j+1} -) \big )\Big)w^*_{j+1}\bigg)  - \xi^\prime(t^\prime_{j+1}-) - \varphi_b\bigg(\sigma^*\Big(\xi^\prime(t^\prime_{j+1}-)\Big)w^\prime_{j+1}\bigg) \bigg|
    \\
    & \leq 
    \Big|\xi^*\big( \lambda(t^\prime_{j+1} -) \big ) - \xi^\prime(t^\prime_{j+1}-)\Big|
    +
    \Big|
    \sigma^*\Big(\xi^*\big( \lambda(t^\prime_{j+1} -) \big )\Big)w^*_{j+1}
    -
    \sigma^*\Big(\xi^\prime(t^\prime_{j+1}-)\Big)w^\prime_{j+1}
    \Big|
    \\
    & \leq 
    \Big|\xi^*\big( \lambda(t^\prime_{j+1} -) \big ) - \xi^\prime(t^\prime_{j+1}-)\Big|
    +
    \Big|\sigma^*\Big(\xi^*\big( \lambda(t^\prime_{j+1} -)\big)\Big) - \sigma^*\Big(\xi^\prime(t^\prime_{j+1}-)\Big)\Big|\cdot |w^*_{j+1}|
    \\ 
    &\qquad
    +
    \Big|\sigma^*\Big(\xi^\prime(t^\prime_{j+1}-)\Big)\Big|\cdot |w^\prime_{j+1} - w^*_{j+1}|
    \\
    & < 
    \Big|\xi^*\big( \lambda(t^\prime_{j+1} -) \big ) - \xi^\prime(t^\prime_{j+1}-)\Big|
    +
    \Big|\sigma^*\Big(\xi\big( \lambda(t^\prime_{j+1} -)\big)\Big) - \sigma^*\Big(\xi^\prime(t^\prime_{j+1}-)\Big)\Big|\cdot M
    + C^*\delta
    \\
    & \leq 
    \big(5C^* + \rho R_j\big)\widetilde{\delta}
    +
   \big(5C^* + \rho R_j\big)\widetilde{\delta} \cdot D \cdot M + C^*\delta
    \qquad\text{because of Assumption \ref{assumption: lipschitz continuity of drift and diffusion coefficients}}
    \\
    & = \Big[ \big(5C^* + \rho R_j\big)(DM + 1) + C^* \Big]\widetilde{\delta}
    \leq 
    \big(6C^* + \rho R_j\big)(DM + 1)\widetilde{\delta}
    \\
    & = 6C^*(DM + 1)\widetilde{\delta} + \rho (DM+1)R_j \widetilde{\delta}
    \leq
    \rho (DM+1)R_j \widetilde{\delta} + \rho (DM+1)R_j \widetilde{\delta}
    \\
    & = 2 \rho (DM+1)R_j \widetilde{\delta} = 
    2^{j}\rho^j(DM+1)^{j+1}(6C^* + \rho)\widetilde{\delta}= R_{j+1}\widetilde{\delta},
\end{align*}
and hence
$
 \sup_{ u \in [0,t^\prime_{j+1}] } \big| \xi^*\big( \lambda(u) \big ) - \xi^\prime(u) \big| \leq R_{j+1}\widetilde{\delta}.
$
By arguing inductively, we yield
$
 \sup_{ u \in [0,1] } \big| \xi^*\big( \lambda(u) \big ) - \xi^\prime(u) \big| \leq R_{k+1}\widetilde{\delta}< \epsilon
$
due to our choice of $\widetilde{\delta}$ in \eqref{choice of delta, proof, lemma: continuity of h k b mapping}.
Combining this bound with \eqref{ineq 1, proof, lemma: continuity of h k b mapping},
we get $
\bm{d}_{J_1}(\xi^*,\xi^\prime) < \epsilon
$
and conclude the proof.
}
\end{proof}

\begin{lemma}
\label
{lemma: continuity of h k b mapping}
\linksinthm
{lemma: continuity of h k b mapping}%
Let Assumption \ref{assumption: lipschitz continuity of drift and diffusion coefficients} and \ref{assumption: boundedness of drift and diffusion coefficients} hold.
Given any $k \in \mathbb N$ and $T > 0$, the mapping
$\bar h^{(k)}_{[0,T]}$ is continuous on $\R^m \times \R^{d \times k} \times \R^{m \times k} \times (0,T)^{k\uparrow}$.
\end{lemma}

\begin{proof}
\linksinpf
{lemma: continuity of h k b mapping}%
To ease notations we focus on the case where $T = 1$, but the proof is identical for any $T > 0$.
Arbitrarily pick some $k \in \mathbb N$,
$\bm x^* \in \R^m$, $\textbf W^* = (\bm w^*_1,\cdots,\bm w_k^*) \in \R^{d \times k}$,
$\textbf V^* = (\bm v^*_1,\cdots,\bm v^*_k) \in \R^{m \times k}$, and $\bm{t}^* = (t^*_1,\cdots,t^*_k) \in (0,1)^{k\uparrow}$.
We claim the existence of some $b = b(\bm x^*,\textbf W^*,\textbf V^*,\bm t^*) > 0$ such that for any $\delta \in (0,1)$, $\bm x^\prime \in \R^m$, $\textbf W^\prime = (\bm w^\prime_1,\cdots,\bm w^\prime_k) \in \R^{d \times k}$,
$\textbf V^\prime = (\bm v^\prime_1,\cdots,\bm v^\prime_k) \in \R^{m \times k}$, and $\bm t^\prime \in (0,1)^{k\uparrow}$ satisfying 
\begin{align}
    \norm{ \bm x^* - \bm x^\prime } < \delta, \qquad 
    \norm{ \bm w^\prime_j - \bm w^*_j} \vee \norm{\bm v^*_j - \bm v^\prime_j} \vee |t^\prime_j - t^*_j| < \delta\ \forall j \in [k].
    \label{condition, x prime and x, lemma: continuity of h k b mapping}
\end{align}
we have $\bar h^{(k)}(\bm x^\prime,\textbf W^\prime,\textbf V^\prime,\bm t^\prime) = \bar h^{(k)|b}(\bm x^\prime,\textbf W^\prime,\textbf V^\prime,\bm t^\prime)$.
Then the continuity of $\bar h^{(k)}$ follows immediately from the continuity of $\bar h^{(k)|b}$ established in Lemma~\ref{lemma: continuity of h k b mapping clipped}.

Now, it only remains to find such $b > 0$.
In particular, we can simply set 
$b = C \cdot \big(\max\{ \norm{\bm w^*_j}:\ j \in [k] \} + 1\big)$
where $C \geq 1$ is the constant in Assumption \ref{assumption: boundedness of drift and diffusion coefficients} satisfying
$\sup_{\bm y \in \R^m}\norm{\bm \sigma(\bm y)} \leq C$.
Indeed, given any $\delta \in (0,1)$ and $\bm x^\prime \in \R^m$, $\textbf W^\prime \in \R^{d\times k}$,
$\textbf V^\prime \in \R^{m \times k}$,
and $\bm t^\prime \in (0,1)^{k\uparrow}$ satisfying \eqref{condition, x prime and x, lemma: continuity of h k b mapping},
for 
$
\xi^\prime = \bar h^{(k)|b}(\bm x^\prime,\textbf W^\prime,\textbf V^\prime,\bm{t}^\prime)
$
we have 
$$\norm{\bm \sigma\big(\xi^\prime(t^\prime_j-) + \bm v_j\big)\bm w^\prime_j} \leq C \cdot \big(\max\{ \norm{\bm w^*_i}:\ i \in [k] \} + \delta\big) < b
\qquad \forall j \in [d].
$$
As a result, the truncation operator $\varphi_b$ at step \eqref{def: perturb ode mapping h k b, 3} is not in effect,
and hence $\xi^\prime = \bar h^{(k)}(\bm x^\prime,\textbf W^\prime,\textbf V^\prime,\bm{t}^\prime)$.
This concludes the proof.
\elaborate{
Given $\epsilon > 0$,
It suffices to show the existence of some $\delta > 0$ such that 
$
\bm{d}_{J_1}(\xi^*,\xi^\prime) < \epsilon
$
holds for any $\xi^\prime = h^{(k)}(x^\prime,\bm{w}^\prime,\bm{t}^\prime)$
with $x^\prime \in \R,\ \bm{w}^\prime = (w^\prime_1,\cdots,w^\prime_k) \in \R^k,\ \bm{t}^\prime = (t^\prime_1,\cdots,t^\prime_k) \in (0,1)^{k \uparrow}$
satisfying
\begin{align}
    |x^*- x^\prime| < \delta;\ \ \ \ |w^\prime_j - w^*_j|\vee |t^\prime_j - t^*_j| < \delta\ \forall j \in [k].
    \nonumber
\end{align}
In particular, let $b = C \cdot \big(\max\{ |w^*_j|:\ j \in [k] \} + 1\big)$
where $C \geq 1$ is the constant in Assumption \ref{assumption: boundedness of drift and diffusion coefficients} satisfying
$\sup_{x \in \R}|\sigma(x)| \leq C$.
As a result,
we have $|\xi^*(t^*_j-)w^*_j| < b$ for any $j \in [k]$,
thus implying $\xi^* = h^{(k)}(x^*,\bm{w}^*,\bm{t}^*) = h^{(k)|b}(x^*,\bm{w}^*,\bm{t}^*)$.
Besides, for any $j \in [k]$, due to $|w^\prime_j - w^*_j| < \delta < 1$ we get $|\xi^\prime(t^\prime_j-)w^\prime_j| \leq C \cdot \big(\max\{ |w^*_j|:\ j \in [k] \} + \delta\big) < b$.
Therefore,
for any
 $x^\prime \in \R,\ \bm{w}^\prime = (w^\prime_1,\cdots,w^\prime_k) \in \R^k,\ \bm{t}^\prime = (t^\prime_1,\cdots,t^\prime_k) \in (0,1)^{k \uparrow}$
 satisfying \eqref{condition, x prime and x, lemma: continuity of h k b mapping},
 we have 
 $
 \xi^\prime = h^{(k)}(x^\prime,\bm{w}^\prime,\bm{t}^\prime) = h^{(k)|b}(x^\prime,\bm{w}^\prime,\bm{t}^\prime).
 $
By appealing to the continuity of $h^{(k)|b}$ (see Lemma \ref{lemma: continuity of h k b mapping clipped}), we conclude the proof.
}
\end{proof}

Next, we move onto the proofs of Lemmas \ref{lemma: LDP, bar epsilon and delta}, \ref{lemma: LDP, bar epsilon and delta, clipped version}, and \ref{lemma: convergence from C k b measure to C k measure}.

\begin{proof}[Proof of Lemma~\ref{lemma: LDP, bar epsilon and delta}]
\linksinpf{lemma: LDP, bar epsilon and delta}%
The claims are trivial if $A$ or $B$ is an empty set.
Also, the claims are trivially true if $k = 0$
(note that in $(b)$ we would have $\mathbb{D}^{(-1)}_A(r) = \emptyset$).
Therefore, in this proof we focus on the case where $A\neq \emptyset$, $B \neq \emptyset$, and $k \geq 1$.

Since $B$ is bounded away from $\mathbb{D}_{A}^{(k-1)}(r)$ under $\dj{}$,
there exists $\bar{\epsilon} > 0$ such that
$\bm{d}_{J_1}\big(B^{\bar{\epsilon}},\mathbb{D}_{A}^{(k - 1)}(r)\big) > 0$
so that part $(b)$ is satisfied. 
Next, we show that there exists $\bar\delta > 0$, which together with $\bar\epsilon$ satisfies $(a)$.
Let $D \in[1,\infty)$ be the Lipschitz coefficient in Assumption \ref{assumption: lipschitz continuity of drift and diffusion coefficients}.
Besides, recall the constant $C\in [1,\infty)$ in Assumption \ref{assumption: boundedness of drift and diffusion coefficients}
that satisfies $\sup_{\bm x \in \mathbb{R}^m}\norm{\bm \sigma(\bm x)} \leq C$.
Let $\rho \delequal \exp(D)$ and 
\begin{align}
    \bar{\delta} \delequal \frac{\bar{\epsilon}}{\rho C + 1}.
    \label{proof: choose bar delta, lemma LDP, bar epsilon and delta}
\end{align}
Note that $\bar\delta < \bar\epsilon$.
To show that the claim $(a)$ holds for such $\bar\epsilon$ and $\bar\delta$, we proceed with proof by contradiction.
Suppose that there is some 
$\bm{t} = (t_1,\cdots,t_k) \in (0,1]^{k\uparrow}$, $\textbf W = (\bm w_1,\cdots,\bm w_k) \in \R^{d\times k}$, and $\bm x_0 \in A$
such that
$\xi \delequal h^{(k)}(\bm x_0,\textbf W, \bm{t}) \in {B^{\bar{\epsilon}}}$
yet $\norm{\bm w_J}\leq \bar{\delta}$ for some $J = 1,2,\cdots,k$.
We construct $\xi' \in \mathbb{D}_{A}^{(k - 1)}(r)$ such that $\bm d_{J_1}(\xi', \xi) < \bar\epsilon$. 
Specifically, we focus on the case where $J<k$, since the proof when $J=k$ is almost identical but only slightly simpler.
Define $\xi'$ by
\begin{equation}
\xi'(s) \delequal 
\begin{cases}
\xi(s) & s \in [0,t_J)\\
h^{(0)}(\xi'(t_J-))(s-t_J) & s \in [t_J, t_{J+1})\\
\xi(s) & s \in [t_{J+1},t].
\end{cases}
\nonumber
\end{equation}
That is, $\xi'$ is driven by the same ODE as $\xi$ on $[t_J, t_{J+1})$, except that at the beginning of the intervals, $\xi'$ starts from $\xi(t_J-)$ instead of $\xi(t_J)$.
On the other hand, $\xi'$ coincides with $\xi$ outside of $[t_J,t_{J+1})$.
To bound the distance between $\xi$ and $\xi'$, note that 
from Assumption~\ref{assumption: boundedness of drift and diffusion coefficients}, we have
$\norm{ \xi(t_J) - \xi(t_J-)} = \norm{\bm \sigma(\xi(t_J-)) \bm w_J}\leq C \bar\delta$. 
Then using Gronwall's inequality, we get
\begin{align}
\norm{\xi(s) - \xi^\prime(s)} 
&
\leq \exp\big( (t_{J+1} - t_J)D \big)\norm{ \xi(t_J) - \xi'(t_J-) }
\nonumber\\
&
\leq \rho\norm{ \xi(t_J) - \xi'(t_J-) }
\nonumber\\
&
\leq \rho C \bar \delta 
< \bar{\epsilon}
\label{bound of difference between xi and xi prime between t_J and t_J+1}
\end{align}
for all $s \in [t_J, t_{J+1})$.
This shows that $\bm d_{J_1} (\xi, \xi') < \bar \epsilon$. 
However, this cannot be the case since $\xi  \in B^{\bar\epsilon}$, $\xi' \in \mathbb D^{(k-1)}_A(r)$, and we chose $\bar\epsilon$ such that $\bm{d}_{J_1}(B^{\bar{\epsilon}},\mathbb{D}_{A}^{(k - 1)}) > 0$.
This concludes the proof for the case with $J < k$.
The proof for the case where $J = k$ is almost identical.
The only difference is that $\xi^\prime$ is set to be $\xi^\prime(s) = \xi(s)$ for all $s < t_k$, and $\xi^\prime(s) = h^{(0)}\big( \xi^\prime(t_k-)\big)(s - t_k)$ for all $s \in [t_k,1]$,
\end{proof}


\begin{proof}[Proof of Lemma~\ref{lemma: LDP, bar epsilon and delta, clipped version}]
\linksinpf{lemma: LDP, bar epsilon and delta, clipped version}
Similar to Lemma~\ref{lemma: LDP, bar epsilon and delta}, all claims hold trivially if $A$ or $B$ is empty, or if $k = 0$.
In this proof, we focus on the case where $A\neq \emptyset$, $B \neq \emptyset$, and $k \geq 1$.

We start by fixing some constant.
Since $B$ is bounded away from $\D^{(k-1)|b}_A(r)$,
we can fix some $\bar\epsilon > 0$ such that 
$\bm{d}_{J_1}\big(B^{\bar\epsilon},\mathbb{D}_{A}^{(k - 1)|b}(r)\big) > 0$ to conclude the proof of part $(b)$.
Next,
let $D \in[1,\infty)$ be the Lipschitz coefficient in Assumption \ref{assumption: lipschitz continuity of drift and diffusion coefficients}.
Besides, recall the constant $C\in [1,\infty)$ in Assumption \ref{assumption: boundedness of drift and diffusion coefficients}
that satisfies $\sup_{\bm x \in \mathbb{R}^m}\norm{\bm \sigma(\bm x)} \leq C$.
Let $\rho \delequal \exp(D)$.
By picking an even smaller $\bar\epsilon > 0$ if necessary, we can w.l.o.g.\ assume that
\begin{align}
    2\rho\bar\epsilon < r\qquad\text{and}\qquad
    \dj{}\big(B^{\bar\epsilon},\D^{(k-1)|b}_A(r)\big) > 2\rho\bar\epsilon.
    \label{proof: choose bar epsilon, lemma: LDP, bar epsilon and delta, clipped version}
\end{align}
Let
\begin{align}
    \bar{\delta} \delequal \bar\epsilon/C.
    \label{proof: choose bar delta, lemma: LDP, bar epsilon and delta, clipped version}
\end{align}
To prove that part $(a)$ holds for such $\bar\delta$, we proceed with a proof by contradiction.
Arbitrarily pick some $\bm x \in A$, $\textbf W = (\bm w_1,\cdots,\bm w_k) \in \R^{d\times k}$, $\textbf V = (\bm v_1,\cdots,\bm v_k) \in \R^{m \times k}$ with $\max_{j \in [k]}\norm{\bm v_j} \leq \bar\epsilon$, $\bm t = (t_1,\cdots,t_k) \in (0,1)^{k\uparrow}$, and $b > 0$.
For $\xi_b = \bar h^{(k)|b}(\bm x,\textbf W, \textbf V, \bm t)$,
suppose that $\xi_b \in B^{\bar\epsilon}$ yet there is some $J \in [k]$ such that $\norm{\bm w_J} \leq \bar \delta$.
Next, construct $\xi \in \mathbb D$ as follows: (recall that $\bm y_\cdot(x)$ is the ODE defined in \eqref{def ODE path y t})
\begin{equation}
\xi(s) \delequal
\begin{cases}
\xi_b(s) & s \in [0,t_J)\\
\bm y_{s - t_J}(\xi(t_J-)) & s \in [t_J, t_{J+1})\\
\xi_b(s) & s \in [t_{J+1},1].
\end{cases}
\nonumber
\end{equation}
That is, $\xi$ is a modified version of $\xi_b$ where the jump at time $t_J$ is removed, but the two paths coincide on $[0,t_J)\cup[t_{J+1},1]$.
Note that by Assumption~\ref{assumption: boundedness of drift and diffusion coefficients},
$$
\norm{\xi(t_J) - \xi_b(t_J)} = \norm{\Delta \xi_b(t_J)} \leq  \norm{ \bm v_J } +  \norm{ \varphi_b\Big(\bm \sigma\big( \xi_b(t_J-) + \bm v_J \big) \bm w_J \Big)}
\leq 
\bar\epsilon + C\bar\delta.
$$
Applying Gronwall's inequality, we then yield that for all $s \in [t_J,t_{J-1})$,
\begin{align*}
    \norm{\xi_b(s) - \xi(s)}
    & \leq \exp\big( D (s - t_J)\big) \cdot \norm{\xi(t_J) - \xi_b(t_J)}
    \\ 
    & \leq \rho \cdot \norm{\xi(t_J) - \xi_b(t_J)}
    \qquad \text{ where }\rho = \exp(D)
    \\ 
    & \leq \rho (\bar\epsilon + C\bar\delta) = 2\rho\bar\epsilon
    \qquad
    \text{due to \eqref{proof: choose bar delta, lemma: LDP, bar epsilon and delta, clipped version}}.
\end{align*}
This implies that $\bm d_{J_1}(\xi,\xi_b) \leq 2\rho \bar\epsilon$
and
$
\xi \in \D^{(k-1)|b}_A(2\rho\bar\epsilon) \subseteq \D^{(k-1)|b}_A(r);
$
see \eqref{proof: choose bar epsilon, lemma: LDP, bar epsilon and delta, clipped version}.
However, 
in light of $\dj{}\big(B^{\bar\epsilon},\D^{(k-1)|b}_A(r)\big) > 2\rho\bar\epsilon$ in \eqref{proof: choose bar epsilon, lemma: LDP, bar epsilon and delta, clipped version},
we arrive at the contraction that $\xi_b \notin B^{\bar\epsilon}$.
This concludes the proof of part $(a)$.
\end{proof}

\begin{proof}[Proof of Lemma~\ref{lemma: convergence from C k b measure to C k measure}]
\linksinpf{lemma: convergence from C k b measure to C k measure}
The proof relies on the following claim: for any $S \in \mathscr S_{\mathbb{D}}$ that is bounded away from $\mathbb{D}^{(k-1)}_A(r)$,
\begin{align}
    \lim_{b \rightarrow \infty}\mathbf{C}^{(k)|b}   (S;\bm x) = \mathbf{C}^{(k)}(S;\bm x).
    \label{subgoal, goal 4, proposition: standard M convergence, LDP unclipped}
\end{align}
Then for any $g \in \mathcal{C}\big(\mathbb{D}\setminus\mathbb{D}^{(k-1)}_A(r)\big)$,
we know that $B = \text{supp}(g)$ is bounded away from $\mathbb{D}^{(k-1)}_A(r)$.
Also, given any $\Delta > 0$,
an approximation to $g$ using simple functions implies the existence of some $N \in \mathbb{N}$,
some sequence of real numbers $\big( c^{(i)}_g \big)_{i = 1}^N$, some sequence $\big( B_g^{(i)} \big)_{i = 1}^N$ of Borel measurable sets on $\mathbb{D}$ that are bounded away from $\mathbb{D}^{(k-1)}_A(r)$ such that the following claims hold for $g^\Delta(\cdot) = \sum_{i = 1}^N c^{(i)}_g\mathbbm{I}\big(\ \cdot\ \in B^{(i)}_g\big)$:
\begin{align*}
    B_g^{(i)} & \subseteq B\ \ \forall i \in [N];\ \ \ \ \big|g^\Delta(\xi) - g(\xi)\big| < \Delta \ \ \forall \xi \in \mathbb{D}.
\end{align*}
Then
\begin{align*}
    \limsup_{b \rightarrow \infty}\Big| \mathbf{C}^{(k)|b}   (g;\bm x) - \mathbf{C}^{(k)}(g;\bm x) \Big|
    & \leq 
     \limsup_{b \rightarrow \infty}\Big| \mathbf{C}^{(k)|b}   (g;\bm x) - \mathbf{C}^{(k)|b}   (g^\Delta;\bm x) \Big|
     \\
     &
     +
      \limsup_{b \rightarrow \infty}\Big| \mathbf{C}^{(k)|b}   (g^\Delta;\bm x) - \mathbf{C}^{(k)}(g^\Delta;\bm x) \Big|
      \\
      &
      +
      \limsup_{b \rightarrow \infty}\Big| \mathbf{C}^{(k)}(g^\Delta;\bm x) - \mathbf{C}^{(k)}(g;\bm x) \Big|
\end{align*}
First, note that $\mathbf{C}^{(k)|b}   (g^\Delta;\bm x) = \sum_{i = 1}^N c^{(i)}_g\mathbf{C}^{(k)|b}   (B^{(i)}_g;\bm x)$ and 
$\mathbf{C}^{(k)}(g^\Delta;\bm x) = \sum_{i = 1}^N c^{(i)}_g\mathbf{C}^{(k)}(B^{(i)}_g;\bm x)$.
Therefore, applying \eqref{subgoal, goal 4, proposition: standard M convergence, LDP unclipped}, we get 
$\limsup_{b \rightarrow \infty}\Big| \mathbf{C}^{(k)|b}   (g^\Delta;\bm x) - \mathbf{C}^{(k)}(g^\Delta;\bm x) \Big| = 0$.
Next, note that $\Big| \mathbf{C}^{(k)|b}   (g^\Delta;\bm x) - \mathbf{C}^{(k)|b}   (g;\bm x) \Big| \leq \Delta \cdot \mathbf{C}^{(k)|b}   (B;\bm x)$
and
$\Big| \mathbf{C}^{(k)}(g^\Delta;\bm x) - \mathbf{C}^{(k)}(g;\bm x) \Big| \leq \Delta \cdot \mathbf{C}^{(k)}(B;\bm x)$.
Thanks to \eqref{subgoal, goal 4, proposition: standard M convergence, LDP unclipped} again, we get
$
 \limsup_{b \rightarrow \infty}\Big| \mathbf{C}^{(k)|b}   (g;\bm x) - \mathbf{C}^{(k)}(g;\bm x) \Big| \leq 2\Delta \cdot \mathbf{C}^{(k)}(B;\bm x).
$
The arbitrariness of $\Delta > 0$ allows us to conclude the proof.

Now, we prove \eqref{subgoal, goal 4, proposition: standard M convergence, LDP unclipped} using Dominated Convergence theorem.
By the definition in \eqref{def: measure mu k b t},
\begin{align*}
    \mathbf{C}^{(k)|b}   (S;\bm x) \delequal &
   \int \mathbbm{I}\Big\{ h^{(k)|b} \big( \bm x,\textbf W,\bm t  \big) \in S \Big\} 
    \big((\nu_\alpha \times \mathbf S)\circ \Phi\big)^k(d \textbf W) \times\mathcal{L}_1^{k\uparrow}(d\bm t).
\end{align*}
where $S \in \mathscr S_\mathbb{D}$ is bounded away from $\mathbb{D}^{(k-1)}_A(r)$.
First, we fix some $\textbf W\in \R^{d \times k}$ and $\bm t\in (0,1)^{k\uparrow}$ and $x_0 \in \R$,
and
let $M \delequal \max_{j \in [k]}\norm{\bm w_j}.$
For any $b > MC$ where $C \geq 1$ is the constant satisfying
such that $\sup_{\bm x \in \R^m}\norm{\bm a(\bm x)} \vee \norm{\bm \sigma(\bm x)} \leq C$
(see Assumption \ref{assumption: boundedness of drift and diffusion coefficients}),
by the definitions of $h^{(k)}$ and $h^{(k)|b}$ it is easy to see that 
$h^{(k)|b}(\bm x,\textbf W,\bm t)=h^{(k)}(\bm x,\textbf W,\bm t)$.
This implies
$$
\lim_{b \rightarrow \infty}
    \mathbbm{I}\big\{ h^{(k)|b} \big( \bm x,\textbf W,\bm t  \big) \in S \big\} 
    =
    \mathbbm{I}\big\{ h^{(k)}\big( \bm x,\textbf W,\bm t  \big) \in S \big\}
    \qquad 
    \forall \textbf W \in \R^{d\times k},\ \bm t \in (0,1]^{k\uparrow}.
$$
\elaborate{
Consider some $b > MC$ where $C \geq 1$ is the constant in Assumption \ref{assumption: boundedness of drift and diffusion coefficients}
such that $|a(x)| \vee \sigma(x) \leq C$ for any $x \in \R$.
For such $b$,
note that $\sup_{x \in \R}|\sigma(x)\cdot w_j| \leq C \cdot M < b$.
As a result, for $\xi = h^{(k)|b}(x,w_1,\cdots,w_k,t_1,\cdots,t_k)$,
we have 
$
\varphi_b\Big(\sigma\big(\xi(t_j-)\big)w_j\Big) = \sigma\big(\xi(t_j-)\big)w_j
$
for any $j \in [k]$.
In summary, we have shown that 
given any $(w_1,\cdots,w_k) \in \R^k, (t_1,\cdots,t_k) \in (0,1)^{k\uparrow}$,
the claim
$h^{(k)|b}(x,w_1,\cdots,w_k,t_1,\cdots,t_k) =  h^{(k)}(x,w_1,\cdots,w_k,t_1,\cdots,t_k)$
holds for any $b$ large enough}
In order to apply Dominated Convergence theorem and conclude the proof of \eqref{subgoal, goal 4, proposition: standard M convergence, LDP unclipped},
it suffices to find an integrable function that dominates $\mathbbm{I}\big\{ h^{(k)|b} \big( \bm x,\textbf W,\bm t  \big) \in S \big\}$.
Specifically, 
since $S$ is bounded away from $\mathbb{D}^{(k-1)}_A(r)$,
we can find some $\bar \epsilon > 0$ such that $\bm d_{J_1}\big(S,\mathbb D^{(k-1)}_A(r)\big) > \bar \epsilon$.
Also, let $\rho  = \exp(D)$ where $D \in[1,\infty)$ is the Lipschitz coefficient in Assumption \ref{assumption: lipschitz continuity of drift and diffusion coefficients}.
Fix some $\bar \delta < \frac{\bar\epsilon}{\rho C}$.
By part $(a)$ of Lemma~\ref{lemma: LDP, bar epsilon and delta, clipped version}, we get
\begin{align}
    \mathbbm{I}\Big\{ h^{(k)|b} \big( \bm x,\textbf W,\bm{t}  \big) \in S \Big\}  \leq \mathbbm{I}\Big\{ \norm{\bm w_j} > \bar\delta\ \forall j \in [k] \Big\}
    \qquad 
    \forall b > 0,\ \textbf W \in \R^{d\times k},\ \bm t \in (0,1)^{k\uparrow}.
    \nonumber
\end{align}
From $
\int \mathbbm{I}\big\{\norm{\bm w_j} > \bar \delta\ \forall j \in [k]\big\} \big((\nu_\alpha \times \mathbf S)\circ \Phi\big)^k(d \textbf W) \times \mathcal{L}^{k\uparrow}_1(d\bm{t}) \leq 1/\bar\delta^{k\alpha} < \infty,
$ we conclude the proof.
\end{proof}

The following result will be applied in the proof of Lemma~\ref{lemma: SGD close to approximation x circ, LDP}.
Let $ \notationdef{notation-discrete-gradient-descent}{\bm{x}^\eta_j(x)}$ be the solution to
\begin{align}
\bm x^\eta_0(\bm x) = \bm x,\qquad 
    {\bm{x}^\eta_j(\bm x)} = \bm{x}^\eta_{j-1}(\bm x) + \eta \bm a\big( \bm{x}^\eta_{j-1}(\bm x) \big)\ \ \ \forall j \geq 1.
    \label{def: gradient descent process y}
\end{align}
After proper scaling of the time parameter, $\bm x^\eta_j$ approximates $\bm y_t$ with small $\eta$.
The next lemma is a direct result from Gronwall's inequality and bounds
the distance between $\bm{x}^\eta_{\floor{t/\eta}}(x)$ and $\bm{y}_t(y)$.
For the sake of completeness we provide the proof.

\begin{lemma} \label{lemma Ode Gd Gap}
\linksinthm{lemma Ode Gd Gap}%
Let Assumptions \ref{assumption: lipschitz continuity of drift and diffusion coefficients} and \ref{assumption: boundedness of drift and diffusion coefficients} hold.
For any $\eta > 0, t > 0$ and $x,y \in \R^m$,
\begin{align*}
    \sup_{s \in [0,t]}\norm{\bm{y}_s(y) - \bm{x}^\eta_{\floor{s/\eta}}(x)} \leq (\eta C + \norm{x - y})\exp(D t)
\end{align*}
where $D,C \in [1,\infty)$ are the constants in Assumptions \ref{assumption: lipschitz continuity of drift and diffusion coefficients} and \ref{assumption: boundedness of drift and diffusion coefficients} respectively.
\end{lemma}

\begin{proof}
\linksinpf{lemma Ode Gd Gap}%
For any $s \geq 0$ that is not an integer,
let $\bm{x}^\eta_s(x) \delequal{} \bm{x}^\eta_{\floor{s}}(x)$ and $\bm{y}_s^\eta(y) \delequal \bm{y}_{s\eta}(y)$.
Now observe that (for any $s \geq 0$)
\begin{align*}
    \bm{y}^\eta_{s}(y) & = \bm{y}^\eta_{\floor{s}}(y) + \eta\int_{\floor{s}}^{s} \bm a(\bm{y}^\eta_u(y) )du \\
    \bm{y}^\eta_{\floor{s}}(y) & = y + \eta\int_0^{\floor{s} }\bm a(\bm{y}^\eta_u(y) )du \\
    \bm{x}^\eta_{\floor{s}}(y)  & = x + \eta \int_0^{ \floor{s} }\bm a( \bm{x}^\eta_u(y) )du.
\end{align*}
Let $\bm b(u) \delequal \bm{y}^\eta_u(y) - \bm{x}^\eta_u(x).$
It suffices to show that $\sup_{u \in [0,t/\eta]}\norm{\bm b(u)} \leq (\eta C + \norm{x-y})\exp(Dt)$.
To this end, we observe that (for any $s > 0$)
\begin{align*}
    \norm{\bm b(s)} & \leq \norm{\bm b(\floor{s})} + \norm{\eta \int_{ \floor{s} }^s \bm a\big( \bm{y}^\eta_u(y) \big)du}
    \leq \norm{\bm b(\floor{s})} + \eta C
    \\
    & \leq \eta \int_0^{ \floor{s} }\norm{ \bm a\big( \bm{y}^\eta_u(y) \big)-\bm a\big( \bm{x}^\eta_u(x) \big) } du + \norm{x - y} + \eta C
    \\
    & \leq 
    \eta D \int_0^s \norm{\bm b(u)}du +\norm{x - y} + \eta C\qquad \text{due to Assumption \ref{assumption: boundedness of drift and diffusion coefficients}}.
\end{align*}
Apply Gronwall's inequality (see Theorem V.68 of \cite{protter2005stochastic}) to $\norm{\bm b(u)}$ on interval $[0,t/\eta]$ and we conclude the proof.
\end{proof}


\section{Technical Results for Metastability Analysis}
\label{subsec: lemma for measure check C}

We first give the proof for Corollary~\ref{corollary: first exit time, untruncated case}.
To do so,
we provide some straightforward bounds for the law of geometric random variables.
\begin{lemma} \label{lemmaGeomFront}
\linksinthm{lemmaGeomFront}
Let $a:(0,\infty) \to (0,\infty)$, $b:(0,\infty) \to (0,\infty)$ 
be two functions
such that 
$\lim_{\epsilon \downarrow 0} a(\epsilon) = 0, \lim_{\epsilon \downarrow 0} b(\epsilon) = 0$.
Let $\{U(\epsilon): \epsilon > 0\}$ be a family of geometric RVs with success rate $a(\epsilon)$,
i.e.
$\P(U(\epsilon) > k) = (1 - a(\epsilon))^{k}$ for $k \in \mathbb{N}$.
For any $c > 1$, there exists $\epsilon_0 > 0$ such that
$$\exp\Big(-\frac{c\cdot a(\epsilon)}{b(\epsilon)}\Big) \leq \P\Big( U(\epsilon) > \frac{1}{b(\epsilon)} \Big) \leq \exp\Big(-\frac{a(\epsilon)}{c\cdot b(\epsilon)}\Big)
\ \ \ \forall \epsilon \in (0,\epsilon_0)
.$$

\end{lemma}

\begin{proof}
\linksinpf{lemmaGeomFront}%
Note that
$\P( U(\epsilon) > \frac{1}{b(\epsilon)} ) = \big(1 - a(\epsilon)\big)^{ \floor{1/b(\epsilon)} }.$
By taking logarithm on both sides, we have
\begin{align*}
    \ln \P\Big( U(\epsilon) > \frac{1}{b(\epsilon)} \Big) & = \floor{1/b(\epsilon)}\ln\Big(1 - a(\epsilon)\Big) = \frac{\floor{1/b(\epsilon)} }{ 1/b(\epsilon) }\frac{\ln\Big(1 - a(\epsilon)\Big) }{-a(\epsilon) }\frac{-a(\epsilon) }{ b(\epsilon) }.
\end{align*}
Since $\lim_{x \rightarrow 0}\frac{\ln(1 + x)}{x} = 1$, we know that for $\epsilon$ sufficiently small, we will have
$-c \frac{a(\epsilon)}{b(\epsilon)}   \leq \ln \P\Big( U(\epsilon) > \frac{1}{b(\epsilon)} \Big) \leq -\frac{a(\epsilon)}{c\cdot b(\epsilon)}.$
By taking exponential on both sides, we conclude the proof.
\end{proof}

\begin{proof}[Proof of Corollary~\ref{corollary: first exit time, untruncated case}]
\linksinpf{corollary: first exit time, untruncated case}
Note that the value of $\bm \sigma(\cdot)$ and $\bm a(\cdot)$ outside of the domain $I$ has no impact on the first exit analysis.
Therefore, by modifying the value of $\bm \sigma(\cdot)$ and $\bm a(\cdot)$ outside of $I$,
we can assume w.l.o.g.\ that
\begin{align}
    \norm{\bm a(\bm x)}\vee \norm{\bm \sigma(\bm x)} \leq C
    \qquad
    \forall \bm x \in \R^m.
    \label{boundedness assumption, corollary: first exit time, untruncated case}
\end{align}
for some $C \in (0,\infty)$.
That is, we can impose the boundedness condition in Assumption~\ref{assumption: boundedness of drift and diffusion coefficients} w.l.o.g.

We start with a few observations.
First, under any $\eta \in (0,\frac{b}{2C})$, 
on the event $\{\eta\norm{\bm Z_j} \leq \frac{b}{2C}\ \forall j \leq t\}$ the norm of the step-size (before truncation) 
$\eta \bm a\big(\bm X^{\eta|b}_{j - 1}(\bm x)\big) + \eta \bm \sigma\big(\bm X^{\eta|b}_{j - 1}(\bm x)\big)\bm Z_j$ of $\bm X_j^{\eta|b}(\bm x)$ is less than $b$ for each $j \leq t$. 
Therefore, $\bm X^{\eta|b}_j(\bm x)$ and $\bm X^\eta_j(\bm x)$ coincide for such $j$'s.
\elaborate{First, note that
given any $\eta \in (0, \frac{b}{2 C })$ and $|w| \leq \frac{b}{2C}$, we have
\begin{align*}
    \Big| \eta a(y) + \sigma(y) w   \Big|
    & \leq 
    \eta |a(y)| + | \sigma(y) | \cdot |w|
    \leq \eta \cdot C + C \cdot \frac{b}{ 2C }
    < \frac{b}{2} + \frac{b}{2} = b
    \qquad\forall y \in I^-.
\end{align*}
As a result, given any $t > 0$, $x \in \R$,  and $\eta \in (0,1 \wedge \frac{b}{2 C })$,
it holds on event $\{\eta|Z_j| \leq \frac{b}{2C}\ \forall j \leq t\}$ that
\begin{align*}
    X^{\eta|b}_j(x) & = X^{\eta|b}_{j - 1}(x) + \varphi_b\Big( \eta a\big(X^{\eta|b}_{j - 1}(x)\big) + \eta \sigma\big(X^{\eta|b}_{j - 1}(x)\big)Z_j  \Big)
    \\
    & = 
    X^{\eta|b}_{j - 1}(x) + \eta a\big(X^{\eta|b}_{j - 1}(x)\big) + \eta \sigma\big(X^{\eta|b}_{j - 1}(x)\big)Z_j
    \ \ \ \ \forall j \leq \tau^{\eta|b}(x) \wedge t.
\end{align*}
}%
In other words, for any $\eta \in (0,\frac{b}{2C})$, on event $ \big\{ \eta\norm{\bm Z_j} \leq \frac{b}{2C}\ \forall j \leq t\big\}$ 
we have
\begin{equation}
    \bm X^{\eta|b}_j(\bm x) = \bm X^\eta_j(\bm x)\qquad \forall j \leq t.
    \label{proof, observation 1, theorem: first exit time, unclipped}
\end{equation}
Next, recall that $I$ is a bounded under Assumption~\ref{assumption: shape of f, first exit analysis}.
Therefore, under $b > \sup_{\bm x \in I}\norm{\bm x}$,
it holds for $\bm w \in \R^d$ that
\begin{align*}
    \varphi_b\big(\bm\sigma(\bm 0)\bm w\big) \notin I\qquad \Longleftrightarrow \qquad \bm \sigma(\bm 0)\bm w \notin I.
\end{align*}
\elaborate{To see why, first suppose that $|\sigma(0)\cdot w| \geq b$. 
Then $\sigma(0)\cdot w \notin I$. 
In the meantime, we also have $\big|\varphi_b\big(\sigma(0)\cdot w\big)\big| = b$, 
and hence, $\varphi_b\big(\sigma(0)\cdot w\big) \notin I$.
On the other hand, the equivalence is trivial if $|\sigma(0)\cdot w| < b$ since $\varphi_b(\sigma(0)\cdot w)  = \sigma(0)\cdot w$ in such a case.
This establishes claim \eqref{proof, observation 2, theorem: first exit time, unclipped}.
}%
As a result, for all $b$ large enough, we have
\begin{align}
   C^I_b = \widecheck{ \mathbf{C} }^{(1)|b}(I^\complement) 
   & = \int \mathbbm{I}\Big\{ \varphi_b\big(\bm \sigma(\bm 0) \bm w\big) \notin  I \Big\} \big((\nu_\alpha \times \mathbf S)\circ \Phi\big)(d \bm w)
   \nonumber
   \\
   & = \int \mathbbm{I}\Big\{ \bm \sigma(\bm 0) \bm w \notin I \Big\} \big((\nu_\alpha \times \mathbf S)\circ \Phi\big)(d \bm w)
   = \widecheck{ \mathbf{C} }(I^\complement) \delequal C^I_\infty.
   \label{proof, observation 2, theorem: first exit time, unclipped}
\end{align}
Similarly, one can show that for all $b$ large enough,
\begin{align}
     \widecheck{ \mathbf{C} }^{(1)|b}(\partial I)
     =
      \widecheck{ \mathbf{C} }^{(1)}(\partial I).
      \label{proof, observation, mass on the boundary set of I, theorem: first exit time, unclipped}
\end{align}
Moreover, given any measurable $A \subseteq \R$ such that $r_A = \inf\{\norm{\bm x}:\ \bm x \in A\} > 0$,
we claim that
\begin{align}
    \lim_{b \to \infty}\widecheck{\mathbf C}^{(1)|b}(A) = \widecheck{ \mathbf C }(A).
    \label{proof, observation 2 general form, theorem: first exit time, unclipped}
\end{align}
This claim follows from a simple application of the dominated convergence theorem.
Indeed, by definition, we have
$
\widecheck{\mathbf C}^{(1)|b}(A) = \int \mathbbm{I}\big\{ \varphi_b\big(\bm \sigma(\bm 0) \bm w\big) \in A \big\}
 \big((\nu_\alpha \times \mathbf S)\circ \Phi\big)(d \bm w).
$
For $f_b(\bm w) \delequal \mathbbm{I}\big\{ \varphi_b\big(\bm \sigma(\bm 0) \bm w\big) \in A \big\}$,
we first note that given $\bm w \in \R^m$, we have $f_b(\bm w) = f(\bm w)\delequal \mathbbm{I}\big\{ \bm \sigma(\bm 0) \bm w \in A \big\}$ for all $b > \norm{\bm w} \norm{\bm  \sigma(\bm 0)}$.
Therefore, the point-wise convergence $\lim_{b \to \infty}f_b(\bm w) = f(\bm w)$ holds for all $\bm w \in \R^m$.
Next, 
observe that
\begin{align*}
    \big\{ \varphi_b\big(\bm \sigma(\bm 0)\bm w\big) \in A \big\}
    &
    \subseteq 
    \big\{
        \norm{ \bm \sigma(\bm 0)\bm w } \geq r_A
    \big\}
    \subseteq 
    \big\{
        \norm{ \bm \sigma(\bm 0)}\cdot \norm{\bm w } \geq r_A
    \big\}
    =
    \big\{
        \norm{\bm w } \geq r_A/\norm{ \bm \sigma(\bm 0)}
    \big\}.
\end{align*}
This implies
$
f_b(\bm w) \leq \mathbbm{I}\big\{ \norm{\bm w } \geq r_A/\norm{ \bm \sigma(\bm 0)} \big\}
$
for all $b > 0$ and $\bm w \in \R^d$.
Also, by definition of the measure $\nu_\alpha$ in \eqref{def: measure nu alpha},
\begin{align}
    \int \mathbbm{I}\big\{ \norm{\bm w} \geq r_A/ \norm{\bm \sigma(\bm 0)} \big\}  \big((\nu_\alpha \times \mathbf S)\circ \Phi\big)(d \bm w) = ( \norm{\bm \sigma(\bm 0)}/r_A)^\alpha < \infty.
    \label{proof: upper bound for measure hat C 1 b, theorem: first exit time, unclipped}
\end{align}
The last inequality follows from $r_A > 0$.
This allows us to apply dominated convergence theorem and establish \eqref{proof, observation 2 general form, theorem: first exit time, unclipped}.

Moving on, we verify a few regularity conditions.
By repeating the calculations in \eqref{proof: upper bound for measure hat C 1 b, theorem: first exit time, unclipped} with  $A = I^\complement$,
we are able to verify the condition ${C^I_\infty} =\widecheck{ \mathbf{C} }(I^\complement) < \infty$
in Corollary~\ref{corollary: first exit time, untruncated case}.
Next, by the convention in \eqref{def: 0 jump coverage set, first exit times},
we have that $\mathcal{G}^{(0)|b}(\epsilon)$ is bounded away from $I^\complement$ for all $\epsilon > 0$ small enough and all $b > 0$.
In the meantime, recall the definition of 
$
\mathcal{G}^{(1)|b} = \big\{ \varphi_b\big(\bm \sigma(\bm 0)\bm w\big):\ \bm w \in \R^d  \big\}.
$
Due to $\norm{\bm \sigma(\bm 0)} > 0$, there exists $\bm w^*$ such that 
$
\norm{\bm \sigma(\bm 0)\bm w^*} > \sup_{\bm x \in I}\norm{\bm x},
$
and hence $\bm \sigma(\bm 0)\bm w^* \notin I$.
As a result, for all $b > \norm{\bm \sigma(\bm 0)\bm w^*}$ we have
$
\mathcal{G}^{(1)|b} \cap I^\complement \neq \emptyset.
$
That is, we have shown that 
\begin{align*}
    \mathcal J^I_b = 1\qquad\text{for all $b > 0$ large enough};
\end{align*}
see \eqref{def: first exit time, J *} for the definition.
Together with \eqref{proof, observation, mass on the boundary set of I, theorem: first exit time, unclipped} and the running assumption $\widecheck{\mathbf C}(\partial I) = 0$ in Corollary~\ref{corollary: first exit time, untruncated case},
we have 
$
 \widecheck{ \mathbf{C} }^{(1)|b}(\partial I) = 0
$
for all $b$ large enough.
These conditions will allow us to apply Theorem~\ref{theorem: first exit time, unclipped}, with $b > 0$ large enough, in the remainder of this proof.

Now, we fix $t \geq 0$ and $B \subseteq I^c$, and recall that
our goal is to study the probability of the event
$$A(\eta,\bm x) \delequal \big\{ C^I_\infty H(\eta^{-1}) \tau^\eta(\bm x) > t,\ \bm X^\eta_{\tau^\eta(\bm x)}(\bm x) \in B \big\}.$$
Here, note that $\lambda(\eta) = \eta^{-1}  H(\eta^{-1})$ and hence $\eta \cdot \lambda(\eta) = H(\eta^{-1})$.
Also, henceforth in the proof we only consider $b$ large enough
such that
$
C^I_\infty = C^I_b;
$
see \eqref{proof, observation 2, theorem: first exit time, unclipped}.
We focus on the case where $C^I_\infty > 0$, but we stress that the proof for the case with $C^I_\infty = 0$ is almost identical.
First,
we arbitrarily pick some $T > t$ and observe that
\begin{align}
    & A(\eta,\bm x) 
    \nonumber \\ 
    & = 
    \underbrace{\Big\{ C^I_\infty H(\eta^{-1}) \tau^\eta(\bm x) \in (t,T],\ \bm X^\eta_{\tau^\eta(\bm x)}(\bm x) \in B \Big\}}_{ \delequal A_1(\eta,\bm x,T) }
    \cup 
    \underbrace{\Big\{ C^I_\infty H(\eta^{-1}) \tau^\eta(\bm x) > T,\ \bm X^\eta_{\tau^\eta(\bm x)}(\bm x) \in B \Big\}}_{ \delequal A_2(\eta,\bm x,T) }.
    \label{proof, decompose event A, theorem: first exit time, unclipped}
\end{align}
Let $E_b(\eta,T) \delequal \big\{ \eta\norm{\bm Z_j} \leq \frac{b}{2C}\ \forall j \leq \frac{T}{ C^I_\infty H(\eta^{-1}) }   \big\}$ and note that
$$
A_1(\eta,\bm x,T) = \Big(A_1(\eta,\bm x,T) \cap E_b(\eta,T) \Big) \cup \Big(A_1(\eta,\bm x,T) \setminus E_b(\eta,T) \Big).
$$
Moreover, for all $\eta \in (0,\frac{b}{2C})$,
\begin{align*}
    & \P\Big(A_1(\eta,\bm x,T) \cap E_b(\eta,T) \Big) 
    \\ 
    & = 
    \P\bigg( \Big\{  C^I_b \eta \cdot \lambda(\eta) \tau^{\eta|b}(\bm x) \in (t,T],\ \bm X^{\eta|b}_{ \tau^{\eta|b}(\bm x) }(\bm x) \in B  \Big\} \cap E_b(\eta,T) \bigg)
    \qquad 
    \text{due to \eqref{proof, observation 1, theorem: first exit time, unclipped} and \eqref{proof, observation 2, theorem: first exit time, unclipped}}
    \\ 
    & \leq 
    \P\bigg( C^I_b \eta \cdot \lambda(\eta) \tau^{\eta|b}(\bm x) \in (t,T],\ \bm X^{\eta|b}_{ \tau^{\eta|b}(\bm x) }(\bm x) \in B  \bigg)
    \\ & = 
    \P\bigg( C^I_b \eta \cdot \lambda(\eta) \tau^{\eta|b}(\bm x) > t,\ \bm X^{\eta|b}_{ \tau^{\eta|b}(\bm x) }(\bm x) \in B  \bigg)
    -
    \P\bigg( C^I_b \eta \cdot \lambda(\eta) \tau^{\eta|b}(\bm x) > T,\ \bm X^{\eta|b}_{ \tau^{\eta|b}(\bm x) }(\bm x) \in B  \bigg).
\end{align*}
By Theorem~\ref{theorem: first exit time, unclipped} and claim~\eqref{proof, observation 2, theorem: first exit time, unclipped}, we get
\begin{align}
\limsup_{\eta \downarrow 0}\sup_{\bm x \in I_\epsilon }\P\Big(A_1(\eta,\bm x,T) \cap E_b(\eta,T) \Big) 
\leq 
\frac{ \widecheck{\mathbf{C}}^{ (1)|b }(B^-) }{ C^I_\infty }\cdot\exp(-t) - \frac{ \widecheck{\mathbf{C}}^{ (1)|b }(B^\circ) }{ C^I_\infty }\cdot\exp(-T).
\label{proof, ineq 0, theorem: first exit time, unclipped}
\end{align}
Meanwhile,
$$
\sup_{\bm x \in I_\epsilon}\P\big(A_1(\eta,\bm x,T) \setminus E_b(\eta,T)\big) \leq \P\big( ( E_b(\eta,T))^\complement\big)
=
\P\bigg( \eta\norm{\bm Z_j} > \frac{b}{2C}\text{ for some }j \leq \frac{T}{C^I_\infty H(\eta^{-1})} \bigg).
$$
Applying Lemma \ref{lemmaGeomFront}, we get (recall that $H(\cdot) = \P(\norm{\bm Z_j} > \cdot)$ and $H(x)\in \RV_{-\alpha}(x)$ as $x \to \infty$)
\begin{align}
    \limsup_{\eta \downarrow 0}\P\bigg( \eta\norm{\bm Z_j} >\frac{b}{2C}\text{ for some }j \leq \frac{T}{C^I_\infty H(\eta^{-1})} \bigg)
    \nonumber
    & = 
    1 - 
    \liminf_{\eta \downarrow 0}\P\bigg( \text{Geom}\Big(H\big( \frac{b}{\eta \cdot 2 C} \big)\Big) > \frac{T}{C^I_\infty H(\eta^{-1})} \bigg)
    \nonumber
    \\
    &
    \leq 
    1 - \lim_{\eta \downarrow 0}\exp\bigg( -\frac{T \cdot  H(\eta^{-1} \cdot \frac{b}{2C} )}{ C^I_\infty H(\eta^{-1})  }\bigg)
    \nonumber
    \\ 
    & = 
    1 - \exp\bigg( - \frac{T}{C^I_\infty} \cdot \Big( \frac{2C}{b} \Big)^\alpha  \bigg).
    \label{proof, ineq 1, theorem: first exit time, unclipped}
\end{align}
Similarly,
\begin{align*}
    A_2(\eta,\bm x,T) & \subseteq \Big\{ C^I_\infty H(\eta^{-1})\tau^\eta(\bm x) > T  \Big\}
    \\ 
    & = \bigg( \Big\{ C^I_\infty H(\eta^{-1})\tau^\eta(\bm x) > T  \Big\} \cap E_b(\eta,T) \bigg) \cup  \bigg( \Big\{ C^I_\infty H(\eta^{-1})\tau^\eta(\bm x) > T  \Big\} \setminus E_b(\eta,T) \bigg).
\end{align*}
On $\{ C^I_\infty  H(\eta^{-1})\tau^\eta(\bm x) > T \} \cap E_b(\eta,T)$, we have $\tau^\eta(\bm x) = \tau^{\eta|b}(\bm x)$
again due to \eqref{proof, observation 1, theorem: first exit time, unclipped}.
By Theorem \ref{theorem: first exit time, unclipped} and \eqref{proof, observation 2, theorem: first exit time, unclipped},
we get
\begin{align}
    & \limsup_{\eta \downarrow 0}\sup_{\bm x \in I_\epsilon}\P\bigg( \Big\{ C^I_\infty H(\eta^{-1})\tau^\eta(\bm x) > T  \Big\} \cap E_b(\eta,T) \bigg)
    \nonumber
    \\ 
    &
    \leq 
    \limsup_{\eta \downarrow 0}\sup_{\bm x \in I_\epsilon}
        \P\Big( C^I_b \eta \cdot \lambda(\eta)\tau^{\eta|b}(\bm x) > T\Big) \leq \exp(-T).
    \label{proof, ineq 2, theorem: first exit time, unclipped}
\end{align}
Meanwhile, the limit of $\sup_{\bm x \in I_\epsilon}\P\big( C^I_\infty H(\eta^{-1})\tau^\eta(\bm x) > T \} \setminus E_b(\eta,T) \big)$ as $\eta \downarrow 0$ is again bounded by \eqref{proof, ineq 1, theorem: first exit time, unclipped}.
Collecting \eqref{proof, ineq 0, theorem: first exit time, unclipped}, \eqref{proof, ineq 1, theorem: first exit time, unclipped}, and \eqref{proof, ineq 2, theorem: first exit time, unclipped},
we yield that for all $b > 0$ large enough and all $T > t$,
\begin{align*}
    & \limsup_{\eta \downarrow 0}\sup_{\bm x \in I_\epsilon}\P\big(A(\eta,\bm x)\big)
    \\
    & \leq 
        \frac{ \widecheck{\mathbf{C}}^{ (1)|b }(B^-) }{ C^I_\infty }\cdot\exp(-t) - \frac{ \widecheck{\mathbf{C}}^{ (1)|b }(B^\circ) }{ C^I_\infty }\cdot\exp(-T) + \exp(-T)
    + 
    2 \cdot \bigg[  1 - \exp\bigg( - \frac{T}{C^I_\infty} \cdot \Big( \frac{2C}{b} \Big)^\alpha  \bigg)\bigg].
\end{align*}
In light of claim \eqref{proof, observation 2 general form, theorem: first exit time, unclipped},
we send $b \to \infty$ and $T \to \infty$
to conclude the proof of the upper bound.

The lower bound can be established analogously. 
By the decomposition of events in \eqref{proof, decompose event A, theorem: first exit time, unclipped},
\begin{align*}
    & \inf_{\bm x \in I_\epsilon}\P\big(A(\eta,\bm x)\big)
    \\ 
    & \geq 
    \inf_{\bm x \in I_\epsilon}\P\big(A_1(\eta,\bm x,T)\big)
    \geq 
    \inf_{\bm x \in I_\epsilon}\P\big(A_1(\eta,\bm x,T) \cap E_b(\eta,T)\big)
    \\
    & = 
    \inf_{\bm x \in I_\epsilon}
    \P\bigg( \Big\{  C^I_b \eta \cdot \lambda(\eta) \tau^{\eta|b}(\bm x) \in (t,T],\ \bm X^{\eta|b}_{ \tau^{\eta|b}(\bm x) }(\bm x) \in B  \Big\} \cap E_b(\eta,T) \bigg)
    \qquad 
    \text{due to \eqref{proof, observation 1, theorem: first exit time, unclipped} and \eqref{proof, observation 2, theorem: first exit time, unclipped}}
    \\ 
    & \geq 
    \inf_{\bm x \in I_\epsilon}
    \P\bigg( C^I_b \eta \cdot \lambda(\eta) \tau^{\eta|b}(\bm x) \in (t,T],\ \bm X^{\eta|b}_{ \tau^{\eta|b}(\bm x) }(\bm x) \in B  \bigg)
    -
    \P\Big( \big(E_b(\eta,T)\big)^\complement\Big)
    \\ 
    & \geq 
    \inf_{\bm x \in I_\epsilon}
    \P\bigg( C^I_b \eta \cdot \lambda(\eta) \tau^{\eta|b}(\bm x) > t,\ \bm X^{\eta|b}_{ \tau^{\eta|b}(\bm x) }(\bm x) \in B  \bigg)
    -
    \sup_{\bm x \in I_\epsilon}
    \P\bigg( C^I_b \eta \cdot \lambda(\eta) \tau^{\eta|b}(\bm x) > T,\ \bm X^{\eta|b}_{ \tau^{\eta|b}(\bm x) }(\bm x) \in B  \bigg)
    \\ 
    &\qquad \qquad 
    -
    \P\Big( \big(E_b(\eta,T)\big)^\complement\Big).
\end{align*}
By Theorem \ref{theorem: first exit time, unclipped} and the limit in \eqref{proof, ineq 1, theorem: first exit time, unclipped},
we yield (for all $b > 0$ large enough and all $T  > t$)
\begin{align*}
    \liminf_{\eta \downarrow 0}\inf_{\bm x \in I_\epsilon}\P\Big(A(\eta,\bm x)\Big)
    & \leq 
        \frac{ \widecheck{\mathbf{C}}^{ (1)|b }(B^\circ) }{ C^I_\infty }\cdot\exp(-t) - \frac{ \widecheck{\mathbf{C}}^{ (1)|b }(B^-) }{ C^I_\infty }\cdot\exp(-T)
    -
    \bigg[  1 - \exp\bigg( - \frac{T}{C^I_\infty} \cdot \Big( \frac{2C}{b} \Big)^\alpha  \bigg)\bigg].
\end{align*}
By claim \eqref{proof, observation 2 general form, theorem: first exit time, unclipped},
we send $b \to \infty$ and $T \to \infty$ to conclude the proof of the lower bound.
\end{proof}

The remainder of this section collects important properties of the measure $\widecheck{ \mathbf C }^{(k)|b}(\cdot)$ defined in \eqref{def: measure check C k b}.
In particular, the proof of Lemma~\ref{lemma: limiting measure, with exit location B, first exit analysis} will be provided at the end of this section.
Throughout the rest of this section,
we impose Assumption~\ref{assumption: lipschitz continuity of drift and diffusion coefficients} and \ref{assumption: shape of f, first exit analysis},
and
fix some $b > 0$ such that the conditions in  Theorem~\ref{theorem: first exit time, unclipped} hold.
We fix some $\bar\epsilon > 0$ small enough such that the conditions in \eqref{constant bar epsilon, new, 1, first exit time analysis}--\eqref{constant bar epsilon, new, 3, first exit time analysis} hold.

Recall that
$
{I_\epsilon} = \{ \bm y:\ \norm{\bm x - \bm y} < \epsilon\ \Longrightarrow\ \bm x \in I \}
$
is  the $\epsilon$-shrinkage of the domain $I$,
and that $I^-_\epsilon$ is the closure of $I_\epsilon$.
We first study the mapping ${\widecheck{g}^{(k)|b}}$ in \eqref{def: mapping check g k b, endpoint of path after the last jump, first exit analysis},
which is defined based on 
$\bar h^{(k)|b}_{[0,T]}$ and $h^{(k)|b}_{[0,T]}$ defined in
\eqref{def: perturb ode mapping h k b, 1}--\eqref{def: perturb ode mapping h k b, 4}.




\begin{lemma}
\label{lemma: choose key parameters, first exit time analysis}
\linksinthm{lemma: choose key parameters, first exit time analysis}
    Let Assumptions~\ref{assumption: lipschitz continuity of drift and diffusion coefficients}  and \ref{assumption: shape of f, first exit analysis}
    hold.
    Let $\bar\epsilon > 0$ be the constant in \eqref{constant bar epsilon, new, 1, first exit time analysis}--\eqref{constant bar epsilon, new, 3, first exit time analysis}. 
    Let $C \in [1,\infty)$ be such that $\sup_{\bm x \in I^-}\norm{\bm a(\bm x)} \vee \norm{\bm \sigma(\bm x)} \leq C$.
    (Below, we adopt the convention that $t_0 = 0$.)
    \begin{enumerate}[(a)]



        \item Given any $T > 0$, the claim $\xi(t) \in I^-_{2\bar\epsilon}\ \forall t \in[0,T]$
        holds for all 
        $
        \xi \in \D^{( \mathcal J^I_b - 1)|b}_{ \bar B_{\bar\epsilon} }[0,T](\bar\epsilon).
        $

        \item 
            Let $\bar c \in (0,1)$ be the constant fixed in \eqref{constant bar c for bar epsilon, first exit time}.
        There exist $\Bar{\delta}>0$ and $\Bar{t} > 0$ such that the following claim holds:
        Given any $T > 0$ and $\bm x_0 \in \R^m$ with $\norm{\bm x_0} \leq \bar\epsilon$,
        if
        \begin{align}
            \xi(t) \notin I_{\bar c\bar\epsilon}\qquad\text{ for some }
            \xi = h^{(\mathcal J^I_b-1)|b}_{[0,T]}\Big( \bm x_0 + \varphi_b\big( \bm \sigma(\bm x_0)\bm w_0\big),\textbf W,(t_1,\cdots,t_{\mathcal J^I_b - 1})\Big),\ t \in [0,T],
            \label{claim, def of xi, lemma: choose key parameters, first exit time analysis}
        \end{align}
        where
        $\textbf W = (\bm w_1,\cdots,\bm w_{\mathcal J^I_b-1}) \in \R^{d \times \mathcal J^I_b-1}$, and 
        $(t_1,\cdots,t_{\mathcal J^I_b-1})\in (0,T]^{ \mathcal J^I_b-1 \uparrow}$,
        then
        \begin{enumerate}[(i)]
        \item 
            $\xi(t) \in I_{2\bar\epsilon}^-$ for all $t \in [0,t_{\mathcal J^I_b - 1})$;
            
        \item 
            $\xi(t_{\mathcal J^I_b-1})\notin I_{\bar\epsilon}$;

        \item 
            $\norm{\xi(t)} \geq \bar\epsilon$ for all $t \leq t_{\mathcal J^I_b - 1}$;

         \item 
            $t_{\mathcal J^I_b-1} < \Bar{t}$;

        \item 
            $\norm{\bm w_j} > \Bar{\delta}$ for all $j = 0,1,\cdots,\mathcal J^I_b - 1$.

        \end{enumerate}

        \item 
            Let $T > 0$, 
            $\bm x \in \R^m,\ \textbf W = (\bm w_1,\cdots, \bm w_{\mathcal J^I_b}) \in \R^{d \times  \mathcal J^I_b},\ (t_1,\cdots,t_{\mathcal J^I_b})\in (0,T]^{\mathcal J^I_b\uparrow}$, and $\epsilon \in (0,\bar\epsilon)$.
            Let
            \begin{align*}
                \xi & =  h^{ (\mathcal J^I_b)|b }_{[0,T]}\big(\bm x, \textbf W, (t_1,\cdots,t_{\mathcal J^I_b})\big),
                \\ 
                \widecheck{\xi}
                & =
                 h^{ (\mathcal J^I_b - 1)|b }_{[0,T]}
                \Big(
                    \varphi_b\big( \bm \sigma(\bm 0)\bm w_1\big), (\bm w_2,\cdots, \bm w_{ \mathcal J^I_b }), 
                    (t_2 - t_1,t_3 - t_1,\cdots,t_{\mathcal J^I_b} - t_1)
                \Big).
            \end{align*}
            If $\norm{\xi(t_1-)} < \epsilon$ 
            and
            $
            \norm{\bm w_j} \leq \epsilon^{-\frac{1}{2\mathcal J^I_b}}\ \forall j \in [\mathcal J^I_b],
            $
            then
        \begin{align*}
           \sup_{t \in [t_1,t_{\mathcal J^I_b}]} 
           \norm{\xi(t) - \widecheck\xi(t - t_1)} \leq \Big( 2\exp\big(D( t_{\mathcal J^I_b}  - t_1)\big) \cdot D\Big)^{\mathcal J^I_b + 1} \cdot  \epsilon,
        \end{align*}
        where
        $D \geq 1$ is the constant in Assumption~\ref{assumption: lipschitz continuity of drift and diffusion coefficients}.

        \item 
            Let $\bar c \in (0,1)$ be the constant fixed in \eqref{constant bar c for bar epsilon, first exit time}.
        Given $\Delta > 0$, there exists $\epsilon_0 = \epsilon_0(\Delta) \in (0,\bar\epsilon)$ such that
        the following claim holds:
        given 
        $T > 0$, 
        $\bm x \in \R^m$,
        $\textbf W = (\bm w_1,\cdots, \bm w_{\mathcal J^I_b}) \in \R^{d \times  \mathcal J^I_b},\ (t_1,\cdots,t_{\mathcal J^I_b})\in (0,T]^{\mathcal J^I_b\uparrow}$,
        if $\norm{\bm x} \leq \epsilon_0$
        and
        $
            \max_{j \in [\mathcal J^I_b] }\norm{\bm w_j} \leq \epsilon_0^{-\frac{1}{2\mathcal J^I_b}},
        $
        then        
        \begin{align*}
        \xi(t) \notin I_{\bar c\bar \epsilon}\text{ or }\widecheck{\xi}(t)\notin I_{\bar c\bar \epsilon}
        \text{ for some }t \in [t_1,T - t_1]
        \qquad
        \Longrightarrow
        \qquad
        \sup_{t \in [t_1,t_{\mathcal J^I_b}]}
        \norm{\widecheck{\xi}(t - t_1) - \xi(t) } < \Delta,
        \end{align*}
        where $\xi$ and $\widecheck{\xi}$ are defined as in part $(c)$.
    \end{enumerate}
\end{lemma}

\begin{proof}
\linksinpf{lemma: choose key parameters, first exit time analysis}
Before the proof of the claims,
we highlight two facts.
First, Assumption~\ref{assumption: lipschitz continuity of drift and diffusion coefficients} and $I$ being a bounded set (so $I^-$ is compact) imply the existence of $C \in (0,\infty)$ such that $\sup_{\bm x \in I^-}\norm{\bm a(\bm x)} \vee \norm{\bm \sigma(\bm x)} \leq C$.
Without loss of generality, in the statement of Lemma~\ref{lemma: choose key parameters, first exit time analysis} we pick some $C \geq 1$.
Next,
one can see that the validity of all claims do not depend on the values of $\bm \sigma(\cdot)$ and $\bm a(\cdot)$ outside of $I^-$. 
Therefore, throughout this proof below we w.l.o.g.\ assume that
\begin{align}
    \norm{\bm a(\bm x)}\vee \norm{\bm \sigma(\bm x)} \leq C
    \qquad
    \forall \bm x \in \R^m.
    \label{constant C, boundedness of a and sigma, lemma: choose key parameters, first exit time analysis}
\end{align}
for some $C \in [1,\infty)$.
That is, we impose the boundedness condition in Assumption~\ref{assumption: boundedness of drift and diffusion coefficients}.


\medskip
\noindent
$(a)$
Arbitrarily pick some $T > 0$ and $\xi \in \D^{( \mathcal J^I_b - 1)|b}_{ \bar B_{\bar\epsilon} }[0,T](\bar\epsilon)$.
To lighten notations, in the proof of part $(a)$ we write $k = J^I_b$.
 By the definition of $\D^{(k-1)|b}_A(\epsilon)$ in \eqref{def: l * tilde jump number for function g, clipped SGD},
there are some $\bm x$ with $\norm{\bm x} \leq \bar\epsilon$,
some $(\bm w_1,\cdots,\bm w_{k-1}) \in \R^{d \times k - 1}$,
some $(\bm v_1,\cdots,\bm v_{k-1}) \in \R^{m \times k-1}$ with $\max_{j \in [k-1]}\norm{\bm v_j} \leq \bar\epsilon$,
and $0 < t_1 < t_2 < \cdots < t_{k-1} < \infty$ such that
\begin{align*}
    \xi = \bar h^{(k-1)|b}_{[0,T]}\big(\bm x, (\bm w_1,\cdots,\bm w_{k-1}), (\bm v_1,\cdots,\bm v_{k-1}), (t_1,\cdots,t_{k-1})\big).
\end{align*}
Given any $t \in [0,T]$,
Let $j^* = j^*(t) = \max\{ j = 0,1,\cdots,k-1:\ t_j \leq t  \}$.
By definition of the mapping $\bar h^{(k-1)|b}_{[0,T]}$ in \eqref{def: perturb ode mapping h k b, 1}--\eqref{def: perturb ode mapping h k b, 3},
we have $\xi(t) = \bm y_{t - t_{j^*}}\big(\xi(t_{j^*})\big)$
where $\bm y_\cdot(\bm x)$ is the ODE under the vector field $\bm a(\cdot)$; see \eqref{def ODE path y t}.
By the definition of $\mathcal G^{(k)|b}(\epsilon)$ and $\bar{\mathcal G}^{(k)|b}$ in \eqref{def: set G k b epsilon}, \eqref{def: bar G k b epsilon, extended k jump coverage set},
we then yield $\xi(t) \in \bar{\mathcal G}^{(k-1)|b}(2\bar\epsilon)$.
However, 
by property~\eqref{constant bar epsilon, new, 3, first exit time analysis},
we must have
\begin{align}
    \bar{\mathcal G}^{(k-1)|b}(2\bar\epsilon) \subseteq I^-_{2\bar\epsilon} \subseteq I_{\bar\epsilon}.
    \label{proof, property, inclusion of bar mathcal G k b and I bar epsilon, lemma: choose key parameters, first exit time analysis}
\end{align}
and hence
$
\xi(t) \in \bar{\mathcal G}^{(k-1)|b}(2\bar\epsilon) \subseteq I_{2\bar\epsilon}.
$
This concludes the proof.

\medskip
\noindent
$(b)$
For simplicity, in the proof of part $(b)$ we write $k = \mathcal J^I_b$.
For claim $(i)$, note that due to $\norm{\bm x_0} \leq \bar\epsilon$,
we have 
$
\bm x_0 + \varphi_b(\bm \sigma(\bm x_0)\bm w_0) \in \mathcal G^{(1)|b}(2\bar\epsilon).
$
Moreover, for all $n = 0,1,\cdots,k - 2$ (recall our convention of $t_0 = 0$),
for the cadlag path $\xi$ defined in \eqref{claim, def of xi, lemma: choose key parameters, first exit time analysis}
we have
$
\xi(t_n) \in \mathcal{G}^{(n + 1)|b}(2\bar\epsilon) \subseteq \mathcal{G}^{(k - 1)|b}(2\bar\epsilon).
$
As a result, 
for all $t \in [0, t_{k-1})$ we have
$
\xi(t) \in \bar{\mathcal G}^{(k-1)|b}(2\bar\epsilon) \subseteq I_{2\bar\epsilon}
$
due to \eqref{proof, property, inclusion of bar mathcal G k b and I bar epsilon, lemma: choose key parameters, first exit time analysis}.
This verifies claim $(i)$.

For claim $(ii)$, we proceed with a proof by contradiction and suppose that $\xi(t_{k-1}) \in I_{\bar\epsilon}$.
By \eqref{constant bar c for bar epsilon, first exit time}, we then get 
$
\xi(t) = \bm y_{ t - t_{k-1} }\big(\xi(t_{k-1})\big) \in I_{\bar c\bar\epsilon}
$
for all $t \in [t_{k-1},T]$.
Together with claim $(i)$, we arrive at the contradiction that $\xi(t) \in I_{\bar c\bar\epsilon}$ for all $t \in [0,T]$.

For claim $(iii)$, the fact $\norm{\xi(t_{k-1})} \geq \bar\epsilon$ follows directly from claim~$(ii)$ and \eqref{constant bar epsilon, new, 1, first exit time analysis}.
For any $j = 1,\cdots,k-1$ and any $t \in [t_{j-1}, t_{j})$,
we proceed with a proof by contradiction and suppose that $\norm{\xi(t)} \leq \bar\epsilon$.
Then we have $\norm{\xi(t_j-)} \leq \bar\epsilon$ due to \eqref{constant bar epsilon, new, 2, first exit time analysis},
and hence
$
\xi(t_j) \in \mathcal G^{(1)|b}(2\bar\epsilon).
$
As a result, we arrive at the contradiction that 
$
\xi(t_{k-1}) \in \mathcal{G}^{(k-1)|b}(2\bar\epsilon) \subseteq I_{\bar\epsilon},
$
due to
$\mathcal{G}^{(k-1)|b}(2\bar\epsilon) \subseteq \mathcal{\bar G}^{(k-1)|b}(2\bar\epsilon)$
and \eqref{proof, property, inclusion of bar mathcal G k b and I bar epsilon, lemma: choose key parameters, first exit time analysis}.
This concludes the proof of claim~$(iii)$.

We prove claim~$(iv)$ for $\bar t \delequal k \cdot \bm t(\bar\epsilon/2)$ where $\bm t(\epsilon)$ is defined in \eqref{def: t epsilon function, first exit analysis}.
Consider the following proof by contradiction.
If $t_{k-1} \geq \bar t = (k - 1) \cdot \bm t(\bar\epsilon/2)$,
then there must be some $j = 1,2,\cdots,k-1$ such that $t_j - t_{j-1} \geq \bar t(\epsilon/2)$.
By claim~$(i)$,
we have $\xi(t_{j-1}) \in I^-_{2\bar\epsilon} \subseteq I_{\bar\epsilon/2}$.
Using the property~\eqref{property: t epsilon function, first exit analysis},
we yield $\xi(t_j-) = \lim_{t \uparrow t_j}\xi(t) \in \bar B_{ \bar\epsilon/2 }(\bm 0)$,
which implies $\norm{\xi(t)} < \epsilon$ for all $t$ less than but close enough to $t_{j}$ and contradicts claim~$(iii)$.
This concludes the proof of claim~$(iv)$.

Lastly, we prove claim~$(v)$ for $\bar\delta > 0$ small enough such that
\begin{align}
    \exp(D\bar t)\cdot C\bar\delta < \bar\epsilon,
    \qquad
    C\bar\delta < b,
    \nonumber
\end{align}
where $D\geq 1$ is the Lipschitz coefficient in Assumption~\ref{assumption: lipschitz continuity of drift and diffusion coefficients} and $C \geq 1$ is the constant in \eqref{constant C, boundedness of a and sigma, lemma: choose key parameters, first exit time analysis}.
Again, we consider a proof by contradiction.
Suppose that for the cadlag path $\xi$ in \eqref{claim, def of xi, lemma: choose key parameters, first exit time analysis}
there is some $j = 0,1,\cdots, k - 1$ such that $\norm{\bm w_j} < \bar\delta$.
First, we consider the case where $j \leq k-2$.
Then note that (for the proof of claim~$(v)$, we interpret $\xi(0-)$ as $\bm x_0$ while, by definition, $\xi(0) = \bm x_0 + \varphi_b\big( \bm \sigma(\bm x_0)\bm w_0\big)$),
we have 
\begin{align*}
    \xi(t_j) - \xi(t_j-) 
    =
    \varphi_b\big(\bm \sigma\big(\xi(t_j-)\bm w_j\big),
\end{align*}
and hence
$
\norm{\xi(t_j) - \xi(t_j-) } \leq C\bar\delta.
$
By Gronwall's inequality, we then get
\begin{align*}
    \norm{\bm y_{t - t_j}\big(\xi(t_j-)\big) - \xi(t) } \leq \exp\big(D(t - t_j)\big)\cdot C\bar\delta
    \qquad
    \forall t \in [t_j, t_{j+1}).
\end{align*}
Recall that we currently focus on the case where $j \leq k-2$.
By claim $(iv)$ and our choice of $\bar\delta$, we get $\exp\big(D(t - t_j)\big)\cdot C\bar\delta \leq \exp\big(D\bar t\big) \cdot C\bar\delta <\bar\epsilon$ in the display above.
This implies the existence of some $\xi^\prime \in \D^{(k-1)|b}_{ \bar B_{\bar\epsilon}(\bm 0) }(\bar\epsilon)$
such that 
$
\sup_{t \in [0,T]}\norm{\xi(t) - \xi^\prime(t)} < \bar\epsilon.
$
However, by results in part $(a)$,
we must have $\xi^\prime(t) \in I^-_{2\bar\epsilon}\ \forall t \in [0,T]$,
which leads to $\xi(t) \in I_{\epsilon}^-\ \forall t \in [0,T]$.
This contradicts the running assumption of part $(b)$
that $\xi(t) \notin I_{\bar c\bar\epsilon}$ for some $t \in [0,T]$,
and allows us to conclude the proof of claim~$(v)$ for the cases where $j \leq k-2$.
In case that $j = k-1$,
by claim~$(i)$ we have $\xi(t_{k-1} - ) = \lim_{t \uparrow t_{k-1}}\xi(t)\in I^-_{2\bar\epsilon}$.
Meanwhile, by definition of the mapping $\bar h^{(k-1)|b}_{[0,T]}$,
we have
$
\xi(t_{k-1})
 = 
\xi(t_{k-1} - ) + \varphi\Big(\bm \sigma\big(\xi(t_{k-1} - )\big)\bm w_{k-1} \Big).
$
By $\norm{\bm w_{k-1}} < \bar\delta$ and our choice of $\bar\delta$ above,
we have 
$
\norm{ \varphi\Big(\bm \sigma\big(\xi(t_{k-1} - )\big)\bm w_{k-1} \Big)} < \bar\epsilon
$
and hence
$
\xi(t_{k-1}) \in I_{\bar\epsilon}.
$
Due to the contradiction with claim~$(ii)$, we conclude the proof.

\medskip
\noindent
$(c)$
The proof is almost identical to that of Lemma~\ref{lemma: SGD close to approximation x breve, LDP clipped} based on an inductive argument.
We omit the details to avoid repetition.

\medskip
\noindent
$(d)$
Let $\bar t$ be the constant specified in part $(b)$.
We claim that: if $\xi(t) \notin I_{\bar c\bar \epsilon}\text{ or }\widecheck{\xi}(t)\notin I_{\bar c\bar \epsilon}$ for some $t \in [0,T]$, then
\begin{align}
        \sup_{t \in [t_1,t_{\mathcal J^I_b}]}
        \norm{\widecheck{\xi}(t - t_1) - \xi(t) }
        < 
        \underbrace{\Big( 
            2\exp\big(D\bar t\big)\cdot D
        \Big)^{\mathcal J^I_b + 1}}_{\delequal \rho^*} \cdot  \epsilon_0
            \qquad \forall \epsilon_0 \in (0,\bar\epsilon].
    \label{goal, part e, lemma: choose key parameters, first exit time analysis}
\end{align}
As a result, claims of part $(d)$ hold for any $\epsilon_0 \in (0,\bar\epsilon)$ small enough such that $\rho^*\epsilon_0 < \Delta$.
Now, it only remains to prove claim~\eqref{goal, part e, lemma: choose key parameters, first exit time analysis}.
Due to $\norm{\bm x} = \norm{\xi(0)} < \epsilon_0$ and \eqref{constant bar epsilon, new, 2, first exit time analysis},
we have 
$
\norm{\xi(t_1-)} \leq \epsilon_0.
$
This allows us to apply results in part $(c)$
and get (recall our choice of $T = t_{\mathcal J_b^*} + 1$)
\begin{align*}
        \sup_{t \in [t_1,t_{\mathcal J^I_b}]} 
           \norm{\xi(t) - \widecheck\xi(t - t_1)} \leq \Big( 2\exp\big(D( t_{\mathcal J^I_b}  - t_1)\big) \cdot D\Big)^{\mathcal J^I_b + 1} \cdot  \epsilon_0,
\end{align*}
Lastly, 
if $\xi(t) \notin I_{\bar c\bar \epsilon}\text{ for some }{t \in [t_1,T]}$, 
then $t_{\mathcal{J}^I_b} - t_1 < \bar t$ by claim~$(iv)$ of part $(b)$.
Likewise, if $\widecheck{\xi}(t) \notin I_{\bar c\bar \epsilon}\text{ for some }{t \in [0,T]}$, 
then we get $t_{\mathcal{J}^I_b} < \bar t$.
In both cases, we get $t_{\mathcal J_b^*}-t_1 \leq \bar t$.
This concludes the proof.
\end{proof}

The next lemma studies the mass the measure $\widecheck{\mathbf C}^{(k)|b}$ charges on the boundary of the domain $I$.

\begin{lemma}
\label{lemma: exit rate strictly positive, first exit analysis}
\linksinthm{lemma: exit rate strictly positive, first exit analysis}
Under Assumptions \ref{assumption: lipschitz continuity of drift and diffusion coefficients} and \ref{assumption: shape of f, first exit analysis},
    $
    \widecheck{\mathbf C}^{(\mathcal{J}^I_b)|b}( I^\complement ) < \infty.
    $
\end{lemma}

\begin{proof}
\linksinpf{lemma: exit rate strictly positive, first exit analysis}
Let $\bar\epsilon > 0$ be such that the conditions in \eqref{constant bar epsilon, new, 1, first exit time analysis}--\eqref{constant bar epsilon, new, 3, first exit time analysis} hold.
Let $\bar t$ and $\bar \delta$ be the constants characterized in Lemma \ref{lemma: choose key parameters, first exit time analysis}.
Observe that
(we write $\textbf W = (\bm w_1,\cdots,\bm w_{\mathcal J^I_b})$)
\begin{align*}
    & 
    \widecheck{\mathbf C}^{(\mathcal{J}^I_b)|b}\big( I^\complement \big)
    \\
    & = 
    \int \mathbbm{I}\bigg\{
    \widecheck{g}^{ (\mathcal{J}^I_b - 1)|b }
    \big( \varphi_b(\bm \sigma(\bm 0) \bm w_1), (\bm w_2,\cdots,\bm w_{\mathcal{J}^I_b}), (t_1,\cdots,t_{\mathcal{J}^I_b-1})\big) \notin I
    \bigg\}
    \\ 
    &\qquad\qquad\qquad\qquad\qquad\qquad\qquad\qquad
    \big((\nu_\alpha \times \mathbf S)\circ \Phi\big)^{ \mathcal J^I_b }(d \textbf W) \times \mathcal{L}^{\mathcal{J}^I_b - 1 \uparrow}_\infty(dt_1,\cdots, dt_{\mathcal J^I_b - 1})
    \\
    & = 
    \int \mathbbm{I}\bigg\{
    h^{ (\mathcal{J}^I_b - 1)|b }_{[0,1 + t_{\mathcal{J}^I_b-1}]}
    \big( \varphi_b(\bm \sigma(\bm 0) \bm w_1), (\bm w_2,\cdots,\bm w_{\mathcal{J}^I_b}), (t_1,\cdots,t_{\mathcal{J}^I_b-1})\big)(t_{\mathcal{J}^I_b-1}) \notin I
    \bigg\}
    \\ 
    &\qquad\qquad\qquad\qquad\qquad\qquad\qquad\qquad
    \big((\nu_\alpha \times \mathbf S)\circ \Phi\big)^{ \mathcal J^I_b }(d \textbf W) \times \mathcal{L}^{\mathcal{J}^I_b - 1 \uparrow}_\infty(dt_1,\cdots, dt_{\mathcal J^I_b - 1})
    \\
    & \leq 
    \int \mathbbm{I}\Big\{ \norm{\bm w_j} > \bar\delta\ \forall j \in [\mathcal{J}^I_b];\ t_{\mathcal{J}^I_b-1} < \bar t\ \Big\}
    \big((\nu_\alpha \times \mathbf S)\circ \Phi\big)^{ \mathcal J^I_b }(d \textbf W) \times \mathcal{L}^{\mathcal{J}^I_b - 1 \uparrow}_\infty(dt_1,\cdots, dt_{\mathcal J^I_b - 1})
    \\ 
    &\qquad\qquad\qquad
    \text{by part $(b)$ of Lemma \ref{lemma: choose key parameters, first exit time analysis}}
    \\
    & \leq \bar t^{ \mathcal{J}^I_b - 1 } \big/ \bar\delta^{\alpha\mathcal{J}^I_b} < \infty.
\end{align*}
This concludes the proof.
\end{proof}

To conclude, we provide the proof of Lemma~\ref{lemma: limiting measure, with exit location B, first exit analysis}.

\begin{proof}[Proof of Lemma~\ref{lemma: limiting measure, with exit location B, first exit analysis}]\linksinpf{lemma: limiting measure, with exit location B, first exit analysis}
Let $\bar c \in (0,1)$ be the constant fixed in \eqref{constant bar c for bar epsilon, first exit time}.
By part $(e)$ of Lemma~\ref{lemma: choose key parameters, first exit time analysis},
for the fixed $\Delta \in (0,\bar\epsilon)$, we are able to fix some $\epsilon_0 \in (0,\frac{\Delta}{2} \wedge \bar c\bar\epsilon)$ such that
the following claim holds:
given 
        $T > 0$, 
        $\bm x \in \R^m$,
        $\textbf W = (\bm w_1,\cdots, \bm w_{\mathcal J^I_b}) \in \R^{d \times  \mathcal J^I_b},\ (t_1,\cdots,t_{\mathcal J^I_b})\in (0,T]^{\mathcal J^I_b\uparrow}$,
        if $\norm{\bm x} \leq \epsilon_0$
        and
        $
            \max_{j \in [\mathcal J^I_b] }\norm{\bm w_j} \leq \epsilon_0^{-\frac{1}{2\mathcal J^I_b}},
        $
        then    
        \begin{align}
        \xi(t) \notin I_{\bar c\bar \epsilon}\text{ or }\widecheck{\xi}(t)\notin I_{\bar c\bar \epsilon}
        \text{ for some }t \in [t_1,T - t_1]
        \quad
        \Longrightarrow
        \quad
        \sup_{t \in [t_1,t_{\mathcal J^I_b}]}
        \norm{\widecheck{\xi}(t - t_1) - \xi(t) } < \Delta,
        \label{choice of epsilon 0, lemma: limiting measure, with exit location B, first exit analysis}
        \end{align}
        where
        \begin{align*}
                \xi & =  h^{ (\mathcal J^I_b)|b }_{[0,T]}\big(\bm x, \textbf W, (t_1,\cdots,t_{\mathcal J^I_b})\big),
                \\ 
                \widecheck{\xi}
                & =
                 h^{ (\mathcal J^I_b - 1)|b }_{[0,T]}
                \Big(
                    \varphi_b\big( \bm \sigma(\bm 0)\bm w_1\big), (\bm w_2,\cdots, \bm w_{ \mathcal J^I_b }), 
                    (t_2 - t_1,t_3 - t_1,\cdots,t_{\mathcal J^I_b} - t_1)
                \Big).
            \end{align*}

Henceforth in the proof, we fix some $\epsilon \in (0,\epsilon_0]$
and
$B \subseteq (I_{\epsilon})^c$.
Due to our choice of $\epsilon \leq \epsilon_0 < \bar c\bar\epsilon$,
we have
$B \subseteq (I_{\bar c \bar \epsilon})^c$.
To prove the lower bound,
let 
$$ 
\widetilde E =  
\Big\{ \xi \in \D[0,T]:\ \exists t \in [0,T]\text{ s.t. }\xi(t) \in B_{\Delta/2},\ \xi(s) \in I_{2\epsilon}\ \forall s \in [0,t)  \Big\}.
$$
For any $\xi \in \widetilde E$ and any $\xi^\prime$ with $\dj{[0,T]}(\xi,\xi^\prime) < \epsilon$,
due to $\epsilon \leq \epsilon_0 < \Delta/2$,
there must be some $t^\prime \in [0,T]$ such that 
$
\xi^\prime(t^\prime) \in B
$
and
$
\xi^\prime(s) \in I_\epsilon\ \forall s \in [0,t^\prime).
$
This implies that
$\xi^\prime \in \widecheck E(\epsilon,B,T)$,
and hence
\begin{align*}
\widetilde E 
    \subseteq
    \Big(\widecheck E(\epsilon,B,T)\Big)_{\epsilon}
    \subseteq
    \Big(\widecheck E(\epsilon,B,T)\Big)^\circ.
\end{align*}
Therefore, for any $\bm x \in \R^m$ with $\norm{\bm x} \leq \epsilon \leq \epsilon_0$,
\begin{align}
            \mathbf{C}^{ (\mathcal{J}^I_b)|b}_{[0,T]}
    \bigg( 
    \Big(\widecheck E(\epsilon,B,T)\Big)^\circ;\ \bm x
    \bigg)
    & \geq 
    \int \mathbbm{I}\Big\{ 
    h^{ (\mathcal{J}^I_b)|b }_{[0,T]}(\bm x,\textbf W,\bm t)\in \widetilde E
    \Big\}
    \big((\nu_\alpha \times \mathbf S)\circ \Phi\big)^{ \mathcal J^I_b }(d \textbf W)
    \times \mathcal{L}^{ \mathcal{J}^I_b \uparrow }_T(d\bm t)
    \nonumber
    \\
    &
    = 
    \int \widetilde \phi_B(t_1,\bm x) \mathcal L_T(dt_1),
    \label{proof: lower bound, intermediate, lemma: limiting measure, with exit location B, first exit analysis}
\end{align}
where $\mathcal L_T$ is the Lebesgue measure on $(0,T)$,
$\mathcal L_T^{k\uparrow}$ is the $k$-fold ofq Lebesgue measure restricted on 
$
\{ (t_1,\cdots,t_k) \in (0,T)^k:\ t_1 < t_2 < \cdots <t_k \},
$
and
\begin{align*}
\widetilde \phi_B(t_1,\bm x)
& = 
\int 
    \mathbbm{I}\bigg\{
        \exists t \in [0,T]\text{ s.t. }
        h^{ (\mathcal{J}^I_b)|b }_{[0,T]}
        \Big(\bm x,\textbf W,
        (t_1,t_1 + u_2, t_1 + u_3,\cdots,t_1 + u_{\mathcal{J}^I_b})\Big)(t) \in B_{\Delta/2}
    \\ 
&\qquad\qquad 
\text{ and }
h^{ (\mathcal{J}^I_b)|b }_{[0,T]}\Big(\bm x,\textbf W,
        (t_1,t_1 + u_2, t_1 + u_3,\cdots,t_1 + u_{\mathcal{J}^I_b})\Big)(s) \in I_{2\epsilon}\ \forall s \in [0,t)
\bigg\} 
\\ 
&\qquad\qquad\qquad\qquad\qquad\qquad
     \big((\nu_\alpha \times \mathbf S)\circ \Phi\big)^{ \mathcal J^I_b }(d \textbf W)
    \times \mathcal{L}^{ \mathcal{J}^I_b - 1 \uparrow }_{T-t_1}(du_2,\cdots,du_{\mathcal{J}^I_b}).
\end{align*}
Set
$
\bm x_0 = \lim_{t \uparrow t_1} \bm y_{t}(\bm x),
$
and note that
\begin{align*}
&  h^{ (\mathcal{J}^I_b)|b }_{[0,T]}\Big(\bm x,(\bm w_1,\cdots,\bm w_{\mathcal J^I_b}),
        (t_1,t_1 + u_2, t_1 + u_3,\cdots,t_1 + u_{\mathcal{J}^I_b})\Big)(t_1 + s)  
\\ 
& = 
h^{ (\mathcal{J}^I_b - 1)|b }_{[0,T - t_1]}\Big(\bm x_0 + \varphi_b(\bm \sigma(\bm x_0) \bm w_1),(\bm w_2,\cdots,\bm w_{\mathcal J^I_b}),
        (u_2, u_3,\cdots,u_{\mathcal{J}^I_b})\Big)(s)  \qquad \forall s \in [0,T - t_1].
\end{align*}
Therefore, for any $t_1 \in [0,T - \bar t]$ and $\bm x$ with $\norm{\bm x} \leq \epsilon$, 
by property \eqref{constant bar epsilon, new, 2, first exit time analysis}
we have $\norm{\bm x_0} \leq \epsilon \leq \epsilon_0 \leq \Delta/2$, and
\begin{align*}
& \widetilde \phi_B(t_1,x)
\\
& \geq \inf_{\bm x_0:\ \norm{\bm x_0} \leq \frac{\Delta}{2}}
\int 
    \mathbbm{I}\bigg\{
        \exists t \in [0,T - t_1]\text{ s.t. }
        h^{ (\mathcal{J}^I_b - 1)|b }_{[0,T - t_1]}
        \Big(\bm x_0 + \varphi_b(\bm \sigma(\bm x_0) \bm w_1),(\bm w_2,\cdots,\bm w_{\mathcal J^I_b}),
        (u_2, \cdots,u_{\mathcal{J}^I_b})\Big)(t)
        \in B_{\Delta/2}
    \\ 
&\qquad\qquad 
\text{ and }
h^{ (\mathcal{J}^I_b - 1)|b }_{[0,T - t_1]}
\Big(\bm x_0 + \varphi_b(\bm\sigma(\bm x_0) \bm w_1),(\bm w_2,\cdots,\bm w_{\mathcal J^I_b}),
        (u_2, \cdots,u_{\mathcal{J}^I_b})\Big)(s) \in I_{2\epsilon}\ \forall s \in [0,t)
\bigg\} 
\\ 
&\qquad\qquad\qquad\qquad\qquad\qquad\qquad\qquad\qquad\qquad
    \big((\nu_\alpha \times \mathbf S)\circ \Phi\big)^{ \mathcal J^I_b }(d \textbf W)
    \times \mathcal{L}^{ \mathcal{J}^I_b - 1 \uparrow }_{T-t_1}(du_2,\cdots,du_{\mathcal{J}^I_b})
\\ 
& = \inf_{\bm x_0:\ \norm{\bm x_0} \leq \frac{\Delta}{2}}
\int
\mathbbm{I}\bigg\{
    h^{ (\mathcal{J}^I_b - 1)|b }_{[0,T - t_1]}
    \Big(\bm x_0 + \varphi_b(\bm \sigma(\bm x_0) \bm w_1),(\bm w_2,\cdots,\bm w_{\mathcal J^I_b}),
        (u_2, \cdots,u_{\mathcal{J}^I_b})\Big)(u_{ \mathcal J^I_b }) \in B_{\Delta/2};\ 
        \min_{j \in [\mathcal J^I_b]}\norm{\bm w_j} > \bar\delta
\bigg\}
\\ 
&\qquad\qquad\qquad\qquad\qquad\qquad\qquad\qquad\qquad\qquad
    \big((\nu_\alpha \times \mathbf S)\circ \Phi\big)^{ \mathcal J^I_b }(d \textbf W)
    \times \mathcal{L}^{ \mathcal{J}^I_b - 1 \uparrow }_{T-t_1}(du_2,\cdots,du_{\mathcal{J}^I_b})
\\ 
&\qquad \qquad \text{by claims $(i)$, $(ii)$, and $(v)$ in part $(b)$ of Lemma~\ref{lemma: choose key parameters, first exit time analysis}}
\\ 
& \geq 
\inf_{\bm x_0:\ \norm{\bm x_0} \leq \frac{\Delta}{2}}
\int
\mathbbm{I}\bigg\{
    h^{ (\mathcal{J}^I_b - 1)|b }_{[0,T - t_1]}
    \Big(\bm x_0 + \varphi_b(\bm \sigma(\bm x_0) \bm w_1),(\bm w_2,\cdots,\bm w_{\mathcal J^I_b}),
        (u_2, \cdots,u_{\mathcal{J}^I_b})\Big)(u_{ \mathcal J^I_b }) \in B_{\Delta/2};
        \\ 
    &\qquad\qquad\qquad\qquad\qquad \qquad\qquad\qquad\qquad\qquad\qquad\qquad
        \min_{j \in [\mathcal J^I_b]}\norm{\bm w_j} > \bar\delta,\ 
         \max_{j \in [\mathcal J^I_b] }\norm{\bm w_j} \leq \epsilon_0^{-\frac{1}{2\mathcal J^I_b}}
\bigg\}
\\ 
&\qquad\qquad\qquad\qquad\qquad\qquad\qquad\qquad\qquad\qquad
    \big((\nu_\alpha \times \mathbf S)\circ \Phi\big)^{ \mathcal J^I_b }(d \textbf W)
    \times \mathcal{L}^{ \mathcal{J}^I_b - 1 \uparrow }_{T-t_1}(du_2,\cdots,du_{\mathcal{J}^I_b})
\\
& \geq 
\int
\mathbbm{I}\bigg\{
    h^{ (\mathcal{J}^I_b - 1)|b }_{[0,T - t_1]}
    \Big(\varphi_b(\bm \sigma(\bm 0) \bm w_1),(\bm w_2,\cdots,\bm w_{\mathcal J^I_b}),
        (u_2, \cdots,u_{\mathcal{J}^I_b})\Big)(u_{ \mathcal J^I_b }) \in B_{\Delta};
        \\ 
    &\qquad\qquad\qquad\qquad\qquad \qquad\qquad\qquad\qquad\qquad\qquad\qquad
        \min_{j \in [\mathcal J^I_b]}\norm{\bm w_j} > \bar\delta,\ 
         \max_{j \in [\mathcal J^I_b] }\norm{\bm w_j} \leq \epsilon_0^{-\frac{1}{2\mathcal J^I_b}}
\bigg\}
\\ 
&\qquad\qquad\qquad\qquad\qquad\qquad\qquad\qquad\qquad\qquad
    \big((\nu_\alpha \times \mathbf S)\circ \Phi\big)^{ \mathcal J^I_b }(d \textbf W)
    \times \mathcal{L}^{ \mathcal{J}^I_b - 1 \uparrow }_{T-t_1}(du_2,\cdots,du_{\mathcal{J}^I_b})
\\
&\qquad \qquad \text{by property \eqref{choice of epsilon 0, lemma: limiting measure, with exit location B, first exit analysis}}
\\
& = 
\int
\mathbbm{I}\bigg\{
    \widecheck{g}^{ (\mathcal{J}^I_b - 1)|b }_{[0,T - t_1]}
    \Big(\varphi_b(\bm \sigma(\bm 0) \bm w_1),(\bm w_2,\cdots,\bm w_{\mathcal J^I_b}),
        (u_2, \cdots,u_{\mathcal{J}^I_b})\Big) \in B_{\Delta};
        \\ 
    &\qquad\qquad\qquad\qquad\qquad \qquad\qquad\qquad\qquad\qquad\qquad\qquad
        \min_{j \in [\mathcal J^I_b]}\norm{\bm w_j} > \bar\delta,\ 
         \max_{j \in [\mathcal J^I_b] }\norm{\bm w_j} \leq \epsilon_0^{-\frac{1}{2\mathcal J^I_b}}
\bigg\}
\\ 
&\qquad\qquad\qquad\qquad\qquad\qquad\qquad\qquad\qquad\qquad
    \big((\nu_\alpha \times \mathbf S)\circ \Phi\big)^{ \mathcal J^I_b }(d \textbf W)
    \times \mathcal{L}^{ \mathcal{J}^I_b - 1 \uparrow }_{T-t_1}(du_2,\cdots,du_{\mathcal{J}^I_b})
\\
&\qquad\qquad \text{by the definition of $\widecheck g^{(k)|b}$ in \eqref{def: mapping check g k b, endpoint of path after the last jump, first exit analysis}}
\\ 
& = 
\int
\mathbbm{I}\bigg\{
    \widecheck{g}^{ (\mathcal{J}^I_b - 1)|b }_{[0,T - t_1]}
    \Big(\varphi_b(\bm \sigma(\bm 0) \bm w_1),(\bm w_2,\cdots,\bm w_{\mathcal J^I_b}),
        (u_2, \cdots,u_{\mathcal{J}^I_b})\Big) \in B_{\Delta};
        \\ 
    &\qquad\qquad\qquad\qquad\qquad \qquad\qquad\qquad\qquad\qquad\qquad\qquad
        \min_{j \in [\mathcal J^I_b]}\norm{\bm w_j} > \bar\delta,\ 
         \max_{j \in [\mathcal J^I_b] }\norm{\bm w_j} \leq \epsilon_0^{-\frac{1}{2\mathcal J^I_b}}
\bigg\}
\\ 
&\qquad\qquad\qquad\qquad\qquad\qquad\qquad\qquad\qquad\qquad
    \big((\nu_\alpha \times \mathbf S)\circ \Phi\big)^{ \mathcal J^I_b }(d \textbf W)
    \times \mathcal{L}^{ \mathcal{J}^I_b - 1 \uparrow }_{\bar t}(du_2,\cdots,du_{\mathcal{J}^I_b})
\\
&\qquad\qquad
\text{by claim $(v)$ in part $(b)$ of Lemma~\ref{lemma: choose key parameters, first exit time analysis}}
\\
& \geq 
\int
\mathbbm{I}\bigg\{
    \widecheck{g}^{ (\mathcal{J}^I_b - 1)|b }_{[0,T - t_1]}
    \Big(\varphi_b(\bm \sigma(\bm 0) \bm w_1),(\bm w_2,\cdots,\bm w_{\mathcal J^I_b}),
        (u_2, \cdots,u_{\mathcal{J}^I_b})\Big) \in B_{\Delta};\ \min_{j \in [\mathcal J^I_b]}\norm{\bm w_j} > \bar\delta\bigg\}
\\ 
&\qquad\qquad\qquad\qquad\qquad\qquad\qquad\qquad\qquad\qquad
    \big((\nu_\alpha \times \mathbf S)\circ \Phi\big)^{ \mathcal J^I_b }(d \textbf W)
    \times \mathcal{L}^{ \mathcal{J}^I_b - 1 \uparrow }_{\bar t}(du_2,\cdots,du_{\mathcal{J}^I_b})
\\ 
&\qquad - 
\int
\mathbbm{I}\bigg\{
    \min_{j \in [\mathcal J^I_b]}\norm{\bm w_j} > \bar\delta,\ \max_{j \in [\mathcal J^I_b]}\norm{\bm w_j} > \epsilon_0^{-\frac{1}{2\mathcal J^I_b}}
    \bigg\}
    \big((\nu_\alpha \times \mathbf S)\circ \Phi\big)^{ \mathcal J^I_b }(d \textbf W)
    \times \mathcal{L}^{ \mathcal{J}^I_b - 1 \uparrow }_{\bar t}(du_2,\cdots,du_{\mathcal{J}^I_b}).
\end{align*}
We focus on the two integrals one the RHS of the last inequality in the display above.
It is easy to see that the latter is upper bounded by 
\begin{align*}
    \widecheck{\bm c}(\epsilon_0) =
    \mathcal J^I_b \cdot (\bar t)^{ \mathcal J^I_b - 1 } \cdot (\bar\delta)^{ -\alpha \cdot (\mathcal J^I_b - 1) }
    \cdot 
    \epsilon_0^{ \frac{\alpha}{2\mathcal J^I_b}  }.
\end{align*}
As for the former, using part $(b)$ of Lemma~\ref{lemma: choose key parameters, first exit time analysis} and the fact that $B_\Delta \subseteq B \subseteq (I_{\bar c \bar\epsilon})^\complement$ again,
we yield
\begin{align*}
    & \int
\mathbbm{I}\bigg\{
    \widecheck{g}^{ (\mathcal{J}^I_b - 1)|b }_{[0,T - t_1]}
    \Big(\varphi_b(\bm \sigma(\bm 0) \bm w_1),(\bm w_2,\cdots,\bm w_{\mathcal J^I_b}),
        (u_2, \cdots,u_{\mathcal{J}^I_b})\Big) \in B_{\Delta};\ \min_{j \in [\mathcal J^I_b]}\norm{\bm w_j} > \bar\delta\bigg\}
\\ 
&\qquad\qquad\qquad\qquad\qquad\qquad\qquad\qquad\qquad\qquad
    \big((\nu_\alpha \times \mathbf S)\circ \Phi\big)^{ \mathcal J^I_b }(d \textbf W)
    \times \mathcal{L}^{ \mathcal{J}^I_b - 1 \uparrow }_{\bar t}(du_2,\cdots,du_{\mathcal{J}^I_b})
\\ 
& = 
\int
\mathbbm{I}\bigg\{
    \widecheck{g}^{ (\mathcal{J}^I_b - 1)|b }_{[0,T - t_1]}
    \Big(\varphi_b(\bm \sigma(\bm 0) \bm w_1),(\bm w_2,\cdots,\bm w_{\mathcal J^I_b}),
        (u_2, \cdots,u_{\mathcal{J}^I_b})\Big) \in B_{\Delta}\}
\\
&\qquad\qquad\qquad\qquad\qquad\qquad\qquad\qquad\qquad\qquad
    \big((\nu_\alpha \times \mathbf S)\circ \Phi\big)^{ \mathcal J^I_b }(d \textbf W)
    \times \mathcal{L}^{ \mathcal{J}^I_b - 1 \uparrow }_{\infty}(du_2,\cdots,du_{\mathcal{J}^I_b})
\\ 
& = \widecheck{\mathbf C}^{ (\mathcal J^I_b)|b}(B_\Delta).
\end{align*}
In summary, for any $\bm x \in \R^m$ with $\norm{\bm x} \leq \epsilon$
and
$t_1 \in [0,T - \bar t]$ ,
we have shown that
$$
\widetilde \phi_B(t_1, \bm x) \geq \widecheck{\mathbf C}^{ (\mathcal J^I_b)|b}(B_\Delta) - 
\widecheck{\bm c}(\epsilon_0).
$$
Together with the trivial bound that 
$
\widetilde \phi_B(t_1,\bm x) \geq 0
$
for all $t_1 > T - \bar t$,
we have in \eqref{proof: lower bound, intermediate, lemma: limiting measure, with exit location B, first exit analysis}
 that
$$
\mathbf{C}^{ (\mathcal{J}^I_b)|b}_{[0,T]}\bigg( \Big(\widecheck E(\epsilon,B,T)\Big)^\circ;\ \bm x\bigg) 
\geq 
(T-\bar t) \cdot \Big( \widecheck{\mathbf C}^{ (\mathcal J^I_b)|b}(B_\Delta) - \widecheck{\bm c}(\epsilon_0)  \Big)
$$
for all $\bm x \in \R^m$ with $\norm{\bm x} \leq \epsilon$.
This concludes the proof of the lower bound.
The proof to the upper bound is almost identical, so we omit the details here.
\end{proof}

Next, we prove the claims in Remark~\ref{remark: regularity conditions for first exit times} regarding the regularity conditions of Theorem~\ref{theorem: first exit time, unclipped} in the non-degenerate one-dimensional cases.
That is, we consider the following iterates in $\R^1$
\begin{align}
    X^\eta_0(x) = x;\qquad
    X^\eta_j(x) = X^\eta_{j - 1}(x) +  \eta a\big(X^\eta_{j - 1}(x)\big) + \eta\sigma\big(X^\eta_{j - 1}(x)\big)Z_j,\ \ \forall j \geq 1
     \nonumber
\end{align}
and impose the following assumptions throughout the rest of Section~\ref{subsec: lemma for measure check C}.

\begin{assumption}[Lipschitz Continuity ($\R^1$)]
\label{assumption: lipschitz continuity of drift and diffusion coefficients, R1}
There exists some $\notationdef{notation-Lipschitz-constant-L-LDP}{D} \in [1,\infty)$ such that
$$|\sigma(x) - \sigma(y)| \vee |a(x)-a(y)| \leq D|x-y|\ \ \ \forall x,y \in \mathbb{R}.$$
\end{assumption}

\begin{assumption}[Attraction Field ($\R^1$)]
\label{assumption: shape of f, first exit analysis, R1}
 $a(0)= 0$.
 Besides, $I \subset \R$ is a bounded set (i.e., $\sup_{x \in I}|x| < \infty)$
 such that $0 \in I$ and 
it holds for all $x \in I\setminus \{0\}$ that $a(x)x < 0$. 
\end{assumption}

\begin{assumption}[Nondegeneracy ($\R^1$)]
\label{assumption: nondegeneracy of diffusion coefficients, R1}
$\sigma(x) > 0\ \ \forall x \in \R.$
\end{assumption}

Specifically, we write
$I = (s_\text{left},s_\text{right})$ where $s_\text{left} < 0 < s_\text{right}$,
and show that the following lemmas hold for all $b > 0$ such that
\begin{align}
s_\text{left}/b \notin \mathbb Z,
\qquad
s_\text{right}/b \notin \mathbb Z.
    \label{condition for boundary of I / b is not an integer, appendix, R1 setting}
\end{align}
In other words, we show that the lemmas below holds for (Lebesgue) almost all $b > 0$.
Besides, we adopt the following choices of $\bar \epsilon > 0$ and $\bm t(\epsilon)$.
We set
\begin{align}
    {l} & = \inf_{x \in I^\complement}|x| = |s_\text{left}| \wedge s_\text{right},\qquad 
    {\mathcal J^*_b} = \ceil{l/b},
    \label{def: first exit time, J *, R1}
\end{align}
and note that
we have $l > (\mathcal J^*_b-1)b$. 
This allows us to fix, throughout this section, some $\bar\epsilon > 0$ small enough such that
\begin{align}
    \bar\epsilon \in (0,1),\qquad 
    l > (\mathcal J^*_b - 1)b + 3\bar\epsilon.
    \label{constant bar epsilon, first exit time analysis, R1}
\end{align}
Next, for any $\epsilon \in (0,\bar\epsilon)$,
let
\begin{align}
    { \bm{t}(\epsilon) } \delequal \min\big\{ t \geq 0:\ \bm{y}_t(s_\text{left} + \epsilon) \in [-\epsilon,\epsilon]\text{ and }\bm{y}_t(s_\text{right} - \epsilon) \in [-\epsilon,\epsilon] \big\}
    \label{def: t epsilon function, first exit analysis, R1}
\end{align}
for the ODE
\begin{align}
\bm y_0(x) = x,\qquad
   \frac{d\bm{y}_t(x)}{dt} = a\big(\bm{y}_t(x)\big) \ \ \forall t \geq 0.
   \label{def ODE path y t, R1}
\end{align}
Also, recall that ${I_\epsilon} \delequal (s_\text{left} + \epsilon,s_\text{right} - \epsilon)$
is the $\epsilon$-shrinkage of set $I$.
We use $I^-_\epsilon = [s_\text{left}+\epsilon,s_\text{right}-\epsilon]$
to denote the closure of $I_\epsilon$.
Then, the definition of $\bm{t}(\cdot)$ immediately implies
\begin{align}
    \bm{y}_{ t }(y) \in [-\epsilon,\epsilon]\qquad \forall y \in I^-_\epsilon,\ t \geq \bm{t}(\epsilon).
    \label{property: t epsilon function, first exit analysis, R1}
\end{align}
The next two lemmas verify the claims in Remark~\ref{remark: regularity conditions for first exit times}
that the regularity conditions in Theorem~\ref{theorem: first exit time, unclipped} holds in the non-degenerate $\R^1$ setting for all $b > 0$ satisfying \eqref{condition for boundary of I / b is not an integer, appendix, R1 setting}.

\begin{lemma}
\label{lemma: measure check C J * b, continuity, first exit analysis, R1}
\linksinthm{lemma: measure check C J * b, continuity, first exit analysis, R1}
Let Assumptions \ref{assumption: lipschitz continuity of drift and diffusion coefficients, R1},
\ref{assumption: shape of f, first exit analysis, R1}, and
\ref{assumption: nondegeneracy of diffusion coefficients, R1} hold.
Let $\bar\epsilon \in (0,b)$ be defined as in \eqref{constant bar epsilon, first exit time analysis}.
For any $|\gamma| > (\mathcal{J}^*_b - 1)b + \bar\epsilon$
such that $\gamma/b \notin \mathbb{Z}$,
$$
\widecheck{\mathbf C}^{(\mathcal{J}^*_b)|b}(\{\gamma\}) = 0.
$$
\end{lemma}

\begin{lemma}
\label{lemma: exit rate strictly positive, first exit analysis, R1}
\linksinthm{lemma: exit rate strictly positive, first exit analysis, R1}
Under Assumptions \ref{assumption: lipschitz continuity of drift and diffusion coefficients, R1},
\ref{assumption: shape of f, first exit analysis, R1}, and
\ref{assumption: nondegeneracy of diffusion coefficients, R1},
    $
    \widecheck{\mathbf C}^{(\mathcal{J}^*_b)|b}( I^\complement ) \in (0,\infty).
    $
\end{lemma}

To provide the proof, we first make one observation related to the truncation operator
\begin{align*}
    {\varphi_c}(w) 
    \delequal{} (w\wedge c)\vee(-c)\ \ \ \forall w \in \mathbb{R}, c>0,
\end{align*}
under a uniform version of Assumption~\ref{assumption: nondegeneracy of diffusion coefficients, R1}, that is,
\begin{align}
\inf_{x \in \R} \sigma(x) \geq c > 0.
    \label{assumption: uniform nondegeneracy of diffusion coefficients}
\end{align}
For any $b,c > 0$, any $w \in \R$ and any $z \geq c$, note that for $\widetilde w \delequal \varphi_{b/c}(w)$,
we have 
$
\varphi_b( z \cdot w) = \varphi_b(z \cdot \widetilde w).
$
Indeed, the claim is obviously true when $|w| \leq b/c$ (so $\widetilde w = w$);
in case that $|w| > b/c$, we simply get $\varphi_b(z \cdot w ) = \varphi_b(z \cdot \widetilde w)$ with the value equal to $b$ or $-b$.
Combining this fact with $|\varphi_b(x) - \varphi_b(y)| \leq |x - y|\ \forall x,y\in \R$,
we yield that under condition~\eqref{assumption: uniform nondegeneracy of diffusion coefficients} (and for any $b,c > 0$, any $w_1,w_2 \in \R$, and any $z_1,z_2 \geq c$),
\begin{align}
|\varphi_b(z_1 \cdot w_1) - \varphi_b(z_2 \cdot w_2)|
\leq |z_1\widetilde w_1 - z_2\widetilde w_2|
    \qquad
    \text{ where }
    \widetilde w_1 = \varphi_{b/c}(w_1),\
    \widetilde w_2 = \varphi_{b/c}(w_2).
    \label{observation, uniform nondegeneracy and truncation operator}
\end{align}

Recall that $I^- = [s_\text{left},s_\text{right}]$.
Also, recall that $l = |s_\text{left}| \wedge s_\text{right}$ and ${\mathcal J^*_b} = \ceil{l/b}$.
We first develop Lemma~\ref{lemma: choose key parameters, first exit time analysis, R1},
which is essentially the $\R^1$ version of Lemma~\ref{lemma: choose key parameters, first exit time analysis},
with more properties established under the non-degeneracy assumption.
We provide the proof for the sake of completeness.

\begin{lemma}
\label{lemma: choose key parameters, first exit time analysis, R1}
\linksinthm{lemma: choose key parameters, first exit time analysis, R1}
    Let Assumptions~\ref{assumption: lipschitz continuity of drift and diffusion coefficients, R1}  and \ref{assumption: shape of f, first exit analysis, R1}
    hold.
    Let $\bar\epsilon > 0$ be the constant characterized in \eqref{constant bar epsilon, first exit time analysis, R1}. 
    Furthermore, suppose that $\sup_{x \in I^-}|a(x)| \vee |\sigma(x)| \leq C$ for some $C \geq 1$
    and
    $\inf_{x \in I^-}\sigma(x) \geq {c}$ for some $c \in (0,1]$.
    (Below, we adopt the convention that $t_0 = 0$.)
    \begin{enumerate}[(a)]
        \item  If $\mathcal J^*_b \geq 2$,
        then
        it holds for all $T > 0$, $x_0 \in [-b - \Bar{\epsilon},b + \bar\epsilon]$, $\bm{w} = (w_1,\cdots,w_{\mathcal J^*_b-2})\in \R^{\mathcal J^*_b-2}$, and $\bm{t} = (t_1,\cdots,t_{\mathcal J^*_b-2})\in (0,T]^{ \mathcal J^*_b-2 \uparrow}$ that
        \begin{align*}
        \sup_{t \in [0,T]}|\xi(t)| \leq (\mathcal J^*_b - 1)b +\bar\epsilon < l - 2\Bar{\epsilon}\qquad\text{ where }
        \xi = h^{(\mathcal J^*_b-2)|b}_{[0,T]}(x_0,\bm{w},\bm{t}).
        \end{align*}

        \item It holds for all $T > 0$, $x_0 \in [-\Bar{\epsilon},\bar\epsilon]$, $\bm{w} = (w_1,\cdots,w_{\mathcal J^*_b-1})\in \R^{\mathcal J^*_b-1}$, and $\bm{t} = (t_1,\cdots,t_{\mathcal J^*_b-1})\in (0,T]^{ \mathcal J^*_b-1 \uparrow}$ that
        \begin{align*}
        \sup_{t \in [0,T]}|\xi(t)|\leq (\mathcal J^*_b - 1)b +\bar\epsilon  < l - 2\Bar{\epsilon}\qquad\text{ where }
        \xi = h^{(\mathcal J^*_b-1)|b}_{[0,T]}(x_0,\bm{w},\bm{t}).
        \end{align*}

        \item There exist $\Bar{\delta}>0$ and $\Bar{t} > 0$ such that the following claim holds:
        If
        \begin{align*}
            \sup_{t \in [0,T]}|\xi(t)| \geq l - \bar\epsilon\qquad\text{ where }
            \xi = h^{(\mathcal J^*_b-1)|b}_{[0,T]}\Big( x_0 + \varphi_b\big( \sigma(x_0)\cdot w_0\big),\bm{w},\bm{t}\Big)
        \end{align*}
        for some 
        $T > 0$, $x_0 \in [-\bar\epsilon,\bar\epsilon]$, $w_0 \in \R$, $\bm{w} = (w_1,\cdots,w_{\mathcal J^*_b-1}) \in \R^{\mathcal J^*_b-1}$, and 
        $\bm{t} = (t_1,\cdots,t_{\mathcal J^*_b-1})\in (0,T]^{ \mathcal J^*_b-1 \uparrow}$,
        then
        \begin{enumerate}[(i)]
        \item $\sup_{t \in [0,t_{\mathcal J^*_b-1})}|\xi(t)| \leq (\mathcal J^*_b - 1)b +\bar\epsilon < l - 2\Bar{\epsilon}$;
        \item 
        $|\xi(t_{\mathcal J^*_b-1})| \geq l - \bar\epsilon$;
        \item $\inf_{t \in [0,t_{\mathcal J^*_b-1}]}|\xi(t)| \geq \Bar{\epsilon}$; 
            \item $|w_{j}| > \Bar{\delta}$ for all $j = 0,1,\cdots,\mathcal J^*_b - 1$;
            \item $t_{\mathcal J^*_b-1} < \Bar{t}$.
        \end{enumerate}

        \item Let $T > 0$, $x \in \R, \bm{w} = (w_1,\cdots, w_{\mathcal J^*_b}) \in \R^{\mathcal J^*_b}, \bm{t} = (t_1,\cdots,t_{\mathcal J^*_b})\in (0,T]^{\mathcal J^*_b\uparrow}$ and $\epsilon \in (0,\bar\epsilon)$.
        If $|\xi(t_1-)| < \epsilon$ for $\xi = h^{(\mathcal J^*_b)|b}_{[0,T]}(x,\bm{w},\bm{t})$, then
        \begin{align*}
           \sup_{t \in [t_1,t_{\mathcal J^*_b}]} |\xi(t) - \hat\xi(t - t_1)| \leq \Big[ \exp\big(D(T - t_1)\big)\cdot\Big(1 + \frac{bD}{c}\Big)\Big]^{\mathcal J^*_b} \cdot  \epsilon
        \end{align*}
        where
        $\hat{\xi} = h^{(\mathcal J^*_b-1)|b}_{[0, T - t_1] }\big( \varphi_b( \sigma(0)\cdot w_1),(w_2,\cdots,w_{\mathcal J^*_b}),(t_2 - t_1, t_3 - t_1,\cdots,t_{ \mathcal J^*_b } - t_1)\big)$
        and
        $D \geq 1$ is the constant in Assumption~\ref{assumption: lipschitz continuity of drift and diffusion coefficients, R1}.

        \item Given $\Delta > 0$, there exists $\epsilon_0 = \epsilon_0(\Delta) \in (0,\bar\epsilon)$ such that
        for any $T > 0$, $x \in [-\epsilon_0$, $\epsilon_0]$, $\bm{w} = (w_1,\cdots,w_{\mathcal J^*_b}) \in \R^{\mathcal J^*_b}$, and $\bm{t} = (t_1,\cdots,t_{\mathcal J^*_b}) \in (0,T]^{{\mathcal J^*_b}\uparrow}$,
        \begin{align*}
        \sup_{t \in [t_1,T - t_1]}
        |\xi(t)| \vee |\hat{\xi}(t - t_1)| \geq l - \bar\epsilon \qquad \Longrightarrow\qquad
        \sup_{t \in [t_1,t_{\mathcal J^*_b}]}
        |\hat{\xi}(t - t_1) - \xi(t) | < \Delta
        \end{align*}
        where $\xi = h^{({\mathcal J^*_b})|b}_{[0,T]}(x,\bm{w},\bm{t})$ and 
        $\hat{\xi} = h^{({\mathcal J^*_b}-1)|b}_{[0, T - t_1] }\big( \varphi_b( \sigma(0)\cdot w_1),(w_2,\cdots,w_{\mathcal J^*_b}),(t_2 - t_1, t_3 - t_1,\cdots,t_{\mathcal J^*_b} - t_1)\big).$

    \end{enumerate}
\end{lemma}

\begin{proof}
\linksinpf{lemma: choose key parameters, first exit time analysis, R1}
Before the proof of the claims,
we highlight two facts.
First, Assumption~\ref{assumption: lipschitz continuity of drift and diffusion coefficients, R1} and $I^-$ being compact immediately imply the existence of $C \in (0,\infty)$ such that $\sup_{x \in I^-}|a(x)| \vee |\sigma(x)| \leq C$.
Without loss of generality, in the statement of Lemma~\ref{lemma: choose key parameters, first exit time analysis, R1} we pick some $C \geq 1$.
Next,
we stress that the validity of all claims do not depend on the values of $\sigma(\cdot)$ and $a(\cdot)$ outside of $I^-$. 
Take part $(a)$ as an example.
Suppose that we can prove part $(a)$ under the stronger assumption that $\sup_{x \in \R}|a(x)| \wedge \sigma(x) \leq C$ for some $C \in [1,\infty)$
and $\inf_{x \in \R}\sigma(x) \geq c$ for some $c \in (0,1]$.
Then due to $\sup_{t \in [0,T]}|\xi(t)| < l = |s_\text{left}| \wedge s_\text{right}$ for $\xi = h^{(\mathcal J^*_b-2)|b}_{[0,T]}(x_0,\bm{w},\bm{t})$,
we have $\xi(t) \in I^-$ for all $t \in [0,T]$.
This implies that part $(a)$ is still valid even if we only have 
$\sup_{x \in I^-}|a(x)| \wedge \sigma(x) \leq C$
and $\inf_{x \in I^-}\sigma(x) \geq c$.
The same applies to all the other claims.
Therefore, in the proof below we assume w.l.o.g.\ that the strong assumptions
$\sup_{x \in \R}|a(x)| \wedge \sigma(x) \leq C$ for some $C \in [1,\infty)$
and $\inf_{x \in \R}\sigma(x) \geq c$ for some $c \in (0,1]$ hold.
Specifically, in this proof we assume w.l.o.g.\ that $a(x) = a(s_\text{left})$ for all $x < s_\text{left}$,
and
$a(x) = a(s_\text{right})$ for all $x > s_\text{right}$.
Then in light of Assumption~\ref{assumption: shape of f, first exit analysis, R1},
we now have $a(x)x \leq 0\ \forall x \in \R$.

\medskip
$(a)$
The proof hinges on the following observation. 
For any $j \geq 0, T > 0, x_0 \in \R, \bm{w} = (w_1,\cdots,w_j) \in \R^j$ and $\bm t = (t_1,\cdots,t_j) \in (0,T]^{j\uparrow}$,
let $\xi = h^{(j)|b}_{[0,T]}(x_0,\bm w,\bm t)$.
The condition $a(x)x \leq 0$ implies that 
\begin{align}
    \frac{d |\xi(t)|}{dt} = -\big|a\big(\xi(t)\big)\big|\ \ \ \forall t \in [0,T]\setminus \{t_1,\cdots,t_j\}
    \label{proof, observation on xi, lemma: choose key parameters, first exit time analysis}
\end{align}
Specifically, suppose that ${\mathcal J^*_b} \geq 2$.
For all $T > 0, x_0 \in [-b - \Bar{\epsilon},b + \bar\epsilon], \bm{w} = (w_1,\cdots,w_{{\mathcal J^*_b}-2})\in \R^{{\mathcal J^*_b}-2}$ and $\bm{t} = (t_1,\cdots,t_{{\mathcal J^*_b}-2})\in (0,T]^{ {\mathcal J^*_b}-2 \uparrow}$,
it holds for $\xi = h^{({\mathcal J^*_b}-2)|b}_{[0,T]}(x_0,\bm{w},\bm{t})$ that
$
d|\xi(t)|/dt \leq 0
$
for any $t \in [0,T]\setminus \{t_1,\cdots,t_{\mathcal{J}^*_b-2}\}$,
thus leading to
\begin{align*}
    \sup_{t \in [0,T]}|\xi(t)| 
    & \leq |\xi(0)| + \sum_{t \leq T}|\Delta \xi(t)| 
    \\
    &
    \leq |\xi(0)| + ({\mathcal J^*_b}-2)b
    \qquad \text{due to truncation operators $\varphi_b$ in $h^{({\mathcal J^*_b}-2)|b}_{[0,T]}$}
    \\
    & 
    \leq b + \bar\epsilon + ({\mathcal J^*_b}-2)b 
    \\
    & = ({\mathcal J^*_b}-1)b + \bar\epsilon < l - 2\bar\epsilon
    \qquad \text{due to \eqref{constant bar epsilon, first exit time analysis, R1}}.
\end{align*}
This concludes the proof of part $(a)$.

\medskip
$(b)$ The proof is almost identical to that of part $(a)$.
In particular, it follows from \eqref{proof, observation on xi, lemma: choose key parameters, first exit time analysis} that 
$
d|\xi(t)|/dt \leq 0
$
for any $t \in [0,T]\setminus \{t_1,\cdots,t_{ \mathcal J^*_b -1}\}$.
Therefore,
we have again that
$
 \sup_{t \in [0,T]}|\xi(t)| 
 \leq 
 |\xi(0)| + ({\mathcal J^*_b}-1)b \leq \bar \epsilon + (\mathcal J^*_b - 1)b < l - 2\bar\epsilon.
$

\medskip
$(c)$
We start from the claim that 
$\sup_{t \in [0,t_{ {\mathcal J^*_b} -1})}|\xi(t)| < l - 2\Bar{\epsilon}$.
The case with ${\mathcal J^*_b} = 1$ is trivial since $[0,t_{{\mathcal J^*_b}-1}) = [0,0) = \emptyset$.
Now consider the case where ${\mathcal J^*_b} \geq 2$.
For $\hat{x}_0 \delequal x_0 + \varphi_b\big( \sigma(x_0)\cdot w_0\big)$, we have $|\hat{x}_0| \leq \bar\epsilon + b$.
By setting $\hat{\bm w} = (w_1,\cdots,w_{{\mathcal J^*_b}-2}),\hat{\bm t} = (t_1,\cdots,t_{{\mathcal J^*_b}-2})$ and $\hat{\xi} = h^{({\mathcal J^*_b}-2)|b}_{[0,T]}(\hat{x}_0,\hat{\bm w},\hat{\bm t})$,
we get $\xi(t) = \hat{\xi}(t)$ for all $t \in [0,t_{{\mathcal J^*_b}-1})$.
It then follows directly from results in part $(a)$ that 
$
\sup_{t \in [0,t_{{\mathcal J^*_b}-1})}|\xi(t)|
=
\sup_{t \in [0,t_{{\mathcal J^*_b}-1})}|\hat\xi(t)| \leq (\mathcal J^*_b - 1)b +\bar\epsilon < l - 2\bar\epsilon.
$

Next, we prove the claim  $|\xi(t_{{\mathcal J^*_b}-1})| \geq l - \bar\epsilon$.
In particular, note that $\sup_{t \in [0,T]}|\xi(t)| \geq l  - \bar\epsilon$
and we just proved that
$\sup_{t \in [0,t_{{\mathcal J^*_b}-1})}|\xi(t)| < l - 2\Bar{\epsilon}$.
Now consider the following proof by contradiction.
Suppose that 
$|\xi(t_{{\mathcal J^*_b}-1})| < l - \bar\epsilon$.
Then by definition of the mapping $h^{({\mathcal J^*_b}-1)|b}_{[0,T]}$, we know that $\xi(t)$ is continuous over $t \in [t_{{\mathcal J^*_b}-1},T]$.
Given observation \eqref{proof, observation on xi, lemma: choose key parameters, first exit time analysis},
we yield the contradiction that $\sup_{t \in [t_{{\mathcal J^*_b}-1},T]}|\xi(t)| \leq |\xi(t_{{\mathcal J^*_b}-1})| \wedge \big( \sup_{t \in [0,t_{{\mathcal J^*_b}-1})}|\xi(t)| \big) < l - \bar\epsilon$.
This concludes the proof.

Similarly, to show the claim $\inf_{t \in [0,t_{{\mathcal J^*_b}-1}]}|\xi(t)| \geq \Bar{\epsilon}$ we proceed with a proof by contradiction.
Suppose there is some $t \in [0,t_{{\mathcal J^*_b}-1}]$ such that $|\xi(t)| < \bar\epsilon$.
Then observation \eqref{proof, observation on xi, lemma: choose key parameters, first exit time analysis} implies that
\begin{align*}
    |\xi(t_{{\mathcal J^*_b}-1})| & \leq |\xi(t)| + \sum_{ s\in (t,t_{{\mathcal J^*_b}-1}] }|\Delta \xi(s)|
    \\
    & \leq \bar\epsilon + ({\mathcal J^*_b}-1)b
    \qquad \text{due to truncation operators $\varphi_b$ in $h^{({\mathcal J^*_b}-1)|b}_{[0,T]}$}
    \\
    & < l - 2\bar\epsilon
    \qquad \text{due to \eqref{constant bar epsilon, first exit time analysis, R1}.}
\end{align*}
However, we have just shown that $|\xi(t_{{\mathcal J^*_b}-1})| \geq l - \bar\epsilon$ must hold.
With this contradiction established we conclude the proof.

Recall our running assumption that $\sup_{x \in \R}|\sigma(x)| \leq C$ for some $C \geq 1$.
By \eqref{constant bar epsilon, first exit time analysis, R1}, we can fix some $\Bar{\delta} > 0$ small enough such that
\begin{align*}
    ({\mathcal J^*_b}-1)b + 3\bar\epsilon + C\bar\delta < l.
\end{align*}
Now we prove that
$|w_j| > \Bar{\delta}$ for all $j = 0,1,\ldots,{\mathcal J^*_b}-1$.
Again, suppose that the claim does not hold. Then there is some $j^* = 0,1,\ldots,{\mathcal J^*_b}-1$ with $|w_{j^*}| \leq \bar\delta$.
From observation \eqref{proof, observation on xi, lemma: choose key parameters, first exit time analysis}, we get
\begin{align*}
    |\xi(t_{{\mathcal J^*_b}-1})| & \leq |\xi(0)| + \sum_{t \in [0,t_{{\mathcal J^*_b}-1}]}|\Delta \xi(t)|
    \\
    & \leq |x_0| + 
    \varphi_b\Big(\Big| \sigma(x_0) \cdot w_0 \Big|\Big) + \sum_{j = 1}^{{\mathcal J^*_b}-1}\varphi_b\Big(\Big| \sigma\big(\xi(t_j-)\big) \cdot w_j \Big|\Big)
    \\
    & \leq \bar\epsilon + 
    ({\mathcal J^*_b}-1)b + C\bar\delta 
    \qquad
    \text{due to $|x_0| \leq \bar\epsilon,\ |w_{j^*}| \leq \bar\delta$, and $|\sigma(y)| \leq C$ for all $y \in \R$}
    \\
    & < l - 2\bar\epsilon\qquad
    \text{due to our choice of $\bar\delta$}.
\end{align*}
This contradiction with the fact $|\xi(t_{{\mathcal J^*_b}-1})| \geq l - \bar\epsilon$ allows us to conclude the proof.

Lastly, we move onto the claim $t_{{\mathcal J^*_b}-1} < \Bar{t}$.
If ${\mathcal J^*_b} = 1$, then due to $t_0 = 0$ the claim is trivially true for any $\Bar{t} > 0$.
Hereafter, we focus on the case where ${\mathcal J^*_b} \geq 2$ and start by specifying the constant $\Bar{t}$.
From the continuity of $a(\cdot)$ (see Assumption~\ref{assumption: lipschitz continuity of drift and diffusion coefficients, R1}) and the fact that $a(y) \neq 0\ \forall y \in (-l,0)\cup(0,l)$ (see Assumption~\ref{assumption: shape of f, first exit analysis, R1}),
we can find some $c_{\bar\epsilon} > 0$ such that $|a(y)| \geq c_{\bar\epsilon}$ for all $y \in [-l + \bar\epsilon,-\bar\epsilon]\cup[\bar\epsilon,l - \bar\epsilon]$.
Now we pick some
\begin{align*}
t_{\bar\epsilon} = l/c_{\bar\epsilon},\qquad
    \Bar{t} = ({\mathcal J^*_b}-1) \cdot t_{\bar\epsilon}.
\end{align*}
We proceed with a proof by contradiction.
Suppose that $t_{{\mathcal J^*_b}-1} \geq \Bar{t} = ({\mathcal J^*_b}-1)\cdot t_{\bar\epsilon}$, then there must exist some $j^* = 1,2,\ldots,{\mathcal J^*_b}-1$ such that $t_{j^*} - t_{j^*-1} \geq t_{\bar\epsilon}$.
First, we have shown that $|\xi(t_{j^*-1})| < l - \bar\epsilon$.
Next, 
we must have $|\xi(t)| < \bar\epsilon$ for some $t \in [t_{j^*-1},t_{j^*})$.
Indeed, suppose that 
$|\xi(t)| \geq \bar\epsilon$ for all $t \in [t_{j^*-1},t_{j^*})$.
Then
from observation \eqref{proof, observation on xi, lemma: choose key parameters, first exit time analysis} and
$|a(y)| \geq c_{\bar\epsilon}$ for all $y \in [-l + \bar\epsilon,-\bar\epsilon]\cup[\bar\epsilon,l - \bar\epsilon]$,
we yield
\begin{align*}
    |\xi(t_{j^*}-)|
    & \leq |\xi(t_{j^*-1})| - c_{\bar\epsilon} \cdot t_{\bar\epsilon}
    \leq l - c_{\bar\epsilon} \cdot \frac{l}{c_{\bar\epsilon}}= 0.
\end{align*}
The continuity of $\xi(t)$ on $t \in [t_{j^*-1},t_{j^*})$ then implies that for any $t \in [t_{j^*-1},t_{j^*})$ close enough to $t_{j^*}$, we have $|\xi(t)| < \bar\epsilon$.
However, we have also shown that 
$\inf_{t \in [0,t_{{\mathcal J^*_b}-1}]}|\xi(t)| \geq \Bar{\epsilon}$.
With this contradiction established, we conclude the proof.

\medskip
$(d)$
Let
$R_j \delequal \sup_{t \in [t_1,t_{j}]} |\xi(t) - \hat{\xi}(t - t_1) |$ for $j \in [ \mathcal{J}^*_b ]$.
Specifically,
$R_1 = |\xi(t_1) - \hat\xi(0 )|$.
We start by analyzing $R_1$.
First, note that
$
\xi(t_1) = \xi(t_1-) + \varphi_b\big( \sigma(\xi(t_1-))\cdot w_1 \big)
$
and 
$
\hat \xi(0) = \varphi_b(\sigma(0) \cdot w_1).
$
Using \eqref{proof, observation on xi, lemma: choose key parameters, first exit time analysis},
By the assumption $|\xi(t_1-)| < \epsilon$,
\begin{align*}
    R_1 & \leq \epsilon + \big|\varphi_b\big( \sigma\big(\xi(t_1-)\big)\cdot w_1 \big) - \varphi_b\big(\sigma(0) \cdot w_1\big)\big|
    \\ 
    & \leq 
    \epsilon + 
    \big|\sigma\big(\xi(t_1-)\big) - \sigma(0) \big|\cdot \big|\varphi_{b/c}(w_1)\big|
    \qquad 
    \text{due to \eqref{observation, uniform nondegeneracy and truncation operator} and $\inf_{x \in \R}\sigma(x) \geq c > 0$}
    \\ 
    & \leq 
    \epsilon + D \epsilon \cdot \frac{b}{c}
    =
    \bigg(1 + \frac{bD}{c}\bigg) \cdot \epsilon
    \qquad 
    \text{ by Assumption~\ref{assumption: lipschitz continuity of drift and diffusion coefficients, R1}.}
\end{align*}
\elaborate{
Besides, due to $\inf_{x \in \R}\sigma(x) \geq c$ with $c \in (0,1]$,
we have
$
\varphi_b\big( \sigma(x) \cdot w_1 \big) = \varphi_b\big( \sigma(x) \cdot \widetilde w_1 \big) =\ \forall x \in \R
$
with $\widetilde{w}_1 = \varphi_{b/c}(w_1)$.
Therefore,
\begin{align*}
    R_0 & \leq \epsilon +
    \Big|\varphi_b\Big( \sigma\big(\xi(t_1-)\big)\cdot \widetilde w_1 \Big) - \varphi_b\big(\sigma(0) \cdot \widetilde w_1\big)\Big|
    \\ 
    & \leq 
    \epsilon + \Big| \sigma\big(\xi(t_1-)\big)\cdot \widetilde w_1 - \sigma(0) \cdot \widetilde w_1\Big|
    \leq 
    \epsilon + \big|\sigma\big(\xi(t_1-)\big) - \sigma(0) \big|\cdot |\widetilde w_1|
    \\ 
    & \leq 
    \epsilon + D \epsilon \cdot \frac{b}{c}
    =
    (1 + \frac{bD}{c}) \cdot \epsilon
    \qquad 
    \text{ because of Assumption \ref{assumption gradient noise heavy-tailed}, $|\xi(t_1-)|\leq \epsilon$, and $|\widetilde w_1| \leq b/c$}.
\end{align*}
}
We proceed with an induction argument.
Suppose that for some $j = 1,\ldots,\mathcal J^*_b -1$, we have $R_j \leq \rho^{j}\cdot\epsilon$ with 
$
\rho \delequal {\exp(D(T-t_1))\cdot(1 + \frac{bD}{c})}.
$
By applying Gronwall's inequality for $u \in [t_j,t_{j+1})$,
\begin{align}
   \sup_{ u \in [t_j,t_{j+1}) }| \xi(u) - \hat{\xi}(u - t_1) | \leq R_j \cdot \exp\big( D(t_{j+1} - t_j) \big) \leq \exp\big(D(T - t_1)\big)R_j
   \leq \rho^{j+1}\epsilon.
   \label{proof: part d, ineq, lemma: choose key parameters, first exit time analysis}
\end{align}
Meanwhile, at $t = t_{j+1}$ we have (set $ \hat t_{j+1} \delequal t_{j+1} - t_1$)
\begin{align*}
    & |\hat{\xi}(\hat t_{j+1}) - \xi(t_{j+1})|
    \\
    & = 
    \bigg|
    \hat{\xi}(\hat t_{j+1}-) + \varphi_b\Big( \sigma\big(\hat{\xi}(\hat t_{j+1}-)\big) \cdot w_{j+1} \Big)
    -
    \Big[
    \xi(t_{j+1}-) + \varphi_b\Big(\sigma\big(\xi(t_{j+1}-) \big) \cdot {w}_{j+1}\Big)
    \Big]
    \bigg|
    \\
    & \leq \Big|\hat{\xi}(\hat t_{j+1}-) - \xi(t_{j+1}-)\Big|
    +
    \Big|
    \varphi_b\Big( \sigma\big(\hat{\xi}(\hat t_{j+1}-)\big) \cdot w_{j+1} \Big)
    -
    \varphi_b\Big(\sigma\big(\xi(t_{j+1}-) \big) \cdot {w}_{j+1}\Big)
    \Big|
    \\ 
    & \leq 
    \exp\big(D(T - t_1)\big)R_j
    +
    \Big|
    \varphi_b\Big( \sigma\big(\hat{\xi}(\hat t_{j+1}-)\big) \cdot w_{j+1} \Big)
    -
    \varphi_b\Big(\sigma\big(\xi(t_{j+1}-) \big) \cdot {w}_{j+1}\Big)
    \Big|
    \qquad \text{by \eqref{proof: part d, ineq, lemma: choose key parameters, first exit time analysis}}
    \\ 
    & \leq 
    \exp\big(D(T - t_1)\big)R_j
    +
    \Big|
   \sigma\big(\hat{\xi}(\hat t_{j+1}-)\big)
    -
    \sigma\big(\xi(t_{j+1}-) \big)
    \Big|
    \cdot \big|\varphi_{b/c}(w_{j+1})\big|
    \qquad 
    \text{by \eqref{observation, uniform nondegeneracy and truncation operator}}
    \\ 
    & \leq 
     \exp\big(D(T - t_1)\big)R_j
     +
    D\big|\hat{\xi}(t_{j+1}-) - \xi(t_{j+1}-)\big| \cdot b/c
    \qquad 
    \text{by Assumption \ref{assumption: lipschitz continuity of drift and diffusion coefficients, R1}}
    \\
    & \leq 
    \exp\big(D(T - t_1)\big)R_j
    +
    \frac{bD}{c}\cdot \exp\big(D(T - t_1)\big)R_j
    = \Big(1 + \frac{bD}{c}\Big)\exp\big(D(T - t_1)\big)R_j \leq \rho^{j+1}\cdot \epsilon.
\end{align*}
\elaborate{
First, since $\xi$ and $\hat\xi$ are left-continuous,
$
\big|\hat{\xi}(\hat t_{j+1}-) - \xi(t_{j+1}-)\big| = \lim_{u \uparrow t_{j+1}}| \hat \xi(u - t_1) - {\xi}(u) | \leq \exp(DT)R_j.
$
On the other hand, by setting $\widetilde{w}_{j+1}\delequal \varphi_{b/c}(w_{j+1})$, we also get
\begin{align*}
    & \Big|
    \varphi_b\Big( \sigma\big(\hat{\xi}(\hat t_{j+1}-)\big) \cdot w_{j+1} \Big)
    -
    \varphi_b\Big(\sigma\big(\xi(t_{j+1}-) \big) \cdot {w}_{j+1}\Big)
    \Big|
    \\
    & = 
    \Big|
     \varphi_b\Big( \sigma\big(\hat{\xi}(\hat t_{j+1}-)\big) \cdot \widetilde w_{j+1} \Big)
    -
    \varphi_b\Big(\sigma\big(\xi(t_{j+1}-) \big) \cdot \widetilde{w}_{j+1}\Big)
    \Big|
    \\
    & \leq \Big|
   \sigma\big(\hat{\xi}(\hat t_{j+1}-)\big) \cdot \widetilde{w}_{j+1}
    -
    \sigma\big(\xi(t_{j+1}-) \big) \cdot \widetilde{w}_{j+1}
    \Big|
    \\ 
    &
    \leq 
    D\big|\hat{\xi}(t_{j+1}-) - \xi(t_{j+1}-)\big| \cdot |\widetilde{w}_{j+1}|
    \qquad 
    \text{due to the Lipschitz continuity of $\sigma$; see Assumption \ref{assumption: lipschitz continuity of drift and diffusion coefficients}}
    \\
    & \leq 
    \frac{bD}{c}\cdot \exp(DT)R_j.
\end{align*}
In summary,
\begin{align*}
    R_{j+1} \leq \Big(1 + \frac{bD}{c}\Big)\cdot \exp(DT) \cdot R_j 
    \leq \Big(1 + \frac{bD}{c}\Big)\cdot \exp(DT) \cdot \rho^j \epsilon = \rho^{j+1}\epsilon.
\end{align*}    
}
By arguing inductively we conclude the proof.

\medskip
$(e)$
Note that the statement is not affected by the values of $\xi(t)$ beyond $ t\in [0,t_{\mathcal J^*_b}]$ or the values of $\hat \xi(t)$ outside of the domain $ t \in [0,t_{\mathcal J^*_b} - t_1]$.
Therefore, without loss of generality we can set $T = t_{\mathcal J_b^*} + 1$.
Let $\bar t$ be the constant specified in part $(c)$.
Suppose that $\sup_{t \in [t_1,T - t_1]}|\xi(t)| \vee |\hat{\xi}(t - t_1)| \geq l - \bar\epsilon$ implies
\begin{align}
        \sup_{t \in [t_1,t_{\mathcal J^*_b}]}
            |\xi(t) - \hat\xi(t - t_1)| < \underbrace{\Big[ \exp\big(D(\bar t+1)\big)\cdot\Big(1 + \frac{bD}{c}\Big)\Big]^{\mathcal J^*_b}}_{\delequal \rho^*} \cdot  \epsilon_0
            \qquad \forall \epsilon_0 \in (0,\bar\epsilon].
    \label{goal, part e, lemma: choose key parameters, first exit time analysis, R1}
\end{align}
Then claims in part $(e)$ hold for any $\epsilon_0 \in (0,\bar\epsilon)$ small enough such that $\rho^*\epsilon_0 < \Delta$.

Now, it only remains to prove claim \eqref{goal, part e, lemma: choose key parameters, first exit time analysis, R1}.
From observation \eqref{proof, observation on xi, lemma: choose key parameters, first exit time analysis},
we get $|\xi(t_1-)| \leq |\xi(0)| \leq \epsilon_0$.
This allows us to apply results in part $(d)$ and get (recall our choice of $T = t_{\mathcal J_b^*} + 1$)
\begin{align*}
        \sup_{t \in [t_1,t_{\mathcal J^*_b}]}
            |\xi(t) - \hat\xi(t - t_1)| \leq \Big[ \exp\big(D( t_{\mathcal J_b^*} - t_1 + 1)\big)\cdot\Big(1 + \frac{bD}{c}\Big)\Big]^{\mathcal J^*_b+1} \cdot  \epsilon_0.
\end{align*}
Lastly, 
if $\sup_{t \in [t_1,T]}|\hat\xi(t-t_1)| \geq l -\bar \epsilon$, then $t_{\mathcal{J}^*_b} - t_1 < \bar t$ follows from part $(c)$.
Likewise, if $\sup_{t \in [0,T]}|\xi(t)| \geq l -\bar \epsilon$, then we get $t_{\mathcal{J}^*_b} < \bar t$.
In both cases, we get $t_{\mathcal J_b^*}-t_1 + 1 \leq \bar t + 1$.
This concludes the proof.
\end{proof}

Now, we are ready to prove Lemmas~\ref{lemma: measure check C J * b, continuity, first exit analysis, R1} and \ref{lemma: exit rate strictly positive, first exit analysis, R1}.

\begin{proof}[Proof of Lemma~\ref{lemma: measure check C J * b, continuity, first exit analysis, R1}]
\linksinpf{lemma: measure check C J * b, continuity, first exit analysis, R1}
First, in case that $\mathcal{J}^*_b = 1$, we have
$
\widecheck{\mathbf C}^{(1)|b}(\{\gamma\})
=
\nu_\alpha(\{ w \in \R:\ \varphi_b(\sigma(0)\cdot w) = \gamma \}).
$
Since $\gamma \neq b$, we know that 
$
\big\{ w:\ \varphi_b\big(\sigma(0)\cdot w\big) = \gamma \big\} \subseteq \{\frac{\gamma}{\sigma(0)}\}.
$
The absolute continuity of $\nu_\alpha$ (w.r.t the Lebesgue measure) then implies that $\widecheck{\mathbf C}^{(1)|b}(\{\gamma\}) = 0$.
Hereafter, we focus on the case where $\mathcal{J}^*_b \geq 2$.
Observe that (recall that $\mathcal L$ is the Lebesgue measure on $(0,\infty)$)
\begin{align*}
    & \widecheck{\mathbf C}^{(\mathcal{J}^*_b)|b}(\{\gamma\})
    \\
    & = 
    \int \mathbbm{I}\bigg(\mathbbm{I}\Big\{ \widecheck g^{(\mathcal{J}^*_b -1)|b}\Big( \varphi_b\big(\sigma(0)\cdot w_1\big),
    (w_2,\cdots,w_{\mathcal{J}^*_b-1},w^*),
    (t_1,\cdots,t_{\mathcal{J}^*_b-2},t_{\mathcal{J}^*_b-2} + t^*) \Big)=\gamma \Big\}
    \\
    &\qquad \qquad \qquad \qquad \qquad \qquad 
    \nu_\alpha(dw^*)\times\mathcal{L}(dt^*)\bigg)
    \nu^{ \mathcal{J}^*_b-1 }_\alpha(d w_1,\cdots,dw_{\mathcal{J}^*_b-1}) \times \mathcal{L}^{\mathcal{J}^*_b-2\uparrow}_\infty(dt_1,\cdots,dt_{\mathcal{J}^*_b-2})
    \\
    & = 
    \int\bigg(\int_{ (w^*,t^*) \in E(\bm w, \bm t )  }\nu_\alpha(dw^*)\times\mathcal{L}(dt^*)\bigg)
    \nu^{ \mathcal{J}^*_b-1 }_\alpha(d\bm w) \times \mathcal{L}^{\mathcal{J}^*_b-2\uparrow}_\infty(d\bm t)
\end{align*}
where
\begin{align*}
 E( \bm w,\bm t )
    & = 
    \Big\{
    (w,t) \in \R \times (0,\infty):\ \varphi_b\Big(
        \bm{y}_{ t }\big( \widetilde{\bm x}(\bm w,\bm t)\big) + \sigma\big(\bm{y}_{ t }( \widetilde{\bm x}(\bm w,\bm t))\big)\cdot w
        \Big) = \gamma
    \Big\},
    \\
    \widetilde{\bm x}(\bm w,\bm t) & = \widecheck{g}^{ ( \mathcal{J}^*_b - 2 ) |b}
    \Big( \varphi_b\big(\sigma(0)\cdot w_1\big), 
    (w_2,\cdots,w_{\mathcal{J}^*_b-1}),
    (t_1,\cdots,t_{\mathcal{J}^*_b-2})\Big).
\end{align*}
Here, $\bm{y}_t(x)$ is the ODE defined in \eqref{def ODE path y t, R1}.
Furthermore, we claim that for any $\bm w,\bm t$, 
there exist some continuous function $w^*:(0,\infty) \to \R$ and some $t^* \in (0,\infty)$ such that
\begin{align}
    E( \bm w,\bm t ) \subseteq \big\{ (w,t) \in \R \times (0,\infty): w = w^*(t)\text{ or }t = t^* \big\}.
    \label{goal, lemma: measure check C J * b, continuity, first exit analysis}
\end{align}
Then set $E( \bm w,\bm t )$ charges zero mass under Lebesgues measure on $\R \times (0,\infty)$.
From the absolute continuity of $\nu_\alpha \times \mathcal{L}$ (w.r.t.\ Lebesgues measure on $\R \times (0,\infty)$)
we get 
$
\widecheck{\mathbf C}^{(\mathcal{J}^*_b)|b}(\{\gamma\}) = 0.
$

Now, it only remains to prove claim \eqref{goal, lemma: measure check C J * b, continuity, first exit analysis}.
Henceforth in this proof we fix some $\bm w \in \R^{\mathcal J^*_b - 1}$ and $\bm t \in (0,\infty)^{\mathcal J^*_b - 2\uparrow}$.
First, due to $|\gamma| > (\mathcal{J}^*_b - 1)b + \bar\epsilon$, it follows from part $(a)$ of Lemma \ref{lemma: choose key parameters, first exit time analysis, R1} that
$
|\widetilde{\bm x}(\bm w,\bm t)| \leq (\mathcal{J}^*_b - 1)b + \bar\epsilon < \gamma.
$
Next, we show that there exists at most one $t^* \in (0,\infty)$ such that 
\begin{align}
    |\bm{y}_{ t^* }\big( \widetilde{\bm x}(\bm w,\bm t)\big) - \gamma| = b.
    \label{proof: goal 2, lemma: measure check C J * b, continuity, first exit analysis}
\end{align}
To see why, we consider two different cases.
If $\widetilde{\bm x}(\bm w,\bm t) = 0$, then $a(0) = 0$ implies that 
$\bm{y}_{ t }\big( \widetilde{\bm x}(\bm w,\bm t)\big) = 0$ for all $t \geq 0$.
By assumption, we have $\gamma \neq b$, 
and hence
$
|\bm{y}_{ t }\big( \widetilde{\bm x}(\bm w,\bm t)\big) - \gamma| = \gamma \neq b
$
for all $t \geq 0$.
If  $\widetilde{\bm x}(\bm w,\bm t) \neq 0$,
then Assumption~\ref{assumption: shape of f, first exit analysis} implies that 
$
|\bm{y}_{ t }(\tilde x)|
$
is strictly monotone decreasing w.r.t.\ $t$ for all $\tilde x \in (-\gamma,\gamma)$.
Due to $|\widetilde{\bm x}(\bm w,\bm t)| < \gamma$,
the only possible scenario for \eqref{proof: goal 2, lemma: measure check C J * b, continuity, first exit analysis} is that
$|\bm{y}_{ t^* }\big( \widetilde{\bm x}(\bm w,\bm t)\big)| = \gamma - b$,
which can only hold for at most one $t^*$ due to the strict monotonicity of $|\bm{y}_{ t }\big( \widetilde{\bm x}(\bm w,\bm t)\big)|$ in $t$.

Now for any $t > 0$ with $t \neq t^*$,
we know that 
$
|\bm{y}_{ t }\big( \widetilde{\bm x}(\bm w,\bm t)\big) - \gamma| \neq b.
$
As a result, the only feasible choice for $w \in \R$ in 
$
\varphi_b\big(\bm{y}_{ t }( \widetilde{\bm x}(\bm w,\bm t)) + \sigma\big(\bm{y}_{ t }( \widetilde{\bm x}(\bm w,\bm t)))\cdot w\big) = \gamma
$
is
$
w = \frac{\gamma - \bm{y}_{ t }\big( \widetilde{\bm x}(\bm w,\bm t)\big)}{ \sigma\big(\bm{y}_{ t }\big( \widetilde{\bm x}(\bm w,\bm t)\big)\big) }.
$
(Note that 
$\sigma(x) > 0\ \forall x \in \R$; see Assumption \ref{assumption: nondegeneracy of diffusion coefficients, R1}.)
By setting $w^*(t) \delequal \frac{\gamma - \bm{y}_{ t }\big( \widetilde{\bm x}(\bm w,\bm t)\big)}{ \sigma\big(\bm{y}_{ t }\big( \widetilde{\bm x}(\bm w,\bm t)\big)\big) }$ we conclude the proof.
\end{proof}

\begin{proof}[Proof of Lemma~\ref{lemma: exit rate strictly positive, first exit analysis, R1}]
\linksinpf{lemma: exit rate strictly positive, first exit analysis, R1}
Let $\bar t$ and $\bar \delta$ be the constants characterized in Lemma \ref{lemma: choose key parameters, first exit time analysis, R1}.
Let $\bar \epsilon$ be the constant in \eqref{constant bar epsilon, first exit time analysis}.
We start with the proof of $\widecheck{\mathbf C}^{(\mathcal{J}^*_b)|b}( I^\complement ) < \infty$.
Recall that $l = |s_\text{left}| \wedge s_\text{right}$, and observe
\begin{align*}
    & 
    \widecheck{\mathbf C}^{(\mathcal{J}^*_b)|b}\big( (-\infty,s_\text{left}]\cup[s_\text{right},\infty) \big)
    \\
    & \leq 
    \widecheck{\mathbf C}^{(\mathcal{J}^*_b)|b}\big( \R \setminus  [-(l - \bar\epsilon),l-\bar\epsilon] \big)
    \\
    & = 
    \int \mathbbm{I}\bigg\{
    \bigg|\widecheck{g}^{ (\mathcal{J}^*_b - 1)|b }\big( \varphi_b(\sigma(0)\cdot w_{\mathcal{J}^*_b}), (w_1,\cdots,w_{\mathcal{J}^*_b-1}), (t_1,\cdots,t_{\mathcal{J}^*_b-1})\big)\bigg| > l - \bar\epsilon 
    \bigg\}
    \\ 
    &\qquad\qquad\qquad\qquad\qquad\qquad\qquad\qquad \times
    \nu^{ \mathcal{J}^*_b }_\alpha(dw_1,\cdots, d w_{\mathcal J^*_b}) \times \mathcal{L}^{\mathcal{J}^*_b - 1 \uparrow}_\infty(dt_1,\cdots, dt_{\mathcal J^*_b - 1})
    \\
    & = 
    \int \mathbbm{I}\bigg\{
    \bigg|h^{ (\mathcal{J}^*_b - 1)|b }_{[0,1 + t_{\mathcal{J}^*_b-1}]}\big( \varphi_b(\sigma(0)\cdot w_{\mathcal{J}^*_b}), (w_1,\cdots,w_{\mathcal{J}^*_b-1}), (t_1,\cdots,t_{\mathcal{J}^*_b-1})\big)(t_{\mathcal{J}^*_b-1})\bigg| > l - \bar\epsilon 
    \bigg\}
    \\ 
    &\qquad\qquad\qquad\qquad\qquad\qquad\qquad\qquad \times
    \nu^{ \mathcal{J}^*_b }_\alpha(dw_1,\cdots, d w_{\mathcal J^*_b}) \times \mathcal{L}^{\mathcal{J}^*_b - 1 \uparrow}_\infty(dt_1,\cdots, dt_{\mathcal J^*_b - 1})
    \\
    & \leq 
    \int \mathbbm{I}\big\{ |w_j| > \bar\delta\ \forall j \in [\mathcal{J}^*_b];\ t_{\mathcal{J}^*_b-1} < \bar t\ \big\}\nu^{ \mathcal{J}^*_b }_\alpha(d\bm w) \times \mathcal{L}^{\mathcal{J}^*_b - 1 \uparrow}_\infty(d\bm t)
    \ \ \ \ 
    \text{using part $(c)$ of Lemma \ref{lemma: choose key parameters, first exit time analysis, R1}}
    \\
    & \leq \bar t^{ \mathcal{J}^*_b - 1 } \big/ \bar\delta^{\alpha\mathcal{J}^*_b} < \infty.
\end{align*}

Next, we move onto the proof of $\widecheck{\mathbf C}^{(\mathcal{J}^*_b)|b}( I^\complement ) > 0$.
Without loss of generality, assume that $s_\text{right} \leq |s_\text{left}|$.
Then due to $l/b \notin \mathbb{Z}$,
we have $(\mathcal{J}^*_b - 1)b <  s_\text{right} < \mathcal{J}^*_b b$.
First, consider the case where $\mathcal{J}^*_b = 1$.
For any $w \geq \frac{b}{\sigma(0)}$, we have $\varphi_b\big(\sigma(0)\cdot w\big) = b > s_\text{right}$.
Therefore,
\begin{align*}
    \widecheck{\mathbf C}^{(1)|b}\big( [s_\text{right},\infty) \big)
    & = \int \mathbbm{I}\big\{ \varphi_b\big( \sigma(0) \cdot w\big) \geq s_\text{right} \big\}\nu_\alpha(dw)
    \geq \int_{ w \in [\frac{b}{\sigma(0)},\infty) }\nu_\alpha(dw) 
    = \bigg(\frac{\sigma(0)}{b}\bigg)^\alpha
    > 0.
\end{align*}
Next, we consider the case where $\mathcal{J}^*_b \geq 2$.
In particular,
we claim the existence of some $(w_1,\cdots,w_{ \mathcal{J}^*_b }) \in \R^{\mathcal{J}^*_b}$ and $\bm{t} = (t_1,\cdots,t_{ \mathcal{J}^*_b - 1 }) \in (0,\infty)^{ \mathcal{J}^*_b - 1 \uparrow}$ such that
\begin{equation}\label{proof, goal, lemma: exit rate strictly positive, first exit analysis}
    \begin{split}
         & \widecheck{g}^{ (\mathcal{J}^*_b)|b }\Big( \varphi_b(\sigma(0) \cdot w_{ \mathcal{J}^*_b } ), (w_1,\cdots,w_{ \mathcal{J}^*_b - 1 }), \bm{t} \Big)
   \\
   & = 
   h^{ (\mathcal{J}^*_b - 1)|b }_{ [0,t_{ \mathcal{J}^*_b - 1 } +1 ] }\Big( \varphi_b(\sigma(0) \cdot w_{ \mathcal{J}^*_b } ), 
   (w_1,\cdots,w_{ \mathcal{J}^*_b - 1 }), \bm{t} \Big)(t_{ \mathcal{J}^*_b - 1 })
   > s_\text{right}.
    \end{split}
\end{equation}
Then from the continuity of mapping $h^{ (\mathcal{J}^*_b - 1)|b }_{ [0,t_{ \mathcal{J}^*_b - 1 } +1 ] }$ (see Lemma \ref{lemma: continuity of h k b mapping clipped}),
we can fix some $\Delta > 0$ such that the following claim holds:
for all $w^\prime_j$'s with $|w^\prime_j - w_j| < \Delta$ and $t^\prime_j$'s with $|t^\prime_j - t_j| < \Delta$,
$$
\widecheck{g}^{ (\mathcal{J}^*_b - 1)|b }\Big( \varphi_b(\sigma(0) \cdot w^\prime_{ \mathcal{J}^*_b } ), (w^\prime_1,\cdots,w^\prime_{ \mathcal{J}^*_b - 1 }), 
(t^\prime_1,\cdots,t^\prime_{ \mathcal{J}^*_b - 1 }) \Big)  > s_\text{right}.
$$
Now, we can conclude the proof with
\begin{align*}
    & 
    \widecheck{\mathbf C}^{(\mathcal{J}^*_b)|b}\big( [s_\text{right},\infty) \big)
    \\
    & \geq  
    \int \mathbbm{I}\big\{ |w^\prime_j - w_j| < \Delta\ \forall j \in [\mathcal{J}^*_b];\ |t^\prime_j - t_j|< \Delta\ \forall j \in [\mathcal{J}^*_b - 1] \big\}
    \nu^{ \mathcal{J}^*_b }_\alpha(d \bm w^\prime)
    \times \mathcal{L}^{ \mathcal{J}^*_b - 1 }_\infty(d \bm t^\prime)
    > 0.
\end{align*}

It only remains to show \eqref{proof, goal, lemma: exit rate strictly positive, first exit analysis}.
By Assumptions \ref{assumption: lipschitz continuity of drift and diffusion coefficients, R1} and \ref{assumption: nondegeneracy of diffusion coefficients, R1},
we can fix some $C_0 >0$ such that $|a(x)|\leq C_0$ for all $x \in [s_\text{left},s_\text{right}]$,
as well as some $c > 0$ such that $\inf_{x \in [s_\text{left},s_\text{right}]}\sigma(x) \geq c$.
Now, we set $w_1 = \cdots = w_{ \mathcal{J}^*_b } = b/c$,
Also, pick some $\Delta > 0$ and set $t_k = k\Delta$ (with convention $t_0 = 0$).
For $\xi =  h^{ (\mathcal{J}^*_b - 1)|b }_{ [0,t_{ \mathcal{J}^*_b - 1 } +1 ] }\big( 
\varphi_b(\sigma(0) \cdot w_{ \mathcal{J}^*_b } ), 
(w_1,\cdots,w_{ \mathcal{J}^*_b - 1 }), 
(t_1,\cdots,t_{ \mathcal{J}^*_b - 1})
\big)$,
part $(c)$ of Lemma \ref{lemma: choose key parameters, first exit time analysis, R1} implies 
$
\sup_{t \in [0,t_{\mathcal J^*_b-1})}|\xi(t)| < l - \bar\epsilon,
$
so 
we must have
$
\xi(t) \in [s_\text{left},s_\text{right}]
$
for all $t < t_{ \mathcal{J}^*_b - 1 }$.
This implies
$
\big|a\big(\xi(t)\big)\big| \leq C_0
$
for all $t < t_{ \mathcal{J}^*_b - 1 }$.
Now we make a few observations.
First,  $\xi(0) = \varphi_b\big(\sigma(0) \cdot w_{\mathcal{J}^*_b}\big) = b$
due to $\sigma(0) \cdot w_\mathcal{J^*} \geq c \cdot \frac{b}{c} = b$.
Also, note that for any $j = 1,2,\ldots,\mathcal{J}^*_b - 1$,
\begin{align*}
    \xi(t_{j}) & = \xi(t_{j-1}) + \int_{ s \in [t_{j-1},t_{j}) }a\big(\xi(s)\big)ds + \varphi_b\big( \sigma( \xi( t_{j}- ) \cdot w_j \big)
    \\
    & = \xi(t_{j-1}) + \int_{ s \in [t_{j-1},t_{j}) }a\big(\xi(s)\big)ds + b
    \ \ \ \text{ due to $\sigma\big(\xi( t_{j}- )\big) \cdot w_j \geq c \cdot \frac{b}{c} = b$}
    \\
    & \geq \xi(t_{j-1}) - C_0 \cdot (t_j - t_{j-1}) + b\ \ \ \text{because of $a(x)x \leq 0$ (see Assumption \ref{assumption: shape of f, first exit analysis, R1})
    and $\big|a\big(\xi(t)\big)\big| \leq C_0$
    }
    \\
    & = \xi(t_{j-1}) - C_0\Delta + b.
\end{align*}
By arguing inductively, we get
$
\widecheck{g}^{ (\mathcal{J}^*_b - 1)|b }\big( 
    \varphi_b(\sigma(0) \cdot w_{ \mathcal{J}^*_b } ), (w_1,\cdots,w_{ \mathcal{J}^*_b - 1 }), \bm{t} 
    \big)
    =
    \xi(t_{ \mathcal{J}^*_b - 1 }) \geq \mathcal{J}^*_b \cdot b - (\mathcal{J}^*_b - 1)C_0\Delta.
$
By definition of $\mathcal J^*_b$ and our running assumption that $s_\text{right} \leq |s_\text{left}|$,
we have $\mathcal J^*_b \cdot b > s_\text{right}$.
It holds for all $\Delta > 0$ small enough that $\mathcal{J}^*_b \cdot b - (\mathcal{J}^*_b - 1)C_0\Delta > s_\text{right}$.
This concludes the proof of claim~\eqref{proof, goal, lemma: exit rate strictly positive, first exit analysis}.
The same arguments apply to the case where $s_\text{right} > |s_\text{left}|$ and we omit the details.
\end{proof}

\ifshowreminders
\newpage
\footnotesize
\newgeometry{left=1cm,right=1cm,top=0.5cm,bottom=1.5cm}

\section*{\linkdest{location of reminders}Some Reminders}
\begin{itemize}[leftmargin=*]
    \item Assumption \ref{assumption gradient noise heavy-tailed}
    \begin{itemize}
        \item $\E Z_j = 0;\ H(\cdot) = \P(|Z_1| > \cdot) \in \RV_{-\alpha}$
        \item
        $\lim_{x \rightarrow \infty}\frac{ H^{(+)}(x) }{H(x)} = p^{(+)},\ \lim_{x \rightarrow \infty}\frac{ H^{(-)}(x) }{H(x)} = p^{(-)} = 1 - p^{(+)}$
        with $p^{(+)},p^{(-)} \in [0,1]$ 
        and $p^{(+)} + p^{(-)} = 1$.
    \end{itemize}
    \item Assumption \ref{assumption: lipschitz continuity of drift and diffusion coefficients} (Lipschitz Continuity)
    \begin{enumerate}
        \item[] There exists some $D \in [1, \infty)$ such that $|\sigma(x) - \sigma(y)| \vee |a(x)-a(y)| \leq D|x-y|\ \ \ \forall x,y \in \mathbb{R}.$
    \end{enumerate}
    \item Assumption \ref{assumption: boundedness of drift and diffusion coefficients} (Boundedness)
    \begin{enumerate}
        \item[] There exist some $0<c \leq 1 \leq C < \infty$ such that $|a(x)| \leq C,\ c \leq \sigma(x) \leq C\ \ \forall x \in \mathbb{R}.$
    \end{enumerate}
    \item Assumption \ref{assumption: nondegeneracy of diffusion coefficients} (Nondegeneracy)
    \begin{enumerate}
        \item[] $\sigma(x) > 0$ $\forall x\in \R$.
    \end{enumerate}
    

    \item Assumption \ref{assumption: shape of f, first exit analysis}
    \begin{enumerate}
    \item[] 
        It holds for all $x \in (s_\text{left},s_\text{right})\setminus \{0\}$ that $a(x)x < 0$. Besides,
        \begin{itemize}
            \item $a(0)= 0$; $a(\cdot)$ is differentiable around $0$ with $a^\prime(0) \neq 0$;
            \item The following claims hold for $s \in \{s_\text{left},s_\text{right}\}$: 
            If $a(s) \neq 0$, then
            $a(\cdot)$ is differentiable around $s$ and $a^\prime(s) \neq 0$. 
        \end{itemize}
    \end{enumerate}
    
\end{itemize}
\fi

\ifshownotationindex
\newpage
\footnotesize
\section{Notation Index}
\label{subsec: Notation Index}
\begin{itemize}[leftmargin=*]

\item 
    \notationidx{asymptotic-equivalence}{Asymptotic Equivalence}:
    $X_n$ is asymptotically equivalent to $Y^\delta_n$ when bounded away from $\mathbb{C}$
    w.r.t.\ $\epsilon_n$ as $\delta \downarrow 0$ if the following holds:
    for each $\Delta > 0$
    and each $B \in \mathscr{S}_\mathbb{S}$ that is bounded away from $\mathbb{C}$,
    \begin{align*}
    \lim_{\delta \downarrow 0}\lim_{n \rightarrow \infty} \frac{\P\big( \bm{d}(X_n, Y^\delta_n)\mathbbm{I}( X_n\in B\text{ or }Y^\delta_n \in B) > \Delta \big)}{ \epsilon_n } = 0.
\end{align*}

\item 
    \notationidx{order-k-time-on-[0,t]}{$A^{k\uparrow}$}:
    Given $A \subseteq \R$,
    ${A^{k \uparrow}} \delequal 
\{
(t_1,\cdots,t_k) \in A^k:\ t_1 < t_2 < \cdots < t_k
\}
    $
    

\item 
    \notationidx{set-for-integers-below-n}{$[n]$}: $[n] = \{1,2,\cdots,n\}$ for any positive integer $n$. For $n = 0$ we set $[n] = \emptyset$.

\item 
    \notationidx{floor-operator}{$\floor{x}$}:
    $\floor{x} = \max\{n \in \mathbb{Z}:\ n \leq x\}$
    
\item
    \notationidx{ceil-operator}{$\ceil{x}$}:
    ${\ceil{x}} = \min\{n \in \Z:\ n \geq x\}$

\item 
    \notationidx{notation-closure-of-set-E}{$E^-$}: closure of set $E$ 

\item 
    \notationidx{notation-interior-of-set-E}{$E^\circ$}: interior of set $E$

\item 
    \notationidx{notation-epsilon-enlargement-of-set-E}{$E^\epsilon$}: 
    $E^\epsilon \delequal 
\{ y \in \mathbb{S}:\ \bm{d}(E,y)\leq\epsilon \}$
    ($\epsilon$-enlargement)

\item 
    \notationidx{notation-epsilon-shrinkage-of-set-E}{$E_{\epsilon}$}:
    $E_{\epsilon} \delequal
    ((E^\complement)^\epsilon)^\complement
    $
    ($\epsilon$-shrinkage)

\linkdest{location, notation index A}

\item 
    \notationidx{a}{$\bm a$}: drift coefficient $\bm a: \mathbb{R}^m \to \mathbb{R}^m$

\item
    \notationidx{a-M}{$\bm a_M$}: drift coefficient with the argument $\bm x$ projected onto the ball $\{ \bm x \in \R^m:\ \norm{\bm x} \leq M \}$

\item 
    \notationidx{alpha-noise-tail-index-LDP}{$\alpha$}:
    $\alpha > 1$; the heavy tail index for $(\bm Z_j)_{j \geq 1}$ in Assumption \ref{assumption gradient noise heavy-tailed}

\item
    \notationidx{notation-event-A-i-concentration-of-small-jumps}{$A_i(\eta,b,\epsilon,\delta,\bm x)$}:
    $
    A_i(\eta,b,\epsilon,\delta,\bm x)
    \delequal 
    \Big\{
    \underset{j \in I_i(\eta,\delta) }{\max}\  
        \eta\norm{\sum_{n = \tau^{>\delta}_{i-1}(\eta) + 1}^j \bm \sigma\big( \bm X^{\eta|b}_{n-1}(\bm x) \big)\bm Z_n} \leq \epsilon 
        \Big\}
    $

\linkdest{location, notation index B}

\item   
    $\notationidx{notation-ball-r-x}{\bar B_r(\bm x)}:\ {\bar B_r(\bm x)} \delequal \{ \bm{y}\in\R^m:\ \norm{\bm y - \bm x} \leq r \}$

\item 
    $
    \notationidx{notation-set-B-0}{B_0}:\ {B_0}\delequal 
    \Big\{ \bm{X}^{\eta|b}(\bm x)\in B\text{ or }\hat{\bm X}^{\eta|b;(k)}(\bm x) \in B;\   
    \bm{d}_{J_1}\big(\bm{X}^{\eta|b}(\bm x),\hat{\bm X}^{\eta|b;(k)}(\bm x)\big) > \Delta
    \Big\}
    $

\item 
    $
    \notationidx{notation-B1}
    {B_1}:\ {B_1} \delequal 
    \{ \tau^{>\delta}_{k+1}(\eta) > \floor{1/\eta}\}
    $

\item 
    $
    \notationidx{notation-B2}
    {B_2}:\ {B_2} \delequal 
    \{ \tau^{>\delta}_{k}(\eta) \leq \floor{1/\eta} \}
    $

\item 
    $
    \notationidx{notation-B3}
    {B_3}:\ {B_3} \delequal 
    \Big\{\eta \norm{ \bm W^{>\delta}_{i}(\eta) } > \bar{\delta}\ \text{for all }i \in [k] \Big\}
    $

\item
    $
    \notationidx{notation-B4}
    {B_4}:\ {B_4} \delequal 
    \Big\{\eta \norm{ \bm W^{>\delta}_{i}(\eta) } \leq 1/\epsilon^{ \frac{1}{2k}  } \ \text{for all }i \in [k] \Big\}
    $

\linkdest{location, notation index C}

\item 
    \notationidx{notation-constant-C-boundedness-assumption}{$C$}:
    $C \in [1,\infty)$ is the constant in Assumption \ref{assumption: boundedness of drift and diffusion coefficients} with $\norm{\bm a(\bm x)} \vee \norm{\bm \sigma(\bm x)} \leq  C\ \ \forall \bm x \in \mathbb{R}^m$.

\item 
    \notationidx{notation-C-*-first-exit-time}{$C^I_\infty$}: 
    $
    C^I_\infty \delequal  \widecheck{ \mathbf{C} }\big( I^\complement\big)
    $

\item 
    \notationidx{notation-C-b-*}{$C_b^I$}: 
    $
    {C^I_b} \delequal  \widecheck{ \mathbf{C} }^{ (\mathcal{J}^I_b)|b }(I^\complement)
    $

\item 
    \notationidx{notation-measure-C-k-t-mu-LDP}{$\mathbf{C}^{(k)}_{[0,T]}(\ \cdot\ ;\bm x)$}:
    $
    {\mathbf{C}^{(k)}_{ {[0,T]} }(\ \cdot\ ;\bm x)} 
    \delequal
    {\mathbf{C}^{(k)|\infty}_{ {[0,T]} }(\ \cdot\ ;\bm x)}
    =
   \int \mathbbm{I}\Big\{ h^{(k)}_{ {[0,T]}}\big( \bm x, \bm W, \bm t  \big) \in\ \cdot\  \Big\}
    \big((\nu_\alpha \times \mathbf S)\circ \Phi\big)^k(d \bm W) \times\mathcal{L}^{k\uparrow}_{{T}}(d\bm t).
    $

\item 
    \notationidx{notation-measure-C-k-t-mu-LDP-T-1}{$\mathbf C^{(k)}$}:
    $\mathbf C^{(k)} = \mathbf C^{(k)}_{[0,1]}$

\item 
    \notationidx{notation-measure-C-k-t-truncation-b-LDP}{${\mathbf{C}}^{(k)|b}_{[0,T]}(\ \cdot\ ;\bm x)$}:
    $
    {\mathbf{C}}^{(k)|b}_{[0,T]}(\ \cdot\ ;\bm x) \delequal 
    \int \mathbbm{I}\Big\{ h^{(k)|b}_{{[0,T]}}\big( \bm x,\textbf W,\bm t  \big) \in\ \cdot\  \Big\} 
   \big((\nu_\alpha \times \mathbf S)\circ \Phi\big)^k(d \textbf W) \times\mathcal{L}^{k\uparrow}_{ {T} }(d\bm t)
   $

\item 
    \notationidx{notation-measure-C-k-t-truncation-b-LDP-T=1}{${\mathbf{C}}^{(k)|b}$}:
    $
    {\mathbf{C}}^{(k)|b} = {\mathbf{C}}^{(k)|b}_{[0,1]}
    $

\item  
    \notationidx{notation-check-C}{$\widecheck{\mathbf C}(\ \cdot\ )$}:
    $
    {\widecheck{\mathbf C}(\ \cdot\ )} \delequal \int \mathbbm{I}\Big\{ \bm \sigma(\bm 0) \bm w \in\ \cdot\ \Big\}
    \big((\nu_\alpha \times \mathbf S)\circ \Phi\big)(d \bm w)
    $
    

\item 
    \notationidx{notation-check-C-k-b}{$\widecheck{ \mathbf C }^{(k)|b}(\ \cdot\ )$}:
    $
    {\widecheck{ \mathbf C }^{(k)|b}(\ \cdot\ )}
    \delequal 
    \int \mathbbm{I}\bigg\{ \widecheck{g}^{(k-1)|b}\Big( \varphi_b\big(\bm\sigma(\bm x)\bm w_1\big),(\bm w_2,\cdots,\bm w_k),\bm t \Big) \in \ \cdot \  \bigg\}
     \big((\nu_\alpha \times \mathbf S)\circ \Phi\big)^k(d \textbf W) \times \mathcal{L}^{k-1\uparrow}_\infty(d\bm t)
    $


\item 
    \notationidx{notation-mathcal-C-S-exclude-C}{$\mathcal{C}({ \mathbb{S}\setminus \mathbb{C} })$}:
the set of all real-valued, non-negative, bounded and continuous functions with support bounded away from $\mathbb{C}$

\item   
    $
    \notationidx{notation-error-function-check-c-epsilon}{\widecheck{\bm c}(\epsilon)}:\ {\widecheck{\bm c}(\epsilon)} 
    \delequal 
    \mathcal J^I_b \cdot (\bar t)^{ \mathcal J^I_b - 1 } \cdot (\bar\delta)^{ -\alpha \cdot (\mathcal J^I_b - 1) }
    \cdot 
    \epsilon^{ \frac{\alpha}{2\mathcal J^I_b}  }.
    $

\linkdest{location, notation index D}
\item
    \notationidx{notation-Lipschitz-constant-L-LDP}{$D$}:
    The Lipschitz $D \in[1,\infty)$ in Assumption \ref{assumption: lipschitz continuity of drift and diffusion coefficients}:
    $\norm{\bm \sigma(\bm x) - \bm \sigma(\bm y)} \vee \norm{\bm a(\bm x)-\bm a(\bm y)} \leq D\norm{\bm x - \bm y}\ \ \ \forall \bm x,\ \bm y \in \mathbb{R}^m$

\item 
    \notationidx{notation-D-0T-cadlag-space}{$\mathbb{D}{[0,T]}$}:
     $\mathbb{D}[0,{T}] = \D\big([0,T],\R^m\big)$ is
    the space of all càdlàg functions with domain $[0,{T}]$ and codomain $\R^m$

\item 
    \notationidx{notation-D-0T-cadlag-space-T=1}{$\mathbb{D}$}:
    $
    \mathbb D \delequal \mathbb D[0,1]
    $

\item 
    \notationidx{notation-D-A-k-t-LDP}{$\mathbb{D}^{(k)}_A[0,T](\epsilon)$}:
    $\mathbb{D}^{(k)}_A[0,T](\epsilon) \delequal h^{(k)}_{ {[0,T]} }\Big( A \times \mathbb{R}^{m \times k} \times \big( \bar B_\epsilon(\bm 0)\big)^k \times (0,{T}]^{k\uparrow} \Big)$ with convention that $\mathbb{D}_A^{(-1)}[0,T](\epsilon) = \emptyset$

\item
    \notationidx{notation-D-A-k-t-LDP-T=1}{$\mathbb{D}^{(k)}_A(\epsilon)$}:
    $
    \mathbb{D}^{(k)}_A(\epsilon) \delequal \mathbb{D}^{(k)}_A[0,1](\epsilon) = \bar h^{(k)}\Big( A \times \mathbb{R}^{m \times k} \times \big(\bar B_\epsilon(\bm 0)\big)^k \times (0,1]^{k\uparrow} \Big)
    $

\item 
    \notationidx{notation-D-A-k-t-truncation-b-LDP}{$\mathbb{D}_{A}^{(k)|b} [0,T](\epsilon)$}:
    ${ \mathbb{D}}_{A}^{(k)|b}[0,T](\epsilon) \delequal h^{(k)|b}_{[0,T]} \Big( A \times \mathbb{R}^{m\times k} \times \big(\bar B_\epsilon(\bm 0)\big)^k\times(0,T]^{k\uparrow} \Big)$ with convention that $\mathbb{D}_{A}^{(-1)|b}[0,T](\epsilon)  = \emptyset$

\item 
    \notationidx{notation-D-A-k-t-truncation-b-LDP-T=1}{${\mathbb{D}}_{A}^{(k)|b}(\epsilon)$}:
    ${\mathbb{D}}_{A}^{(k)|b}(\epsilon) \delequal {\mathbb{D}}_{A}^{(k)|b}[0,1](\epsilon)
    =
    \bar h^{(k)|b} \Big( A \times \mathbb{R}^{m \times k} \times \big(\bar B_\epsilon(\bm 0)\big)^k\times(0,1]^{k\uparrow} \Big)$

\item 
    \notationidx{notation-D-A-k-t-truncation-b-M-LDP}{$\mathbb{D}_{A;M\downarrow}^{(k)|b}(\epsilon)$}: 
    $\mathbb{D}_{A;M\downarrow}^{(k)|b}(\epsilon) \delequal 
\bar h^{(k)|b}_{M\downarrow}\Big( A \times \mathbb{R}^{m \times k} \times \big(\bar B_\epsilon(\bm 0)\big)^k \times (0,1]^{k\uparrow} \Big)
$

    


\item 
    \notationidx{notation-D-J1}{$\dj{[0,T]}$}:
    Skorokhod $J_1$ metric on $\mathbb{D}[0,T]$

\item 
    \notationidx{notation-D-J1-T=1}{$\bm{d}_{J_1}$}:
    $\bm{d}_{J_1} = \dj{[0,1]}$ is the Skorodhod metric on $\mathbb{D} = \mathbb{D}[0,1]$

\linkdest{location, notation index E}

\item 
    \notationidx{notation-set-E-delta-LDP}{$E^\delta_{c,k}(\eta)$}:
    $E^\delta_{c,k}(\eta) \delequal \{ \tau^{>\delta}_{k}(\eta) \leq \floor{1/\eta} < \tau^{>\delta}_{k+1}(\eta);\ \eta\norm{\bm W^{>\delta}_j(\eta)} > c\ \ \forall j \in [k] \}$
    \quad
    $(c > \delta)$
    \quad
    (event that there are exactly $k$ ``big'' jumps by $\floor{1/\eta}$)

\item 
    $
    \notationidx{notation-set-check-E-epsilon-B-T}{\widecheck{E}(\epsilon,B,T)}:\ {\widecheck{E}(\epsilon,B,T)} 
    \delequal 
    \Big\{ \xi \in \mathbb{D}[0,T]:\ \exists t \leq T\ s.t.\ \xi_t \in B\text{ and }\xi_s \in I(\epsilon)\ \forall s \in [0,t) \Big\}
    $

\item 
    \notationidx{notation-eta}{$\eta$}: step length

\linkdest{location, notation index F}

\item
    $\notationidx{notation-sigma-algebra-F}{\mathcal{F}}$:
    the $\sigma-$algebra generated by iid copies $(\bm Z_j)_{j \geq 1}$
    
\item 
    \notationidx{notation-F}{$\mathbb F$}: 
    the filtration
    $\mathbb{F} = (\mathcal{F}_j)_{j \geq 0}$ where  $\mathcal{F}_0 \delequal \{\Omega, \emptyset\}$
    and $\mathcal{F}_j$ is the $\sigma$-algebra generated by $\bm Z_1,\cdots,\bm Z_j$

\linkdest{location, notation index G}

\item 
    $
    \notationidx{notation-mapping-bar-g-k-b}{\bar g^{(k)|b}}:\ 
    \bar g^{(k)|b}\big( \bm x, \textbf W, \textbf V, (t_1,\cdots,t_k)\big)
    \delequal 
    \bar h^{(k)|b}_{ [0,t_k + 1] }
    \Big(
        \bm x,
        \textbf W,
        \textbf V,
        (t_1,\cdots,t_k)
    \Big)(t_k)
    $

\item
    \notationidx{notation-check-g-k-b}{$\widecheck{g}^{(k)|b}$}:
    $
    {\widecheck{g}^{(k)|b}(\bm x,\textbf W,\bm t)}
    \delequal 
    \bar g^{(k)|b}\big(\bm x, \textbf W, (\bm 0,\cdots,\bm 0), \bm t\big)
    =
    h^{(k)|b}_{[0,t_k+1]}(\bm x,\textbf W,\bm t)(t_{k})
    $

\item 
    $
    \notationidx{notation-set-G-k-b-epsilon}{\mathcal G^{(k)|b}(\epsilon)}:\ 
    {\mathcal G^{(k)|b}(\epsilon)}
    \delequal 
    \bigg\{
    \bar g^{(k - 1)|b}
    \Big( \bm v_1 + \varphi_b\big(\bm \sigma(\bm v_1)\bm w_1\big),
    (\bm w_2,\cdots, \bm w_k), (\bm v_2,\cdots,\bm v_k), \bm t 
    \Big):
    \textbf W = (\bm w_1,\cdots, \bm w_k) \in \R^{d\times k},
    \textbf V = (\bm v_1,\cdots, \bm v_k) \in \Big(\bar B_\epsilon(\bm 0)\Big)^k,
    \bm t \in (0,\infty)^{k \uparrow}
    \bigg\}.
    $
    $
    \mathcal{G}^{(0)|b}(\epsilon) \delequal \bar B_{\epsilon}(\bm 0).
    $

\item 
    $
    \notationidx{notation-set-G-k-b}{\mathcal G^{(k)|b}}:\ 
    {\mathcal G^{(k)|b}}
    \delequal 
    \mathcal G^{(k)|b}(0)
    = 
    \bigg\{
    \widecheck{g}^{(k - 1)|b}
    \Big( \varphi_b\big(\bm \sigma(\bm 0)\bm w_1\big),
    (\bm w_2,\cdots, \bm w_k), \bm t 
    \Big):\ 
    \textbf W = (\bm w_1,\cdots,\bm w_k) \in \R^{d \times k},
    \bm t \in (0,\infty)^{k \uparrow}
    \bigg\}.
    $

\item 
    $
    \notationidx{notation-extended-coverage-set-bar-G-k-b-epsilon}{\bar{\mathcal G}^{(k)|b}(\epsilon)}:\ 
    {\bar{\mathcal G}^{(k)|b}(\epsilon)}
    \delequal
    \Big\{
        \bm y_t(\bm x):\ \bm x \in \mathcal G^{(k)|b}(\epsilon),\ t \geq 0
    \Big\},
    $




\item 
    \notationidx{notation-Gamma-M-adapted-process-bounded-by-M-LDP}{$\bm{\Gamma}_M$}:
    $\bm{\Gamma}_M \delequal \big\{ (\bm V_j)_{j \geq 0}\text{ is adapted to }\mathbb{F}:\ \norm{\bm V_j} \leq M\ \forall j \geq 0\text{ almost surely} \big\};$
    see \eqref{def: Gamma M, set of bounded adapted process}
    

\linkdest{location, notation index H}


\item 
    \notationidx{notation-H}{$H$:} $H(x)  \delequal \P(\norm{\bm Z} > x) \in \RV_{-\alpha}(x)$

\item
    \notationidx{notation-H-L}{$H_L$}:
    $H_L(x)\delequal  \nu\big( \{ \bm y \in \R^m:\ \norm{\bm y} > x\} \big) \in \RV_{-\alpha}(x)$



\item
    $\notationidx{notation-h-k-t-bar-mapping-LDP}{\bar h^{(k)}_{[0,T]}}$: 
    ${\bar h^{(k)}_{[0,T]}} = {\bar h^{(k)|\infty}_{[0,T]}}$.
    An operator for perturbed gradient flow under $\bm a(\cdot)$ with initial value $\bm x$, jump sizes $\bm w_j$'s (modulated by $\bm \sigma(\cdot)$) with perturbations $\bm v_j$'s, and jump times $t_j$'s

\item 
    $\notationidx{notation-h-k-t-mapping-LDP}{h^{(k)}_{[0,T]}}:\ 
    h^{(k)}_{[0,T]}(\bm x, \bm W, \bm t)
    \delequal 
    h^{(k)|\infty}_{[0,T]}(\bm x, \bm W, \bm t)
    =
    \bar h^{(k)}_{[0,T]}\big(\bm x,\bm W, (\bm 0,\cdots,\bm 0), \bm t\big)
    $


\item 
    $\notationidx{notation-h-k-t-sigma-mapping-T=1}{h^{(k)}}:\ 
    {h^{(k)}} \delequal {h^{(k)}_{[0,1]}}
    $

\item 
    $\notationidx{notation-h-k-t-bar-mapping-truncation-level-b-LDP}{\bar h^{(k)|b}_{[0,T]}}$:
    an operator for perturbed gradient flow under $\bm a(\cdot)$ with initial value $\bm x$, jump sizes $\bm w_j$'s (modulated by $\bm \sigma(\cdot)$ and truncated under $b > 0$) with perturbations $\bm v_j$'s, and jump times $t_j$'s; see \eqref{def: perturb ode mapping h k b, 1}--\eqref{def: perturb ode mapping h k b, 3}

\item 
    $\notationidx{notation-h-k-b-t-mapping-LDP}{h^{(k)|b}_{[0,T]}}:\ h^{(k)|b}_{[0,T]}(\bm x, \bm W, \bm t)
    \delequal 
    \bar h^{(k)|b}_{[0,T]}\big(\bm x,\bm W, (\bm 0,\cdots,\bm 0), \bm t\big)$



\item 
    $\notationidx{notation-h-k-t-sigma-mapping-truncation-level-b-LDP-T=1}{h^{(k)|b}}:\ {h^{(k)|b}}  \delequal h^{(k)|b}_{[0,1]}$

\item
    \notationidx{notation-mapping-bar-h-k-t-b-M-LDP}{$\bar h^{(k)|b}_{M\downarrow}$}: a modified version of {$\bar h^{(k)|b}$}
    where the truncated drift and diffusion coefficients $\bm a_{M}$, $\bm \sigma_M$ are applied instead of $\bm a$, $\bm\sigma$;
    see \eqref{def: perturb ode mapping h k b, truncated at M, 1}--\eqref{def: perturb ode mapping h k b, truncated at M, 3}

\item
    $\notationidx{mapping-h-k-b-M-down}{h^{(k)|b}_{M\downarrow}}:\ 
    {h^{(k)|b}_{M\downarrow}}\big(\bm x, \textbf W,\bm t\big)
    \delequal 
    \bar h^{(k)|b}_{M\downarrow}\big(\bm x, \textbf W,(\bm 0,\cdots,\bm 0),\bm t\big).$

\linkdest{location, notation index I}

\item 
    $\notationidx{notation-exit-domain-I}{I}$:
    the open, bounded domain $\bm 0 \in {I} \subset \R^m$ that belongs to the attraction field for $\bm 0$ satisfying Assumption~\ref{assumption: shape of f, first exit analysis}.

\item
    $
    \notationidx{notation-I-epsilon-shrinkage}{I_\epsilon}:\ {I_\epsilon} = \{ \bm y:\ \norm{\bm x - \bm y} < \epsilon\ \Longrightarrow\ \bm x \in I \}
    $

\item 
    $\notationidx{notation-domain-check-I-epsilon}{\widecheck I(\epsilon)}:\ 
    {\widecheck I(\epsilon)}\delequal 
    \Big\{
        \bm x \in I:\  \norm{ \bm y_{1/\epsilon}(\bm x) } < \widecheck \epsilon
    \Big\}$
    with $\widecheck{\epsilon} > 0$ defined in \eqref{def: covering sets I epsilon, first exit time}

\item

    \notationidx{notation-A-i-concentration-of-small-jumps-2}{$I_i(\eta,\delta)$}:
    $I_i(\eta,\delta) 
    \delequal 
    \big\{j \in \mathbb{N}:\  \tau^{>\delta}_{i-1}(\eta) + 1 \leq j \leq \big(\tau^{>\delta}_{i}(\eta) - 1 \big) \wedge \floor{1/\eta}\big\}.$


\linkdest{location, notation index J}

\item
    \notationidx{J-Z-c-n}{$\mathcal{J}_Z(c,n):$}
    $\mathcal{J}_{\bm Z}(c,n) \delequal \#\{i \in [n]:\ \norm{\bm Z_i} \geq c \}$





\item  
    \notationidx{notation-J-*-first-exit-analysis}{$\mathcal{J}^I_b$}:
    $
    \mathcal{J}^I_b \delequal \min\big\{ k \geq 1:\ \mathcal G^{(k)|b} \cap I^\complement \neq \emptyset \big\}.
    $
    The ``discretized width'' metric for $I$ w.r.t.\ truncation threshold $b$.

\linkdest{location, notation index K}


\linkdest{location, notation index L}




\item
    \notationidx{notation-levy-process}{$\bm L$}:
    $\bm{L} = \{\bm L_t: t \geq 0\}$ is the L\'evy process taking values in $\R^m$ with the generating triplet $(c_{\bm L},\bm \Sigma_{\bm L},\nu)$ where $c_{\bm L} \in \mathbb{R}^m$ is the drift parameter, 
    $\bm \Sigma_{\bm L}$ is the positive semi-definite matrix that dictates the magnitude of the Brownian motion term in $\bm L_t$, and $\nu$ is the L\'evy measure.


     

    

\item 
    \notationidx{notation-lebesgue-measure-restricted}{$\mathcal{L}_t$}: 
    Lebesgue measure restricted on $(0,t)$

\item 
    \notationidx{notation-lebesgue-measure-on-ordered-[0,t]}{$\mathcal{L}^{k\uparrow}_t$}:
    Lebesgue measure restricted on $(0,t)^{k \uparrow}$

\item
    \notationidx{notation-measure-L-k-up-infty}{$\mathcal{L}^{k\uparrow}_\infty$}:
    Lebesgue measure restricted on $\{ (t_1,\cdots,t_k) \in (0,\infty)^k:\ 0 < t_1 < t_2 < \cdots < t_k \}$

\item 
    \notationidx{notation-law-of-X}{$\mathscr{L}(X)$}:
    law of the random element $X$

\item 
    \notationidx{notation-law-of-X-cond-on-A}{$\mathscr{L}(X|A)$}:
    conditional law of $X$ on event $A$

    

\item 
    $\notationidx{notation-lambda-scale-function}{\lambda(\eta)}$:
    $
    {\lambda(\eta)} \delequal \eta^{-1}H(\eta^{-1}) \in \RV_{\alpha -1}(\eta)
    $
    as $\eta \downarrow 0$. 
    
\item 
    \notationidx{notation-scale-function-lambda-L}{$\lambda_L(\eta;\beta)$}:
    $
    {\lambda_L(\eta;\beta)} \delequal \eta^{-\beta}H_L(\eta^{-1}) \in \RV_{\alpha-1}(\eta)
    $ as $\eta \downarrow 0$


    
\linkdest{location, notation index M}
    
\item
    \notationidx{notation-M-S-exclude-C}{$\mathbb{M}(\mathbb{S}\setminus \mathbb{C})$}:
    $\mathbb{M}(\mathbb{S}\setminus \mathbb{C})
    \delequal 
    \{
    \nu(\cdot)\text{ is a Borel measure on }\mathbb{S}\setminus \mathbb{C} :\ \nu(\mathbb{S}\setminus \mathbb{C}^r) < \infty\ \forall r > 0
    \}.$

    


\linkdest{location, notation index N}

\item
    \notationidx{notation-non-negative-numbers}{$\mathbb N$}:
    $\mathbb{N} = \{0,1,2,\cdots\}$

\item 
    $\notationidx{notation-R-d-unit-sphere}{\mathfrak N_d}:\ {\mathfrak N_d} \delequal \{\bm x \in \R^d:\ \norm{\bm x} = 1\}$,
    unit sphere in $\R^d$

\item 
    \notationidx{notation-measure-nu-alpha}{$\nu_\alpha$}: $\nu_\alpha[x,\infty) = x^{-\alpha}$


\linkdest{location, notation index O}
\linkdest{location, notation index P}


\item 
    \notationidx{notation-truncation-operator-level-b}{$\varphi_c$}:
    ${\varphi_c}(\bm w) 
    \delequal{} 
    \Big(\frac{c}{\norm{\bm w}} \wedge 1\Big)\cdot \bm w$; truncation operator at level $c > 0$

\item 
    $\notationidx{notation-Phi-polar-transform}{\Phi(\bm x)}:\ {\Phi(\bm x)} \delequal \big(\norm{\bm x},\frac{\bm x}{\norm{\bm x}}\big)$ for all $\bm x \neq 0$; polar transform

\linkdest{location, notation index Q}





\linkdest{location, notation index R}

\item 
    \notationidx{notation-rho-LDP}{$\rho$}:
    $\rho \delequal \exp(D)$;  $D$ is the constant in Assumption \ref{assumption: lipschitz continuity of drift and diffusion coefficients}

    

    

\item 
    \notationidx{notation-RV-LDP}{$\RV_\beta$}:
    $\phi \in \RV_\beta$ (as $x \rightarrow \infty$) if $\lim_{x \rightarrow \infty}\phi(tx)/\phi(x) = t^\beta$ for any $t>0$;
    $\phi \in \RV_\beta(\eta)$ (as $\eta \downarrow 0$) if $\lim_{\eta \downarrow 0}\phi(t\eta)/\phi(\eta) = t^\beta$ for any $t>0$

\item  
    \notationidx{notation-R-eta-b-epsilon-return-time}{$R^{\eta|b}_\epsilon(\bm x)$}:
    $
    {R^{\eta|b}_\epsilon(\bm x)}  \delequal \min\big\{ j \geq 0:\ \norm{\bm X^{\eta|b}_j(\bm x)} < \epsilon \big\}
    $ 


\linkdest{location, notation index S}


\item 
    \notationidx{sigma}{$\bm \sigma$}: diffusion coefficient $\bm \sigma: \mathbb{R}^m \to \mathbb{R}^{m\times d}$
    
\item
    \notationidx{sigma-M}{$\bm \sigma_M$}: diffusion coefficient with the argument $\bm x$ projected onto the ball $\{ \bm x \in \R^m:\ \norm{\bm x} \leq M \}$

    

\item 
    \notationidx{notation-support-of-function-g}{$\text{supp} (g)$}:
$\text{supp} (g) \delequal \big(\{ x \in \mathbb S:\ g(x) \neq 0 \}\big)^-$;
support of $g: \mathbb S \to \mathbb{R}$

\item 
    \notationidx{notation-support-of-mu}{$\text{supp}(\mu)$}: the smallest closed set $C$ such that $\mu(\S \setminus C)= 0$

\item 
    \notationidx{notation-borel-sigma-algebra}{$\mathscr{S}_\mathbb{S}$}:
    Borel $\sigma$-algebra of the metric space $(\mathbb{S},\bm{d})$

\linkdest{location, notation index T}

\item
    $
    \notationidx{notation-hitting-time-t-x-epsilon}{\bm t_{\bm x}(\epsilon)}:\ {\bm t_{\bm x}(\epsilon)} \delequal \inf\Big\{ t \geq 0:\ \bm y_t(\bm x) \in \bar B_{\epsilon}(\bm 0) \Big\}
    $

\item 
    $\notationidx{notation-t-epsilon-ode-return-time}{ \bm{t}(\epsilon) }:
        { \bm{t}(\epsilon) } \delequal \sup\Big\{ \bm t_{\bm x}(\epsilon):\ \bm x \in I_\epsilon^-   \Big\}$

\item 
    \notationidx{notation-large-jump-time}{$\tau^{>\delta}_i(\eta)$}:
    $\tau^{>\delta}_i(\eta) \delequal{} \min\{ n > \tau^{>\delta}_{i-1}(\eta):\ \eta\norm{\bm Z_j} > \delta  \},\ \tau^{>\delta}_0(\eta) = 0$; arrival time of $j$\textsuperscript{th} large jump

\item   
    \notationidx{notation-tau-eta-x-first-exit-time}{$\tau^\eta(\bm x)$}:
    $
    \tau^\eta(\bm x) \delequal \min\big\{j \geq 0:\ \bm X^\eta_j(\bm x) \notin I\big\}
    $

\item  
    \notationidx{notation-tau-eta-b-x-first-exit-time}{$\tau^{\eta|b}(\bm x)$}:
    $
    \tau^{\eta|b}(\bm x) \delequal \min\big\{j \geq 0:\ \bm X^{\eta|b}_j(\bm x) \notin I \big\}
    $

\item  
    \notationidx{notation-tau-eta-b-epsilon-exit-time}{$\tau^{\eta|b}_\epsilon(\bm x)$}:
    $
    {\tau^{\eta|b}_\epsilon(\bm x)} \delequal \min\big\{ j \geq 0:\ \bm X^{\eta|b}_j(\bm x) \notin I_\epsilon \big\}
    $




\linkdest{location, notation index U}

\item 
    \notationidx{notation-U-j-t}{$U_j$}:
    iid copies of Unif$(0,1)$
    
\item  
    \notationidx{notation-U-j-k-LDP}{$U_{(j;k)}$}:
    $0 \leq U_{(1;k)} \leq U_{(2;k)} \leq \cdots \leq U_{(k;k)}$;
    the order statistics of iid $\big(U_j\big)_{j = 1}^k$

\linkdest{location, notation index V}



\linkdest{location, notation index W}

   
\item 
    \notationidx{notation-large-jump-size}{$\bm W^{>\delta}_i(\eta)$}: 
    $\bm W^{>\delta}_i(\eta) \delequal{} \bm Z_{\tau^{>\delta}_i(\eta)}$; size of $j$\textsuperscript{th} large jump, i.e., with size above threshold $\delta/\eta$
   
\item 
    \notationidx{notation-W-*_j}{$\bm W^*_j(\cdot)$}:
    iid copies of $\bm W^*(c)$ defined in \eqref{def: prob measure Q, LDP}
    

\linkdest{location, notation index X}

\item 
    \notationidx{notation-discrete-gradient-descent}{$\bm{x}^\eta_{j}(x)$}: (deterministic) difference equation
    $\bm{x}^\eta_{j}(x) = \bm{x}^\eta_{j-1}(x) + \eta \bm a\big(\bm{x}^\eta_{j-1}(x) \big)$ for any $j \geq 1$ with initial condition $\bm{x}^\eta_{0}(x) = x$.

\item 
    \notationidx{notation-breve-X-eta-delta-t}{$\breve{\bm X}^{\eta,\delta}_t(\bm x)$}: ODE that coincides with $\bm X^{\eta}_{\floor{t/\eta} }(\bm x)$ at times $t = \eta \tau^{>\delta}_i(\eta)$, $i=1,2,\ldots$.
    
\item  
    \notationidx{notation-breve-X-eta-b-delta-t}{$\breve{\bm X}^{\eta|b;\delta}_t(\bm x)$}:
    ODE that coincides with $\bm X^{\eta|b}_{\floor{t/\eta} }(\bm x)$ at times $t = \eta \tau^{>\delta}_i(\eta)$, $i=1,2,\ldots$.

\item
    \notationidx{notation-hat-X-clip-b-top-j-jumps}{$\hat{\bm{X}}^{\eta|b; >\delta }(\bm x)$}:
    ODE perturbed by $\bm W^{>\delta}_i(\eta)$'s, with sizes modulated by $\bm \sigma(\cdot)$ and truncated under $b$.
    
    
\item 
    \notationidx{notation-X-j-eta-x}{$\bm X^\eta_j(x)$}: $\bm X^\eta_0(\bm x) = \bm x;\ 
    \bm X^\eta_j(\bm x) = \bm X^\eta_{j - 1}(\bm x) +  \eta \bm a\big(\bm X^\eta_{j - 1}(\bm x)\big) + \eta\bm \sigma\big(\bm X^\eta_{j - 1}(\bm x)\big)\bm Z_j,\ \ \forall j \geq 1$


\item 
    \notationidx{notation-scaled-X-0T-eta-LDP}{$\bm{X}^\eta_{[0,T]}(x)$}:
    $\bm{X}_{[0,T]}^\eta(x) \delequal \big\{ \bm X^\eta_{ \floor{ t/\eta } }(x):\ t \in [0,T] \big\}$

\item 
    \notationidx{notation-scaled-X-eta-LDP}{$\bm{X}^\eta(x)$}: $\bm{X}^\eta(x) = \bm{X}_{[0,1]}^\eta(x) \delequal \big\{ \bm X^\eta_{ \floor{ t/\eta } }(x):\ t \in [0,1] \big\}$


\item 
    \notationidx{notation-X-eta-j-truncation-b-LDP}{$\bm X^{\eta|b}_j(\bm x)$}:
    $\bm X^{\eta|b}_j(\bm x)= \bm X^{\eta|b}_{j - 1}(\bm x)+  \varphi_b\Big(\eta \big[\bm a\big(\bm X^{\eta|b}_{j - 1}(\bm x)\big) + \bm \sigma\big(\bm X^{\eta|b}_{j - 1}(\bm x)\big)\bm Z_j\big]\Big)\ \ \forall j \geq 1$


    

\item 
    \notationidx{notation-scaled-X-eta-mu-truncation-b-LDP}{${\bm{X}}^{\eta|b}_{[0,T]}(\bm x)$}:
    $ {\bm{X}}^{\eta|b}_{[0,T]}(\bm x) \delequal \big\{ \bm X^{\eta|b}_{ \floor{ t/\eta } }(\bm x):\ t \in [0,T] \big\}$

\item 
    \notationidx{notation-scaled-X-eta-mu-truncation-b-LDP-T=1}{${\bm{X}}^{\eta|b}(\bm x)$}:
    ${\bm{X}}^{\eta|b}(\bm x) = {\bm{X}}^{\eta|b}_{[0,1]}(\bm x) \delequal \big\{ \bm X^{\eta|b}_{ \floor{ t/\eta } }(\bm x):\ t \in [0,1] \big\}$



\linkdest{location, notation index Y}

\item 
    \notationidx{notation-continuous-gradient-descent}{$\bm{y}_t(\bm x)$}:  ODE path
    $\frac{d\bm{y}_t(\bm x)}{dt} = \bm a\big(\bm{y}_t(\bm x)\big)$ for any $t > 0$
    with initial condition $\bm{y}_0(\bm x) = \bm x$.

\item
    \notationidx{notation-Y-eta-SDE}{$\bm Y^\eta_t(\bm x)$}:
    $d\bm Y^\eta_t(\bm x) 
    =
    \bm a\big(\bm Y^\eta_{t-}(\bm x)\big)dt
        +
    \bm \sigma\big(\bm Y^\eta_{t-}(\bm x)\big) d\bar{\bm L}^\eta_t$
    
\item 
    \notationidx{notation-bm-Y-0T-eta-SDE}{$\bm Y^\eta_{[0,T]}(\bm x)$}:
    $
    {\bm Y^\eta_{[0,T]}(\bm x)} = \{\bm Y^\eta_t(\bm x):\ t\in[0,T]\}
    $

\item 
    $\notationidx{notation-bm-Y-eta-SDE}{\bm Y^\eta(\bm x)} = \{\bm Y^\eta_t(\bm x):\ t\in[0,1]\}$

\item
    \notationidx{notation-Y-eta-b-SDE}{$\bm Y^{\eta|b}_t(\bm x)$}:
    A modified version of the SDE $\bm Y^\eta_t(\bm x)$ with each discontinuity truncated under $b$
    
\item
    \notationidx{notation-bm-Y-0T-eta-b-SDE}{$\bm{Y}^{\eta|b}_{[0,T]}(\bm x)$}:
    $
    {\bm{Y}^{\eta|b}_{[0,T]}(\bm x)} \delequal \big\{\bm Y^{\eta|b}_t(\bm x):\ t \in [0,T]\big\}
    $

\item 
    $\notationidx{notation-bm-Y-eta-b-SDE}{\bm Y^{\eta|b}(\bm x)}= \{\bm Y^{\eta|b}_t(\bm x):\ t\in[0,1]\}$

\linkdest{location, notation index Z}

\item 
    \notationidx{notation-Z-iid-noise-LDP}{$\bm Z_j$}:
    $(\bm Z_j)_{j \geq 1}$ is a sequence of iid copies of a random vector $\bm Z$ such that
    $\E \bm Z = \bm 0$ and 
    the multivariate regular variation assumption (i.e., Assumption~\ref{assumption gradient noise heavy-tailed}) holds for the law of $\bm Z$.


\end{itemize}
\fi

\ifshowtheoremtree
\newpage
\footnotesize
\newgeometry{left=1cm,right=1cm,top=0.5cm,bottom=1.5cm}

\section*{\linkdest{location of theorem tree}Theorem Tree}
\begin{thmdependence}[leftmargin=*]

\thmtreenode{-}
    {Theorem}{theorem: portmanteau, uniform M convergence}
    {0.8}{Portmanteau Theorem for uniform $\M(\S\setminus\C)$-convergence}
\bigskip

\thmtreenodewopf{}
    {Lemma}{lemma: asymptotic equivalence when bounded away, equivalence of M convergence}
    {0.8}{
        asympt.\ equiv.\ of $X_n$ and $Y_n$ w.r.t.\\ $\epsilon_n^{-1}$ in $\M(\mathbb S\setminus\mathbb{C})$ implies the same $\M$-convergence of $\epsilon_n^{-1}\P(X_n\in \cdot)$ and $\epsilon_n^{-1}\P(Y_n\in \cdot)$.%
    }
\bigskip

        

\thmtreenode{\complete}
    {Theorem}{theorem: LDP 1, unclipped}
    {0.8}{
    [A\ref{assumption gradient noise heavy-tailed},\ref{assumption: lipschitz continuity of drift and diffusion coefficients},\ref{assumption: nondegeneracy of diffusion coefficients},\ref{assumption: boundedness of drift and diffusion coefficients}]
    Sample Path Large Deviations for SGD $\bm X^\eta$.
    $\P ({\bm X}^{\eta}(x)\in \cdot)/\lambda^k(\eta) \to \mathbf{C}^{(k)}(\cdot;x)$ in $\mathbb{M}(\D \setminus \D_A^{(k-1)})$ uniformly in $x$ on any compact set $A$ 
    }
    \begin{thmdependence}
    \thmtreenode{-}
            {Lemma}
            {lemma: continuity of h k b mapping}
            {0.8}{
            [A\ref{assumption: lipschitz continuity of drift and diffusion coefficients},\ref{assumption: boundedness of drift and diffusion coefficients}]
            $h^{(k)}_{[0,T]}$ is continuous on $\R\times \R^k \times (0,T)^{k\uparrow}$.
            }
        \begin{thmdependence}
        \thmtreenode{\issue}
            {Lemma}
            {lemma: continuity of h k b mapping clipped}
            {0.8}{
                [A\ref{assumption: lipschitz continuity of drift and diffusion coefficients},\ref{assumption: nondegeneracy of diffusion coefficients}] 
                $h^{(k)|b}$ is continuous.
            }
            \begin{thmdependence}
                 \thmtreenode{-}
            {Corollary}{corollary: existence of M 0 bar delta bar epsilon, clipped case, LDP}
            {0.8}{
                [A\ref{assumption: lipschitz continuity of drift and diffusion coefficients},\ref{assumption: nondegeneracy of diffusion coefficients}] 
                If $d(B, \D_{A|b}^{(k-1)}) >0$, the shocks of the paths in $B \cap \D_{A}^{(k)|b}$ are bounded away from 0. 
            }
        
            \begin{thmdependence}
            \thmtreenode{-}
                {Lemma}{lemma: boundedness of k jump set under truncation, LDP clipped}
                {0.8}{
                    $\sup_{\xi\in \D_{A}^{(k)|b}} \|\xi\| < \infty$ for any $b\in(0,\infty)$, $k\in \mathbb N$, and compact set $A$.
                }
            \end{thmdependence}
            \end{thmdependence}
        \end{thmdependence}

    \thmtreenode{-}
            {Lemma}{lemma: LDP, bar epsilon and delta}
            {0.8}{
                [A\ref{assumption: lipschitz continuity of drift and diffusion coefficients},\ref{assumption: boundedness of drift and diffusion coefficients}] 
                If $d(B, \D_{A}^{(k-1)}) >0$, the shocks of the paths in $B \cap \D_{A}^{(k)}$ are bounded away from 0. 
            }

    \thmtreenode{-}
        {Lemma}{lemma: sequential compactness for limiting measures, LD of SGD}{0.8}
        {
        Verify $\lim_{n \to \infty}\mathbf C^{(k)}(f;x_{n})
            =
            \mathbf C^{(k)}(f;x^*)$
            and
            $\lim_{n \to \infty}\mathbf C^{(k)|b}(f;x_{n})
            =
            \mathbf C^{(k)|b}(f;x^*)$
        }
        \begin{thmdependence}
            \thmtreeref
                {Lemma}{lemma: continuity of h k b mapping}
            \thmtreeref
                {Lemma}{lemma: continuity of h k b mapping clipped}
            \thmtreeref
                {Lemma}{lemma: LDP, bar epsilon and delta}
            \thmtreenode{-}
                {Lemma}
                {lemma: LDP, bar epsilon and delta, clipped version}
                {0.8}{
                [A\ref{assumption: lipschitz continuity of drift and diffusion coefficients}, A\ref{assumption: boundedness of drift and diffusion coefficients}]
                If $d(B, \D_{A|b}^{(k-1)}) >0$, the shocks of the paths in $B \cap \D_{A}^{(k)|b}$ are bounded away from 0.
                }
        \end{thmdependence}

    \thmtreeref
        {Theorem}{theorem: portmanteau, uniform M convergence}

    \thmtreeref
        {Proposition}{proposition: standard M convergence, LDP unclipped}

    \end{thmdependence}

\bigskip

\thmtreenode{\complete}
    {Theorem}{corollary: LDP 2}
    {0.8}{
    [A\ref{assumption gradient noise heavy-tailed},\ref{assumption: lipschitz continuity of drift and diffusion coefficients},\ref{assumption: nondegeneracy of diffusion coefficients}]
    Sample path large deviations for $\bm X^{\eta|b}$.
    $\P ({\bm X}^{\eta|b}(x)\in \cdot)/\lambda^k(\eta) \to \mathbf{C}_b^{(k)}(\cdot;x)$ in $\mathbb{M}(\D \setminus \D_{A|b}^{(k-1)})$ uniformly in $x$ on any compact set $A$ 
    }
    
    \begin{thmdependence}
        \thmtreeref
            {Proposition}{proposition: standard M convergence, LDP clipped}
        \thmtreeref
            {Lemma}{lemma: sequential compactness for limiting measures, LD of SGD}
        \thmtreeref
            {Lemma}{lemma: LDP, bar epsilon and delta, clipped version}
        \thmtreeref
            {Theorem}{theorem: portmanteau, uniform M convergence}
    \end{thmdependence}

\bigskip

\thmtreenode{-}
    {Proposition}{proposition: standard M convergence, LDP unclipped}
    {0.8}{
    [A\ref{assumption gradient noise heavy-tailed},\ref{assumption: lipschitz continuity of drift and diffusion coefficients},\ref{assumption: boundedness of drift and diffusion coefficients}]
        $\P({\bm X}^{\eta_n}(x_n)\in \cdot)/\lambda^k(\eta_n) \to \mathbf{C}^{(k)}(\cdot; x^*)$ in $\M(\D\setminus \D_{A}^{(k-1)})$ if $\eta_n\to 0$, $x_n \to x^*$, and $x_n, x^* \in A$: cpt
    }
    \begin{thmdependence}
    
    \thmtreenode{-}
    {Proposition}{proposition: standard M convergence, LDP clipped}
    {0.8}{ [A\ref{assumption gradient noise heavy-tailed},\ref{assumption: lipschitz continuity of drift and diffusion coefficients},\ref{assumption: nondegeneracy of diffusion coefficients}]
        $\P({\bm X}^{\eta_n}(x_n)\in \cdot)/\lambda^k(\eta_n) \to \mathbf{C}^{(k)}(\cdot; x^*)$ in $\M(\D\setminus \D_{A}^{(k-1)})$ if $\eta_n\to 0$, $x_n \to x^*$, and $x_n, x^* \in A$: cpt
    }
    
    \begin{thmdependence}
    \thmtreeref{Proposition}{proposition: standard M convergence, LDP clipped, stronger boundedness assumption}
    
    \thmtreeref
            {Corollary}{corollary: existence of M 0 bar delta bar epsilon, clipped case, LDP}
        
    \end{thmdependence}


    \thmtreeref{Lemma}{lemma LDP, small jump perturbation}
    
    \thmtreeref
            {Lemma}{lemma: LDP, bar epsilon and delta}

    \end{thmdependence}

\bigskip

\thmtreenode{-}
    {Proposition}{proposition: standard M convergence, LDP clipped, stronger boundedness assumption}
    {0.8}{ [A\ref{assumption gradient noise heavy-tailed},\ref{assumption: lipschitz continuity of drift and diffusion coefficients},\ref{assumption: boundedness of drift and diffusion coefficients}]
        $\P({\bm X}^{\eta_n|b}(x_n)\in \cdot)/\lambda^k(\eta_n) \to \mathbf{C}^{(j)}_b(\cdot; x^*)$ in $\M(\D\setminus \D_{A|b}^{(k-1)})$ 
        if $\eta_n\to 0$, $x_n \to x^*$, and $x_n, x^* \in A$: cpt
    }
    \begin{thmdependence}
        \thmtreeref
            {Lemma}{lemma: asymptotic equivalence when bounded away, equivalence of M convergence}
        
        \thmtreenode{-}
            {Proposition}{proposition: asymptotic equivalence, clipped}
            {0.8}{
                $\bm{X}^{\eta_n|b}(x_n)$ and $\hat{\bm X}^{\eta_n|b;(k)}(x_n)$ are asympt.\ equiv.\ w.r.t.\\ $\lambda^k(\eta)$ in $\M(\D\setminus \D_{A|b}^{(k-1)})$
            }
    
    \begin{thmdependence}
    
        \thmtreeref
                {Lemma}
                {lemma: LDP, bar epsilon and delta, clipped version}

        \thmtreenode{-}
            {Lemma}{lemma: SGD close to approximation x circ, LDP}
            {0.8}
            {
                (time-scaled) SGD $X^{\eta|b}_{\lfloor t/\eta \rfloor}$ is close to (slow) ODE until first large jump
            }
            \begin{thmdependence}
            \thmtreenode{-} 
                {Lemma}{lemmaBasicGronwall}
                {0.8}{
                SGD and GD are close to each other if the noises are small
                }

            \thmtreenode{-}
                {Lemma}{lemma Ode Gd Gap} 
                {0.8}{
                GD ($\bm y^\eta_{\floor{s}}(y)$) and GF ($\bm x^\eta(\cdot;x)$) are close to each other
                }

            \end{thmdependence}
        
        \thmtreenode{-}
            {Lemma}{lemma LDP, small jump perturbation} 
            {0.8}{
                \begin{minipage}[t]{\linewidth}
                (a) 
                $
                \sup_{ (W_i)_{i \geq 0} \in  \bm{\Gamma}_M}\P\Big( \max_{ j \leq \floor{t/\eta} \wedge \big(\tau^{>\delta}_{1}(\eta) - 1\big) }\ \eta\big|\sum_{i = 1}^j W_{i-1}Z_i \big| > \epsilon \Big)   
                = \bm o({\eta^N}) 
                $\\
                
                (b) 
                $
                \sup_{x \in \R} \P\Big( \big(\bigcap_{i = 1}^k A_i(\eta,\epsilon,\delta,t,x)\big)^c \Big) 
                = \bm o(\eta^N)
                $
                \end{minipage}
            }
        
        \thmtreenode{-}
            {Lemma}{lemma: SGD close to approximation x breve, LDP clipped}
            {0.8}{
                $\hat{X}^{\eta|b;>\delta}_{t}(x)$ and $X_{\lfloor t\rfloor}^{\eta|b}(x)$ are close to each other on $\cap_{i=1}^{k+1} A_i(\eta,\epsilon,\delta, x)$
            }
    \end{thmdependence}
            
    \thmtreenode{-}
        {Proposition}{proposition: uniform weak convergence, clipped}
        {0.8}{ [A\ref{assumption: boundedness of drift and diffusion coefficients}]
        $\P(\hat{\bm X}^{\eta|b;(k)}\in \cdot)/\lambda^k(\eta) \to \mathbf{C}^{(k)|b}   $ in $\M(\D\setminus \D_{A|b}^{(k-1)})$
        }
    
        \begin{thmdependence}
        \thmtreeref{Lemma}{lemma: LDP, bar epsilon and delta, clipped version}
        \thmtreeref
            {Lemma}
            {lemma: continuity of h k b mapping clipped}
        \thmtreenode{-}
            {Lemma}
            {lemma: weak convergence, expectation wrt approximation, LDP, preparation}
            {0.8}{
            On \hyperlink{index, notation-set-E-delta-LDP}{$E^\delta_{c,k}(\eta)$}, continuous and bounded functionals of scaled jump times and jump sizes converge to that of uniform and Pareto distributions.
            }
            \begin{thmdependence}
            \thmtreenode{-}
                {Lemma}{lemma: weak convergence of cond law of large jump, LDP}
                {0.8}{
                Conditional on \hyperlink{index, notation-set-E-delta-LDP}{$E^\delta_{c,k}(\eta)$}, the scaled jump times and jump sizes converge to uniform and Pareto distributions. 
                }
            \end{thmdependence}
        
        \end{thmdependence}
    \end{thmdependence}

\bigskip

\bigskip

\thmtreenode{-}
    {Theorem}{theorem: first exit time, unclipped}{0.8}
    {
    (a)
    $
        \lim_{\eta \downarrow 0}\P\Big(C^I_b \eta\cdot \lambda^{ \mathcal{J}^I_b }(\eta)\tau^{\eta|b}(x) > t
        \Big)
        =
        \exp(-t)
        $\\
    
    (b)
    $
    \lim_{\eta \downarrow 0}\P\Big(C^* \eta \lambda(\eta)\tau^\eta(x) > t
    \Big)
    =
    \exp(-t)
    $
    }
    \begin{thmdependence}
        \thmtreenode{-}
            {Lemma}{lemmaGeomFront}{0.8}{} 
        \thmtreenode{-}
            {Theorem}{thm: exit time analysis framework} {0.8}
            {
            First Exit Time Analysis Framework: $\P\big(
        \gamma(\eta)\tau_{I^\complement}^{\eta}(x)>t,\,V_{\tau}^\eta(x)\in B  
    \big) \approx C(B)e^{-t}$
            }
        \begin{thmdependence}
            \thmtreenode{-}
                {Proposition}{prop: exit time analysis main proposition}{0.8}{
                First Exit Time Analysis Framework ($\epsilon$-relaxed version):
                $
                \P\big(\gamma(\eta) \tau_{I(\epsilon)^\complement}^\eta(x) > t;\; V^\eta_{\tau_\epsilon}(x) \in B\big)
                \approx C(B)e^{-t} + \delta_{t,B}(\epsilon)
                $
                }
        \end{thmdependence}
        \thmtreenodewopf{}
            {Lemma}{lemma: exit prob one cycle, with exit location B, first exit analysis}{0.8}{Apply general framework: Verifying conditions \eqref{eq: exit time condition lower bound} and \eqref{eq: exit time condition upper bound}}
        \thmtreenodewopf{}
            {Lemma}{lemma: fixed cycle exit or return taking too long, first exit analysis}{0.8}{Apply general framework: Verifying condition \eqref{eq:E3}}
        \thmtreenodewopf{}
            {Lemma}{lemma: cycle, efficient return}{0.8}{Apply general framework: Verifying condition \eqref{eq:E4}}
    \end{thmdependence}

\bigskip

\thmtreenode{-}
    {Lemma}{lemma: exit prob one cycle, with exit location B, first exit analysis}{0.8}{Apply general framework: Verifying conditions \eqref{eq: exit time condition lower bound} and \eqref{eq: exit time condition upper bound}}
    
    \begin{thmdependence}
        \thmtreeref
            {Theorem}{corollary: LDP 2}
        \thmtreenode{-}
            {Lemma}{lemma: choose key parameters, first exit time analysis}{0.8}{}
        \thmtreenode{-}
            {Lemma}{lemma: measure check C J * b, continuity, first exit analysis}{0.8}{
            Verifying  $C(\partial I) = 0$ for the measure $C$ defined in \eqref{def: measure C and scale gamma when applying the exit time framework}
            }
            \begin{thmdependence}
                \thmtreeref{Lemma}{lemma: choose key parameters, first exit time analysis}
            \end{thmdependence}

        \thmtreenode{-}
            {Lemma}{lemma: exit rate strictly positive, first exit analysis}{0.8}{
            Verifying $
                    C^I_b = \widecheck{\mathbf C}^{(\mathcal{J}^I_b)|b}\big( I^\complement\big)\in(0,\infty)
                        $
            for the measure $C$ defined in \eqref{def: measure C and scale gamma when applying the exit time framework}
            }
            \begin{thmdependence}
                \thmtreeref{Lemma}{lemma: choose key parameters, first exit time analysis}
                \thmtreeref{Lemma}{lemma: continuity of h k b mapping clipped}
            \end{thmdependence}
        
        \thmtreenode{-}
            {Lemma}{lemma: limiting measure, with exit location B, first exit analysis}{0.8}{}
            \begin{thmdependence}
                \thmtreeref{Lemma}{lemma: choose key parameters, first exit time analysis}
            \end{thmdependence}
        
    \end{thmdependence}

\thmtreenode{-}
    {Lemma}{lemma: fixed cycle exit or return taking too long, first exit analysis}{0.8}{Apply general framework: Verifying condition \eqref{eq:E3}}
    \begin{thmdependence}
        \thmtreeref
            {Theorem}{corollary: LDP 2}
    \end{thmdependence}

\thmtreenode{-}
    {Lemma}{lemma: cycle, efficient return}{0.8}{Apply general framework: Verifying condition \eqref{eq:E4}}
    \begin{thmdependence}
        \thmtreeref
            {Theorem}{corollary: LDP 2}
    \end{thmdependence}

\bigskip

\bigskip

\end{thmdependence}
\fi

\ifshowtheoremlist
\newpage
\footnotesize
\newgeometry{left=1cm,right=1cm,top=0.5cm,bottom=1.5cm}
\linkdest{location of theorem list}
\listoftheorems
\fi

\ifshowequationlist
\newpage
\linkdest{location of equation number list}
\section*{Numbered Equations}

\eqref{def: X eta b j x, unclipped SGD}
\eqref{def: H, law of Z_j}
\eqref{def: measure nu alpha}
\eqref{def: scaled SGD, LDP}
\eqref{def: X eta b j x, clipped SGD}
\eqref{def: perturb ode mapping h k b, 1}
\eqref{def: perturb ode mapping h k b, 2}
\eqref{def: perturb ode mapping h k b, 3}
\eqref{def: l * tilde jump number for function g, clipped SGD}
\eqref{def: measure mu k b t}
\eqref{def: gradient descent process y}
\eqref{defArrivalTime large jump}
\eqref{defSize large jump}
\eqref{property: large jump time probability}
\eqref{def: Gamma M, set of bounded adapted process}
\eqref{def: event A i concentration of small jumps, 1}
\eqref{def: event A i concentration of small jumps, 2}
\eqref{proof: LDP, small jump perturbation, asymptotics part 1}
\eqref{proof: LDP, small jump perturbation, asymptotics part 2}
\eqref{proof: LDP, small jump perturbation, asymptotics part 3}
\eqref{term E Z 1, lemma LDP, small jump perturbation}
\eqref{proof: LDP, small jump perturbation, choose p}
\eqref{proof: LDP, small jump perturbation, ineq 1}
\eqref{proof: applying berstein ineq, lemma LDP, small jump perturbation}
\eqref{term second order moment hat Z 1, lemma LDP, small jump perturbation}
\eqref{def: E eta delta set, LDP}
\eqref{def: prob measure Q, LDP}
\eqref{proof: lemma weak convergence of cond law of large jump, 1, LDP}
\eqref{proof: choose bar delta, lemma LDP, bar epsilon and delta}
\eqref{bound of difference between xi and xi prime between t_J and t_J+1}
\eqref{goal, lemma: LDP, bar epsilon and delta, clipped version}
\eqref{choice of delta, proof, lemma: continuity of h k b mapping}
\eqref{choice of delta, 2, proof, lemma: continuity of h k b mapping}
\eqref{ineq 1, proof, lemma: continuity of h k b mapping}
\eqref{condition, x prime and x, lemma: continuity of h k b mapping}
\eqref{def: a sigma truncated at M, LDP}
\eqref{def: perturb ode mapping h k b, truncated at M, 1}
\eqref{def: perturb ode mapping h k b, truncated at M, 2}
\eqref{def: perturb ode mapping h k b, truncated at M, 3}
\eqref{ineq, no jump time, a, lemma: SGD close to approximation x circ, LDP}
\eqref{ineq, no jump time, b, lemma: SGD close to approximation x circ, LDP}
\eqref{ineq, with jump time, b, lemma: SGD close to approximation x circ, LDP}
\eqref{proof, ineq gap between X and y, SGD close to approximation x circ, LDP}
\eqref{proof, ineq gap between y and xi, SGD close to approximation x circ, LDP}
\eqref{proof, up to t strictly less than eta tau_1^eta, SGD close to approximation x circ, LDP}
\eqref{goal 1, proposition: standard M convergence, LDP unclipped}
\eqref{goal 3, proposition: standard M convergence, LDP unclipped}
\eqref{goal 4, proposition: standard M convergence, LDP unclipped}
\eqref{goal 5, proposition: standard M convergence, LDP unclipped}
\eqref{subgoal for goal 2, proposition: standard M convergence, LDP unclipped}
\eqref{subgoal, goal 4, proposition: standard M convergence, LDP unclipped}
\eqref{property: choice of M 0, new, proposition: standard M convergence, LDP clipped}
\eqref{property: Y and tilde Y, 1, proposition: standard M convergence, LDP clipped}
\eqref{property: Y and tilde Y, 2, proposition: standard M convergence, LDP clipped}
\eqref{goal new 1, proposition: standard M convergence, LDP clipped}
\eqref{goal new 2, proposition: standard M convergence, LDP clipped}
\eqref{def: objects for definition of hat X leq k}
\eqref{def: time and size of top j jumps before n steps}
\eqref{def: hat X truncated b, j top jumps, 1}
\eqref{def: hat X truncated b, j top jumps, 2}
\eqref{property: equivalence between breve X and hat X, clipped at b}
\eqref{goal: asymptotic equivalence claim, proposition: asymptotic equivalence, clipped}
\eqref{choice of bar delta, proof, proposition: asymptotic equivalence, clipped}
\eqref{choice of bar epsilon, proof, proposition: asymptotic equivalence, clipped}
\eqref{goal, event B 1, clipped, proposition: asymptotic equivalence, clipped}
\eqref{goal, event B 2, clipped, proposition: asymptotic equivalence, clipped}
\eqref{goal, event B 3, clipped, proposition: asymptotic equivalence, clipped}
\eqref{goal, event B 4, clipped, proposition: asymptotic equivalence, clipped}
\eqref{def: x eta b M circ approximation, LDP, 2}
\eqref{def: x eta b M circ approximation, LDP, 3}
\eqref{proof, ineq for mathring x, proposition: asymptotic equivalence, clipped}
\eqref{proof: bounded jump size for breve X, goal, event B 2, clipped, proposition: asymptotic equivalence, clipped}
\eqref{proof: goal 1, lemma: atypical 3, large jumps being too small, LDP clipped}
\eqref{proof: ineq 1, lemma: atypical 3, large jumps being too small, LDP, clipped}
\eqref{choice of bar delta, proposition: uniform weak convergence, clipped}
\smallskip

\noindent
\eqref{def: mapping check g k b, endpoint of path after the last jump, first exit analysis}
\eqref{def: first exit time, J *}
\eqref{proof, observation on xi, lemma: choose key parameters, first exit time analysis}
\eqref{def: t epsilon function, first exit analysis}
\eqref{property: t epsilon function, first exit analysis}
\eqref{proof, goal, lemma: exit rate strictly positive, first exit analysis}
\eqref{def: epsilon relaxed first exit time}
\eqref{def: return time R in first cycle, first exit analysis}
\eqref{proof, ineq 1, theorem: first exit time, unclipped}
\smallskip

\noindent
\eqref{defSDE, initial condition x}
\eqref{def: Y eta b 0 f g, SDE clipped}
\eqref{def, tau and W, discont in Y, eta b, 1}
\eqref{def, tau and W, discont in Y, eta b, 2}
\eqref{def: objects for defining Y eta b, clipped SDE, 1}
\eqref{def: objects for defining Y eta b, clipped SDE, 2}
\eqref{def: objects for defining Y eta b, clipped SDE, 3}
\eqref{def: objects for defining Y eta b, clipped SDE, 4}
\eqref{defSDE, initial condition x, clipped}
\fi

\ifshownavigationpage
\newpage
\normalsize
\tableofcontents

\section*{Navigation Links}

\ifshowtheoremlist
    \noindent
    \hyperlink{location of theorem list}{List of Theorems}
    \bigskip
\fi

\ifshowtheoremtree
    \noindent
    \hyperlink{location of theorem tree}{Theorem Tree}
    \begin{itemize}
    \thmtreeref
        {Proposition}{proposition: standard M convergence, LDP clipped}
    
    \thmtreeref
        {Proposition}{proposition: standard M convergence, LDP unclipped}
        
    \thmtreeref
        {Theorem}{theorem: LDP 1, unclipped}
        
    \thmtreeref
        {Theorem}{theorem: sample path LDP, unclipped}
    
    \thmtreeref
        {Proposition}{proposition: standard M convergence, LDP clipped}
    
    \thmtreeref
        {Theorem}{theorem: LDP 1}
    
    \thmtreeref
        {Theorem}{theorem: first exit time, unclipped}
    
        
    \end{itemize}
\fi

\ifshowreminders
    \noindent
    \hyperlink{location of reminders}{Assumptions, etc}
    \bigskip
\fi

\ifshowtheoremlist
    \noindent
    \hyperlink{location of equation number list}{Numbered Equations}
    \bigskip
\fi

\ifshownotationindex
    \noindent
    \hyperlink{location of notation index}{Notation Index}
    \begin{itemize}
    \item[] 
        \hyperlink{location, notation index A}{A},
        \hyperlink{location, notation index B}{B},
        \hyperlink{location, notation index C}{C},
        \hyperlink{location, notation index D}{D},
        \hyperlink{location, notation index E}{E},
        \hyperlink{location, notation index F}{F},
        \hyperlink{location, notation index G}{G},
        \hyperlink{location, notation index H}{H},
        \hyperlink{location, notation index I}{I},
        \hyperlink{location, notation index J}{J},
        \hyperlink{location, notation index K}{K},
        \hyperlink{location, notation index L}{L},
        \hyperlink{location, notation index M}{M},
        \hyperlink{location, notation index N}{N},
        \hyperlink{location, notation index O}{O},
        \hyperlink{location, notation index P}{P},
        \hyperlink{location, notation index Q}{Q},
        \hyperlink{location, notation index R}{R},
        \hyperlink{location, notation index S}{S},
        \hyperlink{location, notation index T}{T},
        \hyperlink{location, notation index U}{U},
        \hyperlink{location, notation index V}{V},
        \hyperlink{location, notation index W}{W},
        \hyperlink{location, notation index X}{X},
        \hyperlink{location, notation index Y}{Y},
        \hyperlink{location, notation index Z}{Z}
    \end{itemize}
\fi

\fi

\end{document}